\documentclass[12pt,twoside]{report}
\usepackage{amsmath}
\usepackage{amssymb}
\usepackage{amsthm}
\usepackage{amscd}
\usepackage{xypic} 
\usepackage{makeidx}

\newcommand{\discrep}[0]{{\operatorname{discrep}}}
\newcommand{\id}[0]{{\operatorname{id}}}
\newcommand{\Ext}[0]{{\operatorname{Ext}}}
\newcommand{\Hom}[0]{{\operatorname{Hom}}}
\newcommand{\Spec}[0]{{\operatorname{Spec}}}
\newcommand{\GL}[0]{{\operatorname{GL}}}
\newcommand{\Proj}[0]{{\operatorname{Proj}}}
\newcommand{\LCS}[0]{{\operatorname{LCS}}}
\newcommand{\Sym}{{\rm{Sym}}}
\newcommand{\Pic}[0]{{\operatorname{Pic}}}
\newcommand{\xIm}[0]{{\operatorname{Im}}}
\newcommand{\Bs}[0]{{\operatorname{Bs}}}
\newcommand{\mult}[0]{{\operatorname{mult}}}
\newcommand{\Exc}[0]{{\operatorname{Exc}}}
\newcommand{\Supp}[0]{{\operatorname{Supp}}}
\newcommand{\Coker}[0]{{\operatorname{Coker}}}
\newcommand{\Nqklt}[0]{{\operatorname{Nqklt}}}
\newcommand{\Nklt}[0]{{\operatorname{Nklt}}}
\newtheorem{thm}{Theorem}[chapter]
\newtheorem{lem}[thm]{Lemma}
\newtheorem{cor}[thm]{Corollary}
\newtheorem{que}[thm]{Problem}
\newtheorem{prop}[thm]{Proposition}
\newtheorem*{claim}{Claim}
\newtheorem{cla}{Claim}
\newtheorem{conj}[thm]{Conjecture}

\theoremstyle{definition}
\newtheorem{ex}[thm]{Example}
\newtheorem{defn}[thm]{Definition}
\newtheorem{rem}[thm]{Remark}
\newtheorem*{ack}{Acknowledgments}       
         
\newtheorem{say}[thm]{}
\newtheorem{steste}{Step}
\newtheorem{step}{Step}
\newtheorem*{apo}{Apologies}   

\makeindex
\begin{document}
\title{Introduction to the log minimal model program for log canonical 
pairs}
\author{Osamu Fujino}
\date{January 8, 2009, Version 6.01}
\maketitle

\begin{abstract}
We describe the foundation of the log minimal model 
program for log canonical pairs according to 
Ambro's idea. 
We generalize Koll\'ar's vanishing and torsion-free 
theorems for embedded simple normal crossing 
pairs. Then we prove the cone and contraction 
theorems for quasi-log varieties, especially, 
for log canonical pairs. 
\end{abstract}

\tableofcontents

\chapter{Introduction}\label{chap1}

In this book, we describe the foundation of the log minimal 
model program (LMMP or MMP, for short) for log canonical 
pairs. We follow Ambro's idea in \cite{ambro}. 
First, we generalize Koll\'ar's vanishing and torsion-free theorems (cf.~\cite{kollar-high}) 
for embedded normal crossing pairs. 
Next, we introduce the notion of quasi-log varieties. 
The key points of the theory of quasi-log varieties are 
adjunction and the vanishing theorem, which 
directly follow from Koll\'ar's vanishing and 
torsion-free theorems for embedded normal crossing 
pairs. 
Finally, we prove the cone and contraction theorems for 
quasi-log varieties. 
The proofs are more or less routine works for experts once 
we know adjunction and the vanishing theorem 
for quasi-log varieties. 
Chapter \ref{chap2} is an expanded version of my preprint \cite{fujino} 
and Chapter \ref{chap3} is based on the preprint \cite{fuji-note}. 

After \cite{km} appeared, the log minimal model program 
has developed drastically. Shokurov's 
epoch-making paper \cite{sho-pre} gave us 
various new ideas. 
The book \cite{corti} explains some of them in details. 
Now, we have \cite{bchm}, where 
the log minimal model program for Kawamata log terminal pairs 
is established on some mild assumptions. 
In this book, we explain nothing on the results in \cite{bchm}. 
It is because many survey articles were and will be 
written for \cite{bchm}. 
See, for example, \cite{chklm}, 
\cite{druel}, and \cite{fuji-ronsetsu}. 
Here, we concentrate basics of the log minimal model program for 
log canonical pairs. 

We do not discuss 
the log minimal model program for toric varieties. 
It is because we have already established the 
foundation of the toric Mori theory. 
We recommend the reader to 
see \cite{reid-toric}, \cite[Chapter 14]{ma-bon}, 
\cite{fuji-sato}, \cite{fuji-to2}, 
and so on. 
Note that we will freely use the toric geometry 
to construct nontrivial 
examples explicitly. 

The main ingredient of this book is 
the theory of mixed Hodge structures. 
All the basic results for Kawamata log terminal pairs 
can be proved without it. 
I think that the classical Hodge 
theory and the theory of 
variation of Hodge structures 
are sufficient for Kawamata log terminal pairs. 
For log canonical 
pairs, the theory of mixed 
Hodge structures seems to be indispensable. 
In this book, we do not discuss the theory of variation of Hodge structures 
nor canonical bundle formulas. 

\begin{apo}
After I finished writing a preliminary version of 
this book, I found a more direct approach 
to the log minimal model program for log canonical 
pairs. 
In \cite{fuji-b}, I obtained a correct generalization 
of Shokurov's non-vanishing theorem for log canonical paris. 
It directly implies the base point free theorem for log 
canonical pairs. 
I also proved the rationality theorem and the cone theorem 
for log canonical pairs without using 
the framework of quasi-log varieties. 
The vanishing and torsion-free theorems we need in \cite{fuji-b} 
are essentially contained in \cite{ev}. 
The reader can learn them by \cite{fuji-a}, 
where I gave a short, easy, and almost self-contained 
proof to them. 
Therefore, now we can prove some of the 
results in this book in a more elementary manner. 
However, the method developed in \cite{fuji-b} can 
be applied only to log canonical pairs. 
So, \cite{fuji-b} will not decrease the value of this book. 
Instead, \cite{fuji-b} will complement 
the theory of quasi-log varieties. 
I am sorry that I do not discuss that new approach here. 
\end{apo}

\begin{ack}
First, I express my gratitude to Professors Shigefumi 
Mori, Yoichi Miyaoka, Noboru Nakayama, 
Daisuke Matsushita, and Hiraku Kawanoue, who 
were the members of my seminars when I 
was a graduate student at RIMS. 
In those seminars, I learned the foundation of the log 
minimal model program according to 
a draft of \cite{km}. 
I was partially supported by the Grant-in-Aid for Young 
Scientists (A) $\sharp$20684001 from JSPS. I was 
also supported by the Inamori Foundation. 
I thank Professors Noboru Nakayama, 
Hiromichi Takagi, Florin Ambro, Hiroshi 
Sato, Takeshi Abe and Masayuki Kawakita for 
discussions, comments, and questions. 
I would like to thank Professor 
J\'anos Koll\'ar for giving me
many comments on the preliminary version of 
this book and 
showing me many examples. 
I also thank Natsuo Saito for drawing a beautiful picture 
of a Kleiman--Mori cone. 
Finally, I thank Professors Shigefumi Mori, 
Shigeyuki Kondo, 
Takeshi Abe, and 
Yukari Ito for warm encouragement. 
\end{ack}

\section{What is a quasi-log variety ?}

In this section, we informally explain why it is 
natural to consider {\em{quasi-log varieties}}. 

Let $(Z, B_Z)$ be a log canonical 
pair and let $f:V\to Z$ be a resolution with 
$$K_V+S+B=f^*(K_Z+B_Z),$$ 
where $\Supp (S+B)$ is a simple normal crossing 
divisor, $S$ is reduced, and $\llcorner B\lrcorner\leq 0$. 
It is very important to consider the 
{\em{locus 
of log canonical singularities}} $W$ of the 
pair $(Z, B_Z)$, that is, $W=f(S)$. 
By the Kawamata--Viehweg vanishing theorem, 
we can easily check that 
$$\mathcal O_W\simeq f_*\mathcal O_S(\ulcorner -B_S\urcorner), 
$$ 
where $K_S+B_S=(K_V+S+B)|_S$. 
In our case, $B_S=B|_S$. 
Therefore, it is natural to introduce the following notion. 
Precisely speaking, a qlc pair is a quasi-log pair with 
only qlc singularities (see Definition \ref{quasi-log}). 

\begin{defn}[Qlc pairs]\index{qlc pair} 
A qlc pair $[X, \omega]$ is a scheme $X$ endowed with an 
$\mathbb R$-Cartier $\mathbb R$-divisor 
$\omega$ such that 
there is a proper morphism $f:(Y, B_Y)\to X$ satisfying 
the following conditions. 
\begin{itemize}
\item[(1)] $Y$ is a simple normal crossing 
divisor on a smooth variety 
$M$ and there exists an 
$\mathbb R$-divisor $D$ on $M$ such that $\Supp (D+Y)$ is 
a simple normal crossing divisor, 
$Y$ and $D$ have no common irreducible components,  
and $B_Y=D|_Y$. 
\item[(2)] $f^*\omega\sim _{\mathbb R}K_Y+B_Y$. 
\item[(3)] $B_Y$ is a subboundary, that is, $b_i\leq 1$ for 
any $i$ when $B_Y=\sum b_i B_i$. 
\item[(4)] $\mathcal O_X\simeq f_*\mathcal O_Y(\ulcorner 
-(B^{<1}_Y)\urcorner)$, 
where $B^{<1}_Y=\sum _{b_i<1}b_i B_i$. 
\end{itemize}
\end{defn}
It is easy to 
see that the pair $[W, \omega]$, where 
$\omega=(K_X+B)|_W$, with 
$f:(S, B_S)\to W$ satisfies the definition of qlc pairs. 
We note that the pair $[Z, K_Z+B_Z]$ with $f:(V, S+B)\to 
Z$ is also a qlc pair since $f_*\mathcal O_V(\ulcorner -B\urcorner)
\simeq \mathcal O_Z$. 
Therefore, we can treat log canonical pairs and 
loci of log canonical singularities in the 
same framework once we introduce the notion of 
qlc pairs. 
Ambro found that a modified version of 
\index{X-method} X-method, 
that is, 
the method introduced by Kawamata and 
used by him to prove the foundational results 
of the log minimal model program 
for Kawamata log terminal pairs, works for qlc pairs 
if we 
generalize Koll\'ar's vanishing and torsion-free theorems 
for embedded normal crossing pairs. 
It is the key idea of \cite{ambro}. 

\section{A sample computation}

The following theorem must motivate the 
reader to study our new framework. 

\begin{thm}[{cf.~Theorem \ref{adj-th} (ii)}]\label{moti1} 
Let $X$ be a normal projective 
variety and $B$ a boundary $\mathbb R$-divisor 
on $X$ such that 
$(X, B)$ is log canonical. 
Let $L$ be a Cartier divisor on $X$. 
Assume that $L-(K_X+B)$ is ample. 
Let $\{C_i\}$ be {\em{any}} set of 
lc centers of the pair $(X, B)$. 
We put $W=\bigcup C_i$ with a reduced scheme structure. 
Then we have 
$$
H^i(X, \mathcal I_W\otimes \mathcal O_X(L))=0 
$$ 
for any $i>0$, where $\mathcal I_W$ is the defining 
ideal sheaf of $W$ on $X$. 
In particular, the restriction 
map 
$$
H^0(X, \mathcal O_X(L))\to H^0(W, \mathcal O_W(L))
$$ 
is surjective. 
Therefore, if $(X, B)$ has a zero-dimensional 
lc center, then the linear system $|L|$ is not empty 
and the base locus of $|L|$ contains no zero-dimensional 
lc centers of $(X, B)$. 
\end{thm}

Let us see a simple setting to understand 
the difference between our new framework and 
the traditional one. 

\begin{say}
Let $X$ be a smooth projective surface and 
let $C_1$ and $C_2$ be smooth curves on $X$. 
Assume that 
$C_1$ and $C_2$ intersect only at a point $P$ 
transversally. 
Let $L$ be a Cartier divisor on $X$ such that 
$L-(K_X+B)$ is ample, where $B=C_1+C_2$. 
It is obvious that $(X, B)$ is log canonical 
and $P$ is an lc center of $(X, B)$. 
Then, by Theorem \ref{moti1}, 
we can directly obtain 
$$
H^i(X, \mathcal I_P\otimes \mathcal O_X(L))=0
$$ 
for any $i>0$, where $\mathcal I_P$ is the 
defining ideal sheaf of $P$ on $X$. 

In the classical framework, we prove it as follows. 
Let $C$ be a general curve passing through 
$P$. 
We take small positive rational numbers $\varepsilon$ 
and $\delta$ such that 
$(X, (1-\varepsilon)B+\delta C)$ is log canonical 
at $P$ and is Kawamata log terminal outside $P$. Since 
$\varepsilon$ and $\delta$ are small, 
$L-(K_X+(1-\varepsilon)B+\delta C)$ is still ample. 
By the Nadel 
vanishing\index{Nadel vanishing theorem} 
theorem, 
we obtain 
$$
H^i(X, \mathcal I_P\otimes \mathcal O_X(L))=0 
$$ 
for any $i>0$. We note that 
$\mathcal I_P$ is nothing but the {\em{multiplier 
ideal sheaf}}\index{multiplier ideal sheaf} 
associated to the pair 
$(X, (1-\varepsilon)B+\delta C)$. 
\end{say}

By our new vanishing theorem, 
the reader will be released from annoyance of 
perturbing coefficients of boundary divisors.  

We give a sample computation here. It may explain 
the reason why Koll\'ar's torsion-free and 
vanishing theorems appear 
in the study of log canonical 
pairs. 
The actual proof of Theorem \ref{moti1} 
depends on much more sophisticated 
arguments on the theory of mixed Hodge structures. 

\begin{ex}
Let $S$ be a normal projective surface which 
has only one simple elliptic Gorenstein 
singularity $Q\in S$. 
We put $X=S\times \mathbb P^1$ and $B=S\times \{0\}$. Then 
the pair $(X, B)$ is log canonical. 
It is easy to see that 
$P=(Q, 0)\in X$ is an lc 
center of $(X, B)$. 
Let $L$ be a Cartier divisor on $X$ such that 
$L-(K_X+B)$ 
is ample. 
We have 
$$
H^i(X, \mathcal I_P\otimes \mathcal O_X(L))=0 
$$ 
for any $i>0$, 
where $\mathcal I_P$ is the defining ideal 
sheaf of $P$ on $X$. We note that $X$ is not Kawamata log terminal and 
that $P$ is not an isolated lc center of 
$(X, B)$. 
\end{ex}
\begin{proof}
Let $\varphi: T\to S$ be the minimal 
resolution. 
Then we can write $K_T+C=\varphi^*K_S$, where 
$C$ is the 
$\varphi$-exceptional 
elliptic curve on $T$. We put $Y=T\times \mathbb P^1$ and 
$f=\varphi \times \id _{\mathbb P^1}: 
Y\to X$, where $\id _{\mathbb P^1}: \mathbb P^1\to \mathbb P^1$ 
is the identity. 
Then $f$ is a resolution of $X$ and we can write 
$$K_Y+B_Y+E=f^*(K_X+B),$$ where 
$B_Y$ is the strict transform 
of $B$ on $Y$ and $E\simeq C\times \mathbb P^1$ is the 
exceptional divisor of $f$. 
Let $g:Z\to Y$ be the blow-up along $E\cap B_Y$. 
Then we can write 
$$K_Z+B_Z+E_Z+F=g^*(K_Y+B_Y+E)=h^*(K_X+B),$$ 
where $h=f\circ g$, $B_Z$ (resp.~$E_Z$) is the strict 
transform of $B_Y$ (resp.~$E$) on $Z$, 
and $F$ is the $g$-exceptional 
divisor. 
We note that $$\mathcal I_P\simeq h_*\mathcal O_Z(-F)\subset 
h_*\mathcal O_Z\simeq \mathcal O_X.$$ Since 
$-F=K_Z+B_Z+E_Z-h^*(K_X+B)$, 
we have 
$$
\mathcal I_P\otimes \mathcal O_X(L)\simeq 
h_*\mathcal O_Z(K_Z+B_Z+E_Z)\otimes 
\mathcal O_X(L-(K_X+B)). 
$$ 
So, it is sufficient to prove that 
$$
H^i(X, h_*\mathcal O_Z(K_Z+B_Z+E_Z)\otimes \mathcal L)=0 
$$ 
for any $i>0$ and any ample line bundle 
$\mathcal L$ on $X$. 
We consider the short exact sequence 
$$
0\to \mathcal O_Z(K_Z)\to \mathcal O_Z(K_Z+E_Z)\to \mathcal 
O_{E_Z}(K_{E_Z})\to 0. 
$$ 
We can easily check that 
$$
0\to h_*\mathcal O_Z(K_Z)\to h_*\mathcal O_Z(K_Z+E_Z)\to h_*\mathcal 
O_{E_Z}(K_{E_Z})\to 0 
$$ 
is exact and 
$$
R^ih_*\mathcal O_Z(K_Z+E_Z)\simeq R^ih_*\mathcal O_{E_Z}(K_{E_Z}) 
$$ 
for any $i>0$ by the Grauert--Riemenschneider 
vanishing theorem. 
We can directly check that 
$$
R^1h_*\mathcal O_{E_Z}(K_{E_Z})\simeq R^1f_*\mathcal O_E(K_E)
\simeq \mathcal O_D(K_D), 
$$ 
where $D=Q\times \mathbb P^1\subset X$. 
Therefore, $R^1h_*\mathcal O_Z(K_Z+E_Z)\simeq 
\mathcal O_D(K_D)$ is a torsion 
sheaf on $X$. 
However, it is torsion-free as a sheaf on $D$. It is a 
generalization of Koll\'ar's torsion-free theorem. 
We consider 
$$
0\to \mathcal O_Z(K_Z+E_Z)\to 
\mathcal O_Z(K_Z+B_Z+E_Z)\to \mathcal 
O_{B_Z}(K_{B_Z})\to 0.
$$ 
We note that $B_Z\cap E_Z=\emptyset$. Thus, we have 
\begin{align*}
0&\to h_*\mathcal O_Z(K_Z+E_Z)\to 
h_*\mathcal O_Z(K_Z+B_Z+E_Z)
\to h_*\mathcal 
O_{B_Z}(K_{B_Z})\\ 
&\overset{\delta}\to R^1h_*\mathcal O_Z(K_Z+E_Z)\to \cdots.
\end{align*}
Since $\Supp h_*\mathcal O_{B_Z}(K_{B_Z})=B$, 
$\delta$ is a zero map by 
$R^1h_*\mathcal O_Z(K_Z+B_Z)\simeq 
\mathcal O_D(K_D)$. Therefore, we know that 
the following sequence 
$$
0\to h_*\mathcal O_Z(K_Z+E_Z)\to 
h_*\mathcal O_Z(K_Z+B_Z+E_Z)
\to h_*\mathcal 
O_{B_Z}(K_{B_Z})\to 0 
$$ 
is exact. 
By Koll\'ar's vanishing theorem on $B_Z$, it is sufficient 
to prove that 
$H^i(X, h_*\mathcal O_Z(K_Z+E_Z)\otimes \mathcal L)=0$ for 
any $i>0$ and any ample line bundle $\mathcal L$. 
We have 
$$H^i(X, h_*\mathcal O_Z(K_Z)\otimes 
\mathcal L)=H^i(X, h_*\mathcal O_{E_Z}(K_{E_Z})\otimes 
\mathcal L)=0$$ for any $i>0$ by Koll\'ar's 
vanishing theorem. 
By the following exact sequence 
\begin{align*}
\cdots&\to H^i(X, h_*\mathcal O_Z(K_Z)\otimes 
\mathcal L)\to H^i(X, h_*\mathcal O_Z(K_Z+E_Z))
\\ &\to 
H^i(X, h_*\mathcal O_{E_Z}(K_{E_Z}))\to \cdots, 
\end{align*}
we obtain the desired vanishing theorem. 
Anyway, we have 
$$
H^i(X, \mathcal I_P\otimes \mathcal O_X(L))=0 
$$ 
for any $i>0$. 
\end{proof}

\section{Overview} 

We summarize the contents of this book. 

In the rest of Chapter \ref{chap1}, 
we collect some preliminary results and 
notations. 
Moreover, we quickly review the classical 
log minimal model 
program. 

In Chapter \ref{chap2}, we discuss Ambro's generalizations of 
Koll\'ar's injectivity, vanishing, and torsion-free theorems for 
embedded normal crossing pairs. 
These results are indispensable for the theory of 
quasi-log varieties. To prove them, we recall 
some results on the mixed Hodge structures. 
For the details of Chapter \ref{chap2}, 
see Section \ref{21-sec}, which is 
the introduction of Chapter \ref{chap2}. 

In Chapter \ref{chap3}, we treat the log minimal model program for 
log canonical pairs. 
In Section \ref{31-sec}, we explicitly state the 
cone and contraction theorems for log canonical 
pairs and prove the log flip conjecture I for 
log canonical pairs in dimension four. 
We also discuss the length of extremal rays for log canonical 
pairs with the aid of the recent result by \cite{bchm}. 
Subsection \ref{313ss} contains 
Koll\'ar's various examples. 
We prove that a log canonical flop does not always 
exist.  
In Section \ref{qlog-sec}, we introduce the notion of 
quasi-log varieties and prove basic results, for example, 
adjunction and the vanishing theorem, for quasi-log varieties. 
Section \ref{33-sec} is devoted to the proofs of the 
fundamental theorems for quasi-log varieties. 
First, we prove the base point free theorem for quasi-log 
varieties. Then, we prove the rationality theorem and 
the cone theorem for quasi-log varieties. 
Once we understand the notion of quasi-log varieties and 
how to use adjunction and the vanishing theorem, 
there are no difficulties to prove the above fundamental theorems. 

In Chapter \ref{chap4}, 
we discuss some supplementary results. 
Section \ref{34-sec} is devoted to the 
proof of the base point free theorem 
of Reid--Fukuda type for quasi-log varieties with only qlc 
singularities. In Section \ref{b-dlt-sec}, 
we prove that the non-klt locus of a dlt pair is Cohen--Macaulay as 
an application of the vanishing theorem in Chapter \ref{chap2}. 
Section \ref{sec-alex} is a detailed description of 
Alexeev's criterion for Serre's $S_3$ condition. 
It is an application of the generalized torsion-free theorem. 
In Section \ref{to-sec}, we recall the notion of toric polyhedra. 
We can easily check that a toric polyhedron has a natural 
quasi-log structure. 
Section \ref{43-sec} is 
a short survey of the 
theory of non-lc ideal sheaves. 
In the finial section, we mention effective 
base point free theorems for log canonical pairs.  

In the final chapter:~Chapter \ref{chap5}, 
we collect various examples of toric flips. 

\section{How to read this book ?}

We assume that the reader is familiar with 
the classical log minimal model program, at 
the level of Chapters 2 and 3 in \cite{km}. 
It is not a good idea to read this book 
without studying the classical results 
discussed in \cite{km}, \cite{kmm}, or 
\cite{ma-bon}. 
We will quickly review the classical 
log minimal model program 
in Section \ref{se-quiqui} for the 
reader's convenience. If the reader understands 
\cite[Chapters 2 and 3]{km}, then it is 
not difficult to read \cite{fuj-lec}, which 
is a gentle introduction 
to the log minimal model program for lc pairs and 
written in the same style as \cite{km}. 
After these preparations, the reader can read Chapter \ref{chap3} in this book without any difficulties. 
We note that Chapter \ref{chap3} can be read 
before Chapter \ref{chap2}. 
The hardest part of this book is Chapter \ref{chap2}. 
It is very technical. So, the reader 
should have strong motivations before 
attacking Chapter \ref{chap2}. 

\section{Notation and Preliminaries} 

We will work over the complex number field $\mathbb C$ throughout 
this book. But we note that by using Lefschetz principle, we can 
extend almost everything to the case where the base 
field is an algebraically closed field of characteristic 
zero. Note that 
every scheme in this book 
is assumed to be separated. 
We deal not only with the usual divisors but 
also with the divisors with rational and real coefficients, which 
turn out to 
be fruitful and natural. 

\begin{say}[Divisors, $\mathbb Q$-divisors, and $\mathbb R$-divisors]
For an $\mathbb R$-Weil divisor 
$D=\sum _{j=1}^r d_j D_j$ such that 
$D_i\ne D_j$ for $i\ne j$, we define 
the {\em{round-up}}\index{round-up} 
$\ulcorner D\urcorner =\sum _{j=1}^{r} 
\ulcorner d_j\urcorner D_j$ 
(resp.~the {\em{round-down}}\index{round-down} 
$\llcorner D\lrcorner 
=\sum _{j=1}^{r} \llcorner d_j \lrcorner D_j$), 
where for any real number $x$, 
$\ulcorner x\urcorner$ (resp.~$\llcorner x\lrcorner$) is 
the integer defined by $x\leq \ulcorner x\urcorner <x+1$ 
(resp.~$x-1<\llcorner x\lrcorner \leq x$).\index{$\llcorner 
D\lrcorner$, the round-down of $D$}
\index{$\ulcorner D\urcorner$, the round-up of $D$}
\index{$\{D\}$, the fractional part of $D$} 
The {\em{fractional part}}\index{fractional part} $\{D\}$ 
of $D$ denotes $D-\llcorner D\lrcorner$. 
We define\index{$D^{=1}$}\index{$D^{\leq 1}$}
\index{$D^{<1}$}\index{$D^{>1}$} 
\begin{align*}&
D^{=1}=\sum _{d_j=1}D_j, \ \ D^{\leq 1}=\sum_{d_j\leq 1}d_j D_j, \\ &
D^{<1}=\sum_{d_j< 1}d_j D_j, \ \ \text{and}\ \ \ 
D^{>1}=\sum_{d_j>1}d_j D_j. 
\end{align*}
The {\em{support}}\index{support}
\index{$\Supp D$, the support of 
$D$} of $D=\sum _{j=1}^{r}d_jD_j$, 
denoted by $\Supp D$, is 
the subscheme $\bigcup_{d_j\ne 0}D_j$. 
We call $D$ a {\em{boundary}} (resp.~{\em{subboundary}}) 
$\mathbb R$-divisor\index{boundary 
$\mathbb R$-divisor}\index{subboundary 
$\mathbb R$-divisor} if 
$0\leq d_j\leq 1$ (resp.~$d_j\leq 1$) for any $j$. 
{\em{$\mathbb Q$-linear equivalence}} (resp.~{\em{$\mathbb R$-linear 
equivalence}})\index{$\sim_{\mathbb Q}$, 
$\mathbb Q$-linear equivalence}\index{$\sim_{\mathbb R}$, 
$\mathbb R$-linear equivalence} 
of two $\mathbb Q$-divisors (resp.~{\em{$\mathbb R$-divisors}}) 
$B_1$ and $B_2$ is denoted by $B_1\sim _{\mathbb Q}B_2$ (resp.~$B_1
\sim_{\mathbb R}B_2$). 
Let $f:X\to Y$ be a morphism and $B_1$ and $B_2$ two 
$\mathbb R$-divisors on $X$. 
We say that they are {\em{linearly $f$-equivalent}} (denoted 
by $B_1\sim _f B_2$)\index{$\sim_f$, linear 
$f$-equivalence} if and 
only if there is a Cartier divisor $B$ on $Y$ such that 
$B_1\sim B_2+f^*B$. 
We can define {\em{$\mathbb Q$-linear}} 
(resp.~{\em{$\mathbb R$-linear}}) {\em{$f$-equivalence}} 
(denoted by $B_1\sim _{\mathbb Q, f}B_2$ (resp.~$B_1\sim 
_{\mathbb R, f}B_2$))\index{$\sim_{\mathbb Q, f}$, 
$\mathbb Q$-linear 
$f$-equivalence}\index{$\sim_{\mathbb R, f}$, $\mathbb R$-linear 
$f$-equivalence} similarly. 

Let $X$ be a normal variety. 
Then $X$ is called {\em{$\mathbb Q$-factorial}} 
if every $\mathbb Q$-divisor is $\mathbb Q$-Cartier. \index{$\mathbb Q$-factorial} 
\end{say}

We quickly review the notion of singularities of pairs. 
For the details, see \cite[\S 2.3]{km}, \cite{ko-sing}, 
and \cite{fujino0}. 
See also the subsection \ref{161-ss}. 
 
\begin{say}[Singularities of pairs]
For a proper birational morphism $f:X\to Y$, 
the {\em{exceptional locus}} $\Exc (f)\subset X$ is the locus where 
$f$ is not an isomorphism.\index{exceptional locus}
\index{$\Exc(f)$, the exceptional locus of $f$} 
Let $X$ be a normal variety and let $B$ be an 
$\mathbb R$-divisor 
on $X$ such that $K_X+B$ is $\mathbb R$-Cartier. 
Let $f:Y\to X$ be a resolution such that 
$\Exc(f)\cup f^{-1}_*B$ has a simple normal crossing support, 
where $f^{-1}_*B$ is the strict transform of $B$ on $Y$. 
We write $K_Y=f^*(K_X+B)+\sum _i a_i E_i$ 
and $a(E_i, X, B)=a_i$. 
We say that $(X, B)$ is {\em{sub log canonical}} (resp.~{\em{sub Kawamata log terminal}}) 
({\em{sub lc}} (resp.~{\em{sub klt}}), for short) 
if and only if $a_i\geq -1$ (resp.~$a_i>-1$) for any $i$. 
If $(X, B)$ is sub lc (resp.~sub klt) and $B$ is 
effective, then $(X, B)$ is called 
{\em{log canonical}}\index{log canonical} 
(resp.~{\em{Kawamata log terminal}}\index{Kawamata log terminal}) 
(lc (resp.~klt), for short). 
Note that the {\em{discrepancy}}
\index{discrepancy} $a(E, X, B)\in \mathbb R$ can be defined 
for any prime divisor $E$ {\em{over}} $X$. 
Let $(X, B)$ be a sub lc pair. 
If $E$ is a prime divisor over $X$ such that $a(E, X, B)=-1$, then 
the center $c_X(E)$ is called an {\em{lc center}}\index{lc center} of 
$(X, B)$. 

\begin{defn}[Divisorial log terminal pairs]\index{divisorial 
log terminal}\label{dlt-def} 
Let $X$ be a normal variety and $B$ a boundary 
$\mathbb R$-divisor 
such that $K_X+B$ is $\mathbb R$-Cartier. 
If there exists a resolution $f:Y\to X$ such that 
\begin{itemize}
\item[(i)] both $\Exc (f)$ and $\Exc (f)\cup 
\Supp (f^{-1}_*B)$ are simple normal crossing divisors on $Y$, and 
\item[(ii)] $a(E, X, B)>-1$ for every exceptional divisor 
$E\subset Y$, 
\end{itemize}
then $(X, B)$ is called {\em{divisorial 
log terminal}} ({\em{dlt}}, for short). 
\end{defn}
For the details of dlt pairs, see Section \ref{b-dlt-sec}. 
The assumption that 
$\Exc(f)$ is a divisor in Definition \ref{dlt-def} (i) 
is very important. 
See Example \ref{416ex} below. 
\end{say}

We often use resolution of singularities. 
We need the following strong statement. 
We sometimes call it Szab\'o's resolution lemma (see \cite{szabo} 
and \cite{fujino0}).  

\begin{say}[Resolution lemma]\index{resolution lemma}\label{15-resol}
Let $X$ be a smooth variety and $D$ a reduced divisor on $X$. 
Then there exists a proper birational morphism 
$f:Y\to X$ with the following properties: 
\begin{itemize}
\item[(1)] $f$ is a composition of blow-ups of smooth 
subvarieties, 
\item[(2)] $Y$ is smooth, 
\item[(3)] $f^{-1}_*D\cup \Exc (f)$ is a simple normal crossing 
divisor, where $f^{-1}_*D$ is the strict transform of $D$ on $Y$, and 
\item[(4)] $f$ is an isomorphism over $U$, where $U$ is the largest 
open set of $X$ such that the restriction $D|_{U}$ is a simple normal 
crossing divisor on $U$. 
\end{itemize} 
Note that $f$ is projective and the exceptional locus $\Exc (f)$ is of 
pure codimension one in $Y$ since $f$ is a composition 
of blowing-ups. 
\end{say}

The Kleiman--Mori cone is the basic object to study in 
the log minimal model program. 

\begin{say}[Kleiman--Mori cone]\index{Kleiman--Mori cone}
Let $X$ be an algebraic scheme over $\mathbb C$ and let 
$\pi:X\to S$ be a proper morphism 
to an algebraic scheme $S$. 
Let $\Pic (X)$ be the group of line bundles on $X$. 
Take a complete curve on $X$ which 
is mapped to a point by $\pi$. For 
$\mathcal L\in \Pic (X)$, we 
define the intersection number 
$\mathcal L\cdot C=\deg _{\overline C}f^*\mathcal L$, 
where $f:\overline C\to C$ is the normalization of $C$. 
Via this intersection pairing, we introduce 
a bilinear form 
$$
\cdot: \Pic (X)\times Z_1(X/S)\to \mathbb Z, 
$$
where $Z_1(X/S)$ is the free abelian group 
generated by integral curves which are mapped to 
points on $S$ by $\pi$. 

Now we have the notion of numerical equivalence both 
in $Z_1(X/S)$ and in $\Pic(X)$, which 
is denoted by $\equiv$\index{$\equiv$, numerical equivalence}, 
and 
we obtain a perfect pairing 
$$
N^1(X/S)\times N_1(X/S)\to \mathbb R,  
$$
where 
$$N^1(X/S)=\{\Pic (X)/\equiv\}\otimes \mathbb R \ \ \  \text{and} 
\ \ \ N_1(X/S)=\{Z_1(X/S)/\equiv\}\otimes \mathbb R, 
$$ 
namely $N^1(X/S)$ and $N_1(X/S)$ are dual to 
each other through this intersection pairing. 
It is well known that $\dim _{\mathbb R}N^1(X/S)=
\dim _{\mathbb R}N_1(X/S)<\infty$. 
We write $\rho (X/S)=\dim _{\mathbb R}N^1(X/S)=
\dim _{\mathbb R}N_1(X/S)$.\index{$\rho(X/S)$} 
We define the Kleiman--Mori cone $\overline {NE}(X/S)$ 
as the closed convex cone 
in $N_1(X/S)$ generated by integral curves 
on $X$ which are mapped to points on $S$ by $\pi$. 
When $S=\Spec \mathbb C$, we drop $/\Spec \mathbb C$ from 
the notation, e.g., we simply write $N_1(X)$ in stead 
of $N_1(X/\Spec \mathbb C)$. 
\end{say}

\begin{defn}
An element $D\in N^1(X/S)$ is called {\em{$\pi$-nef}} 
(or {\em{relatively nef for $\pi$}}), if $D\geq 0$ 
on $\overline {NE}(X/S)$. When $S=\Spec \mathbb C$, 
we simply say that $D$ is {\em{nef}}.\index{nef}
\index{relatively nef} 
\end{defn}

\begin{thm}[Kleiman's criterion for ampleness]
\index{Kleiman's criterion}\label{klei-th}  
Let $\pi:X\to S$ be a {\em{projective}} morphism 
between algebraic schemes. 
Then $\mathcal L\in \Pic (X)$ is $\pi$-ample 
if and only if the numerical class of $\mathcal L$ in $N^1(X/S)$ gives a positive function on $\overline {NE}(X/S)\setminus 
\{0\}$. 
\end{thm}

In Theorem \ref{klei-th}, we have to assume that 
$\pi:X\to S$ is {\em{projective}} 
since there are complete non-projective 
algebraic 
varieties for which Kleiman's criterion does not hold. 
We recall the explicit example given in \cite{fujino-km} for 
the reader's convenience. 
For the details of this example, see 
\cite[Section 3]{fujino-km}. 

\begin{ex}[{cf.~\cite[Section 3]{fujino-km}}] 
We fix a lattice $N= \mathbb Z^3$. 
We take lattice points  
\begin{align*}
v_1  &= (1,0,1), & v_2&=(0,1,1), & v_3&=(-1,-1,1),\\
v_4  &= (1,0,-1), & v_5&=(0,1,-1), & v_6&=(-1,-1,-1).
\end{align*} 
We consider the following fan 
$$
\Delta=
\left \{
\begin{array}{ccc}
\langle v_1, v_2, v_4\rangle, &
\langle v_2, v_4, v_5\rangle, &
\langle v_2, v_3, v_5, v_6\rangle, \\
\langle v_1, v_3, v_4, v_6\rangle, &
\langle v_1, v_2, v_3\rangle, &
\langle v_4, v_5, v_6\rangle, \\ 
\text{and their faces}& &
\end{array}
\right \}. $$ 
Then the toric variety $X=X(\Delta)$ has the 
following properties. 
\begin{itemize}
\item[(i)] $X$ is a non-projective 
complete toric variety with $\rho(X)=1$.  
\item[(ii)] There exists a Cartier divisor $D$ on $X$ such 
that $D$ is positive on $\overline {NE}(X)\setminus \{0\}$. 
In particular, $\overline{NE}(X)$ is a half line. 
\end{itemize}
Therefore, Kleiman's criterion for ampleness 
does not hold for this $X$. 
We note that $X$ is not $\mathbb Q$-factorial 
and that there is a torus invariant curve 
$C\simeq \mathbb P^1$ on $X$ such 
that $C$ is numerically equivalent to zero. 
\end{ex}

If $X$ has only mild singularities, 
for example, $X$ is $\mathbb Q$-factorial, 
then it is known that Theorem \ref{klei-th} holds even when 
$\pi:X\to S$ is {\em{proper}}. 
However, the Kleiman--Mori cone may not have 
enough informations when $\pi$ is only proper. 

\begin{ex}[{cf.~\cite{fuji-pa}}] 
There exists a smooth complete 
toric threefold $X$ such that 
$\overline {NE}(X)=N_1(X)$. 
\end{ex}

The description below helps the reader understand 
examples in \cite{fuji-pa}. 

\begin{ex}\label{fp-example}
Let $\Delta$ be the fan in $\mathbb R^3$ whose 
rays are generated by $v_1=(1,0,0), v_2=(0,1,0), 
v_3=(0,0,1), v_5=(-1,0,-1), 
v_6=(-2,-1,0)$ 
and whose maximal cones are 
$$\langle v_1, v_2, v_3\rangle, 
\langle v_1, v_3, v_6\rangle, 
\langle v_1, v_2, v_5\rangle, 
\langle v_1, v_5,v_6\rangle, 
\langle v_2, v_3, v_5\rangle, 
\langle v_3, v_5, v_6\rangle. 
$$ 
Then the associated toric variety 
$X_1=X(\Delta)$ is $\mathbb P_{\mathbb P^1}(\mathcal 
O_{\mathbb P^1}\oplus \mathcal O_{\mathbb P^1}(2)
\oplus \mathcal O_{\mathbb P^1}(2))$. 
We take a sequence of 
blow-ups 
$$
Y\overset{f_3}\longrightarrow X_3\overset{f_2}
\longrightarrow X_2
\overset{f_1}\longrightarrow X_1, 
$$
where $f_1$ is the blow-up along the ray 
$v_4=(0, -1, -1)=3v_1+v_5+v_6$, $f_2$ is along 
$$v_7=(-1, -1, -1)=\frac{1}{3}(2v_4+v_5+v_6),$$ and 
the final blow-up $f_3$ is along the ray 
$$v_8=(-2, -1, -1)=\frac{1}{2}(v_5+v_6+v_7).$$ 
Then we can directly check that 
$Y$ is a smooth 
projective toric variety 
with $\rho (Y)=5$. 

Finally, we remove the wall $\langle v_1, v_5\rangle$ and 
add the new wall $\langle v_2, v_4\rangle$. Then 
we obtain a flop 
$\phi:Y\dashrightarrow 
X$. We note that 
$v_2+v_4-v_1-v_5=0$. The 
toric variety $X$ is nothing but $X(\Sigma)$ given 
in \cite[Example 1]{fuji-pa}. Thus, 
$X$ is a smooth 
complete toric variety with 
$\rho (X)=5$ and $\overline {NE}(X)=N_1(X)$. Therefore, 
a simple flop $\phi:Y\dashrightarrow X$ completely 
destroys the projectivity of $Y$. 
\end{ex}

We use the following convention throughout this book. 

\begin{say}
$\mathbb R_{>0}$ (resp.~$\mathbb R_{\geq 0}$) 
denotes the set of 
positive (resp.~non-negative) real numbers. 
$\mathbb Z_{>0}$ denotes the set of 
positive integers. 
\end{say} 

\section{Quick review of the classical MMP}\label{se-quiqui}

In this section, we quickly review the classical 
MMP, at the level of \cite[Chapters 2 and 3]{km}, 
for the reader's convenience. 
For the details, see \cite[Chapters 2 and 3]{km} 
or \cite{kmm}. 
Almost all the results explained here will be described 
in more general settings in subsequent chapters. 

\subsection{Singularities of pairs}\label{161-ss}

We quickly review singularities of pairs in the log minimal 
model program. 
Basically, we will only use the notion of 
log canonical pairs in this book. 
So, the reader does not have to 
worry about the various notions of {\em{log terminal}}. 

\begin{defn}[Discrepancy] 
Let $(X, \Delta)$ be a pair 
where $X$ is a normal variety and 
$\Delta$ an $\mathbb R$-divisor 
on $X$ such that 
$K_X+\Delta$ is $\mathbb R$-Cartier. 
Suppose $f:Y\to X$ is a resolution. 
Then, we can write 
$$
K_Y=f^*(K_X+\Delta)+\sum _i a(E_i, X, \Delta)E_i. 
$$ 
This formula means that 
$$f_*(\sum _i a(E_i, X, \Delta)E_i)=-\Delta$$ and that 
$\sum _i a(E_i, X, \Delta)E_i$ is 
numerically equivalent to $K_Y$ over $X$. 
The real number $a(E, X, \Delta)$ is called 
{\em{discrepancy}} of $E$ with 
respect to $(X, \Delta)$. 
The {\em{discrepancy}} of\index{discrepancy} 
$(X, \Delta)$ is given by 
$$
\discrep (X, \Delta)=\inf _E\{a(E, X, \Delta)\, |\, 
E \ {\text{is an exceptional divisor over}} \ X\}.  
$$
\end{defn}

We note that 
it is indispensable to understand how to 
calculate discrepancies for the study of the 
log minimal model program. 

\begin{defn}[Singularities of pairs] 
Let $(X, \Delta)$ be a pair 
where $X$ is a normal variety and 
$\Delta$ an effective $\mathbb R$-divisor 
on $X$ such that 
$K_X+\Delta$ is $\mathbb R$-Cartier. 
We say that 
$(X, \Delta)$ is\index{terminal}\index{canonical}\index{Kawamata log terminal}\index{log canonical} 
$$
\begin{cases}
\text{terminal}\\
\text{canonical}\\
\text{klt}\\
\text{plt}\\
\text{lc}\\
\end{cases}
\quad {\text{if}} \quad \discrep (X,\Delta) 
 \quad
\begin{cases}
> 0,\\
\geq 0,\\
> -1\quad {\text {and \quad $\llcorner \Delta\lrcorner =0$,}}\\
> -1,\\
\geq -1.\\
\end{cases}
$$
Here, plt\index{purely log terminal} is short for {\em{purely log terminal}}. 
\end{defn}
The basic references on this topic are \cite[2.3]{km}, 
\cite{ko-sing}, and \cite{fujino0}. 

\subsection{Basic results for klt pairs} 

In this subsection, we assume that 
$X$ is a projective variety and 
$\Delta$ is an effective $\mathbb Q$-divisor 
for simplicity. 
Let us recall the basic results for klt pairs. 
A starting point is the following 
vanishing theorem. 

\begin{thm}[Vanishing theorem]\label{z-vani}  
Let $X$ be a smooth projective 
variety, $D$ a $\mathbb Q$-divisor 
such that 
$\Supp\{D\}$ is a simple normal crossing 
divisor on $X$. 
Assume that $D$ is ample. Then 
$$
H^i(X, \mathcal O_X(K_X+\ulcorner 
D\urcorner))=0 
$$ 
for $i>0$. 
\end{thm}

It is a special case of the Kawamata--Viehweg vanishing 
theorem. It easily follows from the Kodaira vanishing 
theorem by using the covering trick (see \cite[Theorem 2.64]{km}). In Chapter \ref{chap2}, we will 
prove more general vanishing theorems. 
See, for example, Theorem \ref{8}. 
The next theorem is Shokurov's non-vanishing theorem. 

\begin{thm}[Non-vanishig theorem]\label{z-non}\index{non-vanishing theorem}
Let $X$ be a projective variety, $D$ a nef Cartier 
divisor and $G$ a $\mathbb Q$-divisor. 
Suppose 
\begin{itemize}
\item[$(1)$] $aD+G-K_X$ is $\mathbb Q$-Cartier, ample 
for some $a>0$, and 
\item[$(2)$] $(X, -G)$ is sub klt. 
\end{itemize}
Then, for all $m\gg 0$, $H^0(X, \mathcal O_X(mD+\ulcorner 
G\urcorner))\ne 0$. 
\end{thm}

It plays important roles in the proof of the 
base point free and rationality theorems below. 
In the theory of quasi-log varieties described 
in Chapter \ref{chap3}, the non-vanishing 
theorem will be absorbed into the proof of 
the base point free theorem for 
quasi-log varieties. The following 
two fundamental theorems for klt pairs will be 
generalized for quasi-log varieties in Chapter \ref{chap3}. 
See Theorems \ref{bpf-th}, \ref{rat-th}, and 
\ref{bpf-rd-th} in Chapter \ref{chap4}. 

\begin{thm}[Base point free theorem]\label{z-bpf}\index{base 
point free theorem} 
Let $(X, \Delta)$ be a projective  
klt pair. Let $D$ be a nef Cartier 
divisor such that $aD-(K_X+\Delta)$ is ample 
for 
some $a>0$. Then 
$|bD|$ has no base points for all $b\gg 0$. 
\end{thm}

\begin{thm}[Rationality theorem]\index{rationality theorem} 
Let $(X, \Delta)$ be a projective klt pair such that 
$K_X+\Delta$ is not nef. 
Let $a>0$ be an integer such that 
$a(K_X+\Delta)$ is Cartier. Let $H$ be an ample Cartier 
divisor, and define 
$$
r=\max \{\, t \in \mathbb R\, |\, H+t(K_X+\Delta) \text{\ is nef\ }\}. 
$$ 
Then $r$ is a rational number of the form 
$u/v$ $(u, v\in \mathbb Z)$ where 
$$
0<v\leq a(\dim X+1). 
$$
\end{thm}

The final theorem is the cone theorem. 
It easily follows from 
the base point free and rationality theorems.

\begin{thm}[Cone theorem]\index{cone theorem} 
Let $(X, \Delta)$ be a projective 
klt pair. 
Then we have the following properties. 
\begin{itemize}
\item[$(1)$] There are $($countably many$)$ rational 
curves $C_j\subset X$ such that 
$$
\overline {NE}(X)=\overline {NE}(X)_{(K_X+\Delta)\geq 0} 
+\sum \mathbb R_{\geq 0}[C_j]. 
$$
\item[$(2)$] Let $R\subset \overline {NE}(X)$ be a 
$(K_X+\Delta)$-negative extremal ray. Then there 
is a unique morphism $\varphi_R:X\to Z$ to a projective 
variety such that $(\varphi_R)_*\mathcal O_X\simeq \mathcal O_Z$ and 
an irreducible curve $C\subset X$ is mapped to a point by $\varphi_R$ if and only if $[C]\in R$. 
\end{itemize}
\end{thm}

We note that 
the cone theorem can be 
proved for dlt pairs in the relative setting. 
See, for example, \cite{kmm}. 
We omit it here. It is because 
we will give a complete generalization of 
the cone theorem for quasi-log varieties in 
Theorem \ref{cone-thm}.  

\subsection{X-method}\label{subsec-X}\index{X-method}

In this subsection, 
we give a proof to the base point free theorem 
(see Theorem \ref{z-bpf}) by 
assuming the non-vanishing theorem 
(see Theorem \ref{z-non}). 
The following proof is taken almost verbatim 
from \cite[3.2 Basepoint-free Theorem]{km}. 
This type of argument is sometimes called 
X-method. It has various applications in many different 
contexts. 
So, the reader should understand X-method. 

\begin{proof}[Proof of the base point free theorem]
We prove the base point free theorem:~Theorem \ref{z-bpf}. 
\begin{step}\label{step-z1}
In this step, we establish that 
$|mD|\ne \emptyset$ for every $m\gg 0$. 
We can construct a resolution $f:Y\to X$ such that 
\begin{itemize}
\item[(1)] 
$K_Y=f^*(K_X+\Delta)+\sum a_j F_j$ 
with all $a_j >-1$, 
\item[(2)] $f^*(aD-(K_X+\Delta))-\sum p_j F_j$ is ample 
for some $a>0$ and for suitable 
$0<p_j \ll 1$, and 
\item[(3)] $\sum F_j (\supset \Exc (f)\cup \Supp f^{-1}_*\Delta)$ is a simple normal crossing divisor on $Y$. 
\end{itemize}
We note that the $F_j$ is not necessarily 
$f$-exceptional. 
On $Y$, we write 
\begin{align*}
&f^*(aD-(K_X+\Delta))-\sum p_j F_j\\ 
&=af^*D+\sum (a_j -p_j)F_j -(f^*(K_X+\Delta)+\sum a_j F_j)
\\
&=af^*D+G-K_Y, 
\end{align*} 
where $G=\sum (a_j -p_j)F_j$. By the assumption, 
$\ulcorner G\urcorner$ is an effective 
$f$-exceptional divisor, 
$af^*D+G-K_Y$ is ample, and $$H^0(Y, \mathcal O_Y
(mf^*D+\ulcorner G\urcorner))\simeq 
H^0(X, \mathcal O_X(mD)).$$ We can now apply the non-vanishing 
theorem $($see Theorem \ref{z-non}$)$ to get that 
$H^0(X, \mathcal O_X(mD))\ne 0$ for all $m \gg 0$. 
\end{step}

\begin{step}\label{step-z2}
For a positive integer $s$, let 
$B(s)$ denote the reduced base locus of $|sD|$. Clearly, 
we have 
$B(s^u)\subset B(s^v)$ for any positive 
integers $u>v$. Noetherian induction implies that the 
sequence $B(s^u)$ stabilizes, and we call the limit 
$B_s$. So either $B_s$ is non-empty for some 
$s$ or $B_s$ and $B_{s'}$ are empty for two 
relatively prime integers $s$ and $s'$. In the latter 
case, take $u$ and $v$ such that 
$B(s^u)$ and $B({s'}^v)$ are empty, 
and use the fact that every sufficiently 
large integer is a linear combination 
of $s^u$ and ${s'}^v$ with non-negative 
coefficients to conclude that $|mD|$ is base point free for 
all $m\gg 0$. So, we must show that the 
assumption that some $B_s$ is non-empty leads 
to a contradiction. We let $m=s^u$ such 
that $B_s=B(m)$ and assume that this set is non-empty. 

Starting with the linear system obtained from 
Step \ref{step-z1}, we can blow up further 
to obtain a new $f:Y\to X$ for which 
the conditions of Step \ref{step-z1} hold, and, 
for some $m>0$, 
$$
f^*|mD|=|L|{\text{\ (moving part)}}+
\sum r_j F_j {\text{\ (fixed 
part)}} 
$$ 
such that $|L|$ is base point free. 
Therefore, $\bigcup \{f(F_j) | r_j >0\}$ is the 
base locus of $|mD|$. 
Note that $f^{-1}\Bs|mD|=\Bs |mf^*D|$. 
We obtain the desired contradiction by 
finding some $F_j$ with 
$r_j>0$ such that, for 
all $b\gg 0$, $F_j$ is not contained in the base 
locus of $|bf^*D|$. 
\end{step}

\begin{step}\label{step-z3} 
For an integer $b>0$ and a rational number $c>0$ such 
that $b\geq cm+a$, 
we define divisors: 
\begin{eqnarray*}
N(b, c)&=&bf^*D-K_Y+\sum 
(-cr_j +a_j-p_j)F_j \\ 
&=& (b-cm-a)f^*D {\text{\ \ \ \ \ (nef)}}\\ 
&&+c(mf^*D-\sum r_j F_j) {\text{\ \ \ \ \ (base point free)}}\\ 
&&+f^*(aD-(K_X+\Delta))-\sum p_j F_j {\text{\ \ \ \ \ (ample).}}
\end{eqnarray*}
Thus, $N(b, c)$ is ample for $b\geq cm+a$. 
If that is the case then, by Theorem \ref{z-vani}, $H^1(Y, \mathcal O_Y
(\ulcorner N(b, c)\urcorner+K_Y))=0$, and 
$$
\ulcorner N(b, c)\urcorner=bf^*D+\sum 
\ulcorner -cr_j+a_j-p_j\urcorner F_j-K_Y. 
$$
\end{step}

\begin{step}\label{step-z4} 
$c$ and $p_j$ can be chosen so that 
$$
\sum (-cr_j+a_j-p_j)F_j=A-F 
$$ 
for some $F=F_{j_0}$, where $\ulcorner A\urcorner$ is 
effective and $A$ does not have $F$ as a component. 
In fact, we choose $c>0$ so that 
$$
\min_j (-cr_j+a_j-p_j)=-1. 
$$ 
If this last condition does not 
single out a unique $j$, we wiggle the 
$p_j$ slightly to achieve the desired uniqueness. 
This $j$ satisfies $r_j>0$ and $\ulcorner 
N(b, c)\urcorner+K_Y=
bf^*D+\ulcorner A\urcorner -F$. Now Step \ref{step-z3} 
implies that 
$$
H^0(Y, \mathcal O_Y(bf^*D+\ulcorner A\urcorner))
\to H^0(F, \mathcal O_F(bf^*D+\ulcorner A\urcorner)) 
$$ is 
surjective for $b\geq cm+a$. 
If $F_j$ appears in $\ulcorner A\urcorner$, 
then $a_j>0$, 
so $F_j$ is $f$-exceptional. 
Thus, $\ulcorner A\urcorner$ is $f$-exceptional. 
\end{step}

\begin{step}\label{step-z5} 
Notice that 
\begin{align*}
N(b, c)|_F&=(bf^*D+A-F-K_Y)|_F=
(bf^*D+A)|_F-K_F. 
\end{align*}
So we can apply the non-vanishing theorem (see Theorem 
\ref{z-non}) on $F$ to 
get 
$$
H^0(F, \mathcal O_F(bf^*D+\ulcorner A\urcorner))\ne 0. 
$$ 
Thus, $H^0(Y, \mathcal O_Y(bf^*D+\ulcorner A\urcorner))$ has 
a section not vanishing on $F$. 
Since $\ulcorner A\urcorner$ is $f$-exceptional and effective, 
$$
H^0(Y, \mathcal O_Y(bf^*D+\ulcorner A\urcorner))\simeq 
H^0(X, \mathcal O_X(bD)). 
$$ 
Therefore, $f(F)$ is not contained in the base locus of 
$|bD|$ for all $b\gg 0$. 
\end{step}
This completes the proof of the base point free theorem. 
\end{proof}

In the subsection \ref{331-ssec}, 
we will prove the base point free theorem 
for quasi-log varieties. We 
recommend the reader to 
compare the proof of Theorem \ref{bpf-th} 
with the arguments explained here. 

\subsection{MMP for $\mathbb Q$-factorial dlt pairs}\label{sub153}

In this subsection, we explain the log minimal 
model program for $\mathbb Q$-factorial dlt pairs. 
First, let us recall the definition of 
the log minimal model. 

\begin{defn}[Log minimal model]\label{z-minimal}
Let $(X, \Delta)$ be a log canonical pair and $f:X\to S$ 
a proper morphism. 
A pair $(X', \Delta')$ sitting in a diagram 
$$
\begin{matrix}
X & \overset{\phi}{\dashrightarrow} & \ X'\\
{\ \ \ \ \ f\searrow} & \ &  {\swarrow}f'\ \ \ \ \\
 \ & S &  
\end{matrix}
$$
is called a {\em{log minimal model}} of\index{log 
minimal model} $(X, \Delta)$ 
over $S$ if 
\begin{itemize}
\item[(1)] $f'$ is proper, 
\item[(2)] $\phi^{-1}$ has no exceptional divisors, 
\item[(3)] $\Delta'=\phi_*\Delta$, 
\item[(4)] $K_{X'}+\Delta'$ is $f'$-nef, and 
\item[(5)] $a(E, X, \Delta)<a(E, X', \Delta')$ for every $\phi$-exceptional divisor $E\subset X$. 
\end{itemize} 
\end{defn}

Next, we recall the 
flip theorem for dlt pairs in \cite{bchm} and \cite{hm}. 
We need the notion of small morphisms to 
treat flips. 

\begin{defn}[Small morphism]
Let $f:X\to Y$ be a proper birational 
morphism between normal varieties. If $\Exc (f)$ 
has codimension $\geq 2$, then $f$ is 
called\index{small} {\em{small}}. 
\end{defn}

\begin{thm}[Log flip for dlt pairs]\label{z-flip}  
Let $\varphi:(X, \Delta)\to W$ be an extremal 
flipping contraction, that is, 
\begin{itemize}
\item[$(1)$] $(X, \Delta)$ is dlt, 
\item[$(2)$] $\varphi$ is small projective 
and $\varphi$ has only connected fibers, 
\item[$(3)$] $-(K_X+\Delta)$ is $\varphi$-ample, 
\item[$(4)$] $\rho (X/W)=1$, and 
\item[$(5)$] $X$ is $\mathbb Q$-factorial. 
\end{itemize}
Then we have the following diagram{\em{:}}
$$
\begin{matrix}
X & {\dashrightarrow} & \ X^+\\
{\ \ \ \ \ \searrow} & \ &  {\swarrow}\ \ \ \ \\
 \ & W &  
\end{matrix}
$$
\begin{itemize}
\item[{\em{(i)}}] $X^+$ is a normal variety, 
\item[{\em{(ii)}}] 
$\varphi^+:X^+\to W$ is small projective, and 
\item[{\em{(iii)}}] 
$K_{X^+}+\Delta^+$ is $\varphi^+$-ample, where 
$\Delta^+$ is the strict transform of $\Delta$. 
\end{itemize}
We call $\varphi^+:(X^+, \Delta^+)\to W$ a 
$(K_X+\Delta)$-flip of $\varphi$. 
\end{thm}

Let us explain the relative 
log minimal model program for 
$\mathbb Q$-factorial dlt pairs. 

\begin{say}[MMP for $\mathbb Q$-factorial 
dlt pairs] 
We start with a pair $(X, \Delta)=(X_0, \Delta_0)$. 
Let $f_0: X_0\to S$ be a projective 
morphism. The aim is to set up a recursive 
procedure which creates intermediate pairs 
$(X_i, \Delta_i)$ and 
projective morphisms $f_i: X_i \to S$. 
After some steps, it should stop with a final 
pair $(X', \Delta')$ and 
$f': X'\to S$. 
\setcounter{step}{-1}
\begin{step}[Initial datum]\label{0ste}
Assume that we already constructed $(X_i, \Delta_i)$ 
and $f_i:X_i \to S$ with the following 
properties: 
\begin{itemize}
\item[(1)] $X_i$ is $\mathbb Q$-factorial, 
\item[(2)] $(X_i, \Delta_i)$ is dlt, and 
\item[(3)] $f_i$ is projective. 
\end{itemize}
\end{step}
\begin{step}[Preparation]\label{1ste}
If $K_{X_i}+\Delta_i$ is $f_i$-nef, then we 
go directly to Step \ref{3ste} (2). If 
$K_{X_i}+\Delta_i$ is not $f_i$-nef, 
then we establish two results: 
\begin{itemize}
\item[(1)] (Cone Theorem) We have the following 
equality. 
$$\overline {NE}(X_i/S)=\overline {NE}(X_i/S)_{(K_{X_i}+
\Delta_i)\geq 0}+\sum \mathbb R_{\geq 0}[C_i]. $$
\item[(2)] (Contraction Theorem) 
Any $(K_{X_i}+\Delta_i)$-negative extremal ray $R_i\subset \overline {NE}(X_i/S)$ can 
be contracted. 
Let $\varphi_{R_i}:X_i\to Y_i$ denote the corresponding contraction. 
It sits in a commutative diagram. 
$$
\begin{matrix}
X_i & \overset{\varphi_{R_i}}{\longrightarrow} & \ Y_i\\
{\ \ \ \ \ f_i\searrow} & \ &  {\swarrow}g_i\ \ \ \ \\
 \ & S &  
\end{matrix}
$$
\end{itemize}
\end{step}
\begin{step}[Birational transformations]\label{2ste} 
If $\varphi_{R_i}:X_i\to Y_i$ is birational, 
then we produce a new pair $(X_{i+1}, \Delta_{i+1})$ as 
follows. 
\begin{itemize}
\item[(1)] (Divisorial contraction) 
If $\varphi_{R_i}$ is a divisorial contraction, 
that is, $\varphi_{R_i}$ contracts a divisor, 
then we 
set $X_{i+1}=Y_i$, 
$f_{i+1}=g_i$, and $\Delta_{i+1}=(\varphi_{R_i})_*\Delta_i$. 
\item[(2)] (Flipping contraction) 
If $\varphi_{R_i}$ is a flipping contraction, 
that is, $\varphi_{R_i}$ is small, 
then we set $(X_{i+1}, \Delta_{i+1})=(X^+_i, \Delta^+_i)$, 
where $(X^+_i, \Delta^+_i)$ is the flip 
of $\varphi_{R_i}$, and 
$f_{i+1}=g_i\circ \varphi_{R_i}^+$. See Theorem \ref{z-flip}. 
\end{itemize}
In both cases, we can prove that $X_{i+1}$ is $\mathbb Q$-factorial, 
$f_{i+1}$ is projective and 
$(X_{i+1}, \Delta_{i+1})$ is dlt. 
Then we go back to Step \ref{0ste} with 
$(X_{i+1}, \Delta_{i+1})$ and start anew. 
\end{step}
\begin{step}[Final outcome]\label{3ste} 
We expect that eventually the procedure 
stops, and we get one of the following 
two possibilities: 
\begin{itemize}
\item[(1)] (Mori fiber space) 
If $\varphi_{R_i}$ is a Fano contraction, 
that is, $\dim Y_i<\dim X_i$,  
then we set $(X', \Delta')=(X_i, \Delta_i)$ and 
$f'=f_i$. 
\item[(2)] (Minimal model) If 
$K_{X_i}+\Delta_i$ is $f_i$-nef, then we 
again set $(X', \Delta')=(X_i, \Delta_i)$ and $f'=f_i$. 
We can easily check that $(X', \Delta')$ is 
a log minimal model of $(X, \Delta)$ over $S$ in the 
sense of Definition \ref{z-minimal}.  
\end{itemize}
\end{step}
By the results in \cite{bchm} and \cite{hm}, 
all we have to do is to prove that 
there are no infinite sequence of 
flips in the above process. 
\end{say}
We will discuss the log minimal model program 
for (not necessarily $\mathbb Q$-factorial) {\em{lc}} pairs 
in Section \ref{31-sec}. 

\chapter{Vanishing and Injectivity Theorems for LMMP}\label{chap2} 

\section{Introduction}\label{21-sec} 

The following diagram is well known and 
described, for example, in \cite[3.1]{km}. 
See also Section \ref{se-quiqui}. 

{\small
\begin{equation*}
 \fbox{
\begin{tabular}{c}
\begin{minipage}{5.1cm}
Kawamata--Viehweg vanishing theorem 
\end{minipage}
\end{tabular}
}
\Longrightarrow 
\fbox{
 \begin{tabular}{c}
\begin{minipage}{5.0cm}
Cone, contraction, rationality, and  
base point free theorems for klt pairs
\end{minipage}
\end{tabular}
}
\end{equation*}
}

This means that the Kawamata--Viehweg vanishing theorem 
produces the fundamental theorems of the log minimal model 
program (LMMP, for short) for klt pairs. 
This method is sometimes called X-method\index{X-method} and now classical. 
It is sufficient for the LMMP for $\mathbb Q$-factorial 
dlt pairs. 
In \cite{ambro}, Ambro obtained the same diagram for 
{\em{quasi-log varieties}}. 
Note that the class of quasi-log varieties naturally contains 
lc pairs. 
Ambro introduced the notion of quasi-log varieties for the inductive 
treatments of lc pairs. 
 
{\small
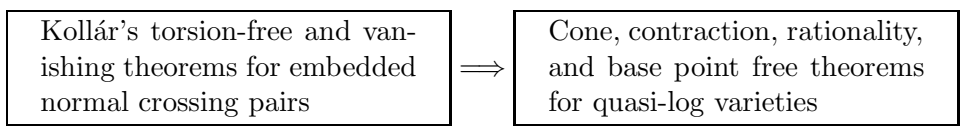
\begin{equation*}
 \fbox{
\begin{tabular}{c}
\begin{minipage}{5cm}
Koll\'ar's torsion-free and 
vanishing theorems for embedded normal crossing pairs
\end{minipage}
\end{tabular}
}
\Longrightarrow
\fbox{
 \begin{tabular}{c}
\begin{minipage}{5cm}
Cone, contraction, rationality, and  
base point free theorems for quasi-log varieties
\end{minipage}
\end{tabular}
}
\end{equation*}
}

Namely, if we obtain Koll\'ar's torsion-free and 
vanishing theorems for {\em{embedded normal crossing 
pairs}}, then X-method works and 
we obtain the fundamental theorems 
of the LMMP for 
quasi-log varieties. 
So, there exists an important problem for 
the LMMP for lc pairs.

\begin{que}\label{1-pro}
Are the injectivity, torsion-free and vanishing theorems 
for embedded normal crossing pairs true? 
\end{que} 

Ambro gave an answer to Problem \ref{1-pro} in \cite[Section 3]{ambro}. 
Unfortunately, the proofs of injectivity, 
torsion-free, and vanishing theorems 
in \cite[Section 3]{ambro} contain various gaps. 
So, in this chapter, we give an affirmative answer to Problem \ref{1-pro} again. 

\begin{thm}\label{ma-th}  
Ambro's formulation of Koll\'ar's injectivity, torsion-free, and 
vanishing theorems for embedded normal crossing pairs 
hold true. 
\end{thm} 

Once we have Theorem \ref{ma-th}, we 
can obtain the fundamental 
theorems of the LMMP for lc 
pairs. 
The X-method for 
quasi-log varieties, 
which was explained in \cite[Section 5]{ambro} and 
will be described in Chapter \ref{chap3}, 
is essentially the same as the klt case. 
It may be more or less a routine work for the experts (see 
Chapter \ref{chap3} and \cite{fuj-lec}). 
We note that Kawamata used Koll\'ar's injectivity, 
vanishing, and torsion-free theorems for 
{\em{generalized normal crossing varieties}} in 
\cite{kawamata1}. 
For the details, see \cite{kawamata1} or 
\cite[Chapter 6]{kmm}. We think that \cite{kawamata1} 
is the first place where X-method was 
used for reducible 
varieties. 

Ambro's proofs of the 
injectivity, torsion-free, 
and vanishing theorems in \cite{ambro} do not work even for 
smooth varieties. So, we need new ideas to prove 
the desired injectivity, torsion-free, vanishing theorems. 
It is the main subject of this chapter. 
We will explain various troubles in the proofs in \cite[Section 3]{ambro} 
below for the reader's convenience. 
Here, we give an application of Ambro's 
theorems to motivate the reader. 
It is the culmination of the works of several authors:~Kawamata, 
Viehweg, Nadel, Reid, Fukuda, Ambro, and many others. 
It is the first time that the following theorem is 
stated explicitly in the literature. 

\begin{thm}[{cf.~Theorem \ref{kvn}}]
Let $(X, B)$ be a proper 
lc pair such that $B$ is a boundary $\mathbb R$-divisor 
and let $L$ be a $\mathbb Q$-Cartier Weil divisor on $X$. 
Assume that $L-(K_X+B)$ is nef 
and log big. Then 
$H^q(X, \mathcal O_X(L))=0$ for any $q>0$. 
\end{thm}

It also contains a complete form of Kov\'acs' Kodaira vanishing 
theorem for lc pairs (see Corollary \ref{lcvar}). 
Let us explain the main trouble in \cite[Section 3]{ambro} 
by the following 
simple example. 

\begin{ex}\label{0}
Let $X$ be a smooth projective variety and $H$ a 
Cartier divisor on $X$. Let $A$ be a smooth 
irreducible member of $|2H|$ and $S$ 
a smooth divisor on $X$ such that $S$ and $A$ are disjoint. 
We put $B=\frac{1}{2}A+S$ and $L=H+K_X+S$. Then $L\sim _{\mathbb Q}
K_X+B$ and 
$2L\sim 2(K_X+B)$. 
We define $\mathcal E=\mathcal O_X(-L+K_X)$ as in the 
proof of \cite[Theorem 3.1]{ambro}. 
Apply the argument in the proof of \cite[Theorem 3.1]{ambro}. 
Then we have a double cover $\pi:Y\to X$ 
corresponding to $2B\in |\mathcal E^{-2}|$. 
Then $\pi_*\Omega ^{p}_Y(\log \pi^*B)\simeq 
\Omega ^p_X(\log B)\oplus \Omega^p_X(\log B)\otimes \mathcal E(S)$. 
Note that $\Omega ^p_X(\log B)\otimes \mathcal E$ is not a direct 
summand of $\pi_*\Omega^p_Y(\log \pi^*B)$. 
Theorem 3.1 in \cite{ambro} claims that 
the homomorphisms 
$H^q(X, \mathcal O_X(L))\to H^q(X, \mathcal O_X(L+D))$ are injective for all 
$q$. Here, we used the notation in \cite[Theorem 3.1]{ambro}. 
In our case, $D=mA$ for some positive integer $m$. 
However, Ambro's argument just implies that 
$H^q(X,\mathcal O_X(L-\llcorner B\lrcorner))\to 
H^q(X, \mathcal O_X(L-\llcorner B\lrcorner +D))$ is injective for 
any $q$. Therefore, his proof works only for the case 
when $\llcorner B\lrcorner=0$ even if $X$ is smooth. 
\end{ex}
This trouble is crucial in several applications on the LMMP. 
Ambro's proof is based on the mixed Hodge structure of 
$H^i(Y-\pi^*B, \mathbb Z)$. 
It is a standard technique for vanishing 
theorems in the LMMP. 
In this chapter, we use the mixed Hodge structure of 
$H^i_{c}(Y-\pi^*S, \mathbb Z)$, where 
$H^i_{c}(Y-\pi^*S, \mathbb Z)$ is the 
cohomology group with 
compact support. 
Let us explain the main idea of this chapter. 
Let $X$ be a smooth projective variety with $\dim X=n$ and 
$D$ a simple normal crossing divisor 
on $X$. 
The main ingredient of our arguments 
is the decomposition 
$$
H^i_c(X-D, \mathbb C)=\bigoplus _{p+q=i}H^q(X, \Omega^p_X(\log D)\otimes 
\mathcal O_X(-D)). 
$$ 
The dual statement 
$$
H^{2n-i}(X-D, \mathbb C)=\bigoplus _{p+q=i}H^{n-q}
(X, \Omega^{n-p}_X(\log D)), 
$$ 
which is well known and is commonly 
used for vanishing theorems, is not useful 
for our purposes. 
To solve Problem \ref{1-pro}, we have to carry out 
this simple idea for reducible varieties. 

\begin{rem}\label{r-1} 
In the proof of \cite[Theorem 3.1]{ambro}, 
if we assume that $X$ is smooth, $B'=S$ is 
a reduced smooth divisor on $X$, and 
$T\sim 0$, 
then 
we need 
the $E_1$-degeneration of 
$$
E^{pq}_1=H^q(X, \Omega^p_X(\log S)\otimes \mathcal O_X(-S))
\Longrightarrow \mathbb H^{p+q}(X, \Omega^{\bullet}_X(\log S)\otimes 
\mathcal O_X(-S)).  
$$ 
However, Ambro seemed to confuse it with 
the $E_1$-degeneration 
of 
$$
E^{pq}_1=H^q(X, \Omega^p_X(\log S))
\Longrightarrow \mathbb H^{p+q}(X, \Omega^{\bullet}_X(\log S)).  
$$ 
Some problems on the Hodge theory 
seem to exist in the proof of \cite[Theorem 3.1]{ambro}. 
\end{rem}

\begin{rem}\label{r0} 
In \cite[Theorem 3.1]{ambro2}, 
Ambro reproved his theorem under some extra assumptions. 
Here, we use the notation in \cite[Theorem 3.1]{ambro2}. 
In the last line of the proof of \cite[Theorem 3.1]{ambro2}, he 
used the $E_1$-degeneration of some 
spectral sequence. 
It seems to be the $E_1$-degeneration of 
$$
E^{pq}_1=H^q(X', \widetilde {\Omega}^p_{X'}(\log \sum 
_{i'}E'_{i'}))\Longrightarrow 
\mathbb H^{p+q}(X', \widetilde{\Omega}^{\bullet}_{X'}
(\log \sum _{i'}E'_{i'})) 
$$ 
since he cited \cite[Corollary 3.2.13]{deligne0}. Or, he 
applied the same type of $E_1$-degeneration to 
a desingularization of $X'$. 
However, we think that the $E_1$-degeneration of 
\begin{eqnarray*}
\lefteqn{E^{pq}_1=H^q(X', \widetilde {\Omega}^p_{X'}(\log (\pi^*R+\sum 
_{i'}E'_{i'}))\otimes 
\mathcal O_{X'}(-\pi^*R))}\\& & \Longrightarrow 
\mathbb H^{p+q}(X', \widetilde{\Omega}^{\bullet}_{X'}
(\log (\pi^*R+\sum _{i'}E'_{i'}))\otimes \mathcal O_{X'}(-\pi^*R)) 
\end{eqnarray*} 
is the appropriate one in his proof. 
If we assume that $T\sim 0$ in \cite[Theorem 3.1]{ambro2}, 
then Ambro's proof seems to imply that the $E_1$-degeneration 
of 
$$
E^{pq}_1=H^q(X, \Omega^p_X(\log R)\otimes \mathcal O_X(-R))
\Longrightarrow \mathbb H^{p+q}(X, \Omega^{\bullet}_X(\log R)\otimes 
\mathcal O_X(-R)) 
$$ 
follows from the usual $E_1$-degeneration of 
$$
E^{pq}_1=H^q(X, \Omega^p_X)\Longrightarrow \mathbb H^{p+q}(X, 
\Omega^{\bullet}_X). 
$$ 
Anyway, there are some problems in the proof of 
\cite[Theorem 3.1]{ambro2}. 
In this chapter, we adopt  
the following spectral sequence 
\begin{eqnarray*}
\lefteqn{E^{pq}_1=H^q(X', \widetilde {\Omega}^p_{X'}(\log 
\pi^*R)\otimes \mathcal O_{X'}(-\pi^*R))}\\& & \Longrightarrow 
\mathbb H^{p+q}(X', \widetilde{\Omega}^{\bullet}_{X'}
(\log \pi^*R)\otimes \mathcal O_{X'}(-\pi^*R)) 
\end{eqnarray*} 
and prove its $E_1$-degeneration. For the 
details, see Sections \ref{sec2} and \ref{sec3}. 
\end{rem}

One of the main contributions of this chapter 
is the rigorous proof of Proposition \ref{2}, 
which we call a fundamental injectivity theorem. 
Even if we 
prove this proposition, there are 
still several technical difficulties to 
recover Ambro's results on injectivity, 
torsion-free, and vanishing theorems:~Theorems \ref{61} and 
\ref{62}. 
Some important arguments are missing in \cite{ambro}. 
We will discuss the other troubles on the arguments in 
\cite{ambro} throughout Section \ref{sec4}. 
See also Section \ref{29-sec}. 

\begin{say}[Background, history, and related topics]\label{bunken}
The standard references for vanishing, 
torsion-free, and injectivity 
theorems for the LMMP are 
\cite[Part III Vanishing Theorems]{kollar} and 
the first half of the book \cite{ev}. 
In this chapter, we closely follow the presentation of 
\cite{ev} and that of \cite{ambro}. 
Some special cases of Ambro's theorems 
were proved in \cite[Section 2]{fujino-high}. 
Chapter 1 in \cite{kmm} is 
still a good source for vanishing 
theorems for the LMMP. 
We note that one of the origins of Ambro's results is \cite[Section 4]
{kawamata}. However, we do not treat Kawamata's 
generalizations of vanishing, torsion-free, 
and injectivity theorems for {\em{generalized}} normal 
crossing varieties. 
It is mainly because we can quickly reprove the 
main theorem of \cite{kawamata1} without appealing 
these difficult vanishing and injectivity 
theorems once we 
know a generalized version of 
Kodaira's canonical bundle formula. 
For the details, see 
\cite{fujino6} or \cite{fuji-ka}. 
\end{say}

We summarize the contents of this chapter. 
In Section \ref{sec-pre}, we collect basic definitions and fix some notations. 
In Section \ref{sec2}, we prove a fundamental cohomology 
injectivity theorem for 
simple normal crossing pairs. 
It is a very special case of Ambro's theorem. 
Our proof heavily depends on the 
$E_1$-degeneration of a certain Hodge to de Rham type 
spectral sequence. 
We postpone the proof of the 
$E_1$-degeneration in Section \ref{sec3} 
since it is a purely Hodge theoretic argument.  
Section \ref{sec3} consists of a short survey of mixed Hodge 
structures on various objects and the proof 
of the key $E_1$-degeneration. 
We could find no references on mixed Hodge 
structures which 
are appropriate for our purposes. 
So, we write it for the reader's convenience. 
Section \ref{sec4} is 
devoted to the proofs of Ambro's theorems for 
embedded simple normal crossing 
pairs. 
We discuss various problems in \cite[Section 3]{ambro} and 
give the 
first rigorous proofs to 
\cite[Theorems 3.1, 3.2]{ambro} 
for embedded simple normal crossing pairs. 
We think that several indispensable 
arguments such as Lemmas \ref{re-vani-lem}, \ref{6}, and \ref{comp} 
are missing in \cite[Section 3]{ambro}. 
We treat some further generalizations of vanishing and 
torsion-free theorems in Section \ref{genera}. 
In Section \ref{sec6}, we recover Ambro's theorems in full 
generality. 
We recommend the reader to compare this chapter 
with \cite{ambro}. 
We note that Section \ref{sec6} 
seems to be unnecessary for applications. 
Section \ref{28-sec} is devoted to describe some 
examples. 
In Section \ref{29-sec}, 
we will quickly review the structure of 
our proofs of the injectivity, torsion-free, and 
vanishing theorems. It may 
help the reader to understand 
the reason why our proofs are much longer than 
the original proofs in \cite[Section 3]{ambro}. 
In Chapter \ref{chap3}, we will treat the fundamental 
theorems of the LMMP for lc pairs as 
an application of our vanishing and torsion-free 
theorems. 
The reader can find various other applications 
of our new cohomological results in \cite{fujino8}, 
\cite{fujino9}, and \cite{fujino10}. 
See also Sections \ref{to-sec}, \ref{43-sec}, 
and \ref{44-sec}. 

We note that we will work over $\mathbb C$, the complex number field, 
throughout this chapter. 

\section{Preliminaries}\label{sec-pre}
We explain basic notion according to \cite[Section 2]{ambro}. 

\begin{defn}[Normal and simple normal crossing varieties]
\index{normal crossing variety}\index{simple normal 
crossing variety}\label{01}
A variety $X$ has {\em{normal crossing}} singularities 
if, for every closed point $x\in X$, 
$$
\widehat{\mathcal O}_{X,x}\simeq \frac{\mathbb C[[x_0, \cdots, x_{N}]]}
{(x_0\cdots x_k)}
$$
for some $0\leq k \leq N$, where 
$N=\dim X$. 
Furthermore, if each irreducible component of $X$ is smooth, 
$X$ is called a {\em{simple normal crossing}} 
variety. If $X$ is a normal crossing variety, then $X$ has only  
Gorenstein singularities. Thus, it has an 
invertible dualizing sheaf $\omega_X$. 
So, we can define the {\em{canonical divisor}} 
$K_X$ such that $\omega_X\simeq \mathcal O_X(K_X)$. 
It is a Cartier divisor on $X$ and is well defined 
up to linear equivalence. 
\end{defn}

\begin{defn}[Mayer--Vietoris simplicial resolution]\label{011}
Let $X$ be a simple normal crossing variety with 
the irreducible decomposition $X=\bigcup_{i\in I}X_i$. 
Let $I_n$ be the set of strictly increasing 
sequences $(i_0, \cdots, i_n)$ in $I$ and $X^n=\coprod_{I_n} X_{i_0}\cap 
\cdots \cap X_{i_n}$ the disjoint union of the intersections of 
$X_i$. 
Let $\varepsilon_n:X^n \to X$ be the disjoint union 
of the natural inclusions. 
Then $\{X^n, \varepsilon_n\}_n$ 
has a natural semi-simplicial 
scheme structure. 
The face operator is induced by $\lambda_{j,n}$, where 
$\lambda_{j, n}: X_{i_0}\cap \cdots 
\cap X_{i_n}\to 
X_{i_0}\cap \cdots \cap X_{i_{j-1}}\cap X_{i_{j+1}}\cap 
\cdots \cap X_{i_n}$ is the natural closed embedding for 
$j\leq n$ (cf.~\cite[3.5.5]{elzein2}). 
We denote it by $\varepsilon:X^{\bullet}\to X$ and call it the 
{\em{Mayer--Vietoris simplicial resolution}} of 
$X$. The complex 
$$
0\to \varepsilon_{0*}\mathcal O_{X^0}
\to \varepsilon _{1*}\mathcal O_{X^1} \to 
\cdots \to \varepsilon _{k*}\mathcal O_{X^k}\to \cdots,  
$$ 
where the differential 
$d_k:\varepsilon _{k*}\mathcal O_{X^k}\to 
\varepsilon _{k+1*}\mathcal O_{X^{k+1}}$ is 
$\sum ^{k+1}_{j=0} (-1)^j 
\lambda^* _{j, k+1}$ for any $k\geq 0$, 
is denoted by $\mathcal O_{X^\bullet}$.  
It is easy to see that $\mathcal O_{X^\bullet}$ is 
quasi-isomorphic to $\mathcal O_X$. 
By tensoring $\mathcal L$, any line bundle 
on $X$, 
to $\mathcal O_{X^\bullet}$, we obtain a complex 
$$
0\to \varepsilon_{0*}\mathcal L^0
\to \varepsilon _{1*}\mathcal L^1 \to 
\cdots \to \varepsilon _{k*}\mathcal L^k\to \cdots,  
$$ 
where $\mathcal L^n=\varepsilon ^*_n\mathcal L$. 
It is denoted by $\mathcal L^{\bullet}$. 
Of course, $\mathcal L^{\bullet}$ is quasi-isomorphic to 
$\mathcal L$. 
We note that $H^q(X^\bullet, \mathcal L^\bullet)$ is 
$\mathbb H^q(X, \mathcal L^{\bullet})$ by the definition 
and it is obviously isomorphic to 
$H^q(X, \mathcal L)$ for any $q\geq 0$ 
because $\mathcal L^\bullet$ is quasi-isomorphic to $\mathcal L$. 
\end{defn}

\begin{defn}\label{0111} 
Let $X$ be a simple normal crossing variety. A {\em{stratum}}\index{stratum} 
of $X$ is the image on $X$ of some irreducible 
component of $X^{\bullet}$. 
Note that an irreducible component of $X$ is a stratum of $X$. 
\end{defn}

\begin{defn}[Permissible and normal crossing divisors]
\index{permissible divisor}\index{normal crossing 
divisor}\label{02}
Let $X$ be a simple normal crossing variety. A Cartier 
divisor $D$ on $X$ is called {\em{permissible}} 
if it induces a Cartier divisor $D^{\bullet}$ on $X^{\bullet}$. 
This means that $D^n=\varepsilon^*_n D$ is a Cartier divisor 
on $X_n$ for any $n$. 
It is equivalent to the condition 
that $D$ contains no strata of $X$ in its 
support. 
We say that $D$ is a {\em{normal crossing}} 
divisor on $X$ if, in the notation of Definition \ref{01}, 
we have 
$$
\widehat{\mathcal O}_{D, x}\simeq \frac{\mathbb C[[x_0, \cdots, 
x_{N}]]}{(x_0\cdots x_k, x_{i_1}\cdots x_{i_l})}
$$ 
for some $\{i_1, \cdots, i_l\}\subset\{k+1, \cdots, N\}$. 
It is equivalent to the condition that 
$D^n$ is a normal crossing divisor on $X^n$ for any $n$ 
in the usual sense. 
Furthermore, let $D$ be a 
normal crossing divisor on a simple normal crossing 
variety $X$. 
If $D^n$ is a simple normal crossing divisor 
on $X^n$ for any $n$, then $D$ is called a 
{\em{simple normal crossing divisor}} on $X$. \index{simple normal crossing divisor}
\end{defn}

The following lemma is easy but important. We 
will repeatedly use it in 
Sections \ref{sec2} and \ref{sec4}. 

\begin{lem}\label{7} 
Let $X$ be a simple normal crossing variety and $B$ 
a permissible $\mathbb R$-Cartier $\mathbb R$-divisor 
on $X$, that is, $B$ is an $\mathbb R$-linear 
combination of permissible Cartier divisor on $X$, 
such that $\llcorner B\lrcorner=0$. 
Let $A$ be a Cartier divisor on $X$. Assume that 
$A\sim _{\mathbb R}B$. 
Then there exists a $\mathbb Q$-Cartier $\mathbb Q$-divisor 
$C$ on $X$ such that $A\sim _{\mathbb Q}C$, $\llcorner 
C\lrcorner =0$, 
and $\Supp C=\Supp B$. 
\end{lem}

\begin{proof}[Sketch of the proof] 
We can write $B=A+\sum _i r_i(f_i)$, where 
$f_i\in \Gamma (X, \mathcal K_X^*)$ and 
$r_i\in \mathbb R$ for any $i$. 
Here, $\mathcal K_X$ is the sheaf of total quotient 
ring of $\mathcal O_X$. 
First, we assume that $X$ is smooth. 
In this case, the claim is well known and easy to check. 
Perturb $r_i$'s suitably. Then we obtain a desired 
$\mathbb Q$-Cartier $\mathbb Q$-divisor 
$C$ on $X$. 
It is an elementary problem of the linear algebra. 
In the general case, we take the normalization 
$\varepsilon _0: X^0\to X$ and apply the above result 
to $X^0$, $\varepsilon _0^*A$, 
$\varepsilon _0^*B$, 
and $\varepsilon_0^*(f_i)$'s. 
We note that $\varepsilon _0:X_i\to X$ is a closed 
embedding 
for any irreducible component $X_i$ of $X^0$. 
So, we get a desired $\mathbb Q$-Cartier 
$\mathbb Q$-divisor $C$ on $X$. 
\end{proof}

\begin{defn}[Simple normal crossing pair]\index{simple normal 
crossing pair}\label{03} 
We say that the pair $(X, B)$ is a {\em{simple 
normal crossing pair}} if the 
following 
conditions are satisfied. 
\begin{itemize}
\item[(1)] $X$ is a simple normal crossing 
variety, and 
\item[(2)] $B$ is an $\mathbb R$-Cartier 
$\mathbb R$-divisor whose 
support is a simple normal crossing 
divisor on $X$. 
\end{itemize}
We say that a simple normal crossing 
pair $(X, B)$ is {\em{embedded}} 
if there exists a closed embedding 
$\iota:X\to M$, where 
$M$ is a smooth variety 
of dimension $\dim X+1$. 
We put $K_{X^0}+\Theta=\varepsilon^*_0(K_X+B)$, where 
$\varepsilon_0:X^0\to X$ is the normalization of $X$. 
From now on, we assume that $B$ is a subboundary $\mathbb R$-divisor. 
A {\em{stratum}}\index{stratum} of $(X, B)$ is an irreducible 
component of $X$ or 
the image of some lc center of $(X^0, \Theta)$ on $X$. 
It is compatible with Definition \ref{0111} when $B=0$. 
A Cartier divisor $D$ on a simple normal 
crossing pair $(X, B)$ is called 
{\em{permissible with respect to $(X, B)$}} 
if $D$ contains no strata of the pair $(X, B)$.\index{permissible divisor} 
\end{defn}

\begin{rem}\label{smooth}
Let $(X, B)$ be a simple normal crossing pair. 
Assume that $X$ is smooth. 
Then $(X, B)$ is embedded. 
It is because $X$ is a divisor 
on $X\times C$, where $C$ is 
a smooth curve. 
\end{rem}

We give a typical example of 
embedded simple normal crossing 
pairs. 

\begin{ex}
Let $M$ be a smooth variety and $X$ a simple 
normal crossing divisor on $M$. 
Let $A$ be an $\mathbb R$-Cartier $\mathbb R$-divisor on $M$ such that 
$\Supp (X+A)$ is simple normal crossing on $M$ and 
that $X$ and $A$ have no common irreducible 
components. 
We put $B=A|_X$. Then 
$(X, B)$ is an embedded simple normal crossing 
pair. 
\end{ex}

The reader will find that it is very useful 
to introduce the notion of {\em{global 
embedded simple normal crossing 
pairs}}. 

\begin{defn}[Global embedded simple normal crossing 
pairs]\index{global embedded simple normal crossing 
pair}\label{gsnc0} 
Let $Y$ be a simple normal crossing divisor 
on a smooth 
variety $M$ and let $D$ be an $\mathbb R$-divisor 
on $M$ such that 
$\Supp (D+Y)$ is simple normal crossing and that 
$D$ and $Y$ have no common irreducible components. 
We put $B_Y=D|_Y$ and consider the pair $(Y, B_Y)$. 
We call $(Y, B_Y)$ a {\em{global embedded simple normal 
crossing pair}}. 
\end{defn}

The following lemma is obvious. 

\begin{lem}\label{use}
Let $(X, S+B)$ be an embedded simple normal crossing 
pair such that $S+B$ is a boundary $\mathbb R$-divisor, 
$S$ is reduced, and $\llcorner B\lrcorner=0$. 
Let $M$ be the ambient space of $X$ and $f:N\to M$ 
the blow-up along a smooth irreducible 
component $C$ of $\Supp (S+B)$. 
Let $Y$ be the strict transform of $X$ on $N$. 
Then $Y$ is a simple normal crossing divisor on $N$. 
We can write $K_Y+S_Y+B_Y=f^*(K_X+S+B)$, 
where $S_Y+B_Y$ is 
a boundary $\mathbb R$-Cartier $\mathbb R$-divisor 
on $Y$ such that $S_Y$ is reduced and $\llcorner B_Y\lrcorner 
=0$. In particular, $(Y, S_Y+B_Y)$ 
is an embedded simple normal crossing pair. 
By the construction, we can easily check the 
following properties. 
\begin{itemize}
\item[{\em{(i)}}] $S_Y$ is the strict transform of $S$ on $Y$ if 
$C\subset \Supp B$, 
\item[{\em{(ii)}}] $B_Y$ is the strict transform of 
$B$ on $Y$ if $C\subset \Supp S$, 
\item[{\em{(iii)}}] the 
$f$-image of any stratum of $(Y, S_Y+B_Y)$ 
is a stratum of $(X, S+B)$, and 
\item[{\em{(iv)}}] $R^if_*\mathcal O_Y=0$ for $i>0$ and 
$f_*\mathcal O_Y\simeq \mathcal O_X$. 
\end{itemize}
\end{lem}

As a consequence of Lemma \ref{use}, 
we obtain a very useful lemma. 

\begin{lem}\label{useful-lemma}
Let $(X, B_X)$ be an embedded simple normal crossing pair, 
$B_X$ a boundary $\mathbb R$-divisor, and 
$M$ the ambient space of $X$. 
Then there is a projective birational morphism 
$f:N\to M$, which is a sequence of 
blow-ups as in {\em{Lemma \ref{use}}}, with the 
following properties. 
\begin{itemize}
\item[{\em{(i)}}] Let $Y$ be the strict transform of $X$ on $N$. 
We put $K_Y+B_Y=f^*(K_X+B_X)$. Then 
$(Y, B_Y)$ is an embedded simple normal crossing 
pair. Note that $B_Y$ is a boundary $\mathbb R$-divisor. 
\item[{\em{(ii)}}] $f:Y\to X$ is an isomorphism 
at the generic point 
of any stratum of $Y$. $f$-image of any stratum 
of $(Y, B_Y)$ is a stratum of $(X, B_X)$. 
\item[{\em{(iii)}}] $R^if_*\mathcal O_Y=0$ for any $i>0$ and 
$f_*\mathcal O_Y\simeq \mathcal O_X$. 
\item[{\em{(iv)}}] There exists an $\mathbb R$-divisor $D$ on $N$ 
such that $D$ and $Y$ have no common irreducible 
components and $\Supp (D+Y)$ is simple normal crossing on $N$, and 
$B_Y=D|_Y$. This means that the pair 
$(Y, B_Y)$ is a global embedded 
simple normal crossing pair. 
\end{itemize}
\end{lem}

The next lemma is also easy to 
prove. 

\begin{lem}[{cf.~\cite[p.216 embedded log transformation]{ambro}}]
\index{embedded log transformation}
\label{useful-lem2}
Let $X$ be a simple normal crossing divisor on a smooth variety 
$M$ and let $D$ be an 
$\mathbb R$-divisor on $M$ such that 
$\Supp (D+X)$ is simple normal crossing and that 
$D$ and $X$ have  no common irreducible components. 
We put $B=D|_X$. 
Then $(X, B)$ is a global 
embedded simple normal crossing 
pair. 
Let $C$ be a smooth stratum of $(X, B^{=1})$. 
Let $\sigma:N\to M$ be the blow-up 
along $C$. We denote by $Y$ the reduced structure of the total 
transform of $X$ in $N$. we put $K_Y+B_Y=
f^*(K_X+B)$, where $f=\sigma|_{Y}$. 
Then we have the following properties. 
\begin{itemize}
\item[{\em{(i)}}] $(Y, B_Y)$ is an embedded simple normal 
crossing pair. 
\item[{\em{(ii)}}] $f_*\mathcal O_Y\simeq \mathcal O_X$ and 
$R^if_*\mathcal O_Y=0$ for any $i>0$. 
\item[{\em{(iii)}}] The strata of $(X, B^{=1})$ 
are exactly 
the images of the strata of $(Y, B^{=1}_Y)$. 
\item[{\em{(iv)}}] $\sigma^{-1}(C)$ is a maximal 
$($with respect to the inclusion$)$ stratum of 
$(Y, B^{=1}_Y)$. 
\item[{\em{(v)}}] There exists an $\mathbb R$-divisor $E$ 
on $N$ such that $\Supp (E+Y)$ is simple normal 
crossing and that $E$ and $Y$ have no common irreducible 
components such that 
$B_Y=E|_Y$. 
\item[{\em{(vi)}}] If $B$ is a boundary $\mathbb R$-divisor, 
then so is $B_Y$. 
\end{itemize}
\end{lem}
In general, normal crossing varieties are much more 
difficult than 
{\em{simple}} normal crossing 
varieties. 
We postpone the definition of {\em{normal crossing 
pairs}} in Section \ref{sec6} 
to avoid unnecessary confusion. 
Let us recall the notion of semi-ample 
$\mathbb R$-divisors since 
we often use it in this book. 

\begin{say}[Semi-ample $\mathbb R$-divisor]\index{semi-ample 
$\mathbb R$-divisor} 
Let $D$ be an $\mathbb R$-Cartier $\mathbb R$-divisor on a variety $X$ 
and $\pi:X\to S$ a proper morphism. 
Then, $D$ is $\pi$-semi-ample if $D\sim _{\mathbb R}f^*H$, 
where $f:X\to Y$ is a proper morphism over $S$ and 
$H$ a relatively ample $\mathbb R$-Cartier $\mathbb R$-divisor 
on $Y$. 
It is not difficult to see that $D$ is $\pi$-semi-ample 
if and only if $D\sim _{\mathbb R} \sum _ia_i D_i$, where 
$a_i$ is a positive real number and $D_i$ is a $\pi$-semi-ample 
Cartier divisor on $X$ for any $i$.  
\end{say}
 
In the following sections, we have to 
treat algebraic varieties with quotient singularities. 
All  the $V$-manifolds in this book are obtained 
as cyclic covers of smooth 
varieties whose ramification loci are contained 
in simple normal crossing divisors. 
So, they also have toroidal structures. 
We collect basic definitions according to 
\cite[Section 1]{steenbrink}, which 
is the best reference for our purposes. 

\begin{say}[$V$-manifold]\index{$V$-manifold}\label{v-n}
A {\em{$V$-manifold}} of dimension $N$ is a complex analytic 
space that admits an open covering $\{U_i\}$ such 
that each $U_i$ is analytically isomorphic to $V_i/G_i$, 
where $V_i\subset \mathbb C^N$ is an 
open ball and $G_i$ is a finite subgroup of $\GL(N, \mathbb C)$. 
In this paper, $G_i$ is always a cyclic group for any $i$. 
Let $X$ be a $V$-manifold and $\Sigma$ its 
singular locus. Then we define $\widetilde {\Omega}^{\bullet}_X=
j_*\Omega^{\bullet}_{X-\Sigma}$, where 
$j:X-\Sigma \to X$ is the natural open immersion. 
A divisor $D$ on $X$ is called a {\em{divisor with 
$V$-normal crossings}} if locally on $X$ we have $(X, D)\simeq 
(V, E)/G$ with $V\subset \mathbb C^N$ an open domain, 
$G\subset \GL(N, \mathbb C)$ a small 
subgroup acting on $V$, and $E\subset V$ a $G$-invariant 
divisor with only normal crossing singularities. 
We define $\widetilde {\Omega}^{\bullet}_X(\log D)
=j_*\Omega^{\bullet}_{X-\Sigma}(\log D)$. 
Furthermore, if $D$ is Cartier, then 
we put $\widetilde {\Omega}^{\bullet}_X(\log D)(-D)
=\widetilde \Omega^{\bullet}_X(\log D)\otimes \mathcal O_X(-D)$. 
This complex will play crucial roles in Sections \ref{sec2} and 
\ref{sec3}. 
\end{say}

\section{Fundamental injectivity theorems}\label{sec2}
The following proposition is a reformulation of 
the well-known result by Esnault--Viehweg (cf.~\cite[3.2.~Theorem.~c), 
5.1.~b)]{ev}). 
Their proof in \cite{ev} depends on the characteristic $p$ 
methods obtained by Deligne and Illusie. 
Here, we give another proof for the later usage. 
Note that all we want to do in this section is to generalize the following 
result for simple normal crossing pairs. 

\begin{prop}[Fundamental injectivity theorem I]
\index{fundamental injectivity theorem}\label{1}
Let $X$ be a proper smooth variety and $S+B$  
a boundary $\mathbb R$-divisor on $X$ such that 
the support of $S+B$ is simple normal crossing, $S$ is 
reduced, and 
$\llcorner B\lrcorner =0$. 
Let $L$ be a Cartier divisor on $X$ and let $D$ be an 
effective Cartier divisor whose support is contained 
in $\Supp B$. 
Assume that 
$L\sim_{\mathbb R}K_X+S+B$. Then 
the natural homomorphisms 
$$
H^q(X, \mathcal O_X(L))\to H^q(X, \mathcal O_X(L+D)), 
$$ 
which are induced by the inclusion 
$\mathcal O_X\to \mathcal O_X(D)$, 
are injective for all $q$. 
\end{prop}

\begin{proof} 
We can assume that $B$ is a $\mathbb Q$-divisor 
and $L\sim _{\mathbb Q}K_X+S+B$ by Lemma \ref{7}. 
We put $\mathcal L=\mathcal O_X(L-K_X-S)$. 
Let $\nu$ be the smallest positive integer 
such that $\nu L\sim \nu (K_X+ S+ B)$. 
In particular, $\nu B$ is an integral Weil divisor. 
We take the $\nu$-fold cyclic cover 
$\pi': Y'=\Spec_X\!\bigoplus _{i=0}^{\nu-1} \mathcal L^{-i}\to 
X$ associated to the section $\nu B\in |\mathcal L^{\nu}|$. 
More precisely, let $s\in H^0(X, \mathcal L^{\nu})$ be a section 
whose zero divisor is $\nu B$. 
Then the dual of $s:\mathcal O_X\to \mathcal L^{\nu}$ 
defines a $\mathcal O_X$-algebra structure on 
$\bigoplus ^{\nu-1}_{i=0} \mathcal L^{-i}$.   
For the details, see, for example, \cite[3.5.~Cyclic covers]{ev}. 
Let $Y\to Y'$ be the normalization and 
$\pi:Y\to X$ the composition morphism. 
Then $Y$ has only quotient singularities 
because 
the support of $\nu B$ is simple normal crossing (cf.~\cite[3.24.~Lemma]
{ev}). 
We put $T=\pi^*S$. 
The usual differential $d:\mathcal O_Y\to \widetilde {\Omega}^1_Y\subset 
\widetilde {\Omega}^1_Y(\log T)$ gives the 
differential $d:\mathcal O_Y(-T)\to \widetilde {\Omega}^1_Y(\log T)(-T)$.  
This induces a natural connection 
$\pi_*(d): \pi_*\mathcal O_Y(-T)\to \pi_*(\widetilde{\Omega}^1_Y(\log T)(-T))$. 
It is easy to see that $\pi_*(d)$ decomposes into $\nu$ eigen 
components. 
One of them is $\nabla: \mathcal L^{-1}(-S)\to 
\Omega^1_X(\log (S+B))\otimes \mathcal L^{-1}(-S)$ 
(cf.~\cite[3.2.~Theorem.~c)]{ev}). 
It is well known and easy to check that 
the inclusion 
$\Omega^{\bullet}_X(\log (S+B))\otimes \mathcal L^{-1}(-S-D) 
\to \Omega^{\bullet}_X(\log (S+B))\otimes \mathcal L^{-1}(-S)$ is a 
quasi-isomorphism (cf.~\cite[2.9.~Properties]{ev}). 
On the other hand, the following spectral sequence 
\begin{eqnarray*}
\lefteqn{E^{pq}_1=H^q(X, \Omega^p_X(\log (S+B))
\otimes \mathcal L^{-1}(-S))}\\
& &\Longrightarrow \mathbb H^{p+q}(X, 
\Omega^{\bullet}_X(\log (S+B))\otimes 
\mathcal L^{-1}(-S))
\end{eqnarray*}
degenerates in $E_1$. 
This follows from 
the $E_1$-degeneration of 
$$
H^{q}(Y, \widetilde \Omega^p_Y(\log T)(-T))
\Longrightarrow \mathbb H^{p+q}(Y, \widetilde {\Omega}^{\bullet}_Y
(\log T)(-T))
$$ 
where the right hand side is isomorphic to $H^{p+q}_c(Y-T, \mathbb C)$. 
We will discuss this $E_1$-degeneration in Section \ref{sec3}. 
For the details, see \ref{s5} in Section \ref{sec3} below. 
We note that $\Omega^{\bullet}_X(\log (S+B))
\otimes \mathcal L^{-1}(-S)$ is a 
direct summand of $\pi_*(\widetilde {\Omega}^{\bullet}_Y(\log T)(-T))$. 
We consider the following 
commutative diagram for any $q$. 
$$
\begin{CD}
\mathbb H^{q}(X, \Omega^{\bullet}_X(\log (S+B))\otimes \mathcal L^{-1}(-S))
@>{\alpha}>> H^q(X, \mathcal L^{-1}(-S)) \\ 
@AA{\gamma}A @AA{\beta}A \\ 
\mathbb H^{q}(X, \Omega^{\bullet}_X(\log (S+B))\otimes \mathcal L^{-1}(-S-D))
@>>> H^q(X, \mathcal L^{-1}(-S-D)) 
\end{CD}
$$
Since $\gamma$ is an isomorphism by the 
above quasi-isomorphism and $\alpha$ is surjective 
by the $E_1$-degeneration, we obtain that $\beta$ 
is surjective. 
By the Serre duality, we obtain 
$H^q(X, \mathcal O_X(K_X)\otimes \mathcal L(S))\to H^q(X, \mathcal 
O_X(K_X)\otimes \mathcal L(S+D))$ is 
injective for any $q$. 
This means that $H^q(X, \mathcal O_X(L))\to 
H^q(X, \mathcal O_X(L+D))$ is injective for any $q$. 
\end{proof}

The next result is a key result of this chapter. 

\begin{prop}[Fundamental injectivity theorem II]
\index{fundamental injectivity theorem}\label{2}
Let $(X, S+B)$ be a simple normal crossing 
pair such that $X$ is proper, $S+B$ is a boundary $\mathbb R$-divisor, 
$S$ is reduced, and 
$\llcorner B\lrcorner =0$. 
Let $L$ be a Cartier divisor on $X$ and let $D$ be an 
effective Cartier divisor whose support is contained 
in $\Supp B$. 
Assume that 
$L\sim_{\mathbb R}K_X+S+B$. Then 
the natural homomorphisms 
$$
H^q(X, \mathcal O_X(L))\to H^q(X, \mathcal O_X(L+D)), 
$$ 
which are induced by the inclusion 
$\mathcal O_X\to \mathcal O_X(D)$, 
are injective for all $q$. 
\end{prop}

\begin{proof} 
We can assume that $B$ is a $\mathbb Q$-divisor 
and $L\sim _{\mathbb Q}K_X+S+B$ by 
Lemma \ref{7}. 
Without loss of generality, we can assume that 
$X$ is connected. 
Let $\varepsilon: X^{\bullet}\to X$ be the Mayer--Vietoris 
simplicial resolution of $X$. 
Let $\nu$ be the smallest positive integer such that 
$\nu L\sim \nu (K_X+ S+ B)$. 
We put $\mathcal L=\mathcal O_X(L-K_X-S)$. 
We take the $\nu$-fold cyclic cover $\pi':Y'\to X$ associated 
to $\nu B\in |\mathcal L^{\nu}|$ as in the proof of Proposition \ref{1}. 
Let $\widetilde Y\to Y'$ be the normalization of $Y'$. 
We can glue 
$\widetilde Y$ 
naturally along the inverse image of $\varepsilon_1(X^1)\subset 
X$ and then obtain 
a connected 
reducible variety $Y$ and a finite morphism 
$\pi:Y\to X$. For a supplementary argument, see 
Remark \ref{su} below. 
We can construct the Mayer--Vietoris simplicial resolution 
$\varepsilon:Y^{\bullet}\to Y$ and 
a natural morphism $\pi_{\bullet}:Y^{\bullet}\to X^{\bullet}$. 
Note that Definition \ref{011} makes sense without any 
modifications though $Y$ has singularities. 
The finite morphism $\pi_0:Y^0\to X^0$ 
is essentially the same as the finite cover 
constructed in Proposition \ref{1}. 
Note that the inverse image of an irreducible 
component $X_{i}$ of $X$ by $\pi_0$ 
may be a disjoint union of copies of the finite 
cover constructed in the proof of 
Proposition \ref{1}. 
More precisely, let $V$ be any stratum of $X$. 
Then $\pi^{-1}(V)$ 
is not necessarily 
connected and $\pi:\pi^{-1}(V)\to V$ may be a disjoint union of 
copies of the finite cover constructed in the 
proof of 
the Proposition \ref{1}. 
Since $H^q(X^{\bullet}, (\mathcal L^{-1}(-S-D))^{\bullet})\simeq 
H^q(X, \mathcal L^{-1}(-S-D))$ and 
$H^q(X^{\bullet}, (\mathcal L^{-1}(-S))^{\bullet})\simeq 
H^q(X, \mathcal L^{-1}(-S))$, it is sufficient to 
see that 
$H^q(X^{\bullet}, (\mathcal L^{-1}(-S-D))^{\bullet})\to 
H^q(X^{\bullet}, (\mathcal L^{-1}(-S))^{\bullet})$ is surjective. 
First, we note that the natural inclusion 
$$
\Omega^{\bullet}_{X^n}(\log (S^n+B^n))\otimes 
(\mathcal L^{-1}(-S-D))^{n} \to 
\Omega ^{\bullet}_{X^n}(\log (S^n+B^n))\otimes 
(\mathcal L^{-1}(-S))^{n} 
$$ 
is a quasi-isomorphism for any $n\geq 0$ (cf.~\cite[2.9.~Properties]{ev}). 
So, 
$$
\Omega^{\bullet}_{X^{\bullet}}(\log (S^{\bullet}+B^{\bullet}))\otimes 
(\mathcal L^{-1}(-S-D))^{\bullet} \to 
\Omega ^{\bullet}_{X^{\bullet}}(\log (S^{\bullet}+B^{\bullet}))\otimes 
(\mathcal L^{-1}(-S)^{\bullet}) 
$$ 
is a quasi-isomorphism. 
We put $T=\pi^*S$. 
Then $\Omega^{\bullet}_{X^n}(\log (S^n +B^n))\otimes 
(\mathcal L^{-1}(-S))^{n}$ is a direct summand of 
${\pi_n}_*\widetilde {\Omega}^{\bullet}_Y(\log T^n)(-T^n)$ 
for any $n\geq 0$. 
Next, we can check that 
$$
E^{pq}_1=H^q(Y^{\bullet}, 
\widetilde {\Omega}^p_{Y^{\bullet}}(\log T^{\bullet})(-T^{\bullet}))
\Longrightarrow 
\mathbb H^{p+q}(Y, 
s(\widetilde {\Omega}^{\bullet}_{Y^{\bullet}}
(\log T^{\bullet})(-T^{\bullet}))) 
$$
degenerates in $E_1$. We will discuss this 
$E_1$-degeneration in Section \ref{sec3}.  See \ref{s6} in 
Section \ref{sec3}. 
The right hand side is isomorphic to 
$H^{p+q}_c(Y-T, \mathbb C)$. 
Therefore, 
\begin{eqnarray*}
\lefteqn{E^{pq}_1=H^q(X^{\bullet}, \Omega^p_{X^{\bullet}}(\log 
(S^{\bullet}+B^{\bullet}))\otimes 
(\mathcal L^{-1}(-S))^{\bullet})}\\
&&\Longrightarrow 
\mathbb H^{p+q}(X, 
s(\Omega^{\bullet}_{X^{\bullet}}(\log (S^{\bullet}+B^{\bullet})) 
\otimes (\mathcal L^{-1}(-S))^{\bullet})) 
\end{eqnarray*}
degenerates in $E_1$. 
Thus, we have the following commutative diagram. 
$$
\begin{CD}
\mathbb H^{q}(X, s(\Omega^{\bullet}_{X^{\bullet}}
(\log (S^{\bullet}+B^{\bullet}))\otimes (\mathcal L^{-1}(-S))^{\bullet}))
@>{\alpha}>> H^q(X^{\bullet}, (\mathcal L^{-1}(-S))^{\bullet}) \\ 
@AA{\gamma}A @AA{\beta}A \\ 
\mathbb H^{q}(X, 
s(\Omega^{\bullet}_{X^{\bullet}}
(\log (S^{\bullet}+B^{\bullet}))
\otimes (\mathcal L^{-1}(-S-D))^{\bullet}))
@>>> H^q(X^{\bullet}, (\mathcal L^{-1}(-S-D))^{\bullet}) 
\end{CD}
$$ 
As in the proof of Proposition \ref{1}, $\gamma$ is an isomorphism 
and $\alpha$ is surjective. 
Thus, $\beta$ is surjective. 
This implies the desired injectivity results. 
\end{proof}

\begin{rem}\label{su} 
For simplicity, we assume that 
$X=X_1\cup X_2$, where 
$X_1$ and $X_2$ are smooth, and that $V=X_1\cap X_2$ is irreducible. 
We consider the natural projection $p:\widetilde Y\to X$. 
We note that $\widetilde Y=\widetilde Y_1\coprod \widetilde Y_2$, 
where $\widetilde Y_i$ is the inverse image 
of $X_i$ by $p$ for $i=1$ and $2$. 
We put $p_i=p|_{\widetilde Y_i}$ for $i=1$ and $2$. 
It is easy to see that $p_1^{-1}(V)$ is 
isomorphic to $p_2^{-1}(V)$ over $V$. We denote 
it by $W$. 
We consider the following surjective 
$\mathcal O_X$-module homomorphism $\mu:
p_*\mathcal O_{\widetilde Y_1}\oplus 
p_*\mathcal O_{\widetilde Y_2}\to p_*\mathcal O_W: 
(f, g)\mapsto f|_W-g|_W$. Let $\mathcal A$ be the kernel 
of $\mu$. 
Then $\mathcal A$ is an $\mathcal O_X$-algebra and $\pi:Y
\to X$ is nothing but $\Spec _X\mathcal A\to X$. 
We can 
check that $\pi^{-1}(X_i)\simeq \widetilde Y_i$ for $i=1$ and $2$ 
and that $\pi^{-1}(V)\simeq W$. 
\end{rem}

\begin{rem}\label{rem-1}
As pointed out in the introduction, 
the proof of \cite[Theorem 3.1]{ambro} 
only implies that the homomorphisms 
$$H^q(X, \mathcal O_X(L-S))\to 
H^q(X, \mathcal O_X(L-S+D))$$ 
are injective for all $q$. 
When $S=0$, we do not need 
the mixed Hodge structure on the cohomology 
with compact support. The mixed Hodge structure on 
the usual singular 
cohomology is sufficient for the case when $S=0$. 
\end{rem}

We close this section with an easy 
application of Proposition \ref{2}. 
The following vanishing theorem is the Kodaira vanishing theorem 
for simple normal crossing varieties. 

\begin{cor}\label{3}
Let $X$ be a projective simple 
normal crossing variety and $\mathcal L$ 
an ample line bundle on $X$. 
Then $H^q(X, \mathcal O_X(K_X)\otimes \mathcal L)=0$ for 
any $q>0$. 
\end{cor}
\begin{proof}
We take a general member $B\in |\mathcal L^l|$ for 
some $l\gg 0$. 
Then we can find a Cartier divisor $M$ such that 
$M\sim _{\mathbb Q}K_X+\frac{1}{l}B$ and 
$\mathcal O_X(K_X)\otimes \mathcal L\simeq \mathcal O_X(M)$. 
By Proposition \ref{2}, we obtain injections 
$H^q(X, \mathcal O_X(M))\to H^q(X, \mathcal O_X(M+mB))$ for 
any $q$ and any positive integer $m$. 
Since $B$ is ample, Serre's vanishing theorem implies 
the desired vanishing theorem. 
\end{proof}

\section{$E_1$-degenerations of Hodge to de Rham type spectral sequences}
\label{sec3}

From \ref{s1} to \ref{s3}, we recall some well-known 
results on mixed Hodge structures. 
We use the notations in \cite{deligne} freely. 
The basic references on this topic are 
\cite[Section 8]{deligne}, \cite[Part II]{elzein}, 
and \cite[Chapitres 2 and 3]{elzein2}. 
The recent book \cite{ps} may be useful. 
The starting point is the pure Hodge structures on proper 
smooth algebraic varieties. 

\begin{say}\label{s1}
(Hodge structures for proper smooth varieties). 
Let $X$ be a proper smooth algebraic variety over 
$\mathbb C$. 
Then the triple $(\mathbb Z_X, 
(\Omega^{\bullet}_X, F), \alpha)$, 
where $\Omega^{\bullet}_X$ is the holomorphic de Rham complex 
with the filtration b\^ete $F$ and 
$\alpha:\mathbb C_X\to \Omega^{\bullet}_X$ is the inclusion, 
is a cohomological Hodge complex (CHC, for short) of weight 
zero. 
\end{say}
The next one is also a fundamental 
example. For the details, see \cite[I.1.]{elzein} or 
\cite[3.5]{elzein2}. 

\begin{say}\label{s2}
(Mixed Hodge structures for 
proper simple normal crossing varieties). 
Let $D$ be a proper simple normal crossing algebraic 
variety over $\mathbb C$. 
Let $\varepsilon:D^{\bullet}\to D$ be the Mayer--Vietoris simplicial 
resolution (cf.~Definition \ref{011}). 
The following complex of sheaves, denoted by 
$\mathbb Q_{D^{\bullet}}$, 
$$
0\to {\varepsilon_0}_*\mathbb Q_{D^0}
\to {\varepsilon_1}_*\mathbb Q_{D^1}\to \cdots 
\to {\varepsilon_k}_*\mathbb Q_{D^k}\to \cdots, 
$$ is 
a resolution of $\mathbb Q_D$. 
More explicitly, the differential 
$d_k: \varepsilon _{k*} \mathbb Q_{D^k} 
\to \varepsilon _{k+1*}\mathbb Q_{D^{k+1}}$ is $\sum _{j=0}^{k+1} 
(-1)^j\lambda^*_{j, k+1}$ for any $k\geq 0$. 
For the details, see \cite[I.1.]{elzein} or 
\cite[3.5.3]{elzein2}. 
We obtain the resolution $\Omega^{\bullet}_{D^{\bullet}}$ of 
$\mathbb C_D$ as follows, 
$$
0\to {\varepsilon_0}_*\Omega^{\bullet}_{D^0}
\to {\varepsilon_1}_*\Omega^{\bullet}_{D^1}\to \cdots 
\to {\varepsilon_k}_*\Omega^{\bullet}_{D^k}\to \cdots. 
$$ 
Of course, 
$d_k: \varepsilon _{k*} \Omega^{\bullet}_{D^k} 
\to \varepsilon _{k+1*}\Omega^{\bullet}_{D^{k+1}}$ is $\sum _{j=0}^{k+1} 
(-1)^j\lambda^*_{j, k+1}$. 
Let $s(\Omega^{\bullet}_{D^{\bullet}})$ be the simple complex 
associated to the double complex $\Omega^{\bullet}_{D^{\bullet}}$. 
The Hodge filtration $F$ on 
$s(\Omega^{\bullet}_{D^{\bullet}})$ is defined 
by  
$F^p=s(0\to \cdots \to 0\to 
\varepsilon _*\Omega^p_{D^{\bullet}}\to 
\varepsilon_*\Omega^{p+1}_{D^{\bullet}}\to \cdots)$. 
We note that $\varepsilon _*\Omega^p_{D^{\bullet}}
=(0\to \varepsilon _{0*}\Omega^p_{D^0}
\to \varepsilon _{1*}\Omega^p_{D^1}\to \cdots 
\to \varepsilon _{k*}\Omega^p_{D^k}\to \cdots)$. 
There exist natural weight filtrations $W$'s on 
$\mathbb Q_{D^{\bullet}}$ and $s(\Omega^{\bullet}_{D^{\bullet}})$. 
We omit the definition of the weight filtrations $W$'s on 
$\mathbb Q_{D^{\bullet}}$ and $s(\Omega^{\bullet}_{D^{\bullet}})$  
since we do not need their explicit descriptions. 
See \cite[I.1.]{elzein} or \cite[3.5.6]{elzein2}. 
Then $(\mathbb Z_{D}, (\mathbb Q_{D^{\bullet}}, W), 
(s(\Omega^{\bullet}_{D^{\bullet}}), W, F))$ is a cohomological 
mixed Hodge complex (CMHC, for short). 
This CMHC induces a natural mixed Hodge structure on $H^{\bullet}
(D, \mathbb Z)$. 
\end{say}
For the precise definitions of CHC and CMHC (CHMC, in French), 
see \cite[Section 8]{deligne} or 
\cite[Chapitre 3]{elzein2}. 
The third example is not so standard but is indispensable for our injectivity 
theorems. 

\begin{say}\label{s3}
(Mixed Hodge structure on the cohomology 
with compact support). 
Let $X$ be a proper 
smooth algebraic variety over $\mathbb C$ and 
$D$ a simple normal crossing divisor on $X$. 
We consider the mixed cone of $\mathbb Q_X\to 
\mathbb Q_{D^{\bullet}}$ with suitable shifts of complexes 
and weight filtrations (for the details, 
see \cite[I.3.]{elzein} or 
\cite[3.7.14]{elzein2}). We obtain a complex $\mathbb Q_{X-
D^{\bullet}}$, which is quasi-isomorphic to $j_{!}\mathbb Q_{X-D}$, 
where $j:X-D\to X$ is the natural open immersion, and 
a weight filtration $W$ on $\mathbb Q_{X-D^{\bullet}}$. 
We define in the same way, that is, 
by taking a cone of a morphism of complexes 
$\Omega^{\bullet}_X\to \Omega^{\bullet}_{D^{\bullet}}$, 
a complex $\Omega^{\bullet}_{X-D^{\bullet}}$ with 
filtrations $W$ and $F$. 
Then we obtain that the triple 
$(j_!\mathbb Z_{X-D}, (\mathbb Q_{X-D^{\bullet}}, 
W), (\Omega^{\bullet}_{X-D^{\bullet}}, W, F))$ is 
a CMHC. 
It defines a natural mixed Hodge structure on $H^{\bullet}_c(X-D, 
\mathbb Z)$. 
Since we can check that the complex 
\begin{eqnarray*}
0\to \Omega^{\bullet}_X(\log D)(-D)\to \Omega^{\bullet}_X
\to {\varepsilon_0}_*\Omega^{\bullet}_{D^0}\\
\to {\varepsilon_1}_*\Omega^{\bullet}_{D^1}\to 
\cdots 
\to {\varepsilon_k}_*\Omega^{\bullet}_{D^k}\to \cdots
\end{eqnarray*} 
is exact by direct local calculations, we see that 
$(\Omega^{\bullet}_{X-D^{\bullet}}, F)$ is quasi-isomorphic to 
$(\Omega^{\bullet}_X(\log D)(-D), F)$ in $D^+F(X, \mathbb C)$, 
where 
\begin{eqnarray*}
\lefteqn{F^p\Omega^{\bullet}_X(\log D)(-D)}\\ & & 
=(0\to \cdots \to 0 \to \Omega^p_X(\log D)(-D)\to 
\Omega^{p+1}_X(\log D)(-D)\to \cdots). 
\end{eqnarray*}
Therefore, the spectral sequence 
$$
E^{pq}_1=H^q(X, \Omega^p_X(\log D)(-D))
\Longrightarrow 
\mathbb H^{p+q}(X, \Omega^{\bullet}_X(\log D)(-D))
$$ 
degenerates in $E_1$ and the right hand side 
is isomorphic to $H^{p+q}_c(X-D, \mathbb C)$. 
\end{say}

From here, we treat mixed Hodge structures on 
much more complicated algebraic varieties. 

\begin{say}\label{s4} 
(Mixed Hodge structures for 
proper simple normal crossing pairs). 
Let $(X, D)$ be a proper simple normal crossing pair over $\mathbb C$ 
such that $D$ is reduced. 
Let $\varepsilon:X^{\bullet}\to X$ be the 
Mayer--Vietoris simplicial resolution of $X$. 
As we saw in the previous step, we have a CHMC 
$$(j_{n!}\mathbb Z_{X^n-D^n}, (\mathbb Q_{X^n-(D^n)^{\bullet}}, 
W), (\Omega^{\bullet}_{X^n-(D^n)^{\bullet}}, W, F))$$ on $X^n$, 
where $j_n:X^n-D^n\to X^n$ is the natural 
open immersion, and that $(\Omega^{\bullet}_{X^n-(D^n)^{\bullet}}, F)$ 
is quasi-isomorphic to $(\Omega^{\bullet}_{X^n}(\log D^n)(-D^n), 
F)$ in $D^+F(X^n, \mathbb C)$ for 
any $n\geq 0$. 
Therefore, by using the Mayer--Vietoris 
simplicial resolution $\varepsilon: X^{\bullet}\to X$, 
we can construct a CMHC 
$(j_!\mathbb Z_{X-D}, (K_{\mathbb Q}, 
W), (K_{\mathbb C}, W, F))$ 
on $X$ that induces a natural mixed Hodge structure on $H^{\bullet}_c
(X-D, \mathbb Z)$. We can see that 
$(K_{\mathbb C}, F)$ is quasi-isomorphic to $(s(\Omega^{\bullet}_{X^{\bullet}}
(\log D^{\bullet})(-D^{\bullet})), F)$ in $D^+F(X, \mathbb C)$, 
where 
\begin{eqnarray*}
F^p=s(0\to \cdots \to 0 \to 
\varepsilon _*\Omega^p_{X^{\bullet}}(\log D^{\bullet})(-D^{\bullet})
\\ 
\to \varepsilon _*\Omega^{p+1}_{X^{\bullet}}(\log D^{\bullet})(-D^{\bullet})
\to \cdots).
\end{eqnarray*}
We note that $\Omega^{\bullet}_{X^{\bullet}}(\log D^{\bullet})(-D^{\bullet})$ is 
the double complex 
\begin{eqnarray*}
0\to {\varepsilon _0}_*\Omega^{\bullet}_{X^0}(\log D^0)(-D^0) 
\to {\varepsilon _1}_*\Omega^{\bullet}_{X^1}(\log D^1)(-D^1) 
\to \cdots\\
\to 
{\varepsilon _k}_*\Omega^{\bullet}_{X^k}(\log D^k)(-D^k) 
\to \cdots. 
\end{eqnarray*}
Therefore, the spectral sequence 
$$
E^{pq}_1=H^q(X^{\bullet}, \Omega^p_{X^{\bullet}}(\log D^{\bullet})(-D
^{\bullet}))
\Longrightarrow 
\mathbb H^{p+q}(X, s(\Omega^{\bullet}_{X^{\bullet}}(\log D^{\bullet})
(-D^{\bullet})))
$$ 
degenerates in $E_1$ and the right hand side 
is isomorphic to $H^{p+q}_c(X-D, \mathbb C)$. 
\end{say} 

Let us go to the proof of the $E_1$-degeneration that 
we already used in the proof of Proposition \ref{1}. 

\begin{say}[$E_1$-degeneration for Proposition \ref{1}]\label{s5}
Here, we use the notation in the proof of 
Proposition \ref{1}. 
In this case, $Y$ has only quotient singularities. 
Then $(\mathbb Z_Y, (\widetilde \Omega^{\bullet}_Y, F), \alpha)$ 
is a CHC, where $F$ is the filtration b\^ete and 
$\alpha:\mathbb C_Y\to \widetilde {\Omega}^{\bullet}_Y$ is 
the inclusion. For the 
details, see \cite[(1.6)]{steenbrink}. 
It is easy to see that $T$ is a divisor with $V$-normal crossings 
on $Y$ (see \ref{v-n} or \cite[(1.16) Definition]{steenbrink}). 
We can easily check that $Y$ is singular only over the singular 
locus of $\Supp B$. 
Let $\varepsilon:T^{\bullet}\to T$ be the Mayer--Vietoris 
simplicial resolution. Though $T$ has singularities, 
Definition \ref{011} makes sense without any modifications. 
We note that 
$T^n$ has only quotient singularities for any $n\geq 0$ 
by the construction of $\pi:Y\to X$. 
We can also check that the same construction in \ref{s2} 
works with minor modifications and we have 
a CMHC $(\mathbb Z_T, (\mathbb Q_{T^{\bullet}}, W), 
(s(\widetilde {\Omega}^{\bullet}_{T^{\bullet}}), W, F))$ that 
induces a natural mixed Hodge structure on $H^{\bullet}(T, \mathbb Z)$. 
By the same arguments as in \ref{s3}, we can construct a triple 
$(j_!\mathbb Z_{Y-T}, (\mathbb Q_{Y-T^{\bullet}}, W), (K_{\mathbb C}, 
W, F))$, where $j:Y-T\to Y$ is the natural 
open immersion. 
It is a CHMC that induces a natural mixed Hodge structure on 
$H^{\bullet}_c
(Y-T, \mathbb Z)$ and $(K_{\mathbb C}, F)$ is quasi-isomorphic 
to $(\widetilde \Omega^{\bullet}_Y(\log T)(-T), F)$ 
in $D^+F(Y, \mathbb C)$, where 
\begin{eqnarray*}
\lefteqn{F^p\widetilde\Omega^{\bullet}_Y(\log T)(-T)}\\& &
=(0\to \cdots \to 0 \to \widetilde \Omega^p_Y(\log T)(-T)\to 
\widetilde\Omega^{p+1}_Y(\log T)(-T)\to \cdots). 
\end{eqnarray*} 
Therefore, the spectral sequence 
$$
E^{pq}_1=H^q(Y, \widetilde\Omega^p_Y(\log T)(-T))
\Longrightarrow 
\mathbb H^{p+q}(Y, \Omega^{\bullet}_Y(\log T)(-T))
$$ 
degenerates in $E_1$ and the right hand side 
is isomorphic to $H^{p+q}_c(Y-T, \mathbb C)$. 
\end{say}

The final one is the $E_1$-degeneration that we used in the 
proof of Proposition \ref{2}. 
It may be one of the main contributions of this chapter. 

\begin{say}[$E_1$-degeneration for Proposition \ref{2}]\label{s6} 
We use the notation in the proof of Proposition \ref{2}. 
Let $\varepsilon:Y^{\bullet}\to Y$ be the Mayer--Vietoris 
simplicial resolution. 
By the previous step, we can obtain a CHMC 
$$(j_{n!}\mathbb Z_{Y^n-T^n}, (\mathbb Q_{Y^n-(T^n)^{\bullet}}, 
W), (K_{\mathbb C}, 
W, F))$$ for each $n\geq 0$. 
Of course, $j_n:Y^n-T^n\to Y^n$ is the natural 
open immersion for any $n\geq 0$. 
Therefore, we can construct 
a CMHC 
$$(j_!\mathbb Z_{Y-T}, (K_{\mathbb Q}, W), (K_{\mathbb C}, 
W, F))$$ on $Y$. 
It induces a natural mixed Hodge structure on $H^{\bullet}_c(Y-T, \mathbb Z)$. 
We note that $(K_{\mathbb C}, F)$ is quasi-isomorphic to 
$(s(\widetilde \Omega^{\bullet}_{Y^{\bullet}}(\log T^{\bullet})(-T^{\bullet})), 
F)$ in $D^+F(Y, \mathbb C)$, 
where 
\begin{eqnarray*}
F^p=s(0\to \cdots \to 0 \to 
\varepsilon _*\widetilde \Omega^p_{Y^{\bullet}}(\log T^{\bullet})(-T^{\bullet})
\\
\to \varepsilon _*\widetilde
\Omega^{p+1}_{Y^{\bullet}}(\log T^{\bullet})(-T^{\bullet})
\to \cdots).
\end{eqnarray*} 
See \ref{s4} above. 
Thus, the desired spectral sequence 
$$
E^{pq}_1=H^q(Y^{\bullet}, \widetilde 
\Omega^p_{Y^{\bullet}}(\log T^{\bullet})(-T
^{\bullet}))
\Longrightarrow 
\mathbb H^{p+q}(Y, s(\widetilde
\Omega^{\bullet}_{Y^{\bullet}}(\log T^{\bullet})
(-T^{\bullet})))
$$ 
degenerates in $E_1$. 
It is what we need in the proof of Proposition \ref{2}. 
Note that 
$
\mathbb H^{p+q}(Y, s(\widetilde
\Omega^{\bullet}_{Y^{\bullet}}(\log T^{\bullet})
(-T^{\bullet})))
\simeq H^{p+q}_c(Y-T, \mathbb C)$. 
\end{say}

\section{Vanishing and injectivity theorems}\label{sec4}

The main purpose of this section is to prove Ambro's 
theorems (cf.~\cite[Theorems 3.1 and 3.2]{ambro}) 
for embedded simple normal crossing pairs. 
The next lemma (cf.~\cite[Proposition 1.11]{fujino-high}) 
is missing in the proof of 
\cite[Theorem 3.1]{ambro}. It justifies the 
first three lines in the proof of \cite[Theorem 3.1]{ambro}. 

\begin{lem}[Relative vanishing lemma]\label{re-vani-lem} 
Let $f:Y\to X$ be a proper morphism 
from a simple normal crossing 
pair $(Y, T+D)$ 
such that $T+D$ is a boundary 
$\mathbb R$-divisor, 
$T$ is reduced, and $\llcorner D\lrcorner =0$. 
We assume that $f$ is an isomorphism 
at the generic point of any stratum of the pair $(Y, T+D)$. 
Let $L$ be a Cartier divisor on $Y$ such 
that $L\sim _{\mathbb R}K_Y+T+D$. 
Then $R^qf_*\mathcal O_Y(L)=0$ for $q>0$. 
\end{lem} 

\begin{proof} 
By Lemma \ref{7}, we can assume that 
$D$ is a $\mathbb Q$-divisor and 
$L\sim _{\mathbb Q}K_Y+T+D$. 
We divide the proof into two steps. 
\setcounter{step}{0}
\begin{step}\label{1ne}
We assume that $Y$ is irreducible. 
In this case, $L-(K_Y+T+D)$ is 
nef and log big over $X$ with respect to the pair $(Y, T+D)$ 
(see Definition \ref{2-46}). 
Therefore, $R^qf_*\mathcal O_Y(L)=0$ for any $q>0$ 
by the vanishing theorem (see, for example, Lemma \ref{vani-rf-le}). 
\end{step}
\begin{step}
Let $Y_1$ be an irreducible component of $Y$ and $Y_2$ the union 
of the other irreducible components of $Y$. 
Then we have a short exact sequence 
$0\to i_*\mathcal O_{Y_1}(-Y_2|_{Y_1})\to 
\mathcal O_Y\to \mathcal O_{Y_2}\to 0$, 
where $i:Y_1\to Y$ is the natural closed immersion 
$($cf.~\cite[Remark 2.6]{ambro}$)$. 
We put $L'=L|_{Y_1}-{Y_2}|_{Y_1}$. 
Then we have a short exact sequence 
$0\to i_*\mathcal O_{Y_1}(L')\to 
\mathcal O_Y(L)\to \mathcal O_{Y_2}(L|_{Y_2})\to 0$ 
and $L'\sim _{\mathbb Q}K_{Y_1}+T|_{Y_1}+D|_{Y_1}$. 
On the other hand, we can check 
that $L|_{Y_2}\sim 
_{\mathbb Q}K_{Y_2}+{Y_1}|_{Y_2}+T|_{Y_2}
+D|_{Y_2}$. 
We have already known that $R^qf_*\mathcal O_{Y_1}(L')=0$ 
for any $q>0$ by Step \ref{1ne}. 
By the induction on the number of 
the irreducible components of 
$Y$, we have $R^qf_*\mathcal O_{Y_2}(L|_{Y_2})=0$ 
for any $q>0$. 
Therefore, $R^qf_*\mathcal O_Y(L)=0$ for any 
$q>0$ by the exact sequence: 
$$\cdots 
\to R^qf_*\mathcal O_{Y_1}(L')\to 
R^qf_*\mathcal O_Y(L)\to R^qf_*\mathcal O_{Y_2}
(L|_{Y_2})\to \cdots. $$
\end{step} 
So, we finish the proof of Lemma \ref{re-vani-lem}. 
\end{proof}

The following lemma is a variant of Szab\'o's resolution lemma 
(see the resolution lemma in \ref{15-resol}). 

\begin{lem}\label{6} 
Let $(X, B)$ be an embedded simple normal 
crossing pair and $D$ a permissible Cartier 
divisor on $X$. Let $M$ be an ambient space of $X$. 
Assume that there exists an $\mathbb R$-divisor $A$ on $M$ 
such that $\Supp (A+X)$ is simple normal 
crossing on $M$ and that $B=A|_{X}$. 
Then there exists a projective birational 
morphism $g:N\to M$ from a smooth 
variety $N$ with the following properties. 
Let $Y$ be the strict transform of $X$ on $N$ and 
$f=g|_Y:Y\to X$. 
Then we have 
\begin{itemize}
\item[{\em{(i)}}] $g^{-1}(D)$ is a 
divisor on $N$. $\Exc (g)\cup g^{-1}_*(A+X)$ is 
simple normal crossing on $N$, where $\Exc(g)$ is 
the exceptional locus of $g$. 
In particular, $Y$ is a simple normal crossing divisor on $N$.  
\item[{\em{(ii)}}] $g$ and $f$ are isomorphisms outside $D$, in 
particular, $f_*\mathcal O_Y\simeq 
\mathcal O_X$. 
\item[{\em{(iii)}}] $f^*(D+B)$ has a simple normal crossing 
support on $Y$. More precisely, there exists an $\mathbb R$-divisor 
$A'$ on $N$ such that 
$\Supp (A'+Y)$ is simple normal crossing on $N$, $A'$ and $Y$ have 
no common irreducible components, and 
that $A'|_{Y}=f^*(D+B)$. 
\end{itemize}
\end{lem}
\begin{proof}
First, we take a blow-up $M_1\to M$ along $D$. 
Apply Hironaka's desingularization theorem 
to $M_1$ and obtain a projective birational morphism $M_2\to M_1$ 
from a smooth variety $M_2$. 
Let $F$ be the reduced divisor that coincides 
with the support of the inverse image of $D$ on $M_2$. 
Apply Szab\'o's resolution lemma to 
$\Supp \sigma^*(A+X)\cup F$ on $M_2$ 
(see, for example, \ref{15-resol} 
or \cite[3.5.~Resolution lemma]{fujino0}), 
where $\sigma:M_2\to M$. 
Then, we obtain desired projective birational 
morphisms $g:N\to M$ from 
a smooth variety $N$, and $f=g|_Y:Y\to X$, 
where $Y$ is the strict transform of $X$ on $N$, such that 
$Y$ is a simple normal crossing 
divisor on $N$, $g$ and $f$ are isomorphisms outside $D$, and 
$f^*(D+B)$ has a simple normal crossing support on $Y$. 
Since $f$ is an isomorphism outside $D$ and 
$D$ is permissible on $X$, $f$ is 
an isomorphism at the generic point of any 
stratum of $Y$. 
Therefore, every fiber of $f$ is connected and 
then $f_*\mathcal O_Y\simeq \mathcal O_X$. 
\end{proof}

\begin{rem}
In Lemma \ref{6}, 
we can directly check that 
$f_*\mathcal O_Y(K_Y)\simeq \mathcal O_X(K_X)$. By Lemma 
5.1, $R^qf_*\mathcal O_Y(K_Y)=0$ for 
$q>0$. 
Therefore, we obtain $f_*\mathcal O_Y\simeq \mathcal O_X$ 
and $R^qf_*\mathcal O_Y=0$ for 
any $q>0$ by the Grothendieck duality. 
\end{rem}

Here, we treat the compactification problem. 
It is because we can use the same technique as in the proof 
of Lemma \ref{6}. 
This lemma plays important roles in this section. 

\begin{lem}\label{comp}
Let $f:Z\to X$ be a proper morphism 
from an embedded simple normal crossing pair $(Z, B)$. 
Let $M$ be the ambient space of $Z$. 
Assume that there is an $\mathbb R$-divisor $A$ on $M$ such that 
$\Supp (A+Z)$ is simple normal crossing 
on $M$ and that $B=A|_Z$. 
Let $\overline X$ be a projective variety such that $\overline 
X$ contains $X$ as a Zariski open set. 
Then there exist a proper embedded simple normal crossing pair 
$(\overline Z, \overline B)$ that is a compactification of $(Z, B)$ 
and a proper morphism $\overline f: \overline Z\to 
\overline X$ that compactifies $f:Z\to X$. 
Moreover, $\Supp \overline B\cup 
\Supp (\overline Z\setminus Z)$ is a simple normal 
crossing divisor on $\overline Z$, and 
$\overline Z\setminus Z$ has no common 
irreducible components with $\overline B$. 
We note that $\overline B$ is $\mathbb R$-Cartier. 
Let $\overline M$, which is a compactification of 
$M$, be the ambient space of $(\overline Z, \overline B)$. 
Then, by the construction, 
we can find an $\mathbb R$-divisor 
$\overline A$ on $\overline M$ such that 
$\Supp (\overline A+\overline Z)$ is simple normal crossing 
on $\overline M$ and that $\overline B=\overline A|_{\overline Z}$. 
\end{lem}
\begin{proof}
Let $\overline Z, \overline A\subset \overline M$ be any compactification. 
By blowing up $\overline M$ inside $\overline Z\setminus Z$, 
we can assume that $f:Z\to X$ extends to $\overline f: 
\overline Z\to \overline X$. 
By Hironaka's desingularization and the resolution lemma, 
we can assume that $\overline M$ is smooth and $\Supp(\overline Z
+\overline A)\cup \Supp (\overline M\setminus M)$ 
is a simple normal crossing 
divisor on $\overline M$. 
It is not difficult to see that 
the above compactification has the desired 
properties. 
\end{proof}

\begin{rem}\label{rem-2} 
There exists a big trouble 
to compactify normal crossing varieties. 
When we treat normal crossing varieties, 
we can not directly compactify them.
For the details, see \cite[3.6.~Whitney umbrella]{fujino0}, 
especially, Corollary 3.6.10 and 
Remark 3.6.11 in \cite{fujino0}. 
Therefore, the first two lines in the 
proof of \cite[Theorem 3.2]{ambro} 
is nonsense.  
\end{rem}

It is the time to state the main injectivity theorem 
(cf.~\cite[Theorem 3.1]{ambro}) for embedded {\em{simple}} 
normal crossing pairs. 
For applications, this formulation seems to be sufficient. 
We note that we will recover \cite[Theorem 3.1]{ambro} 
in full generality in Section \ref{sec6} (see 
Theorem \ref{61}). 

\begin{thm}[{cf.~\cite[Theorem 3.1]{ambro}}]\label{5.1} 
Let $(X, S+B)$ be an embedded simple normal crossing 
pair such that $X$ is proper, $S+B$ is a boundary 
$\mathbb R$-divisor, $S$ is reduced, and $\llcorner 
B\lrcorner =0$. Let $L$ be a Cartier 
divisor on $X$ and $D$ an effective Cartier divisor that is permissible 
with 
respect to $(X, S+B)$. 
Assume the following conditions. 
\begin{itemize}
\item[{\em{(i)}}] $L\sim _{\mathbb R}K_X+S+B+H$, 
\item[{\em{(ii)}}] $H$ is a semi-ample $\mathbb R$-Cartier 
$\mathbb R$-divisor, and 
\item[{\em{(iii)}}] $tH\sim _{\mathbb R} D+D'$ for some 
positive real number $t$, where 
$D'$ is an effective $\mathbb R$-Cartier 
$\mathbb R$-divisor that is permissible with respect to $(X, S+B)$. 
\end{itemize}
Then the homomorphisms 
$$
H^q(X, \mathcal O_X(L))\to H^q(X, \mathcal O_X(L+D)), 
$$ 
which are induced by the natural inclusion 
$\mathcal O_X\to \mathcal O_X(D)$, 
are injective for all $q$. 
\end{thm}

\begin{proof}
First, we use Lemma \ref{useful-lemma}. Thus, 
we can assume that there exists a divisor 
$A$ on $M$, where $M$ is the ambient space of $X$, such that 
$\Supp (A+X)$ is simple normal crossing on $M$ 
and that $A|_X=S$. 
Apply Lemma \ref{6} to an embedded simple normal crossing 
pair $(X, S)$ and a divisor $\Supp (D+D'+B)$ on $X$. 
Then we obtain a projective birational 
morphism $f:Y\to X$ from 
an embedded simple normal crossing 
variety $Y$ such that $f$ is an isomorphism 
outside $\Supp (D+D'+B)$, and that the 
union of the support of $f^*(S+B+D+D')$ and the 
exceptional locus of $f$ has a simple normal crossing support on $Y$. 
Let $B'$ be the strict transform of $B$ on $Y$. 
We can assume that $\Supp B'$ is disjoint from 
any strata of $Y$ that are 
not irreducible components of $Y$ 
by taking blow-ups. 
We write $K_Y+S'+B'=f^*(K_X+S+B)+E$, 
where $S'$ is the strict transform of $S$, and 
$E$ is $f$-exceptional. 
By the construction of $f:Y\to X$, 
$S'$ is Cartier and $B'$ is $\mathbb R$-Cartier. 
Therefore, $E$ is also $\mathbb R$-Cartier. 
It is easy to see that $E_+=\ulcorner E\urcorner \geq 0$. 
We put $L'=f^*L+E_+$ and $E_-=E_+-E\geq 0$. 
We note that $E_+$ is Cartier and $E_-$ is 
$\mathbb R$-Cartier 
because $\Supp E$ is simple normal crossing on 
$Y$. 
Since $f^*H$ is an $\mathbb R_{>0}$-linear 
combination of semi-ample Cartier 
divisors, we can write $f^*H\sim _{\mathbb R}
\sum _i a_i H_i$, where 
$0< a_i <1$ and $H_i$ is a general 
Cartier divisor on $Y$ for any $i$. 
We put $B''=B'+E_-+\frac{\varepsilon}{t} 
f^*(D+D')+(1-\varepsilon) \sum _i a_i H_i$ for 
some $0<\varepsilon \ll 1$. 
Then $L'\sim _{\mathbb R}K_Y+S'+B''$. 
By the construction, $\llcorner B''\lrcorner =0$, the 
support of $S'+B''$ is simple normal crossing 
on $Y$, and $\Supp B''\supset \Supp f^*D$. 
So, Proposition \ref{2} implies that 
the homomorphisms 
$H^q(Y, \mathcal O_Y(L'))\to H^q(Y, \mathcal O_Y(L'+f^*D))$ are 
injective for all $q$.
By Lemma \ref{re-vani-lem}, $R^qf_*\mathcal O_Y(L')=0$ for 
any $q>0$ and it is easy to see that $f_*\mathcal O_Y(L')\simeq 
\mathcal O_X(L)$. By the 
Leray spectral sequence, 
the homomorphisms $H^q(X, \mathcal O_X(L))\to 
H^q(X, \mathcal O_X(L+D))$ are injective for all $q$. 
\end{proof}

The following theorem is another main theorem of 
this section. It is essentially the same as \cite[Theorem 3.2]{ambro}. 
We note that we assume that $(Y, S+B)$ is a 
{\em{simple}} normal crossing pair. 
It is a small but technically important difference. 
For the full statement, see Theorem \ref{62} below. 

\begin{thm}[{cf.~\cite[Theorem 3.2]{ambro}}]\label{8}
Let $(Y, S+B)$ be an embedded simple normal crossing 
pair such that $S+B$ is a boundary $\mathbb R$-divisor, 
$S$ is reduced, and $\llcorner B\lrcorner =0$. 
Let $f:Y\to X$ be a proper morphism and $L$ a Cartier 
divisor on $Y$ such that 
$H\sim _{\mathbb R}L-(K_Y+S+B)$ is $f$-semi-ample. 
\begin{itemize}
\item[{\em{(i)}}] every non-zero local section of $R^qf_*\mathcal O_Y(L)$ contains 
in its support the $f$-image of 
some strata of $(Y, S+B)$. 
\item[{\em{(ii)}}] 
let $\pi:X\to V$ be a projective morphism and 
assume that $H\sim _{\mathbb R}f^*H'$ for 
some $\pi$-ample $\mathbb R$-Cartier 
$\mathbb R$-divisor $H'$ on $X$. 
Then $R^qf_*\mathcal O_Y(L)$ is $\pi_*$-acyclic, that is, 
$R^p\pi_*R^qf_*\mathcal O_Y(L)=0$ for any $p>0$. 
\end{itemize}
\end{thm}
\begin{proof}
Let $M$ be the ambient space of $Y$. 
Then, by Lemma \ref{useful-lemma}, 
we can assume that there exists an $\mathbb R$-divisor $D$ on 
$M$ such that $\Supp (D+Y)$ is simple normal 
crossing on $M$ and that $D|_Y=S+B$. Therefore, 
we can use Lemma \ref{comp} in Step \ref{8-2} of (i) 
and Step \ref{o3} of (ii) below. 

(i) We have already proved a very spacial case in Lemma \ref{re-vani-lem}. 
The argument in Step 1 is not new and it is well known. 
\setcounter{step}{0}
\begin{step}
First, we assume that $X$ is projective. 
We can assume that $H$ is semi-ample by replacing 
$L$ (resp.~$H$) with 
$L+f^*A'$ (resp.~$H+f^*A'$), where 
$A'$ is a very ample Cartier divisor. 
Assume that $R^qf_*\mathcal O_Y(L)$ 
has a local section whose support does not contain 
the image of any 
$(Y, S+B)$-stratum. Then 
we can find a very ample Cartier divisor 
$A$ with the following properties. 
\begin{itemize}
\item[(a)] $f^*A$ is permissible with respect to 
$(Y, S+B)$, and 
\item[(b)] $R^qf_*\mathcal O_Y(L)\to R^qf_*\mathcal O_Y(L)\otimes 
\mathcal O_X(A)$ is not injective. 
\end{itemize}
We can assume that $H-f^*A$ is semi-ample by replacing 
$L$ (resp.~$H$) with 
$L+f^*A$ (resp.~$H+f^*A$). 
If necessary, we replace $L$ (resp.~$H$) with 
$L+f^*A''$ (resp.~$H+f^*A''$), where $A''$ is a very ample 
Cartier divisor. 
Then, we have $H^0(X, R^qf_*\mathcal O_Y(L))\simeq 
H^q(Y, \mathcal O_Y(L))$ and 
$H^0(X, R^qf_*\mathcal O_Y(L)\otimes \mathcal O_X(A))\simeq 
H^q(Y, \mathcal O_Y(L+f^*A))$. 
We obtain that 
$$H^0(X, R^qf_*\mathcal O_Y(L))
\to H^0(X, R^qf_*\mathcal O_Y(L)\otimes \mathcal O_X(A))
$$ 
is not injective by (b) if $A''$ is sufficiently ample. 
So, 
$H^q(Y, \mathcal O_Y(L))
\to 
H^q(Y, \mathcal O_Y(L+f^*A))$ is not injective. 
It contradicts Theorem \ref{5.1}. 
We finish the proof when $X$ is projective. 
\end{step}
\begin{step}\label{8-2}
Next, we assume that $X$ is not projective. 
Note that the problem is local. So, we can shrink $X$ and assume that 
$X$ is affine. 
By the argument similar to the one in Step 1 in the proof of 
(ii) below, we can assume that $H$ is a semi-ample $\mathbb Q$-Cartier 
$\mathbb Q$-divisor. 
We compactify $X$ and apply Lemma \ref{comp}. 
Then we obtain a compactification $\overline f: \overline Y\to \overline X$ of 
$f:Y\to X$. 
Let $\overline H$ be the closure of $H$ on $\overline Y$. If $\overline H$ is not a semi-ample $\mathbb Q$-Cartier 
$\mathbb Q$-divisor, then we take blowing-ups of $\overline Y$ inside $\overline Y\setminus Y$ 
and obtain a semi-ample $\mathbb Q$-Cartier $\mathbb Q$-divisor 
$\overline H$ on 
$\overline Y$ such that $\overline H|_{Y}=H$. 
Let $\overline L$ (resp.~$\overline B$, $\overline S$) be the 
closure of $L$ (resp.~$B$, $S$) on $\overline Y$. 
We note that $\overline H\sim _{\mathbb R}\overline L-(K_{\overline Y}
+\overline S +\overline B)$ does not necessarily hold. 
We can write $H+\sum _i a_i (f_i)=L-(K_Y+S+B)$, 
where $a_i$ is a real number and $f_i\in \Gamma(Y, \mathcal K_Y^*)$ for 
any $i$. 
We put $E=\overline H+\sum _i a_i (f_i)-
(\overline L-(K_{\overline Y}+\overline S+
\overline B))$.  
We replace $\overline L$ (resp.~$\overline B)$ with 
$\overline L+\ulcorner E\urcorner$ (resp.~$\overline B+\{-E\}$). 
Then we obtain the desired property of 
$R^q\overline f_*\mathcal O_{\overline Y}
(\overline L)$ since $\overline X$ is projective. We 
note that $\Supp E$ is in $\overline Y\setminus Y$. 
So, we finish the whole proof. 
\end{step}

(ii) We divide the proof into three steps. 
\setcounter{step}{0}
\begin{step} 
We assume that $\dim V=0$. 
The following arguments are well known and 
standard. 
We describe them for the reader's convenience. 
In this case, we can write $H'\sim _{\mathbb R}H'_1+H'_2$, 
where $H'_1$ (resp.~$H'_2$) is a $\pi$-ample $\mathbb Q$-Cartier 
$\mathbb Q$-divisor 
(resp.~$\pi$-ample $\mathbb R$-Cartier $\mathbb R$-divisor) on $X$. 
So, we can write $H'_2\sim _{\mathbb R}\sum _i a_i H_i$, 
where $0<a_i<1$ and $H_i$ is a general very ample Cartier 
divisor on $X$ for any $i$. 
Replacing $B$ (resp.~$H'$) with $B+\sum _i a_i f^*H_i$ 
(resp.~$H'_1$), we can assume that 
$H'$ is a $\pi$-ample $\mathbb Q$-Cartier $\mathbb Q$-divisor. 
We take a general member $A\in |mH'|$, 
where $m$ is a sufficiently large and divisible integer, such 
that $A'=f^*A$ and 
$R^qf_*\mathcal O_Y(L+A')$ is $\pi_*$-acyclic for all $q$. 
By (i), we have the following short exact sequences, 
$$
0\to R^qf_*\mathcal O_Y(L)\to 
R^qf_*\mathcal O_Y(L+A')
\to R^qf_*\mathcal O_{A'}(L+A')\to 0. 
$$ 
for any $q$. 
Note that $R^qf_*\mathcal O_{A'}(L+A')$ is $\pi_*$-acyclic by 
induction on $\dim X$ and $R^qf_*\mathcal O_Y(L+A')$ is also 
$\pi_*$-acyclic by the above assumption. 
Thus, $E^{pq}_2=0$ for $p\geq 2$ in the 
following commutative diagram of spectral sequences. 
$$
\xymatrix{
&E^{pq}_2=R^p\pi_*R^qf_*\mathcal O_Y(L) \ar[d]_{\varphi^{pq}}
\ar@{=>}[r]&R^{p+q}(\pi\circ f)_*
\mathcal O_Y(L)\ar[d]_{\varphi^{p+q}}\\ 
&\overline{E}^{pq}_2=R^p\pi_*R^qf_*\mathcal O_Y(L+A') 
\ar@{=>}[r]& R^{p+q}
(\pi\circ f)_*
\mathcal O_Y(L+A')
}
$$
Since $\varphi^{1+q}$ is injective by Theorem \ref{5.1}, 
$E^{1q}_2\to R^{1+q}(\pi\circ f)_*\mathcal O_Y(L)$ 
is injective by the fact that 
$E^{pq}_2=0$ for $p\geq 2$, and $\overline 
{E}^{1q}_2=0$ by the above assumption, we 
have $E^{1q}_2=0$. 
This implies that $R^p\pi_*R^qf_*\mathcal O_Y(L)=0$ for any 
$p>0$. 
\end{step}

\begin{step}\label{o2}  
We assume that $V$ is projective. 
By replacing $H'$ (resp.~$L$) with 
$H'+\pi^*G$ (resp.~$L+(\pi\circ f)^*G$), where 
$G$ is a very ample Cartier divisor on $V$, 
we can assume that $H'$ is an ample $\mathbb R$-Cartier 
$\mathbb R$-divisor. 
By the same argument as in Step 1, we can assume that $H'$ is ample 
$\mathbb Q$-Cartier $\mathbb Q$-divisor and $H\sim _{\mathbb Q}f^*H'$. 
If $G$ is a 
sufficiently ample Cartier 
divisor on $V$, $H^k(V, R^p\pi_*R^qf_*\mathcal O_Y(L)\otimes G)
=0$ for 
any $k\geq 1$, $H^0(V, R^p\pi_*R^qf_*\mathcal O_Y(L)\otimes G)
\simeq H^p(X, R^qf_*\mathcal O_Y(L)\otimes \pi^*G)$, 
and $R^p\pi_*R^qf_*\mathcal O_Y(L)\otimes G$ is generated by its 
global sections. 
Since $H+f^*\pi^*G\sim _{\mathbb R}
L+f^*\pi^*G-(K_Y+S+B)$, 
$H+f^*\pi^*G\sim _{\mathbb Q}f^*(H'+\pi^*G)$, 
and $H'+\pi^*G$ is ample, 
we can apply Step 1 and obtain $H^p(X, R^qf_*\mathcal O_Y(L+
f^*\pi^*G))=0$ for any $p>0$. 
Thus, $R^p\pi_*R^qf_*\mathcal 
O_Y(L)=0$ for any $p>0$ by the above arguments.
\end{step}
\begin{step}\label{o3}  
When $V$ is not projective, we shrink $V$ and assume 
that $V$ is affine. 
By the same argument as in Step 1 above, we can assume that 
$H'$ is $\mathbb Q$-Cartier. 
We compactify $V$ and $X$, and can assume that 
$V$ and $X$ are projective. 
By Lemma \ref{comp}, we can reduce it to the case when 
$V$ is projective. This step is essentially the same as Step 
2 in the proof of (i). So, we omit the details here.  
\end{step}
We finish the whole proof of (ii). 
\end{proof}

\begin{rem}\label{smooth-ne} 
In Theorem \ref{5.1}, if $X$ is smooth, then 
Proposition \ref{1} is enough for the proof of Theorem \ref{5.1}. 
In the proof of Theorem \ref{8}, if $Y$ is smooth, then 
Theorem \ref{5.1} for a smooth $X$ is sufficient. 
Lemmas \ref{re-vani-lem}, 
\ref{6}, and \ref{comp} are easy and 
well known for smooth varieties. 
Therefore, the reader can find that our proof of Theorem \ref{8} 
becomes much easier if we assume that $Y$ is smooth. 
Ambro's original proof of \cite[Theorem 3.2 (ii)]{ambro} 
used embedded simple normal crossing pairs even 
when $Y$ is smooth (see (b) in the proof of 
\cite[Theorem 3.2 (ii)]{ambro}). 
It may be a technically important difference. 
I could not follow Ambro's original 
argument in (a) in the proof of 
\cite[Theorem 3.2 (ii)]{ambro}. 
\end{rem}

\begin{rem}\label{9-1}
It is easy to see that Theorem \ref{5.1} is a generalization of 
Koll\'ar's injectivity theorem. 
Theorem \ref{8} (i) (resp.~(ii)) is a generalization of 
Koll\'ar's torsion-free (resp.~vanishing) theorem. 
\end{rem}

We treat an easy vanishing theorem 
for lc pairs as an application of Theorem \ref{8} (ii). 
It seems to be buried in \cite{ambro}. 
We note that we do not need the notion of embedded simple normal 
crossing pairs to prove Theorem \ref{lc}. See Remark \ref{smooth-ne}. 

\begin{thm}[Kodaira vanishing theorem for lc pairs]
\index{Kodaira vanishing theorem}\label{lc} 
Let $(X, B)$ be an lc pair such that 
$B$ is a boundary $\mathbb R$-divisor. 
Let $L$ be a $\mathbb Q$-Cartier Weil divisor on $X$ such that 
$L-(K_X+B)$ is $\pi$-ample, where 
$\pi:X\to V$ is a projective morphism. 
Then $R^q\pi_*\mathcal O_X(L)=0$ for any $q>0$.  
\end{thm}
\begin{proof}
Let $f:Y\to X$ be a log resolution of $(X, B)$ such 
that $K_Y=f^*(K_X+B)+\sum _i a_i E_i$ with 
$a_i\geq -1$ for any $i$. 
We can assume that $\sum _i E_i\cup \Supp f^*L$ is a 
simple normal crossing divisor on $Y$. 
We put $E=\sum _i a_i E_i$ and $F=\sum _{a_j=-1}(1-b_j)E_j$, 
where $b_j=\mult _{E_j}\{f^*L\}$. 
We note that $A=L-(K_X+B)$ is $\pi$-ample by the 
assumption. So, we have 
$f^*A=f^*L-f^*(K_X+B)=
\ulcorner f^*L+E+F\urcorner-(K_Y+F+\{-(f^*L+E+F)\})$. 
We can easily check that $f_*\mathcal O_Y(\ulcorner 
f^*L+E+F\urcorner)\simeq \mathcal O_X(L)$ and 
that $F+\{-(f^*L+E+F)\}$ has a simple 
normal crossing support and is a boundary $\mathbb R$-divisor 
on $Y$. By Theorem \ref{8} (ii), we obtain that 
$\mathcal O_X(L)$ is $\pi_*$-acyclic. 
Thus, we have $R^q\pi_*\mathcal O_X(L)=0$ for 
any $q>0$. 
\end{proof}
We note that Theorem \ref{lc} contains 
a complete form of \cite[Theorem 0.3]{kovacs} as a 
corollary. For the related topics, 
see \cite[Corollary 1.3]{kss}. 

\begin{cor}[Kodaira vanishing theorem for lc varieties]\label{lcvar}  
Let $X$ be a projective lc variety and $L$ 
an ample Cartier divisor on $X$. 
Then $$H^q(X, \mathcal O_X(K_X+L))=0$$ for any $q>0$. 
Furthermore, if we assume that $X$ is Cohen--Macaulay, 
then $H^q(X, \mathcal O_X(-L))=0$ for any $q<\dim X$. 
\end{cor}

\begin{rem}
We can see that Corollary \ref{lcvar} is contained in 
\cite[Theorem 2.6]{fujino-high}, 
which is a very special case of Theorem \ref{8} (ii). 
I forgot to state Corollary \ref{lcvar} explicitly in \cite{fujino-high}. 
There, we do not need embedded simple normal crossing pairs. 
We note that there are typos in the proof of 
\cite[Theorem 2.6]{fujino-high}. 
In the commutative diagram, $R^if_*\omega_X(D)$'s 
should be replaced by $R^jf_*\omega_X(D)$'s. 
\end{rem}

We close this section with an easy example. 

\begin{ex}
Let $X$ be a projective lc threefold which has the 
following properties:~(i) there 
exists a projective birational morphism $f:Y\to X$ from a 
smooth projective threefold, and (ii) the exceptional 
locus $E$ of $f$ is an Abelian surface with 
$K_Y=f^*K_X-E$. 
For example, $X$ is a cone over a normally projective 
Abelian surface in $\mathbb P^N$ 
and $f:Y\to X$ is the blow-up 
at the vertex of $X$. 
Let $L$ be an ample Cartier divisor on $X$. 
By the Leray spectral sequence, 
we have 
\begin{align*}
0&\to H^1(X, f_*f^*\mathcal O_X(-L))\to 
H^1(Y, f^*\mathcal O_X(-L))\to 
H^0(X, R^1f_*f^*\mathcal O_X(-L))\\ &\to 
H^2(X, f_*f^*\mathcal O_X(-L))\to 
H^2(Y, f^*\mathcal O_X(-L))\to \cdots. 
\end{align*} 
Therefore, we obtain 
$$H^2(X, \mathcal O_X(-L))\simeq 
H^0(X, \mathcal O_X(-L)\otimes R^1f_*\mathcal O_Y), $$
because $H^1(Y, f^*\mathcal O_X(-L))=H^2
(Y, f^*\mathcal O_X(-L))=0$ by the 
Kawamata--Viehweg vanishing theorem. 
On the other hand, we have $R^qf_*\mathcal O_Y\simeq 
H^q(E, \mathcal O_E)$ for any $q>0$ since 
$R^qf_*\mathcal O_Y(-E)=0$ for 
every $q>0$. Thus, $H^2(X, \mathcal O_X(-L))\simeq 
\mathbb C^2$. In particular, $H^2(X, \mathcal O_X(-L))\ne 0$. 
We note that $X$ is not Cohen--Macaulay. 
In the above example, if we assume that $E$ is a $K3$-surface, 
then $H^q(X, \mathcal O_X(-L))=0$ for 
$q<3$ and $X$ is Cohen--Macaulay. 
For the details, 
see the subsection \ref{431sss}, especially, 
Lemma \ref{437lem}. 
\end{ex}

\section{Some further generalizations}\label{genera} 
Here, we treat some generalizations of Theorem \ref{8}. 
First, we introduce the notion of nef and log big (resp.~nef and 
log abundant) divisors. 
See also Definition \ref{nlb-def}. 

\begin{defn}\index{nef and log 
big divisor}\index{nef and 
log abundant divisor}\label{2-46}
Let $f:(Y, B)\to X$ be a proper morphism 
from a simple normal crossing 
pair $(Y, B)$ such that 
$B$ is a subboundary. 
Let $\pi:X\to V$ be a proper 
morphism and $H$ an $\mathbb R$-Cartier $\mathbb R$-divisor 
on $X$. 
We say that $H$ is {\em{nef and log big}} 
(resp.~{\em{nef and log abundant}}) 
over $V$ if and only if $H|_C$ is nef and big (resp.~nef and 
abundant) over $V$ for any $C$, 
where $C$ is the image 
of a stratum of $(Y, B)$. When $(X, B_X)$ is 
an lc pair, we choose a log resolution of $(X, B_X)$ 
to be $f:(Y, B)\to X$, 
where $K_Y+B=f^*(K_X+B_X)$. 
\end{defn}

We can generalize Theorem \ref{8} as follows. 
It is \cite[Theorem 7.4]{ambro} for embedded {\em{simple}} normal crossing 
pairs. His idea of the proof is very clever. 

\begin{thm}[{cf.~\cite[Theorem 7.4]{ambro}}]\label{74} 
Let $f:(Y, S+B)\to X$ be a proper morphism from an embedded simple 
normal crossing pair such that 
$S+B$ is a boundary $\mathbb R$-divisor, 
$S$ is reduced, and $\llcorner B\lrcorner=0$. 
Let $L$ be a Cartier divisor on $Y$ and $\pi:X\to V$ 
a proper morphism. 
Assume that $f^*H\sim _{\mathbb R}L-(K_Y+S+B)$, 
where $H$ is nef and log big over $V$. 
Then 
\begin{itemize}
\item[{\em{(i)}}] 
every non-zero local section of $R^qf_*\mathcal O_Y(L)$ contains 
in its support the $f$-image of some strata of $(Y, S+B)$, and 
\item[{\em{(ii)}}] $R^qf_*\mathcal O_Y(L)$ is $\pi_*$-acyclic. 
\end{itemize}
\end{thm}

\begin{proof}We note that we can assume that $V$ is affine 
without loss of generality. 
By using Lemma \ref{useful-lemma},  
we can assume that there exists a divisor 
$D$ on $M$, where $M$ is the ambient space of $Y$, such that 
$\Supp (D+Y)$ is simple normal crossing on $M$ 
and that $D|_Y=S+B$. 

\setcounter{step}{0}
\begin{step}\label{so-no-1} 
We assume that each stratum of $(Y, S+B)$ dominates some 
irreducible 
component of $X$. 
By taking the Stein factorization, 
we can assume that $f$ has connected fibers. 
Then we can assume that $X$ is irreducible and 
each stratum of $(Y, S+B)$ dominates $X$. 
By Chow's lemma, there exists a projective birational 
morphism $\mu:X'\to X$ such that 
$\pi':X'\to V$ is projective. 
By taking blow-ups $\varphi: Y'\to Y$ 
that is an isomorphism over the generic point 
of any stratum of 
$(Y, S+B)$, 
we have the following commutative diagram. 
$$
\begin{CD}
Y'@>{\varphi}>>Y\\ 
@V{g}VV @VV{f}V\\
X'@>>{\mu}> X
\end{CD}
$$ 
Then, we can write 
$$
K_{Y'}+S'+B'=\varphi^*(K_Y+S+B)+E, 
$$
where 
\begin{itemize}
\item[(1)] 
$(Y', S'+B')$ is a global embedded simple normal crossing 
pair such that 
$S'+B'$ is a boundary $\mathbb R$-divisor, 
$S'$ is reduced, and $\llcorner B'\lrcorner=0$. 
\item[(2)] $E$ is an effective $\varphi$-exceptional 
Cartier divisor. 
\item[(3)] Each stratum of $(Y', S'+B')$ dominates 
$X'$. 
\end{itemize}
We note that each stratum of $(Y, S+B)$ dominates $X$. 
Therefore, 
$$
\varphi^*L+E\sim _{\mathbb R}K_{Y'}+S'+B'+\varphi^*f^*H. 
$$
We note that 
$\varphi_*\mathcal 
O_{Y'}(\varphi^*L+E)\simeq \mathcal O_Y(L)$ and 
$R^i\varphi_*\mathcal O_{Y'}(\varphi^*L+E)=0$ for 
any $i>0$ by 
Theorem \ref{8} (i). 
Thus, we can assume that 
$\varphi:Y'\to Y$ is an identity, that is, 
we have 
$$
\begin{CD}
Y@=Y\\ 
@V{g}VV @VV{f}V\\
X'@>>{\mu}> X. 
\end{CD}
$$
We put $\mathcal F=R^qg_*\mathcal O_Y(L)$. 
Since $\mu^*H$ is nef and big over $V$ and 
$\pi':X'\to V$ is projective, 
we can write 
$\mu^*H=E+A$, where 
$A$ is a $\pi'$-ample $\mathbb R$-divisor 
on $X'$ and $E$ is an effective $\mathbb R$-divisor. 
By the same arguments as above, we take some blow-ups 
and can assume that 
$(Y, S+B+g^*E)$ is a global 
embedded simple normal crossing pair. 
If $k\gg 1$, then $\llcorner B+\frac{1}{k}g^*E
\lrcorner 
=0$, 
$$\mu^*H=\frac{1}{k}E+\frac{1}{k}A+\frac{k-1}{k}\mu^*
H, $$ 
and 
$$\frac{1}{k}A+\frac{k-1}{k}\mu^*H$$ is $\pi'$-ample. 
Thus, $\mathcal F$ is $\mu_*$-acyclic and $(\pi\circ 
\mu)_*=\pi'_*$-acyclic by Theorem \ref{8} 
(ii). 
We note that 
$$
L-\Big(K_Y+S+B+\frac{1}{k}g^*E\Big)\sim _{\mathbb R}
g^*\Big(\frac{1}{k}A+\frac{k-1}{k}\mu^*H\Big). 
$$
So, we have $R^qf_*\mathcal O_Y(L)\simeq \mu_*\mathcal F$ and 
$R^qf_*\mathcal O_Y(L)$ is $\pi_*$-acyclic. 
It is easy to see that $\mathcal F$ is torsion-free by Theorem 
\ref{8} (i). 
Therefore, $R^qf_*\mathcal O_Y(L)$ is also 
torsion-free. 
Thus, we finish the proof when each stratum of $(Y, S+B)$ dominates some irreducible component of $X$. 
\end{step}
\begin{step}\label{so-no-2} 
We treat the general case by induction on $\dim f(Y)$. 
By taking embedded log transformation 
(see Lemma \ref{useful-lem2}), 
we can decompose $Y=Y'\cup Y''$ as follows:~$Y'$ is 
the union of all strata of $(Y, S+B)$ that are not mapped 
to irreducible components of $X$ and $Y''=Y-Y'$. 
We put $K_{Y''}+B_{Y''}=(K_Y+S+B)|_{Y''}-Y'|_{Y''}$. 
Then $f:(Y'', B_{Y''})\to X$ and 
$L''=L|_{Y''}-Y'|_{Y''}$ satisfy the assumption in Step \ref{so-no-1}. We consider the following short exact sequence 
$$
0\to \mathcal O_{Y''}(L'')\to \mathcal O_Y(L)\to \mathcal 
O_{Y'}(L)\to 0. 
$$ 
By taking $R^qf_*$, 
we have short exact sequence 
$$
0\to R^qf_*\mathcal O_{Y''}(L'')\to 
R^qf_*\mathcal O_Y(L)\to R^qf_*\mathcal O_{Y'}(L)\to 0 
$$ 
for any $q$ by Step \ref{so-no-1}. 
It is because 
the connecting homomorphisms 
$R^qf_*\mathcal O_{Y'}(L)\to 
R^{q+1}f_*\mathcal O_{Y''}(L'')$ are zero maps 
by Step \ref{so-no-1}. 
Since (i) and (ii) hold for the first and third members 
by Step \ref{so-no-1} and by induction on dimension, 
respectively, they also hold for $R^qf_*\mathcal O_Y(L)$. 
\end{step}
So, we finish the proof. 
\end{proof}

In Step \ref{so-no-2} 
in the proof of Theorem \ref{74}, 
we used the embedded log 
transformation and 
the d\'evissage (see \cite[Remark 2.6]{ambro}). 
So, we need the notion of embedded simple normal crossing 
pairs to prove Theorem \ref{74} even when $Y$ is 
smooth. It is a key point. 

As a corollary of Theorem \ref{74}, 
we can prove the following vanishing theorem, which is 
stated implicitly in the introduction of \cite{ambro}. 
It is the culmination of the works of several authors:~Kawamata, 
Viehweg, Nadel, Reid, Fukuda, Ambro, and 
many others (cf.~\cite[Theorem 1-2-5]{kmm}). 

\begin{thm}\label{kvn}
Let $(X, B)$ be an lc pair such that 
$B$ is a boundary $\mathbb R$-divisor 
and let $L$ be a $\mathbb Q$-Cartier 
Weil divisor on $X$. 
Assume that 
$L-(K_X+B)$ is nef and log big over $V$, 
where $\pi:X\to V$ is a proper morphism. 
Then $R^q\pi_*\mathcal O_X(L)=0$ for any $q>0$. 
\end{thm} 

As a special case, we have the Kawamata--Viehweg vanishing 
theorem. 

\begin{cor}[Kawamata--Viehweg vanishing theorem]\index{Kawamata--Viehweg vanishing theorem}\label{kv-vani}
Let $(X, B)$ be a klt pair and let $L$ be 
a $\mathbb Q$-Cartier Weil divisor on $X$. 
Assume that $L-(K_X+B)$ is nef and big 
over $V$, where $\pi:X\to V$ is a proper morphism. 
Then $R^q\pi_*\mathcal O_X(L)=0$ for 
any $q>0$. 
\end{cor}

The proof of Theorem \ref{lc} works for Theorem \ref{kvn} 
without any changes if we adopt Theorem \ref{74}. 
We add one example. 

\begin{ex}\label{ex111}
Let $Y$ be a projective surface which has the following 
properties:~(i) there exists a projective birational 
morphism $f:X\to Y$ from a smooth projective surface $X$, and (ii) 
the exceptional locus $E$ of $f$ is an elliptic curve with 
$K_X+E=f^*K_Y$. 
For example, $Y$ is a cone over a smooth plane 
cubic curve and $f:X\to Y$ is the blow-up at the vertex of $Y$. 
We note that $(X, E)$ is a plt pair. 
Let $H$ be an ample Cartier divisor on $Y$. We consider 
a Cartier divisor $L=f^*H+K_X+E$ on $X$. 
Then $L-(K_X+E)$ is nef and big, but not log big. 
By the short exact sequence $$0\to 
\mathcal O_X(f^*H+K_X)\to 
\mathcal O_X(f^*H+K_X+E)
\to \mathcal O_E(K_E)\to 0, $$ 
we obtain $$R^1f_*\mathcal O_X(f^*H+K_X+E)\simeq 
H^1(E, \mathcal O_E(K_E))\simeq 
\mathbb C(P), $$ 
where $P=f(E)$. 
By the Leray spectral sequence, we have 
\begin{align*}
0&\to H^1(Y, f_*\mathcal O_X(K_X+E)\otimes 
\mathcal O_Y(H))\to 
H^1(X, \mathcal O_X(L))\to 
H^0(Y, \mathbb C(P))\\ 
&\to 
H^2(Y, f_*\mathcal O_X(K_X+E)\otimes 
\mathcal O_Y(H))\to \cdots. 
\end{align*} 
If $H$ is sufficiently ample, then 
$H^1(X, \mathcal O_X(L))\simeq 
H^0(Y, \mathbb C(P))\simeq \mathbb C(P)$. 
In particular, $H^1(X, \mathcal O_X(L))\ne 0$. 
\end{ex}

\begin{rem}
In Example \ref{ex111}, 
there exists an effective $\mathbb Q$-divisor $B$ on 
$X$ such that 
$L-\frac{1}{k}B$ is ample for any $k>0$ by Kodaira's 
lemma. 
Since $L\cdot E=0$, we have $B\cdot E<0$. 
In particular, $(X, E+\frac{1}{k}B)$ is not lc 
for any $k>0$. 
This is the main reason why $H^1(X, \mathcal O_X(L))\ne 0$. 
If $(X, E+\frac{1}{k}B)$ were 
lc, then the ampleness of $L-(K_X+E+\frac{1}{k}B)$ would 
imply $H^1(X, \mathcal O_X(L))=0$. 
\end{rem}

We modify the proof of Theorem \ref{74}. 
Then 
we can easily obtain the following generalization 
of Theorem \ref{8} (i). 
We leave the details for the reader's exercise. 

\begin{thm}\label{nef-labun-th}
Let $f:(Y, S+B)\to X$ be a proper 
morphism from an embedded simple normal crossing pair such that 
$S+B$ is a boundary, $S$ is reduced, and $\llcorner B\lrcorner =0$. 
Let $L$ be a Cartier divisor on $Y$ and 
$\pi:X\to V$ a proper morphism. Assume that 
$f^*H\sim _{\mathbb R}L-(K_Y+S+B)$, where 
$H$ is nef and log abundant over $V$. 
Then, every non-zero local section of 
$R^qf_*\mathcal O_Y(L)$ contains 
in its support the $f$-image of some strata of $(Y, S+B)$. 
\end{thm}

\begin{proof}[Sketch of the proof]
In Step \ref{so-no-1} in the proof of Theorem \ref{74}, 
we can write $\mu^*H=E+A$, where 
$E$ is an effective $\mathbb R$-divisor such that 
$k\mu^*H-E$ is $\pi'$-semi-ample for any 
positive integer $k$ (cf.~\cite[5.11.~Lemma]{ev}). 
Therefore, Theorem \ref{nef-labun-th} 
holds when each stratum of $(Y, S+B)$ dominates 
some irreducible component of $X$. 
Step \ref{so-no-2} in the proof of Theorem \ref{74} 
works without any changes. 
\end{proof}

\section{From SNC pairs to NC pairs}\label{sec6} 

In this section, we recover Ambro's theorems 
from Theorems \ref{5.1} and \ref{8}. We repeat Ambro's 
statements for the reader's convenience. 

\begin{thm}[{cf.~\cite[Theorem 3.1]{ambro}}]\label{61} 
Let $(X, S+B)$ be an embedded normal crossing 
pair such that $X$ is proper, $S+B$ is a boundary 
$\mathbb R$-divisor, $S$ is reduced, and $\llcorner 
B\lrcorner =0$. Let $L$ be a Cartier 
divisor on $X$ and $D$ an effective Cartier divisor that is permissible 
with 
respect to $(X, S+B)$. 
Assume the following conditions. 
\begin{itemize}
\item[{\em{(i)}}] $L\sim _{\mathbb R}K_X+S+B+H$, 
\item[{\em{(ii)}}] $H$ is a semi-ample $\mathbb R$-Cartier 
$\mathbb R$-divisor, and 
\item[{\em{(iii)}}] $tH\sim _{\mathbb R} D+D'$ for some 
positive real number $t$, where 
$D'$ is an effective $\mathbb R$-Cartier 
$\mathbb R$-divisor that is permissible with respect to $(X, S+B)$. 
\end{itemize}
Then the homomorphisms 
$$
H^q(X, \mathcal O_X(L))\to H^q(X, \mathcal O_X(L+D)), 
$$ 
which are induced by the natural inclusion 
$\mathcal O_X\to \mathcal O_X(D)$, 
are injective for all $q$. 
\end{thm}

\begin{thm}[{cf.~\cite[Theorem 3.2]{ambro}}]\label{62}
Let $(Y, S+B)$ be an embedded normal crossing 
pair such that $S+B$ is a boundary $\mathbb R$-divisor, 
$S$ is reduced, and $\llcorner B\lrcorner =0$. 
Let $f:Y\to X$ be a proper morphism and $L$ a Cartier 
divisor on $Y$ such that 
$H\sim _{\mathbb R}L-(K_Y+S+B)$ is $f$-semi-ample. 
\begin{itemize}
\item[{\em{(i)}}] 
every non-zero local section of $R^qf_*\mathcal O_Y(L)$ contains 
in its support the $f$-image of 
some strata of $(Y, S+B)$. 
\item[{\em{(ii)}}] 
let $\pi:X\to V$ be a projective morphism and 
assume that $H\sim _{\mathbb R}f^*H'$ for 
some $\pi$-ample $\mathbb R$-Cartier 
$\mathbb R$-divisor $H'$ on $X$. 
Then $R^qf_*\mathcal O_Y(L)$ is $\pi_*$-acyclic, that is, 
$R^p\pi_*R^qf_*\mathcal O_Y(L)=0$ for any $p>0$. 
\end{itemize}
\end{thm}

Before we go to the proof, let us recall the 
definition of {\em{normal crossing pairs}}, 
which is a slight generalization of Definition \ref{03}. 
The following definition is the same as 
\cite[Definition 2.3]{ambro} though 
it may look different. 

\begin{defn}[Normal crossing pair]\index{normal crossing 
pair}\label{625} 
Let $X$ be a normal crossing variety. 
We say that a reduced divisor $D$ on $X$ 
is {\em{normal crossing}} if, 
in the notation of Definition \ref{01}, 
we have 
$$
\widehat{\mathcal O}_{D, x}\simeq \frac{\mathbb C[[x_0, \cdots, 
x_{N}]]}{(x_0\cdots x_k, x_{i_1}\cdots x_{i_l})}
$$ 
for some $\{i_1, \cdots, i_l\}\subset\{k+1, \cdots, N\}$. 
We say that the pair $(X, B)$ is a 
{\em{normal crossing pair}} if the 
following 
conditions are satisfied. 
\begin{itemize}
\item[(1)] $X$ is a normal crossing 
variety, and 
\item[(2)] $B$ is an $\mathbb R$-Cartier $\mathbb R$-divisor whose 
support is normal crossing on $X$. 
\end{itemize}
We say that a normal crossing 
pair $(X, B)$ is {\em{embedded}} 
if there exists a closed embedding 
$\iota:X\to M$, where 
$M$ is a smooth variety 
of dimension $\dim X+1$. 
We put $K_{X^0}+\Theta=\eta^*(K_X+B)$, where 
$\eta:X^0\to X$ is the normalization of $X$. 
From now on, we assume that $B$ is a subboundary $\mathbb R$-divisor. 
A {\em{stratum}} of $(X, B)$\index{stratum} 
is an irreducible 
component of $X$ or 
the image of some lc center of $(X^0, \Theta)$ on $X$. 
A Cartier divisor $D$ on a normal 
crossing pair $(X, B)$ is called 
{\em{permissible with respect to $(X, B)$}} 
if $D$ contains no strata of the pair $(X, B)$.\index{permissible divisor} 
\end{defn}

The following three lemmas are easy to 
check. So, we omit the proofs. 

\begin{lem}\label{63} 
Let $X$ be a normal crossing divisor on a smooth 
variety $M$. 
Then there exists a sequence of blow-ups 
$M_k\to M_{k-1}\to \cdots 
\to M_0=M$ with 
the following properties. 
\begin{itemize}
\item[{\em{(i)}}] $\sigma_{i+1}: M_{i+1}\to M_i$ 
is the blow-up 
along a smooth stratum of $X_i$ for any $i\geq 0$, 
\item[{\em{(ii)}}] $X_0=X$ and $X_{i+1}$ is 
the inverse image of $X_{i}$ with 
the reduced structure for any $i\geq 0$, 
and 
\item[{\em{(iii)}}] $X_k$ is a simple normal crossing 
divisor on $M_k$. 
\end{itemize}
For each step $\sigma_{i+1}$, we can directly 
check that $\sigma_{i+1*}\mathcal O_{X_{i+1}}\simeq 
\mathcal O_{X_i}$ and $R^q\sigma_{i+1*}\mathcal O_{X_{i+1}}=0$ 
for any $i\geq 0$ and $q\geq 1$. 
Let $B$ be an $\mathbb R$-Cartier $\mathbb R$-divisor 
on $X$ such that $\Supp B$ is normal crossing. We put $B_0=B$ 
and $K_{X_{i+1}}+B_{i+1}=\sigma^*_{i+1}
(K_{X_i}+B_i)$ for all $i\geq 0$. 
Then it is obvious that $B_i$ is 
an $\mathbb R$-Cartier $\mathbb R$-divisor 
and $\Supp B_i$ is normal crossing on $X_i$ for 
any $i\geq 0$. 
We can also check that $B_i$ is a boundary $\mathbb R$-divisor 
$($resp.~$\mathbb Q$-divisor$)$ for any $i\geq 0$ if so is 
$B$. If $B$ is a boundary, then 
the $\sigma_{i+1}$-image of 
any stratum of $(X_{i+1}, B_{i+1})$ is a stratum 
of $(X_i, B_i)$. 
\end{lem}

\begin{rem}\label{635} 
Each step in Lemma \ref{63} is called 
{\em{embedded log transformation}} in \cite[Section 2]{ambro}. 
See also Lemma \ref{useful-lem2}. 
\end{rem}

\begin{lem}\label{64} 
Let $X$ be a simple normal crossing divisor on a smooth 
variety $M$. 
Let $S+B$ be a boundary $\mathbb R$-Cartier 
$\mathbb R$-divisor 
on $X$ such that $\Supp (S+B)$ is normal crossing, 
$S$ is reduced, and $\llcorner B\lrcorner=0$. 
Then there exists a sequence of blow-ups 
$M_k\to M_{k-1}\to \cdots 
\to M_0=M$ with 
the following properties. 
\begin{itemize}
\item[{\em{(i)}}] $\sigma_{i+1}: M_{i+1}\to M_i$ is the blow-up 
along a smooth stratum 
of $(X_i, S_i)$ that 
is contained in $S_i$ for any $i\geq 0$, 
\item[{\em{(ii)}}] we put $X_0=X$, $S_0=S$, and $B_0=B$, 
and $X_{i+1}$ is 
the strict transform of $X_{i}$ for any $i\geq 0$, 
\item[{\em{(iii)}}] we define $K_{X_{i+1}}+S_{i+1}+B_{i+1}
=\sigma^*_{i+1}(K_{X_i}+S_i+B_i)$ for 
any $i\geq 0$, where 
$B_{i+1}$ is the strict transform of $B_i$ on $X_{i+1}$, 
\item[{\em{(iv)}}] the $\sigma_{i+1}$-image of 
any stratum of $(X_{i+1}, S_{i+1}+B_{i+1})$ is a stratum 
of $(X_i, S_i+B_i)$, and  
\item[{\em{(v)}}] $S_k$ is a simple normal crossing 
divisor on $X_k$. 
\end{itemize}
For each step $\sigma_{i+1}$, we can easily 
check that $\sigma_{i+1*}\mathcal O_{X_{i+1}}\simeq 
\mathcal O_{X_i}$ and $R^q\sigma_{i+1*}\mathcal O_{X_{i+1}}=0$ 
for any $i\geq 0$ and $q\geq 1$. 
We note that $X_i$ is simple normal crossing, 
$\Supp (S_i+B_i)$ is normal crossing on $X_i$, 
and $S_i$ is reduced for any $i\geq 0$. 
\end{lem}

\begin{lem}\label{65} 
Let $X$ be a simple normal crossing divisor on a smooth 
variety $M$. 
Let $S+B$ be a boundary $\mathbb R$-Cartier 
$\mathbb R$-divisor 
on $X$ such that $\Supp (S+B)$ is normal crossing, 
$S$ is reduced and simple normal crossing, 
and $\llcorner B\lrcorner=0$. 
Then there exists a sequence of blow-ups 
$M_k\to M_{k-1}\to \cdots 
\to M_0=M$ with 
the following properties. 
\begin{itemize}
\item[{\em{(i)}}] $\sigma_{i+1}: M_{i+1}\to M_i$ 
is the blow-up 
along a smooth stratum 
of $(X_i, \Supp B_i)$ that is 
contained in $\Supp B_i$ for any $i\geq 0$, 
\item[{\em{(ii)}}] we put $X_0=X$, $S_0=S$, and $B_0=B$, 
and $X_{i+1}$ is 
the strict transform of $X_{i}$ for any $i\geq 0$, 
\item[{\em{(iii)}}] we define $K_{X_{i+1}}+S_{i+1}+B_{i+1}
=\sigma^*_{i+1}(K_{X_i}+S_i+B_i)$ for 
any $i\geq 0$, where 
$S_{i+1}$ is the strict transform of $S_i$ on $X_{i+1}$, and 
\item[{\em{(iv)}}] $\Supp (S_k+B_k)$ 
is a simple normal crossing 
divisor on $X_k$. 
\end{itemize}
We note that $X_i$ is simple normal crossing on $M_i$ and 
$\Supp (S_i+B_i)$ is normal crossing on $X_i$ for 
any $i\geq 0$. 
We can easily check that 
$\llcorner B_i\lrcorner \leq 0$ for any $i\geq 0$. 
The composition morphism $M_k\to M$ is denoted by 
$\sigma$. 
Let $L$ be any Cartier divisor on $X$. 
We put $E=\ulcorner -B_k\urcorner$. 
Then $E$ is an effective $\sigma$-exceptional Cartier 
divisor on $X_k$ 
and we obtain $\sigma_*\mathcal O_{X_k}(\sigma ^*L+E)
\simeq \mathcal O_X(L)$ and 
$R^q\sigma_*\mathcal O_{X_k}(\sigma ^*L+E)=0$ 
for any $q\geq 1$ by {\em{Theorem \ref{8} (i)}}. 
We note that $\sigma^*L+E-(K_{X_k}+S_k+\{B_k\})
=\sigma^*L-\sigma^*(K_X+S+B)$ is $\mathbb R$-linearly 
trivial over $X$ and $\sigma$ 
is an isomorphism at the generic point of any stratum 
of $(X_k, S_k+\{B_k\})$. 
\end{lem}

Let us go to the proof of Theorems \ref{61} and \ref{62}. 

\begin{proof}[{Proof of {\em{Theorems \ref{61} and \ref{62}}}}]
We take a sequence of blow-ups and obtain a projective 
morphism $\sigma:X'\to X$ (resp.~$\sigma:Y'\to Y$) 
from an embedded simple normal crossing variety $X'$ 
(resp.~$Y'$) in 
Theorem \ref{61} (resp.~Theorem \ref{62}) by Lemma 
\ref{63}. 
We can replace $X$ (resp.~$Y$) and $L$ 
with $X'$ (resp.~$Y'$) and $\sigma^*L$ 
by Leray's spectral sequence. 
So, we can assume that $X$ (resp.~$Y$) 
is simple normal crossing. 
Similarly, we can assume that $S$ is simple normal crossing 
on $X$ (resp.~$Y$) by applying Lemma 
\ref{64}. 
Finally, we use Lemma \ref{65} and obtain a birational 
morphism $\sigma:(X', S'+B')\to (X, S+B)$ 
(resp.~$(Y', S'+B')\to (Y, S+B)$) from 
an embedded 
simple normal crossing pair $(X', S'+B')$ (resp.~$(Y', S'+B')$) 
such that $K_{X'}+S'+B'=\sigma^*(K_X+S+B)$ 
(resp.~$K_{Y'}+S'+B'=\sigma^*
(K_Y+S+B)$) as in Lemma \ref{65}. 
By Lemma \ref{65}, we 
can replace $(X, S+B)$ (resp.~$(Y, S+B)$) and $L$ 
with $(X', S'+\{B'\})$ (resp.~$(Y', S'+\{B'\})$) and 
$\sigma^*L+\ulcorner -B'\urcorner$ by Leray's 
spectral sequence. 
Then we apply Theorem \ref{5.1} (resp.~Theorem 
\ref{8}). Thus, we obtain Theorems \ref{61} and 
\ref{62}. 
\end{proof}

\section{Examples}\label{28-sec}

In this section, we treat various supplementary 
examples. 

\begin{say}[Examples for Section \ref{sec2}]
Let $X$ be a smooth 
projective variety and let $M$ be a Cartier 
divisor on $X$ such that 
$N\sim mM$, where $N$ is a simple normal crossing 
divisor on $X$ and $m\geq 2$. 
We put $B=\frac{1}{m}N$ and 
$L=K_X+M$. 
In this setting, we can apply Proposition \ref{1}. 
If $M$ is semi-ample, then the existence of 
such $N$ and $m$ is obvious by Bertini. 
Here, we give explicit examples where $M$ is not nef. 

\begin{ex}
We consider the $\mathbb P^1$-bundle 
$\pi:X=\mathbb P_{\mathbb P^1}(\mathcal O_{\mathbb P^1}
\oplus \mathcal O_{\mathbb P^1}(2))\to \mathbb P^1$. 
Let $E$ and $G$ be the sections 
of $\pi$ such that $E^2=-2$ and $G^2=2$. 
We note that 
$E+2F\sim G$, 
where $F$ is a fiber of $\pi$. 
We consider $M=E+F$. Then 
$2M=2E+2F\sim E+G$. 
In this case, $M\cdot E=-1$. In particular, 
$M$ is not nef. Furthermore, 
we can easily check that 
$H^i(X, \mathcal O_X(K_X+M))=0$ for 
any $i$. So, it is not interesting to 
apply Proposition \ref{1}. 
\end{ex} 

\begin{ex}
We consider the $\mathbb P^1$-bundle 
$\pi:Y=\mathbb P_{\mathbb P^1}(\mathcal O_{\mathbb P^1}
\oplus \mathcal O_{\mathbb P^1}(4))\to \mathbb P^1$. 
Let $G$ (resp.~$E$) be the 
positive (resp.~negative) section of $\pi$, that is, 
the section corresponding 
to the projection $\mathcal O_{\mathbb P^1}\oplus 
\mathcal O_{\mathbb P^1}(4)\to \mathcal O_{\mathbb P^1}(4)$ 
(resp.~$\mathcal O_{\mathbb P^1}\oplus 
\mathcal O_{\mathbb P^1}(4)\to \mathcal O_{\mathbb P^1})$. 
We put $M'=-F+2G$, where 
$F$ is a fiber of $\pi$. Then 
$M'$ is not nef and 
$2M'\sim G+E+F_1+F_2+H$, where 
$F_1$ and $F_2$ are distinct 
fibers of $\pi$, and $H$ is a general 
member of the free linear system $|2G|$. Note that 
$G+E+F_1+F_2+H$ is a reduced 
simple normal crossing divisor on $Y$. We put 
$X=Y\times C$, where $C$ is an elliptic curve, and 
$M=p^*M'$, where 
$p:X\to Y$ is the projection. 
Then $X$ is a smooth projective 
variety and $M$ is a Cartier divisor on $X$. 
We note that $M$ is not nef and that 
we can find a reduced simple normal 
crossing divisor such 
that $N\sim 2M$. By the K\"unneth formula, 
we have 
\begin{eqnarray*}
H^1(X, \mathcal O_X(K_X+M))\simeq 
H^0(\mathbb P^1, \mathcal O_{\mathbb P^1}(1))
\simeq \mathbb C^2.
\end{eqnarray*}
Therefore, $X$ with $L=K_X+M$ satisfies the conditions in 
Proposition \ref{1} and we have 
$H^1(X, \mathcal O_X(L))\ne 0$. 
\end{ex}
\end{say}

\begin{say}[Kodaira vanishing theorem for 
singular varieties]
The following example is due to Sommese (cf.~\cite[(0.2.4) 
Example]{som}). 
It shows that the Kodaira vanishing theorem does not 
necessarily hold for varieties with non-lc singularities. 
Therefore, Corollary \ref{lcvar} is sharp. 

\begin{prop}[Sommese]\index{Sommese's example}\label{no1} 
We consider the $\mathbb P^3$-bundle 
$$\pi:Y=\mathbb P_{\mathbb P^1}(\mathcal O_{\mathbb P^1}
\oplus \mathcal O_{\mathbb P^1}(1)^{\oplus 3})\to 
\mathbb P^1$$ over $\mathbb P^1$. Let $\mathcal M=\mathcal O_Y(1)$ be the 
tautological line bundle 
of $\pi:Y\to \mathbb P^1$. 
We take a general member $X$ of 
the linear system 
$|(\mathcal M\otimes \pi^*\mathcal O_{\mathbb P^1}(-1))^{\otimes 4}|$. 
Then $X$ is a normal projective Gorenstein threefold and 
$X$ is not lc. 
We put $\mathcal L=\mathcal M\otimes \pi^*\mathcal O_{\mathbb P^1} (1)$. Then 
$\mathcal L$ is ample. 
In this case, we can check that $H^2(X, \mathcal L^{-1})=
\mathbb C$. 
By the Serre duality, $H^1(X, \mathcal O_X(K_X)\otimes 
\mathcal L)=\mathbb C$. Therefore, 
the Kodaira vanishing theorem does not 
hold for $X$. 
\end{prop}

\begin{proof}
We consider the following short exact sequence 
$$
0\to \mathcal L^{-1}(-X)\to \mathcal L^{-1}\to 
\mathcal L^{-1}|_X\to 0. 
$$ 
Then we have the long exact sequence 
\begin{align*}
\cdots \to H^i(Y, \mathcal L^{-1}(-X))
\to H^i(Y, \mathcal L^{-1}) 
\to H^i(X, \mathcal L^{-1})\\ 
\to H^{i+1}(Y, \mathcal L^{-1}(-X))\to 
\cdots. 
\end{align*}
Since $H^i(Y, \mathcal L^{-1})=0$ for 
$i<4$ by the original 
Kodaira vanishing theorem, we 
obtain that 
$H^2(X, \mathcal L^{-1})=H^3(Y, \mathcal L^{-1}
(-X))$. 
Therefore, it is sufficient to prove that 
$H^3(Y, \mathcal L^{-1}(-X))=\mathbb C$. 

We have 
$$
\mathcal L^{-1}(-X)=\mathcal M^{-1}
\otimes \pi^*\mathcal O_{\mathbb P^1}(-1)
\otimes \mathcal M^{-4}\otimes \pi^*\mathcal O_{\mathbb P^1}
(4)=\mathcal M^{-5}\otimes \pi^*\mathcal O_{\mathbb P^1}(3). 
$$
We note that $R^i\pi_*\mathcal M^{-5}=0$ for $i\ne 3$ 
because $\mathcal M=\mathcal O_Y(1)$. 
By the Grothendieck duality, 
$$
R\mathcal Hom(R\pi_*\mathcal M^{-5}, 
\mathcal O_{\mathbb P^1}(K_{\mathbb P^1})[1])
=R\pi_*R\mathcal Hom (\mathcal M^{-5}, \mathcal O_Y(K_Y)[4]).
$$ 
By the Grothendieck duality again, 
\begin{eqnarray*}
R\pi_*\mathcal M^{-5}=
R\mathcal Hom (R\pi_*R\mathcal Hom 
(\mathcal M^{-5}, \mathcal O_Y(K_Y)[4]), 
\mathcal O_{\mathbb P^1}
(K_{\mathbb P^1})[1])\\
=R\mathcal Hom (R\pi_*(\mathcal O_Y(K_Y)\otimes \mathcal 
M^5), \mathcal O_{\mathbb P^1}(K_{\mathbb P^1}))[-3]=(*).  
\end{eqnarray*} 
By the definition, we have 
$$
\mathcal O_Y(K_Y)=
\pi^*(\mathcal O_{\mathbb P^1}(K_{\mathbb P^1})\otimes 
\det (\mathcal O_{\mathbb P^1}\oplus \mathcal O_{\mathbb P^1}
(1)^{\oplus 3}))\otimes 
\mathcal M^{-4}=\pi^*\mathcal O_{\mathbb P^1}(1)\otimes 
\mathcal M^{-4}. 
$$
By this formula, we obtain 
$$\mathcal O_Y(K_Y)
\otimes \mathcal M^5=\pi^*\mathcal O_{\mathbb P^1}
(1)\otimes \mathcal M. 
$$
Thus, $R^i\pi_*(\mathcal O_Y(K_Y)\otimes 
\mathcal M^5)=0$ for any $i>0$. We note 
that 
\begin{eqnarray*}
\pi_*(\mathcal O_Y(K_Y)\otimes \mathcal 
M^5)&=&\mathcal O_{\mathbb P^1}
(1)\otimes \pi_*\mathcal M \\&=&
\mathcal O_{\mathbb P^1}(1)\otimes 
(\mathcal O_{\mathbb P^1}\oplus \mathcal O_{\mathbb P^1}(1)
^{\oplus 3})=\mathcal O_{\mathbb P^1}(1)\oplus 
\mathcal O_{\mathbb P^1}(2)^{\oplus 3}. 
\end{eqnarray*}
Therefore, we have 
\begin{eqnarray*}
(*)&=&R\mathcal Hom(\mathcal O_{\mathbb P^1}(1)\oplus 
\mathcal O_{\mathbb P^1}(2)^{\oplus 3}, 
\mathcal O_{\mathbb P^1}(-2))[-3]\\
&=&(\mathcal O_{\mathbb P^1}(-3)\oplus \mathcal O_{\mathbb P^1}
(-4)^{\oplus 3})[-3]. 
\end{eqnarray*}
So, we obtain 
$R^3\pi_*\mathcal M^{-5}=\mathcal O_{\mathbb P^1}(-3)\oplus \mathcal O_{\mathbb P^1}(-4)^{\oplus 
3}$.  
Thus, 
$R^3\pi_*\mathcal M^{-5}\otimes \mathcal 
O_{\mathbb P^1}(3)=\mathcal O_{\mathbb P^1}\oplus \mathcal O_{\mathbb P^1}(-1)
^{\oplus 3}$. 

By the spectral sequence, we have 
\begin{eqnarray*}
H^3(Y, \mathcal L^{-1}(-X))&=&H^3(Y, \mathcal M^{-5}
\otimes \pi^*\mathcal O_{\mathbb P^1}(3))\\
&=&H^0(\mathbb P^1, R^3\pi_*(\mathcal M^{-5}\otimes 
\pi^*\mathcal O_{\mathbb P^1}(3)))\\
&=&H^0(\mathbb P^1, \mathcal O_{\mathbb P^1}\oplus 
\mathcal O_{\mathbb P^1}(-1)^{\oplus 3})=\mathbb C. 
\end{eqnarray*} 
Therefore, 
$H^2(X, \mathcal L^{-1})=\mathbb C$. 

Let us recall that $X$ is a general member of 
the linear system 
$|(\mathcal M\otimes \pi^*\mathcal O_{\mathbb P^1}
(-1))^{\otimes 4}|$. 
Let $C$ be the negative section of $\pi:Y\to 
\mathbb P^1$, that is, 
the section corresponding to the projection 
$$
\mathcal O_{\mathbb P^1}\oplus \mathcal O_{\mathbb P^1}
(1)^{\oplus 3}\to \mathcal O_{\mathbb P^1}\to 0. 
$$ 
From now, we will check that 
$|\mathcal M\otimes \pi^*\mathcal O_{\mathbb P^1}
(-1)|$ is free outside $C$. Once we 
checked it, we know that 
$|(\mathcal M\otimes \pi^*\mathcal O_{\mathbb P^1}
(-1))^{\otimes 4}|$ 
is free outside $C$. 
Then $X$ is smooth in codimension one. 
Since $Y$ is smooth, $X$ is normal 
and Gorenstein by adjunction. 

We take 
$Z\in 
|\mathcal M\otimes \pi^*\mathcal O_{\mathbb P^1}
(-1)|\ne \emptyset$. 
Since 
$H^0(Y, \mathcal M\otimes \pi^*\mathcal O_{\mathbb P^1}(-1)
\otimes \pi^*\mathcal O_{\mathbb P^1}(-1))=0$, 
$Z$ can not have a fiber of $\pi$ as an irreducible 
component, that is, 
any irreducible component of $Z$ is mapped onto 
$\mathbb P^1$ by $\pi:Y\to \mathbb P^1$. 
On the other hand, let $l$ be a line in a fiber 
of $\pi:Y\to \mathbb P^1$. Then 
$Z\cdot l=1$. Therefore, 
$Z$ is irreducible. 
Let $F=\mathbb P^3$ be a fiber of $\pi:Y\to 
\mathbb P^1$. We consider 
\begin{eqnarray*}
0=H^0(Y, \mathcal M\otimes \pi^*\mathcal O_{\mathbb P^1}(-1)
\otimes \mathcal O_Y(-F))\to 
H^0(Y, \mathcal M\otimes \pi^*\mathcal O_{\mathbb P^1}(-1))\\
\to H^0(F, \mathcal O_F(1))\to 
H^1(Y, \mathcal M\otimes \pi^*\mathcal O_{\mathbb P^1}(-1)
\otimes \mathcal O_Y(-F))\to \cdots. 
\end{eqnarray*}
Since $(\mathcal M\otimes \pi^*\mathcal O_{\mathbb P^1}
(-1))\cdot C=-1$, 
every member of $|\mathcal M\otimes \pi^*\mathcal O_{\mathbb P^1}
(-1)|$
contains $C$. We put 
$P=F\cap C$. Then 
the image of 
$$\alpha:
H^0(Y, \mathcal M\otimes \pi^*\mathcal O_{\mathbb P^1}(-1))\\
\to H^0(F, \mathcal O_F(1))$$ 
is 
$H^0(F, m_P\otimes \mathcal O_{F}(1))$, 
where 
$m_P$ is the maximal 
ideal of 
$P$. It is because 
the dimension of 
$H^0(Y, \mathcal M\otimes \pi^*\mathcal O_{\mathbb P^1}(-1))$ 
is three. 
Thus, we know that 
$|\mathcal M\otimes \pi^*\mathcal O_{\mathbb P^1}
(-1)|$ 
is free outside $C$. 
In particular, 
$|(\mathcal M\otimes \pi^*\mathcal O_{\mathbb P^1}
(-1))^{\otimes 4}|$ 
is free outside $C$. 

More explicitly, 
the image of the injection 
$$\alpha:
H^0(Y, \mathcal M\otimes \pi^*\mathcal O_{\mathbb P^1}(-1))\\
\to H^0(F, \mathcal O_F(1))$$ 
is 
$H^0(F, m_P\otimes \mathcal O_{F}(1))$. 
We note that 
$$H^0(Y, \mathcal M\otimes \pi^*\mathcal O_{\mathbb P^1}
(-1))=H^0(\mathbb P^1, \mathcal O_{\mathbb P^1}(-1)
\oplus \mathcal O_{\mathbb P^1}^{\oplus 3})=\mathbb C^3, 
$$ 
and 
$$
H^0(Y, (\mathcal M\otimes \pi^*\mathcal O_{\mathbb P^1}
(-1))^{\otimes 4})=H^0(\mathbb P^1, 
\Sym ^4 (\mathcal O_{\mathbb P^1}(-1)\oplus \mathcal O_{\mathbb P^1}^{\oplus 3}))=\mathbb C^{15}. 
$$ 
We can check that the restriction of 
$H^0(Y, (\mathcal M\otimes \pi^*\mathcal O_{\mathbb P^1}
(-1))^{\otimes 4})$ 
to $F$ is 
$\Sym ^4 H^0(F, m_P\otimes \mathcal O_F(1))$. 
Thus, the general fiber $f$ of $\pi:X\to \mathbb P^1$ 
is a cone in $\mathbb P^3$ on a smooth 
plane curve of degree $4$ with the 
vertex $P=f\cap C$. 
Therefore, $(Y, X)$ is not lc because 
the multiplicity of $X$ along $C$ is four. 
Thus, $X$ is not lc by the inversion of adjunction 
(cf.~Corollary \ref{inv-cor}). 
Anyway, $X$ is the required variety. 
\end{proof}

\begin{rem}
We consider the $\mathbb P^{k+1}$-bundle 
$$\pi:Y=\mathbb P_{\mathbb P^1}(\mathcal O_{\mathbb P^1} 
\oplus \mathcal O_{\mathbb P^1}(1)^{\oplus (k+1)})\to 
\mathbb P^1$$ over $\mathbb P^1$ for $k\geq 2$. 
We put $\mathcal M=\mathcal O_Y(1)$ and 
$\mathcal L=\mathcal M\otimes \pi^*
\mathcal O_{\mathbb P^1}(1)$. 
Then $\mathcal L$ is ample. 
We take a general member $X$ of 
the linear system 
$|(\mathcal M\otimes \pi^*\mathcal O_{\mathbb P^1}(-1))
^{\otimes {(k+2)}}|$. 
Then we can check the following properties. 
\begin{itemize}
\item[(1)] $X$ is a normal projective 
Gorenstein $(k+1)$-fold. 
\item[(2)] $X$ is not lc. 
\item[(3)] We can check that 
$R^{k+1}\pi_*\mathcal M^{-(k+3)}
=\mathcal O_{\mathbb P^1}(-1-k)\oplus 
\mathcal O_{\mathbb P^1}(-2-k)^{\oplus (k+1)}$ and that 
$R^i\pi_*\mathcal M^{-(k+3)}=0$ for $i\ne k+1$. 
\item[(4)] Since 
$\mathcal L^{-1}(-X)=\mathcal M^{-(k+3)}
\otimes 
\pi^*\mathcal O_{\mathbb P^1}(k+1)$, 
we have 
\begin{eqnarray*}
H^{k+1}(Y, \mathcal L^{-1}(-X))
=H^0(\mathbb P^1, 
R^{k+1}\pi_*\mathcal M^{-(k+3)}
\otimes \mathcal O_{\mathbb P^1}(k+1))\\
=H^0(\mathbb P^1, \mathcal O_{\mathbb P^1}\oplus \mathcal 
O_{\mathbb P^1}(-1)^{\oplus (k+1)})=\mathbb C. 
\end{eqnarray*} 
\end{itemize} Thus, 
$H^k(X, \mathcal L^{-1})=H^{k+1}(Y, \mathcal L^{-1}
(-X))=\mathbb C$. 
\end{rem}

We note that the first cohomology 
group of an anti-ample 
line bundle on a normal variety with 
$\dim \geq 2$ always vanishes by the following 
Mumford vanishing theorem. 

\begin{thm}[Mumford]\index{Mumford vanishing theorem} 
Let $V$ be a normal complete algebraic 
variety and $\mathcal L$ be a semi-ample line bundle 
on $V$. 
Assume that $\kappa (V, \mathcal L)\geq 2$. 
Then $H^1(V, \mathcal L^{-1})=0$. 
\end{thm}
\begin{proof}
Let $f:W\to V$ be a resolution. 
By Leray's spectral sequence, 
$$
0\to H^1(V, f_*f^*\mathcal L^{-1})\to H^1(W, f^*\mathcal L^{-1})\to \cdots. 
$$
By the Kawamata--Viehweg vanishing theorem, 
$H^1(W, f^*\mathcal L^{-1})=0$. 
Thus, $H^1(V, \mathcal L^{-1})=H^1(V, f_*f^*\mathcal 
L^{-1})=0$. 
\end{proof}
\end{say}

\begin{say}[On the Kawamata--Viehweg vanishing theorem]
The next example shows that a naive generalization of 
the Kawamata--Viehweg vanishing theorem does not 
necessarily hold for varieties with lc singularities. 

\begin{ex}\label{ex4} 
We put $V=\mathbb P^2\times \mathbb P^2$. 
Let $p_i:V\to \mathbb P^2$ be the $i$-th 
projection for $i=1$ and $2$. 
We define $\mathcal L=p_1^*\mathcal O_{\mathbb P^2} 
(1)\otimes p_2^*\mathcal O_{\mathbb P^2}(1)$ and 
consider the $\mathbb P^1$-bundle 
$\pi:W=\mathbb P_{V}(\mathcal L\oplus 
\mathcal O_V)\to V$. 
Let $F=\mathbb P^2\times \mathbb P^2$ be the negative 
section of 
$\pi:W\to V$, that is, 
the section of $\pi$ corresponding to $\mathcal L
\oplus \mathcal O_V\to \mathcal O_V\to 0$. 
By using the linear system 
$|\mathcal O_W(1)\otimes \pi^*p_1^*\mathcal O_{\mathbb P^2}
(1)|$, we can contract $F=\mathbb P^2\times \mathbb P^2$ to 
$\mathbb P^2\times\{\text{point}\}$. 

Next, we consider an elliptic curve 
$C\subset \mathbb P^2$ and 
put $Z=C\times C\subset V=\mathbb P^2\times \mathbb P^2$. 
Let $\pi:Y\to Z$ be the restriction of $\pi:W\to V$ to 
$Z$. 
The restriction of the above contraction morphism 
$\Phi_{|\mathcal O_W(1)\otimes \pi^*p_1^*
\mathcal O_{\mathbb P^2}(1)|}: W\to U$ to $Y$ is denoted 
by $f:Y\to X$. Then, the exceptional 
locus of $f:Y\to X$ is $E=F|_Y=C\times C$ and $f$ contracts 
$E$ to $C\times \{\text{point}\}$. 

Let $\mathcal O_W(1)$ be the tautological 
line bundle of the $\mathbb P^1$-bundle 
$\pi:W\to V$. 
By the construction, 
$\mathcal O_W(1)=\mathcal O_W(D)$, where 
$D$ is the positive 
section of $\pi$, that is, the section 
corresponding to $\mathcal L\oplus 
\mathcal O_W\to \mathcal L\to 0$. 
By the definition, 
$$\mathcal O_W(K_W)=\pi^*(\mathcal O_V(K_V)\otimes 
\mathcal L)\otimes \mathcal O_W(-2). 
$$ 
By adjunction, 
$$
\mathcal O_Y(K_Y)=\pi^*(\mathcal O_Z(K_Z)\otimes \mathcal L
|_Z)\otimes \mathcal O_Y(-2)=\pi^*(\mathcal L|_Z)\otimes 
\mathcal O_Y(-2). 
$$ 
Therefore, 
$$
\mathcal O_Y(K_Y+E)=\pi^*(\mathcal L|_Z)\otimes 
\mathcal O_Y(-2)\otimes \mathcal O_Y(E). 
$$ 
We note that $E=F|_Y$. 
Since $\mathcal O_Y(E)\otimes \pi^*(\mathcal L|_Z)\simeq 
\mathcal O_Y(D)$, we have 
$\mathcal O_Y(-(K_Y+E))
=\mathcal O_Y(1)$ because $\mathcal O_Y(1)=\mathcal O_Y(D)$. 
Thus, $-(K_Y+E)$ is nef and big. 

On the other hand, it is not difficult to see 
that $X$ is a normal projective Gorenstein 
threefold, $X$ is lc but not klt along $G=f(E)$, and 
that $X$ is smooth outside $G$. 
Since we can check that 
$f^*K_X=K_Y+E$, $-K_X$ is nef and big. 

Finally, we consider the short exact sequence 
$$
0\to \mathcal J\to \mathcal O_X\to \mathcal 
O_X/\mathcal J\to 0, 
$$ 
where $\mathcal J$ is the multiplier ideal sheaf 
of $X$. 
In our case, we can easily check that $\mathcal J=
f_*\mathcal O_Y(-E)=\mathcal 
I_G$, where $\mathcal I_G$ is the defining ideal 
sheaf of $G$ on $X$. 
Since $-K_X$ is nef and big, 
$H^i(X, \mathcal J)=0$ for any $i>0$ by Nadel's 
vanishing\index{Nadel vanishing theorem} theorem. 
Therefore, 
$H^i(X, \mathcal O_X)=H^i(G, \mathcal O_G)$ for 
any $i>0$. 
Since $G$ is an elliptic curve, 
$H^1(X, \mathcal O_X)=H^1(G, \mathcal O_G)=\mathbb C$. 
We note that $-K_X$ is nef and big but $-K_X$ is not 
log big with respect to $X$. 
\end{ex}
\end{say}

\begin{say}[On the injectivity theorem]
The final example in this section supplements 
Theorem \ref{5.1}. 

\begin{ex}
We consider the $\mathbb P^1$-bundle 
$\pi:X=\mathbb P_{\mathbb P^1}(\mathcal O_{\mathbb P^1}
\oplus \mathcal O_{\mathbb P^1}(1))\to \mathbb P^1$. 
Let $S$ (resp.~$H$) be the negative (resp.~positive) 
section of $\pi$, that is, 
the section corresponding to 
the projection $\mathcal O_{\mathbb P^1}\oplus 
\mathcal O_{\mathbb P^1}(1)\to \mathcal O_{\mathbb P^1}$ 
(resp.~$\mathcal O_{\mathbb P^1}\oplus 
\mathcal O_{\mathbb P^1}(1)\to \mathcal O_{\mathbb P^1}(1)$). 
Then $H$ is semi-ample and $S+F\sim H$, where 
$F$ is a fiber of $\pi$. 
\begin{claim}
The homomorphism 
$$
H^1(X, \mathcal O_X(K_X+S+H))\to 
H^1(X, \mathcal O_X(K_X+S+H+S+F)) 
$$
induced by the natural inclusion $\mathcal O_X\to 
\mathcal O_X(S+F)$ is not injective. 
\end{claim}
\begin{proof}[Proof of {\em{Claim}}]
It is sufficient to see that the homomorphism 
$$
H^1(X, \mathcal O_X(K_X+S+H))\to 
H^1(X, \mathcal O_X(K_X+S+H+F)) 
$$
induced by the natural inclusion $\mathcal O_X\to 
\mathcal O_X(F)$ is not injective. 
We consider the short exact sequence 
\begin{align*}
0\to \mathcal O_X(K_X+S+H)\to 
\mathcal O_X(K_X+S+H+F)\\ \to 
\mathcal O_F(K_F+(S+H)|_F)\to 0. 
\end{align*}
We note that $F\simeq \mathbb P^1$ and 
$\mathcal O_F(K_F+(S+H)|_F)\simeq \mathcal O_{\mathbb P^1}$. 
Therefore, we obtain the following exact sequence 
$$
0\to \mathbb C\to 
H^1(X, \mathcal O_X(K_X+S+H))\to 
H^1(X, \mathcal O_X(K_X+S+H+F))\to 0. 
$$
Thus, $H^1(X, \mathcal O_X(K_X+S+H))\to 
H^1(X, \mathcal O_X(K_X+S+H+F))$ is 
not injective. 
We note that $S+F$ is not {\em{permissible}} 
with respect to $(X, S)$. 
\end{proof}
\end{ex}
\end{say}

\section{Review of the proofs}\label{29-sec}
We close this chapter with the review of our proofs of 
Theorems \ref{61} and \ref{62}. 
It may help the reader to compare this chapter 
with \cite[Section 3]
{ambro}. We think that our proofs are not so long. Ambro's 
proofs seem to be too short. 

\begin{say}[Review]\label{680} 
We review our proofs of the injectivity, 
torsion-free, and 
vanishing theorems. 
\setcounter{steste}{0}
\begin{steste}\label{step1}
($E_1$-degeneration of a certain 
Hodge to de Rham type spectral sequence). 
We discuss this $E_1$-degeneration in \ref{s6}. 
As we pointed out in the introduction, the appropriate 
spectral sequence was not chosen in \cite{ambro}. 
It is one of 
the crucial technical problems in \cite[Section 3]{ambro}. 
This step is purely Hodge theoretic. 
We describe it in Section \ref{sec3}. 
\end{steste}
\begin{steste}\label{step2}
(Fundamental injectivity theorem:~Proposition \ref{2}). 
This is a very special case of \cite[Theorem 3.1]{ambro} and 
follows from the $E_1$-degeneration in Step \ref{step1} by 
using covering arguments. This step is in Section \ref{sec2}. 
\end{steste}
\begin{steste}\label{step3}
(Relative vanishing lemma:~Lemma \ref{re-vani-lem}). 
This step is missing in \cite{ambro}. 
It is a very special case of \cite[Theorem 3.2 (ii)]{ambro}. 
However, we can not use \cite[Theorem 3.2 (ii)]{ambro} 
in this stage. 
Our proof of this lemma does {\em{not}} work directly for 
normal 
crossing pairs. 
So, we need to assume that the varieties are {\em{simple}} 
normal crossing pairs. 
\end{steste}
\begin{steste}\label{step4}
(Injectivity theorem for embedded simple normal 
crossing pairs:~Theorem \ref{5.1}). 
It is \cite[Theorem 3.1]{ambro} for embedded {\em{simple}} 
normal crossing pairs. It follows easily from 
Step \ref{step2} since we already have the relative 
vanishing lemma in Step \ref{step3}. 
A key point in this step is Lemma \ref{6}, which 
is missing in \cite{ambro} and 
works only for embedded {\em{simple}} normal crossing pairs. 
\end{steste}
\begin{steste}\label{step5}
(Torsion-free and vanishing theorems for 
embedded simple normal crossing pairs:~Theorem \ref{8}). 
It is \cite[Theorem 3.2]{ambro} for embedded {\em{simple}} normal 
crossing pairs. 
The proof uses the lemmas on desingularization and compactification 
(see Lemmas \ref{6} and \ref{comp}), which 
hold only for embedded {\em{simple}} 
normal crossing pairs, and the injectivity 
theorem proved for embedded {\em{simple}} normal 
crossing pairs in Step \ref{step4}. 
Therefore, this step also works only for embedded {\em{simple}} 
normal crossing pairs. Our proof of 
the vanishing theorem is slightly 
different from Ambro's 
one. 
Compare Steps \ref{o2} and \ref{o3} in the 
proof of Theorem \ref{8} with (a) and (b) in the proof 
of \cite[Theorem 3.2 (ii)]{ambro}. See Remark \ref{smooth-ne}. 
\end{steste}
\begin{steste}\label{step6} 
(Ambro's theorems:~Theorems \ref{61} and 
\ref{62}). 
In this final step, we recover Ambro's theorems, that is, 
\cite[Theorems 3.1 and 3.2]{ambro}, in full generality. 
Since we have already proved \cite[Theorem 3.2 (i)]{ambro} for 
embedded {\em{simple}} normal crossing pairs in Step \ref{step5}, 
we can reduce the problems to the case when 
the varieties are embedded {\em{simple}} normal 
crossing pairs 
by blow-ups and Leray's spectral sequences. 
This step is described in Section \ref{sec6}. 
\end{steste}
\end{say}

\chapter{Log Minimal Model Program for lc pairs}\label{chap3} 

In this chapter, 
we discuss the log minimal model program (LMMP, 
for short) for log canonical pairs. 

In Section \ref{31-sec}, we 
will explicitly state the LMMP for lc 
pairs. 
We state the cone and contraction theorems 
explicitly for lc pairs with the 
additional estimate of lengths of extremal rays. 
We also write the flip conjectures for 
lc pairs. We note 
that the flip conjecture I (existence of an lc flip) 
is still open and that the flip 
conjecture II (termination of a sequence of 
lc flips) follows from the termination of klt flips. 
We give a proof of the flip conjecture I in dimension four. 

\begin{thm}[{cf.~Theorem \ref{lcflip-th}}] 
Log canonical flips exist in dimension four. 
\end{thm}

In Section \ref{qlog-sec}, we introduce the notion of 
quasi-log varieties. 
We think that the notion of quasi-log varieties is indispensable 
for investigating lc pairs. 
The reader can find that the key points of the theory of 
quasi-log varieties in \cite{ambro} 
are adjunction and the vanishing theorem (see \cite[Theorem 4.4]{ambro} and Theorem \ref{adj-th}). 
Adjunction and the 
vanishing theorems for quasi-log varieties follow from \cite[3.~Vanishing 
Theorems]{ambro}. 
However, Section 3 of \cite{ambro} 
contains various troubles. 
Now Chapter \ref{chap2} gives us sufficiently 
powerful vanishing and torsion-free theorems 
for the theory of 
quasi-log varieties. 
We succeed in removing all the troublesome 
problems for the foundation of the theory of 
quasi-log varieties. It is one of the main contributions 
of this chapter and \cite{fuj-lec}. 
We slightly change Ambro's formulation. By this 
change, 
the theory of quasi-log varieties becomes more accessible. 
As a byproduct, we have the following definition of 
quasi-log varieties.  

\begin{defn}[Quasi-log varieties]\index{quasi-log 
variety}\label{new-def}
A {\em{quasi-log variety}} is a scheme $X$ endowed with an 
$\mathbb R$-Cartier $\mathbb R$-divisor 
$\omega$, a proper closed subscheme 
$X_{-\infty}\subset X$, and a finite collection $\{C\}$ of reduced 
and irreducible subvarieties of $X$ such that there is a 
proper morphism $f:Y\to X$ from 
a simple normal crossing divisor $Y$ on a smooth 
variety $M$ 
satisfying the following properties: 
\begin{itemize}
\item[(0)] there exists an $\mathbb R$-divisor 
$D$ on $M$ such that 
$\Supp (D+Y)$ is simple normal crossing on $M$ and 
that $D$ and $Y$ have no common 
irreducible components. 
\item[(1)] $f^*\omega\sim_{\mathbb R}K_Y+B_Y$, 
where $B_Y=D|_{Y}$. 
\item[(2)] The natural map 
$\mathcal O_X
\to f_*\mathcal O_Y(\ulcorner -(B_Y^{<1})\urcorner)$ 
induces an isomorphism 
$$
\mathcal I_{X_{-\infty}}\to f_*\mathcal O_Y(\ulcorner 
-(B_Y^{<1})\urcorner-\llcorner B_Y^{>1}\lrcorner),  
$$
where $\mathcal I_{X_{-\infty}}$ is the defining ideal sheaf of $X_{-\infty}$. 
\item[(3)] The collection of subvarieties $\{C\}$ coincides with the image 
of $(Y, B_Y)$-strata that are not included in $X_{-\infty}$. 
\end{itemize}
\end{defn} 
Definition \ref{new-def} is equivalent to 
Ambro's original definition (see \cite[Definition 4.1]{ambro} and Definition \ref{qlog-def-ambro}). 
For the details, see the subsection \ref{3-2-6}. 
However, we think Definition \ref{new-def} is much 
better than Ambro's. 
Once we adopt Definition \ref{new-def}, 
we do not need the notion of {\em{normal 
crossing pairs}} to define quasi-log varieties and 
get flexibility in the choice of {\em{quasi-log resolutions}} 
$f:Y\to X$ by Proposition \ref{taisetsu3}. 

In Section \ref{33-sec}, we will prove the fundamental 
theorems for the theory of quasi-log varieties 
such as cone, contraction, rationality, and 
base point free theorems. 

The paper \cite{fuj-lec} is a gentle introduction 
to the log minimal model 
program for lc pairs. It may be better to 
see \cite{fuj-lec} before reading this chapter. 

\section{LMMP for log canonical pairs}\label{31-sec}
\subsection{Log minimal model program}
In this subsection, we explicitly state the log minimal model program 
(LMMP, for short) for log canonical pairs. It is known to some experts but 
we can not find it in the standard literature. 
The following cone theorem is a consequence of 
Ambro's cone theorem for 
quasi-log varieties (see Theorem 5.10 in \cite{ambro}, 
Theorems \ref{cont-th} and \ref{cone-thm} below) 
except for the existence of $C_j$ with 
$0<-(K_X+B)\cdot C_j\leq 2 \dim X$ in 
Theorem \ref{coco-th} (1). 
We will discuss the estimate of lengths of extremal rays 
in the subsection \ref{leng-ssec}. 

\begin{thm}[Cone and contraction theorems]\label{coco-th} 
Let $(X, B)$ be an 
lc pair, $B$ an $\mathbb R$-divisor, and $f:X\to Y$ a 
projective morphism between algebraic  
varieties. Then we have 
\begin{itemize}
\item[{\em{(i)}}] There are $($countably many$)$ 
rational curves $C_j\subset X$ such that 
$f(C_j)$ is a point, $0<-(K_X+B)\cdot C_j\leq 
2\dim X$, and 
$$
\overline {NE}(X/Y)=\overline {NE}(X/Y)_{(K_X+B)\geq 0}
+\sum \mathbb R_{\geq 0}[C_j]. 
$$
\item[{\em{(ii)}}] For any $\varepsilon >0$ and $f$-ample 
$\mathbb R$-divisor $H$, 
$$
\overline {NE}(X/Y)=\overline {NE}(X/Y)_{(K_X+B+\varepsilon 
H)\geq 0}
+\underset{\text{\em{finite}}}\sum \mathbb R_{\geq 0}[C_j]. 
$$
\item[{\em{(iii)}}] Let $F\subset \overline {NE}(X/Y)$ be a 
$(K_X+B)$-negative 
extremal face. Then there is a unique morphism 
$\varphi_F:X\to Z$ over $Y$ such that 
$(\varphi_F)_*\mathcal O_X\simeq \mathcal O_Z$, 
$Z$ is projective over 
$Y$, and 
an irreducible curve $C\subset X$ is mapped to a point 
by $\varphi_F$ if and only if $[C]\in F$. 
The map 
$\varphi_F$ is called the contraction of $F$. 
\item[{\em{(iv)}}] Let $F$ and $\varphi_F$ be as in {\em{(iii)}}. Let $L$ 
be a line bundle on $X$ such that 
$L\cdot C=0$ for every curve $C$ with 
$[C]\in F$. Then there is a line bundle $L_Z$ on $Z$ such that 
$L\simeq \varphi_F^*L_Z$. 
\end{itemize}
\end{thm}

\begin{rem}[Lengths of extremal rays]\label{est-re} 
In Theorem \ref{coco-th} (i), 
the estimate $-(K_X+B)\cdot C_j\leq 2 \dim X$ 
should be replaced by $-(K_X+B)\cdot C_j\leq \dim X+1$. 
For toric varieties, this conjectural estimate and some generalizations 
were obtained in \cite{fuji-tor} and \cite{fuji-to2}. 
\end{rem}

The following proposition is obvious. 
See, for example, \cite[Proposition 3.36]{km}. 

\begin{prop}\label{di-fa-prop}
Let $(X, B)$ be a $\mathbb Q$-factorial 
lc pair and let $\pi: X\to S$ be a projective morphism. 
Let $\varphi_R:X\to Y$ be the contraction of a $(K_X+B)$-negative extremal ray $R\subset \overline {NE}(X/S)$. 
Assume that $\varphi_R$ is either a divisorial 
contraction $($that is, $\varphi_R$ contracts a divisor 
on $X$$)$ 
or a Fano contraction $($that is, $\dim Y<\dim X$$)$. 
Then 
\begin{itemize}
\item[$(1)$] $Y$ is $\mathbb Q$-factorial, and 
\item[$(2)$] $\rho (Y/S)=\rho (X/S)-1$. 
\end{itemize}
\end{prop}

By the above cone and contraction theorems, we can easily see that the 
LMMP, 
that is, a recursive procedure explained in \cite[3.31]{km} 
(see also the subsection \ref{sub153}),  
works for $\mathbb Q$-factorial 
log canonical pairs if the flip 
conjectures (Flip Conjectures I and II) hold. 

\begin{conj}\label{fI-conj}{\em{((Log) 
Flip Conjecture I:~The existence of a (log) flip)}}\index{flip conjecture 
I}.  
Let $\varphi\colon (X,B)\to W$ be an 
extremal flipping contraction of an
$n$-dimensional pair, that is,
\begin{itemize}
\item[$(1)$] $(X, B)$ is lc, $B$ is an $\mathbb R$-divisor, 
\item[$(2)$] $\varphi$ is small projective and $\varphi$ has only connected fibers, 
\item[$(3)$] $-(K_X+B)$ is $\varphi$-ample, 
\item[$(4)$] $\rho(X/W)=1$, and 
\item[$(5)$] $X$ is $\mathbb{Q}$-factorial. 
\end{itemize}
Then there should be a diagram{\em{:}}
$$
\begin{matrix}
X & \dashrightarrow & \ X^+ \\
{\ \ \ \ \ \searrow} & \ &  {\swarrow}\ \ \ \ \\
 \ & W &  
\end{matrix}
$$
which satisfies the following conditions\emph{:} 
\begin{itemize}
\item[{\em{(i)}}] $X^+$ is a normal variety, 
\item[{\em{(ii)}}] $\varphi^+\colon X^+\to W$ 
is small projective, and 
\item[{\em{(iii)}}] $K_{X^+}+B^+$ is $\varphi^+$-ample, 
where $B^+$ is the strict transform of $B$. 
\end{itemize}
We call $\varphi^+:(X^+, B^+)\to W$ a {\em{$(K_X+B)$-flip}} 
of $\varphi$. 
\end{conj}

We note the following proposition. 
See, for example, \cite[Proposition 3.37]{km}. 

\begin{prop}\label{fli-prop} 
Let $(X, B)$ be a $\mathbb Q$-factorial 
lc pair and let $\pi: X\to S$ be a projective morphism. 
Let $\varphi_R:X\to Y$ be the contraction of a $(K_X+B)$-negative extremal ray $R\subset \overline {NE}(X/S)$. 
Let $\varphi_R:X\to Y$ be the flipping contraction 
of $R\subset \overline {NE}(X/S)$ with 
flip $\varphi^+_R: X^+\to Y$. 
Then we have 
\begin{itemize}
\item[$(1)$] $X^+$ is $\mathbb Q$-factorial, and 
\item[$(2)$] $\rho (X^+/S)=\rho (X/S)$. 
\end{itemize}
\end{prop}

Note that to prove Conjecture \ref{fI-conj} we can assume that $B$ is 
a $\mathbb{Q}$-divisor, by
perturbing $B$ slightly. It is known that 
Conjecture \ref{fI-conj} holds when $\dim X=3$ (see 
\cite[Chapter 8]{FA}). 
Moreover, if there exists an $\mathbb R$-divisor $B'$ on $X$ such 
that $K_X+B'$ is klt and $-(K_X+B')$ is $\varphi$-ample, then 
Conjecture \ref{fI-conj} 
is true by \cite{bchm}. The following famous 
conjecture is stronger than Conjecture \ref{fI-conj}. 
We will see it in Lemma \ref{fg-lem}. 

\begin{conj}[Finite generation\index{finite generation 
conjecture}]\label{fg-conj}
Let $X$ be an $n$-dimensional smooth projective variety and 
$B$ a boundary $\mathbb Q$-divisor on $X$ such that 
$\Supp B$ is a simple normal crossing divisor on $X$. 
Assume that $K_X+B$ is big. 
Then the log canonical ring 
$$R(X, K_X+B)=\bigoplus _{m\geq 0}H^0(X, 
\mathcal O_X(\llcorner m(K_X+B)\lrcorner))$$ is a 
finitely generated $\mathbb 
C$-algebra. 
\end{conj}

Note that if there exists a $\mathbb Q$-divisor 
$B'$ on $X$ such that $K_X+B'$ is klt and 
$K_X+B'\sim _{\mathbb Q}K_X+B$, then Conjecture \ref{fg-conj} holds 
by \cite{bchm}. 
See Remark \ref{trun-re}. 

\begin{lem}\label{fg-lem}
Let $f:X\to S$ be a proper surjective morphism 
between normal varieties with connected fibers. 
We assume $\dim X=n$. 
Let $B$ be a $\mathbb Q$-divisor 
on $X$ such that $(X, B)$ is lc. 
Assume that $K_X+B$ is $f$-big. 
Then the relative log canonical ring 
$$R(X/S, K_X+B)=\bigoplus _{m\geq 0}
f_*\mathcal O_X(\llcorner m(K_X+B)\lrcorner)$$ is a finitely generated 
$\mathcal O_S$-algebra if 
{\em{Conjecture \ref{fg-conj}}} holds. 
In particular, {\em{Conjecture \ref{fg-conj}}} implies {\em{Conjecture \ref{fI-conj}}}. 
\end{lem}

The following conjecture is the most general one. 
\begin{conj}[Finite Generation Conjecture]\index{finite 
generation conjecture}\label{fg-conj2} 
Let $f:X\to S$ be a proper surjective morphism 
between normal varieties. 
Let $B$ be a $\mathbb Q$-divisor 
on $X$ such that $(X, B)$ is lc. 
Then the relative log canonical ring 
$$R(X/S, K_X+B)=\bigoplus _{m\geq 0}
f_*\mathcal O_X(\llcorner m(K_X+B)\lrcorner)$$ is a finitely generated 
$\mathcal O_S$-algebra. 
\end{conj}

When $(X, B)$ is klt, we can reduce Conjecture \ref{fg-conj2} 
to the 
case when $K_X+B$ is $f$-big 
by using a canonical bundle formula (see \cite{fuji-mori}). 
Thus, Conjecture \ref{fg-conj2} holds for klt pairs 
by \cite{bchm}. 
When $(X, B)$ is lc but not klt, we do not know if we can reduce 
it to the case when $K_X+B$ is $f$-big or not. 

Before we go to the proof of Lemma \ref{fg-lem}, we note one easy remark. 

\begin{rem}\label{trun-re}
For a graded integral domain 
$R=\underset {m\geq 0}\bigoplus R_m$ and a positive 
integer $k$, 
the truncated ring $R^{(k)}$ is defined by 
$R^{(k)}=\underset{m\geq 0}\bigoplus R_{km}$. 
Then $R$ is finitely generated if and only if 
so is $R^{(k)}$. 
We consider $\Proj R$ when $R$ is 
finitely generated. 
We note that $\Proj R^{(k)}=\Proj R$. 
\end{rem} 
The following argument is well known to the experts. 

\begin{proof}[Proof of {\em{Lemma \ref{fg-lem}}}] 
Since the problem is local, we can shrink $S$ and assume that 
$S$ is affine. 
By compactifying $X$ and $S$ and by the desingularization 
theorem, we can 
further assume that 
$X$ and $S$ are projective, $X$ is smooth, $B$ is effective, 
and $\Supp B$ is a simple normal crossing divisor. 
Let $A$ be a very ample divisor on $S$ and $H\in |{rA}|$ 
a general member 
for $r\gg 0$. 
Note that $K_X+
B+(r-1)f^*A$ is big for $r\gg 0$ (cf.~\cite[Corollary 0-3-4]{kmm}). 
Let $m_0$ be a positive integer such that 
$m_0(K_X+B+f^*H)$ is Cartier. 
By Conjecture \ref{fg-conj}, 
$\underset{m\geq 0}\bigoplus H^0({X}, 
\mathcal O_{X}(mm_0(K_{X}+
{B}+f^*H)))
$ 
is finitely 
generated. 
Thus, the relative 
log canonical model $X'$ over $S$ exists. 
Indeed, by assuming that $m_0$ is sufficiently large and divisible, 
$R(X, K_X+B+f^*H)^{(m_0)}$ is 
generated by 
$R(X, K_X+B+f^*H)_{m_0}$ and 
$|m_0(K_X+B+(r-1)f^*A)|\ne \emptyset$.  
Then 
$X'=\Proj \underset{m\geq 0}\bigoplus H^0({X}, 
\mathcal O_{X}(mm_0(K_{X}+ {B}+f^*H)))
$ and $X'$ is the closure of the image of $X$ by the rational map 
defined by 
the complete linear system $|m_0(K_X+B+rf^*A)|$. 
More precisely, 
let $g:X''\to X$ be the elimination of the 
indeterminacy of the 
rational map defined by $|m_0(K_X+B+rf^*A)|$. 
Let $g':X''\to X'$ be the induced morphism and 
$h:X''\to S$ the morphism defined by the complete linear 
system $|m_0g^*f^*A|$. 
Then it is not difficult to see that $h$ factors through $X'$. 

Therefore, 
$\underset{m\geq 0}\bigoplus f_*\mathcal O_X(mm_0(K_X+B))$ 
is a finitely generated $\mathcal O_S$-algebra 
by the existence of the relative 
log canonical model $X'$ over $S$. 
We finish the proof.  
\end{proof}

The next theorem is an easy consequence of 
\cite{bchm}, \cite{ahk}, \cite{abun}, 
and \cite{re-fu}. 

\begin{thm}
Let $(X, B)$ be a proper four-dimensional 
lc pair such that $B$ is a $\mathbb Q$-divisor 
and $K_X+B$ is big. 
Then the log canonical ring 
$$\underset{m\geq 0}\bigoplus H^0(X, \mathcal O_X
(\llcorner m(K_X+B)\lrcorner))$$ is finitely generated. 
\end{thm}
\begin{proof}
Without loss of generality, we can assume that 
$X$ is smooth projective and 
$\Supp B$ is simple normal crossing. 
Run a $(K_X+B)$-LMMP. 
Then we obtain a log minimal model 
$(X', B')$ by \cite{sho-pre}, \cite{hm} and \cite{ahk} 
with the aid of the special termination theorem 
(cf.~\cite[Theorem 4.2.1]{special}). 
By \cite[Theorem 3.1]{re-fu}, which is 
a consequence of the main theorem in \cite{abun}, $K_{X'}+B'$ is 
semi-ample. In particular, 
$\underset{m\geq 0}\bigoplus H^0(X, \mathcal O_X
(\llcorner m(K_X+B)\lrcorner))
\simeq \underset {m\geq 0}\bigoplus H^0(X', \mathcal O_{X'}
(\llcorner m(K_{X'}+B')\lrcorner))$ 
is finitely generated. 
\end{proof}

As a corollary, we obtain the next theorem by Lemma \ref{fg-lem}. 

\begin{thm}\label{lcflip-th} 
{\em{Conjecture \ref{fI-conj}}} is true if $\dim X\leq 4$. 
\end{thm}

More generally, we have the following theorem. 

\begin{thm}
{\em{Conjecture \ref{fg-conj2}}} is true if $\dim X\leq 4$. 
\end{thm}
For the proof, see \cite{birkar}, \cite{fuji-finite}, 
and \cite{fukuda2}. 
Let us go to the flip conjecture II. 

\begin{conj}\label{fII-conj}{\em{((Log) 
Flip Conjecture II:~Termination 
of a sequence of $($log$)$ flips)}}\index{flip conjecture II}.  
A sequence of $($log$)$ flips
$$
(X_0,B_0)\dasharrow 
(X_1,B_1)\dasharrow (X_2,B_2)\dasharrow \cdots  
$$
terminates after finitely many steps.  Namely, there does 
not exist an infinite
sequence of $($log$)$ flips.
\end{conj}
Note that it is sufficient to prove Conjecture \ref{fII-conj} 
for any sequence of klt flips. 
The termination of dlt flips with dimension $\leq n-1$ implies the 
special termination in dimension $n$. Note that 
we use the formulation in \cite[Theorem 4.2.1]{special}. 
The special termination and the termination of klt flips 
in dimension $n$ implies the termination of 
dlt flips in dimension $n$. 
The termination of 
dlt flips 
in dimension $n$ implies the termination of lc flips in dimension $n$. 
It is because we can use the LMMP for $\mathbb Q$-factorial 
dlt pairs in full generality by \cite{bchm} once we obtain 
the termination of dlt flips. 
The reader can find all the necessary arguments in \cite[4.2, 4.4]{special}. 

\begin{rem}[Analytic spaces]\label{ana-re} 
The proofs of the vanishing theorems in Chapter \ref{chap2} only work for algebraic varieties. 
Therefore, the cone, contraction, and base point free theorems 
stated here 
for lc pairs hold only for algebraic 
varieties. 
Of course, all the results should be proved for complex analytic spaces that 
are projective over any fixed analytic spaces. 
\end{rem}

\subsection{Non-$\mathbb Q$-factorial log minimal model 
program}

In this subsection, we explain the log minimal model 
program for non-$\mathbb Q$-factorial 
lc pairs. It is the most general log minimal model 
program. 
First, let us recall the definition of 
log canonical models. 

\begin{defn}[Log canonical model]\label{z-log-cano}
Let $(X, \Delta)$ be a log canonical pair and $f:X\to S$ 
a proper morphism. 
A pair $(X', \Delta')$ sitting in a diagram 
$$
\begin{matrix}
X & \overset{\phi}{\dashrightarrow} & \ X'\\
{\ \ \ \ \ f\searrow} & \ &  {\swarrow}f'\ \ \ \ \\
 \ & S &  
\end{matrix}
$$
is called a {\em{log canonical model}} of\index{log canonical 
model} $(X, \Delta)$ 
over $S$ if 
\begin{itemize}
\item[(1)] $f'$ is proper, 
\item[(2)] $\phi^{-1}$ has no exceptional divisors, 
\item[(3)] $\Delta'=\phi_*\Delta$, 
\item[(4)] $K_{X'}+\Delta'$ is $f'$-ample, and 
\item[(5)] $a(E, X, \Delta)\leq a(E, X', \Delta')$ for every $\phi$-exceptional divisor $E\subset X$. 
\end{itemize} 
\end{defn}

Next, we explain the minimal model program 
for non-$\mathbb Q$-factorial 
lc pairs (cf.~\cite[4.4]{special}). 

\begin{say}[MMP for non-$\mathbb Q$-factorial 
lc pairs]
We start with a pair $(X, \Delta)=(X_0, \Delta_0)$. 
Let $f_0: X_0\to S$ be a projective 
morphism. The aim is to set up a recursive 
procedure which creates intermediate pairs 
$(X_i, \Delta_i)$ and 
projective morphisms $f_i: X_i \to S$. 
After some steps, it should stop with a final 
pair $(X', \Delta')$ and 
$f': X'\to S$. 
\setcounter{step}{-1}
\begin{step}[Initial datum]\label{0ste-new}
Assume that we already constructed $(X_i, \Delta_i)$ 
and $f_i:X_i \to S$ with the following 
properties: 
\begin{itemize}
\item[(1)] $(X_i, \Delta_i)$ is lc, 
\item[(2)] $f_i$ is projective, and 
\item[(3)] $X_i$ is not necessarily $\mathbb Q$-factorial. 
\end{itemize}
If $X_i$ is $\mathbb Q$-factorial, then 
it is easy to see that $X_k$ is also $\mathbb Q$-factorial 
for any $k\geq i$. 
Even when $X_i$ is not $\mathbb Q$-factorial, 
$X_{i+1}$ sometimes becomes $\mathbb Q$-factorial. 
See, for example, Example \ref{final-ex} below. 

\end{step}
\begin{step}[Preparation]\label{1ste-new}
If $K_{X_i}+\Delta_i$ is $f_i$-nef, then we 
go directly to Step \ref{3ste-new} (2). If 
$K_{X_i}+\Delta_i$ is not $f_i$-nef, 
then we establish two results: 
\begin{itemize}
\item[(1)] (Cone Theorem) We have the following 
equality. 
$$\overline {NE}(X_i/S)=\overline {NE}(X_i/S)_{(K_{X_i}+
\Delta_i)\geq 0}+\sum \mathbb R_{\geq 0}[C_i]. $$
\item[(2)] (Contraction Theorem) 
Any $(K_{X_i}+\Delta_i)$-negative extremal ray $R_i\subset \overline {NE}(X_i/S)$ can 
be contracted. 
Let $\varphi_{R_i}:X_i\to Y_i$ denote the corresponding contraction. 
It sits in a commutative diagram. 
$$
\begin{matrix}
X_i & \overset{\varphi_{R_i}}{\longrightarrow} & \ Y_i\\
{\ \ \ \ \ f_i\searrow} & \ &  {\swarrow}g_i\ \ \ \ \\
 \ & S &  
\end{matrix}
$$
\end{itemize}
\end{step}
\begin{step}[Birational transformations]\label{2ste-new} 
If $\varphi_{R_i}:X_i\to Y_i$ is birational, 
then we can find an effective 
$\mathbb Q$-divisor $B$ on $X_i$ such that 
$(X_i, B)$ is log canonical 
and $-(K_{X_i}+B)$ is $\varphi_{R_i}$-ample 
since $\rho (X_i/S)=1$ (cf.~Lemma \ref{sho-lem}). 
Here, we {\em{assume}} that 
$\bigoplus _{m\geq 0}(\varphi_{R_i})_*\mathcal O_{X_i}
(\llcorner m(K_{X_i}+B)\lrcorner)
$ is a finitely generated $\mathcal O_{Y_i}$-algebra. 
We put 
$$
X_{i+1}=\Proj _{Y_i}\bigoplus _{m\geq 0}
(\varphi_{R_i})_*\mathcal O_{X_i}
(\llcorner m(K_{X_i}+B)\lrcorner), 
$$ 
where $\Delta_{i+1}$ is the strict transform 
of $(\varphi_{R_i})_*\Delta_i$ on $X_{i+1}$. 

We note 
that $(X_{i+1}, \Delta_{i+1})$ is the log canonical 
model of $(X_i, \Delta_i)$ over $Y_i$ (see Definition \ref{z-log-cano}). 
It can be checked easily that 
$\varphi_{R_i}^+: X_{i+1}\to Y_i$ is a small 
projective morphism and that $(X_{i+1}, \Delta_{i+1})$ is 
log canonical. 
Then we go back to Step \ref{0ste-new} with 
$(X_{i+1}, \Delta_{i+1})$, $f_{i+1}=
g_i\circ \varphi_{R_i}^+$ and start anew. 

If $X_i$ is $\mathbb Q$-factorial, then so is 
$X_{i+1}$. If $X_i$ is $\mathbb Q$-factorial and 
$\varphi_{R_i}$ is not small, then 
$\varphi_{R_i}^+:X_{i+1}\to Y_i$ is an isomorphism. 
It may happen that $\rho (X_i/S)<\rho (X_{i+1}/S)$ when 
$X_i$ is not $\mathbb Q$-factorial. See, 
for example, Example \ref{final-ex} below.  

\end{step}
\begin{step}[Final outcome]\label{3ste-new} 
We expect that eventually the procedure 
stops, and we get one of the following 
two possibilities: 
\begin{itemize}
\item[(1)] (Mori fiber space) 
If $\varphi_{R_i}$ is a Fano contraction, 
that is, $\dim Y_i<\dim X_i$,  
then we set $(X', \Delta')=(X_i, \Delta_i)$ and 
$f'=f_i$. 
\item[(2)] (Minimal model) If 
$K_{X_i}+\Delta_i$ is $f_i$-nef, then we 
again set $(X', \Delta')=(X_i, \Delta_i)$ and $f'=f_i$. 
We can easily check that $(X', \Delta')$ is 
a log minimal model of $(X, \Delta)$ over $S$ in the 
sense of Definition \ref{z-minimal}.  
\end{itemize}
\end{step}
Therefore, all we have to do is to prove 
Conjecture \ref{fg-conj2} for birational 
morphisms and Conjecture \ref{fII-conj}. 
\end{say}
We close this subsection with 
an example of a non-$\mathbb Q$-factorial log 
canonical variety. 

\begin{ex}
Let $C\subset \mathbb P^2$ be 
a smooth 
cubic curve and 
$Y\subset \mathbb P^3$ be a cone 
over $C$. 
Then 
$Y$ is log canonical. In this case, 
$Y$ is not $\mathbb Q$-factorial. 
We can check it as follows. 
Let $f:X=\mathbb P_C(\mathcal O_C\oplus \mathcal L)\to Y$ 
be a resolution such that 
$K_X+E=f^*K_Y$, where 
$\mathcal L=\mathcal O_{\mathbb P^2}(1)|_C$ 
and $E$ is the exceptional curve. 
We take $P, Q\in C$ such that 
$\mathcal O_C(P-Q)$ is not a torsion in 
$\Pic ^0(C)$. 
We consider $D=\pi^*P-\pi^*Q$, where 
$\pi:X=\mathbb P_C(\mathcal O_C\oplus \mathcal L)\to C$. 
We put $D'=f_*D$. 
If $D'$ is $\mathbb Q$-Cartier, 
then $mD=f^*mD'+aE$ for 
some $a\in \mathbb Z$ and $m\in \mathbb Z_{>0}$. 
Restrict it to $E$. Then 
$\mathcal O_C(m(P-Q))\simeq \mathcal O_E(aE)\simeq 
(\mathcal L^{-1})^{\otimes a}$. 
Therefore, we obtain that 
$a=0$ and $m(P-Q)\sim 0$. It is 
a contradiction. 
Thus, $D'$ is not $\mathbb Q$-Cartier. 
In particular, $Y$ is not $\mathbb Q$-factorial. 
\end{ex}

\subsection{Lengths of extremal rays}\label{leng-ssec}
In this subsection, we consider the estimate of lengths of extremal rays. 
Related topics are in \cite{bchm}. 
Let us recall the following easy lemma. 

\begin{lem}[{cf.~\cite[Lemma 1]{sho-7}}]\label{sho-lem}
Let $(X, B)$ be an lc pair, 
where $B$ is 
an $\mathbb R$-divisor. Then there 
are positive real numbers 
$r_i$ and effective $\mathbb Q$-divisors $B_i$ for $1\leq i\leq l$ 
and a positive integer $m$ such that 
$\sum^{l}_{i=1}r_i=1$, $K_X+B=\sum ^{l}_{i=1}
r_i(K_X+B_i)$, $(X, B_i)$ is lc, and 
$m(K_X+B_i)$ is Cartier for any $i$. 
\end{lem}

The next result is essentially due to \cite{kawamata} and 
\cite[Proposition 1]{sho-7}. 

\begin{prop}\label{leng-prop}
We use the notation in {\em{Lemma \ref{sho-lem}}}. 
Let $(X, B)$ be an lc pair, $B$ an $\mathbb R$-divisor, and $f:X\to Y$ 
a projective morphism between algebraic 
varieteis. 
Let $R$ be a $(K_X+B)$-negative extremal ray 
of $\overline {NE}(X/Y)$. 
Then we can find a rational curve $C$ on $X$ such that 
$[C]\in R$ and $-(K_X+B_i)\cdot C\leq 2\dim X$ for 
any $i$. 
In particular, $-(K_X+B)\cdot C\leq 2\dim X$. More 
precisely, we can write $-(K_X+B)\cdot C=
\sum ^{l}_{i=1} \frac{r_in_i}{m}$, 
where $n_i\in \mathbb Z$ and 
$n_i\leq 2m\dim X$ for 
any $i$. 
\end{prop}
\begin{proof}
By replacing $f:X\to Y$ with 
the extremal contraction $\varphi_R:X\to W$ over $Y$, we can 
assume that the relative Picard number $\rho (X/Y)=1$. 
In particular, $-(K_X+B)$ is $f$-ample. 
Therefore, we can assume that 
$-(K_X+B_1)$ is $f$-ample and $-(K_X+B_i)=
-s_i(K_X+B_1)$ in $N^1(X/Y)$ with 
$s_i\leq 1$ for any $i\geq 2$. 
Thus, it is sufficient to find a rational curve $C$ such that 
$f(C)$ is a point and that $-(K_X+B_1)\cdot C\leq 
2\dim X$. So, we can assume that $K_X+B$ is 
$\mathbb Q$-Cartier and lc. 
By \cite{bchm}, there is a birational morphism $g:(W, B_W)\to (X, B)$ such that $K_W+B_W=g^*(K_X+B)$, 
$W$ is $\mathbb Q$-factorial, $B_W$ is effective, and 
$(W, \{B_W\})$ is klt. 
By \cite[Theorem 1]{kawamata}, we can find 
a rational curve $C'$ on $W$ such that 
$-(K_W+B_W)\cdot C'\leq 2\dim W=2\dim X$ and 
that $C'$ spans 
a $(K_W+B_W)$-negative extremal ray. Note that Kawamata's 
proof works in the above situation with only small modifications. 
See the proof of Theorem 10-2-1 in \cite{ma-bon} 
and Remark \ref{re-03} below. 
By the projection formula, the $g$-image of $C'$ is a desired 
rational curve. So, we finish the 
proof.    
\end{proof}

\begin{rem}\label{re-03}
Let $(X,D)$ be an lc pair, $D$ an $\mathbb R$-divisor. 
Let $\phi:X\to Y$ be a projective morphism and 
$H$ a Cartier divisor on $X$. Assume that 
$H-(K_X+D)$ is $f$-ample. 
By Theorem \ref{kvn}, $R^q\phi_*\mathcal O_X(H)=0$ for 
any $q>0$ if $X$ and $Y$ are {\em{algebraic}} varieties. 
If this vanishing theorem holds for analytic spaces $X$ and $Y$, 
then Kawamata's original argument in \cite{kawamata} 
works directly for lc pairs. 
In that case, we do not need the results in \cite{bchm} in the 
proof of Proposition \ref{leng-prop}. 

We consider the proof of \cite[Theorem 10-2-1]{ma-bon} when 
$(X, D)$ is lc such that $(X, \{D\})$ is klt. 
We need $R^1\phi_*\mathcal O_X(H)=0$ after 
shrinking $X$ and $Y$ analytically. 
In our situation, $(X, D-\varepsilon \llcorner D\lrcorner)$ is klt for 
$0<\varepsilon \ll 1$. 
Therefore, $H-(K_X+D-\varepsilon \llcorner D\lrcorner)$ is $\phi$-ample 
and $(X, D-\varepsilon \llcorner D\lrcorner)$ is klt for $0<\varepsilon 
\ll 1$. Thus, we can apply the analytic 
version of the relative Kawamata--Viehweg  vanishing theorem. 
So, we do not need the analytic version of 
Theorem \ref{kvn}. 
\end{rem}

By Proposition \ref{leng-prop}, Lemma 2.6 in \cite{birkar} holds for lc pairs. 
For the proof, see \cite[Lemma 2.6]{birkar}. 
It may be useful for the LMMP with scaling. 

\begin{prop}\label{bir-prop} 
Let $(X,B)$ be an lc pair, $B$ an $\mathbb R$-divisor, 
and $f:X\to Y$ a projective morphism between algebraic 
varieties. Let $C$ be an effective $\mathbb R$-Cartier 
divisor on $X$ such that 
$K_X+B+C$ is $f$-nef and $(X, B+C)$ is lc. 
Then, either $K_X+B$ is also $f$-nef or there 
is a $(K_X+B)$-negative 
extremal ray $R$ such that $(K_X+B+\lambda C)\cdot R=0$, 
where 
$$\lambda:=\inf\{ t\geq 0 \, |\,  K_X+B+tC \ {\text{is $f$-nef}} \,\}. 
$$ 
Of course, $K_X+B+\lambda C$ is $f$-nef.  
\end{prop}
The following picture helps the reader 
understand Proposition \ref{bir-prop}. 
 
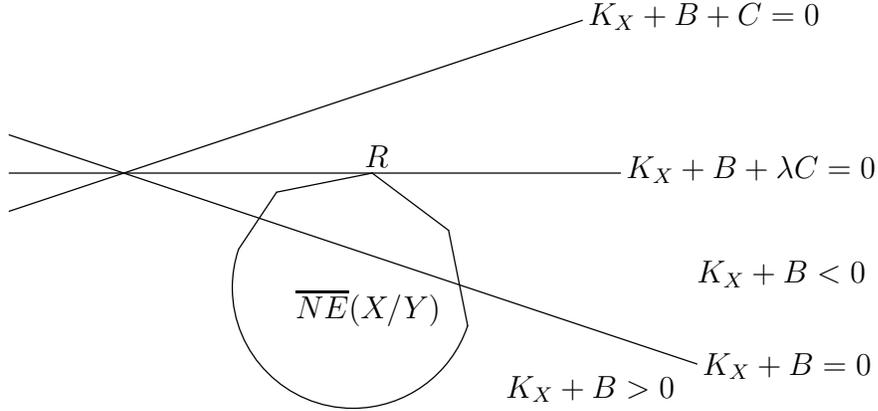
\begin{center}
\unitlength 0.1in
\begin{picture}( 46.0000, 21.5200)(  2.0000,-24.3200)
%
\special{pn 8}%
\special{pa 200 1200}%
\special{pa 3400 1200}%
\special{fp}%
\special{pa 200 1400}%
\special{pa 3200 400}%
\special{fp}%
\special{pa 200 1000}%
\special{pa 3800 2200}%
\special{fp}%
%
\special{pn 8}%
\special{ar 2000 1800 632 632  0.3217506 3.4633432}%
%
\special{pn 8}%
\special{pa 2600 2000}%
\special{pa 2500 1500}%
\special{fp}%
\special{pa 2500 1500}%
\special{pa 2100 1200}%
\special{fp}%
\special{pa 2100 1200}%
\special{pa 1600 1300}%
\special{fp}%
\special{pa 1600 1300}%
\special{pa 1400 1600}%
\special{fp}%
\put(17.0000,-20.0000){\makebox(0,0)[lb]{$\overline{NE}(X/Y)$}}%
\put(20.6000,-11.7000){\makebox(0,0)[lb]{$R$}}%
\put(32.3000,-4.5000){\makebox(0,0)[lb]{$K_{X}+B+C=0$}}%
\put(34.3000,-12.5000){\makebox(0,0)[lb]{$K_{X}+B+\lambda C=0$}}%
\put(38.0000,-18.0000){\makebox(0,0)[lb]{$K_{X}+B<0$}}%
\put(38.3000,-22.8000){\makebox(0,0)[lb]{$K_{X}+B=0$}}%
%
\put(48.0000,-10.0000){\makebox(0,0)[lb]{}}%
\put(28.0000,-24.0000){\makebox(0,0)[lb]{$K_{X}+B>0$}}%
\end{picture}%
\end{center}

\subsection{Log canonical flops}\label{313ss}

The following theorem is an easy consequence of 
\cite{bchm}. 

\begin{thm}\label{exe85} 
Let $(X, \Delta)$ be a klt pair 
and $D$ a $\mathbb Q$-divisor 
on $X$. 
Then $\bigoplus _{m\geq 0}
\mathcal O_X(\llcorner mD\lrcorner)$ is 
a finitely generated $\mathcal O_X$-algebra. 
\end{thm}

\begin{proof}[Sketch of the proof]
If $D$ is $\mathbb Q$-Cartier, then the claim is 
obvious. So, we assume that $D$ is not $\mathbb Q$-Cartier.  
We can also assume that 
$X$ is quasi-projective. 
By \cite{bchm}, we take a birational 
morphism $f:Y\to X$ such that 
$Y$ is $\mathbb Q$-factorial, 
$f$ is small projective, and $(Y, \Delta_Y)$ is 
klt, where $K_Y+\Delta_Y=f^*(K_X+\Delta)$. 
Then the strict transform 
$D_Y$ of $D$ on $Y$ is $\mathbb Q$-Cartier. 
Let $\varepsilon$ be a small positive number. 
By applying the MMP with scaling for the 
pair $(Y, \Delta _Y+\varepsilon D_Y)$ over $X$, we can assume 
that $D_Y$ is $f$-nef, 
Therefore, by the base point free theorem, 
$\bigoplus _{m\geq 0}
f_*\mathcal O_Y(\llcorner mD_Y\lrcorner)\simeq 
\bigoplus _{m\geq 0}\mathcal O_X(\llcorner mD\lrcorner)$ 
is finitely generated as an $\mathcal O_X$-algebra. 
\end{proof}

The next example shows that 
Theorem \ref{exe85} is not true 
for lc pairs. 
In other words, 
if $(X, \Delta)$ is lc, then $\bigoplus _{m\geq 0}
\mathcal O_X(\llcorner mD\lrcorner)$ 
is not necessarily finitely generated as 
an $\mathcal O_X$-algebra. 

\begin{ex}[{cf.~\cite[Exercise 95]{ko-exe}}]\label{exe88} 
Let $E\subset \mathbb P^2$ be a smooth cubic 
curve. 
Let $S$ be a surface obtained by blowing up nine 
general points on $E$ and $E_S\subset S$ 
the strict transform of $E$. Let $H$ be a very 
ample divisor on $S$ giving 
a projectively normal embedding $S\subset \mathbb P^n$. 
Let $X\subset \mathbb A^{n+1}$ be the cone 
over $S$ and $D\subset X$ the cone over $E_S$. 
Then $(X, D)$ is lc since $K_S+E_S\sim 0$ (cf.~Proposition \ref{438p}). 
Let $P\in D\subset X$ be the vertex 
of the cones $D$ and $X$. 
Since $X$ is normal, we have 
\begin{align*}
H^0(X, \mathcal O_X(mD)) &=H^0(X\setminus P, \mathcal O_X(mD))
\\& \simeq \bigoplus _{r\in \mathbb Z}H^0(S, 
\mathcal O_S(mE_S+rH)). 
\end{align*} 
By the construction, 
$\mathcal O_S(mE_S)$ has only the obvious section which 
vanishes 
along $mE_S$ for any $m>0$. 
It can be checked by the induction on $m$ 
using the following exact sequence 
$$
0\to H^0(X, \mathcal O_S((m-1)E_S))\to 
H^0(S, \mathcal O_S(mE_S))\to H^0(E_S, \mathcal 
O_{E_S}(mE_S))\to \cdots 
$$
since $\mathcal O_{E_S}(E_S)$ is not a torsion element 
in $\Pic ^0(E_S)$. 
Therefore, 
$H^0(S, \mathcal O_S(mE_S+rH))=0$ for any $r<0$. So, we 
have 
$$
\bigoplus _{m\geq 0} \mathcal O_X(mD)
\simeq \bigoplus _{m\geq 0}\bigoplus _{r\geq 0} 
H^0(S, \mathcal O_S(mE_S+rH)). 
$$ 
Since $E_S$ is nef, $\mathcal O_S(mE_S+4H)\simeq 
\mathcal O_S(K_S+E_S+mE_S+4H)$ is very 
ample for any $m\geq 0$. 
Therefore, by replacing $H$ with 
$4H$, we can assume that 
$\mathcal O_S(mE_S+rH)$ is very ample for 
any $m\geq 0$ and $r>0$. 
In this setting, the multiplication maps 
\begin{align*}
\bigoplus _{a=0}^{m-1}H^0(S, \mathcal O_S(aE_S+H))\otimes 
H^0(S, \mathcal O_S((m-a)E_S))
\\ \to H^0(S, \mathcal O_S(mE_S+H))
\end{align*} 
are never surjective.  
This implies that 
$\bigoplus _{m\geq 0}\mathcal O_X(mD)$ 
is not finitely generated as an 
$\mathcal O_X$-algebra. 
\end{ex}

Let us recall the definition of log canonical flops 
(cf.~\cite[6.8 Definition]{FA}). 

\begin{defn}[Log canonical flop]\index{log canonical 
flop}\label{lc-flop-def}
Let $(X, B)$ be an lc pair. Let $H$ be a 
Cartier divisor on $X$. 
Let $f:X\to Z$ be 
a small contraction such that 
$K_X+B$ is numerically $f$-trivial and 
$-H$ is $f$-ample. 
The opposite of $f$ with respect to $H$ is called 
an $H$-flop with respect to 
$K_X+B$ or simply say an $H$-flop. 
\end{defn} 

The following example shows that 
log canonical flops do not always exist. 

\begin{ex}[{cf.~\cite[Exercise 96]{ko-exe}}]\label{exe89} 
Let $E$ be an elliptic curve and $L$ a degree 
zero line bundle on $E$. 
We put $S=\mathbb P_E(\mathcal O_E\oplus L)$. Let 
$C_1$ and $C_2$ be the sections of the $\mathbb P^1$-bundle 
$p:S\to E$. 
We note that $K_S+C_1+C_2\sim 0$. 
As in Example \ref{exe88}, 
we take a sufficiently ample divisor 
$H=aF+bC_1$ on $S$ giving a projectively normal embedding 
$S\subset 
\mathbb P^n$, where $F$ is a fiber of the $\mathbb P^1$-bundle 
$p:S\to E$, $a>0$, and $b>0$. 
We can assume that $\mathcal O_S(mC_i+rH)$ is very ample 
for any $i$, $m\geq 0$, and $r>0$. 
Moreover, we can assume that 
$\mathcal O_S(M+rH)$ is 
very ample for any nef divisor $M$ and 
any $r>0$. 
Let $X\subset \mathbb A^{n+1}$ be a cone over $S$ and 
$D_i\subset X$ the cones over 
$C_i$. 
Since $K_S+C_1+C_2\sim 0$, $(X, D_1+D_2)$ is lc and 
$K_X+D_1+D_2\sim 0$ (cf.~Proposition \ref{438p}). 
By the same arguments as in Example \ref{exe88}, 
we can prove the following statement. 
\begin{cla}\label{exe-cl1}
If $L$ is a non-torsion element in 
$\Pic ^0(E)$, then 
$\bigoplus _{m\geq 0} \mathcal O_X(mD_i)$ is not 
a finitely generated sheaf of $\mathcal O_X$-algebra 
for $i=1$ and $2$. 
\end{cla}
We note that $\mathcal O_S(mC_i)$ has only the obvious 
section which vanishes along $mC_i$ for any $m>0$. 
Let $B\subset X$ be the cone over $F$. 
Then we have the following result. 
\begin{cla}\label{exe-cl2} 
The graded $\mathcal O_X$-algebra $\bigoplus _{m\geq 0} 
\mathcal O_X(mB)$ is a finitely generated 
$\mathcal O_X$-algebra. 
\end{cla} 
\begin{proof}[Proof of {\em{Claim \ref{exe-cl2}}}]
By the same arguments as in Example \ref{exe88}, 
we have 
$$
\bigoplus _{m\geq 0} \mathcal O_X(mB)\simeq 
\bigoplus _{m\geq 0} 
\bigoplus _{r\geq 0} H^0(S, \mathcal O_S(mF+rH)). 
$$ 
We consider $V=\mathbb P_S(\mathcal O_S(F)\oplus \mathcal 
O_S(H))$. 
Then $\mathcal O_V(1)$ is semi-ample. 
Therefore, 
$$
\bigoplus _{n\geq 0} H^0(V, \mathcal O_V(n))\simeq 
\bigoplus _{m\geq 0} \bigoplus _{r\geq 0} 
H^0(S, \mathcal O_S(mF+rH))  
$$ 
is finitely generated. 
\end{proof}
Let $P \in X$ be the vertex of the cone $X$ and 
let $f:Y\to X$ be the blow-up at $P$. 
Let $A\simeq S$ be the exceptional divisor 
of $f$. 
We consider the $\mathbb P^1$-bundle $\pi:
\mathbb P_S(\mathcal O_S\oplus \mathcal O_S(H))\to 
S$. Then $Y\simeq \mathbb P_S(\mathcal O_S\oplus \mathcal O_S(H))\setminus G$, where $G$ is the section of 
$\pi$ corresponding to $\mathcal O_S\oplus \mathcal O_S(H)
\to \mathcal O_S(H)\to 0$. 
We consider $\pi^*F$ on $Y$. 
Then $\mathcal O_Y(\pi^*F)$ is $f$-semi-ample. 
So, we obtain a contraction morphism 
$g:Y\to Z$ over $X$. 
It is easy to see that 
$Z\simeq \Proj _X\bigoplus _{m \geq 0} \mathcal O_X(mB)$ 
and that $h:Z\to X$ is a small projective contraction. 
On $Y$, we have 
$-A\sim \pi^*H=a\pi^*F+b\pi^*C_1$. 
Therefore, we obtain $aB+bD_1\sim 0$ on $X$. 
If $L$ is not a torsion element, then 
the flop of $h:Z\to X$ with respect to 
$D_1$ does not 
exist since $\bigoplus _{m\geq 0} \mathcal O_X(mD_1)$ is 
not finitely generated as an $\mathcal O_X$-algebra. 

Let $C$ be any Cartier divisor on $Z$ such that 
$-C$ is $h$-ample. 
Then the flop of $h:Z\to X$ with respect to 
$C$ exists if and only if 
$\bigoplus _{m\geq 0} h_*\mathcal O_Z(mC)$ is a finitely 
generated $\mathcal O_X$-algebra. 
We can find a positive constant $m_0$ and 
a degree zero Cartier divisor $N$ on $E$ such that 
the finite generation of 
$\bigoplus _{m\geq 0} h_*\mathcal O_Z(mC)$ is equivalent to 
that of $\bigoplus _{m\geq 0}
\mathcal O_X(m(m_0D_1+\widetilde N))$, 
where $\widetilde N\subset X$ is the cone 
over $p^*N\subset S$. 
\begin{cla}\label{exe-cl3}
If $L$ is not a torsion element in $\Pic ^0(E)$, 
then $\bigoplus_{m\geq 0} 
\mathcal O_X(m(m_0D_1+\widetilde N))$ 
is not finitely generated as an $\mathcal O_X$-algebra. 
In particular, the flop of $h:Z\to X$ with 
respect to $C$ does not exist. 
\end{cla}
\begin{proof}[Proof of {\em{Claim \ref{exe-cl3}}}] 
By the same arguments as in Example \ref{exe88}, 
we have 
\begin{align*}
&\bigoplus_{m\geq 0} \mathcal O_X(m(m_0D_1+\widetilde N))
\\
&\simeq  
\bigoplus _{m\geq 0}\bigoplus _{r\in \mathbb Z}
H^0(S, \mathcal O_S(m(m_0C_1+p^*N)+rH)). 
\end{align*} 
Since $\dim H^0(S, \mathcal O_S(m(m_0C_1+p^*N)))\leq 1$ for 
any $m\geq 0$, we can easily check that the above 
$\mathcal O_X$-algebra is not finitely generated. 
See the arguments in Example \ref{exe88}. 
We note that $\mathcal O_S(m(m_0C_1+p^*N)+rH)$ is very ample 
for any $m\geq 0$ and $r>0$ because 
$m_0C_1+p^*N$ is nef. 
\end{proof}
Anyway, if $L$ is not a torsion element in $\Pic ^0(E)$, 
then the flop of $h:Z\to X$ does not exist. 

In the above setting, we assume that $L$ is a torsion 
element in $\Pic ^0(E)$. 
Then $\mathcal O_Y(\pi^*C_1)$ is $f$-semi-ample. So, 
we obtain a contraction morphism 
$g':Y\to Z^+$ over $X$. 
It is easy to see that $\bigoplus _{m\geq 0}
\mathcal O_X(mD_i)$ is 
finitely generated as an $\mathcal O_X$-algebra 
for $i=1, 2$ (cf.~Claim \ref{exe-cl2}), 
$Z^+\simeq \Proj _X\bigoplus _{m\geq 0}\mathcal O_X(mD_1)$,  
and that $Z^+\to X$ is the flop of $Z\to X$ with 
respect to $D_1$. 

Let $C$ be any Cartier divisor on $Z$ such that 
$-C$ is $h$-ample. 
If $-C\sim _{\mathbb Q, h} cB$ 
for some positive rational number 
$c$, then it is obvious that the above $Z^+\to X$ is the flop of 
$h:Z\to X$ with respect to $C$. 
Otherwise, the flop of $h:Z\to X$ with respect to $C$ does not 
exist. As above, 
we can find a positive constant $m_0$ and 
a non-torsion element 
$N$ in $\Pic^0(E)$ such that 
$\bigoplus _{m\geq 0} h_*\mathcal O_Z(mC)$ is 
finitely generated if and only if 
so is $\bigoplus _{m\geq 0} \mathcal O_X(m(m_0D_1+\widetilde 
N))$, where $\widetilde N\subset X$ is the cone 
over $p^*N\subset S$. By the 
same arguments as in the proof of Claim \ref{exe-cl3}, 
we can easily check that $\bigoplus _{m\geq 0}
\mathcal O_X(m(m_0D_1+\widetilde N))$ is not finitely 
generated as an $\mathcal O_X$-algebra. 
We note that 
$\dim H^0(S, \mathcal O_S(m(m_0D_1+p^*N)))=0$ for any $m>0$ 
since $N$ is a non-torsion element and 
$L$ is a torsion element in $\Pic ^0(E)$.  
\end{ex}

\section{Quasi-log varieties}\label{qlog-sec}

\subsection{Definition of quasi-log varieties}

In this subsection, we introduce the notion of 
{\em{quasi-log varieties}} according to \cite{ambro}. 
Our definition requires slightly stronger assumptions 
than 
Ambro's original 
one. However, we will check that 
our definition is equivalent to Ambro's in the subsection 
\ref{3-2-6}. 

Let 
us recall the definition of global embedded simple normal 
crossing pairs (see Definition \ref{gsnc0}). 

\begin{defn}[Global embedded simple normal crossing 
pairs]\index{global embedded simple normal 
crossing pair}\label{gsnc} 
Let $Y$ be a simple normal crossing divisor 
on a smooth 
variety $M$ and let $D$ be an $\mathbb R$-divisor 
on $M$ such that 
$\Supp (D+Y)$ is simple normal crossing and that 
$D$ and $Y$ have no common irreducible components. 
We put $B_Y=D|_Y$ and consider the pair $(Y, B_Y)$. 
We call $(Y, B_Y)$ a {\em{global embedded simple normal 
crossing pair}}. 
\end{defn}

It's time for us to define {\em{quasi-log varieties}}. 

\begin{defn}[Quasi-log varieties]\index{quasi-log variety}
\label{quasi-log}
A {\em{quasi-log variety}} is a scheme $X$ endowed with an 
$\mathbb R$-Cartier $\mathbb R$-divisor 
$\omega$, a proper closed subscheme 
$X_{-\infty}\subset X$, and a finite collection $\{C\}$ of reduced 
and irreducible subvarieties of $X$ such that there is a 
proper morphism $f:(Y, B_Y)\to X$ from a global 
embedded simple 
normal crossing pair satisfying the following properties: 
\begin{itemize}
\item[(1)] $f^*\omega\sim_{\mathbb R}K_Y+B_Y$. 
\item[(2)] The natural map 
$\mathcal O_X
\to f_*\mathcal O_Y(\ulcorner -(B_Y^{<1})\urcorner)$ 
induces an isomorphism 
$$
\mathcal I_{X_{-\infty}}\to f_*\mathcal O_Y(\ulcorner 
-(B_Y^{<1})\urcorner-\llcorner B_Y^{>1}\lrcorner),  
$$ 
where $\mathcal I_{X_{-\infty}}$ is the defining ideal sheaf of 
$X_{-\infty}$. 
\item[(3)] The collection of subvarieties $\{C\}$ coincides with the image 
of $(Y, B_Y)$-strata that are not included in $X_{-\infty}$. 
\end{itemize}
We sometimes simply say that 
$[X, \omega]$ is a {\em{quasi-log pair}}\index{quasi-log pair}. 
We use the following terminology according to Ambro.  
The subvarieties $C$ 
are the {\em{qlc centers}}\index{qlc center} of $X$, 
$X_{-\infty}$ is the {\em{non-qlc locus}}\index{non-qlc locus}\index{$X_{-\infty}$, non-qlc locus} 
of $X$, and $f:(Y, B_Y)\to X$ is 
a {\em{quasi-log resolution}}\index{quasi-log resolution} 
of $X$. We say that $X$ has {\em{qlc singularities}}\index{qlc 
singularity} 
if $X_{-\infty}=\emptyset$. 
Assume that $[X, \omega]$ is a quasi-log pair 
with $X_{-\infty}=\emptyset$. Then 
we simply say that $[X, \omega]$ is a 
{\em{qlc pair}}.\index{qlc pair} 
Note that a quasi-log variety $X$ is the union of 
its qlc centers and $X_{-\infty}$. A {\em{relative quasi-log variety}} 
$X/S$ is a quasi-log variety $X$ endowed 
with a proper morphism $\pi:X\to S$. 
\end{defn}

\begin{rem}[Quasi-log canonical class]\index{quasi-log 
canonical class}\label{cano}
In Definition \ref{quasi-log}, we assume that 
$\omega$ is an $\mathbb R$-Cartier $\mathbb R$-divisor.\index{$\omega$, quasi-log canonical class} 
However, it may be better to see $\omega\in \Pic (X)\otimes 
_{\mathbb Z}\mathbb R$. 
It is because the {\em{quasi-log canonical class}} 
$\omega$ is 
defined up to $\mathbb R$-linear equivalence and 
we often restrict $\omega$ to a subvariety of $X$. 
\end{rem}

\begin{ex}\label{new-example}
Let $X$ be a normal variety and let $B$ be an effective 
$\mathbb R$-divisor on $X$ such that 
$K_X+B$ is $\mathbb R$-Cartier. 
We take a resolution $f:Y\to X$ such that 
$K_Y+B_Y=f^*(K_X+B)$ and that $\Supp B_Y$ is 
a simple normal crossing divisor on $Y$. 
Then the pair $[X, K_X+B]$ is a 
quasi-log variety with a  
quasi-log resolution $f:(Y, B_Y)\to X$. 
By this quasi-log structure, 
$[X, K_X+B]$ is qlc if and only if 
$(X, B)$ is lc. 
See also Corollary \ref{346346}. 
\end{ex}

\begin{rem}\label{hosoku}
By Definition \ref{quasi-log}, $X$ has only qlc singularities 
if and only if $B_Y$ is a subboundary. 
In this case, $f_*\mathcal O_Y\simeq \mathcal O_X$ 
since
$\mathcal O_X\simeq f_*\mathcal O_Y(\ulcorner 
-(B^{<1}_Y)\urcorner)$. 
In particular, $f$ is surjective when $X$ has 
only qlc singularities. 
\end{rem}

\begin{rem}[Semi-normality]\label{hosoku2}
In general, we have 
$$\mathcal O_{X\setminus X_{-\infty}}\simeq 
f_*\mathcal O_{f^{-1}(X\setminus X_{-\infty})}
(\ulcorner -(B^{<1}_Y)\urcorner 
-\llcorner B^{>1}_Y\lrcorner)
=f_*\mathcal O_{f^{-1}(X\setminus X_{-\infty})}
(\ulcorner -(B^{<1}_Y)\urcorner).$$  
This implies that 
$\mathcal O_{X\setminus X_{-\infty}}\simeq f_*\mathcal 
O_{f^{-1}(X\setminus X_{-\infty})}$. 
Therefore, 
$X\setminus X_{-\infty}$ is 
semi-normal\index{semi-normal} 
since $f^{-1}(X\setminus X_{-\infty})$ is a simple normal 
crossing variety. 
\end{rem}

\begin{rem}\label{sing}
To prove the cone and contraction theorems for 
lc pairs, it is enough to 
treat quasi-log varieties with only qlc singularities. 
For the details, see \cite{fuj-lec}. 
\end{rem}

We close this subsection with an obvious lemma. 

\begin{lem}\label{ob-lem} 
Let $[X, \omega]$ be a quasi-log pair. 
Assume that $X=V\cup X_{-\infty}$ and 
$V\cap X_{-\infty}=\emptyset$. 
Then $[V, \omega']$ is a qlc pair, where 
$\omega'=\omega|_V$. 
\end{lem}

\subsection{Quick review of vanishing and torsion-free 
theorems}

In this subsection, we quickly review Ambro's formulation 
of torsion-free and vanishing theorems in a simplified form. 
For more advanced topics and the proof, see Chapter 
\ref{chap2}. 

We consider a global embedded simple normal crossing 
pair $(Y, B)$. More precisely, 
let $Y$ be a simple normal crossing divisor 
on a smooth 
variety $M$ and let $D$ be an $\mathbb R$-divisor 
on $M$ such that 
$\Supp (D+Y)$ is simple normal crossing and that 
$D$ and $Y$ have no common irreducible components. 
We put $B=D|_Y$ and consider the pair $(Y, B)$. 
Let $\nu:Y^{\nu}\to Y$ be the normalization. 
We put $K_{Y^\nu}+\Theta=\nu^*(K_Y+B)$. 
A {\em{stratum}} of $(Y, B)$ is an irreducible component of $Y$ or 
the image of some lc center of $(Y^\nu, \Theta^{=1})$. 

When $Y$ is smooth and $B$ is an $\mathbb R$-divisor 
on $Y$ such that 
$\Supp B$ is simple normal crossing, we 
put $M=Y\times \mathbb A^1$ and $D=B\times \mathbb A^1$. 
Then $(Y, B)\simeq (Y\times \{0\}, B\times \{0\})$ satisfies 
the above conditions. 

The following theorem is a special 
case of Theorem \ref{8}. 

\begin{thm}\label{ap1}
Let $(Y, B)$ be as above. 
Assume that $B$ is a boundary $\mathbb R$-divisor. 
Let 
$f:Y\to X$ be a proper morphism and $L$ a Cartier 
divisor on $Y$. 

$(1)$ Assume that $H\sim _{\mathbb R}L-(K_Y+B)$ is $f$-semi-ample.
Then 
every non-zero local section of $R^qf_*\mathcal O_Y(L)$ contains 
in its support the $f$-image of 
some strata of $(Y, B)$. 

$(2)$ Let $\pi:X\to V$ be a proper morphism and 
assume that $H\sim _{\mathbb R}f^*H'$ for 
some $\pi$-ample $\mathbb R$-Cartier 
$\mathbb R$-divisor $H'$ on $X$. 
Then, $R^qf_*\mathcal O_Y(L)$ is $\pi_*$-acyclic, that is, 
$R^p\pi_*R^qf_*\mathcal O_Y(L)=0$ for any $p>0$. 
\end{thm}

We need a slight generalization of Theorem 
\ref{ap1} in Section \ref{34-sec}. 
Let us recall the definition of {\em{nef and log big divisors}} for 
the vanishing theorem. 

\begin{defn}[Nef and log big divisors]\index{nef and log 
big divisor}\label{nlb-def}
Let $f:(Y, B_Y)\to X$ be a proper morphism 
from a simple normal crossing 
pair $(Y, B_Y)$. 
Let $\pi:X\to V$ be a proper 
morphism and $H$ an $\mathbb R$-Cartier $\mathbb R$-divisor 
on $X$. 
We say that $H$ is {\em{nef and log big}} 
over $V$ if and only if $H|_C$ is nef and big 
over $V$ for any $C$, where 
\begin{itemize}
\item[(i)] $C$ is a qlc center when $X$ is a quasi-log variety and 
$f:(Y, B_Y)\to X$ is a quasi-log resolution, or 
\item[(ii)] $C$ is the image of a stratum of $(Y, B_Y)$ when  
$B_Y$ is a subboundary.  
\end{itemize} 
If $X$ is a quasi-log variety with only qlc singularities and 
$f:(Y, B_Y)\to X$ is a quasi-log resolution, 
then the above two cases (i) and (ii) coincide. 
When $(X, B_X)$ is 
an lc pair, we choose a log resolution of $(X, B_X)$ 
to be $f:(Y, B_Y)\to X$, 
where $K_Y+B_Y=f^*(K_X+B_X)$. 
We note that if $H$ is ample over $V$ then it is 
obvious that $H$ is nef and log big over $V$.  
\end{defn}

\begin{thm}[cf.~Theorem \ref{74}]\label{nlb-vani-th}
Let $(Y, B)$ be as above. 
Assume that $B$ is a boundary $\mathbb R$-divisor. 
Let $f:Y\to X$ be a proper morphism and $L$ a Cartier 
divisor on $Y$. We put $H\sim _{\mathbb R}L-(K_X+B)$. 
Let $\pi:X\to V$ be a proper morphism and 
assume that $H\sim _{\mathbb R}f^*H'$ for 
some $\pi$-nef and $\pi$-log big $\mathbb R$-Cartier 
$\mathbb R$-divisor $H'$ on $X$. 
Then, 
every non-zero local section of $R^qf_*\mathcal O_Y(L)$ contains 
in its support the $f$-image of 
some strata of $(Y, B)$, and 
$R^qf_*\mathcal O_Y(L)$ is $\pi_*$-acyclic, that is, 
$R^p\pi_*R^qf_*\mathcal O_Y(L)=0$ for any $p>0$. 
\end{thm}

For the proof, see Theorem \ref{74}. 

\subsection{Adjunction and Vanishing Theorem}

The following theorem is one of the key results in the 
theory of quasi-log varieties (cf.~\cite[Theorem 4.4]{ambro}). 

\begin{thm}[Adjunction and vanishing theorem]\index{adjunction and 
the vanishing theorem}\label{adj-th}
Let $[X, \omega]$ be a quasi-log pair and $X'$ the union 
of $X_{-\infty}$ with 
a $($possibly empty$)$ union of some qlc centers of $[X, \omega]$. 
\begin{itemize}
\item[{\em{(i)}}] Assume that $X'\ne X_{-\infty}$. Then 
$X'$ is a quasi-log variety, with $\omega'=\omega|_{X'}$ and 
$X'_{-\infty}=X_{-\infty}$. 
Moreover, the qlc centers of $[X', \omega']$ are exactly the qlc centers of 
$[X, \omega]$ that are included in $X'$. 
\item[{\em{(ii)}}] Assume that $\pi:X\to S$ is proper. 
Let $L$ be a Cartier divisor on $X$ such that 
$L-\omega$ is nef and log big over $S$. Then $\mathcal I_{X'}\otimes 
\mathcal O_X(L)$ is $\pi_*$-acyclic, where $\mathcal I_{X'}$ is 
the defining ideal sheaf of $X'$ on $X$. 
\end{itemize}
\end{thm}

Theorem \ref{adj-th} is the hardest part to 
prove in the theory of quasi-log varieties. 
It is because it depends on the non-trivial 
vanishing and torsion-free theorems for simple normal 
crossing pairs. 
The adjunction for normal divisors 
on normal varieties is investigated 
in \cite{fujino10}. 
See also Section \ref{43-sec}. 
 
\begin{proof}
By blowing up the ambient space $M$ of $Y$, we 
can assume that the union of all strata of $(Y, B_Y)$ mapped 
to $X'$, which 
is denoted by $Y'$, is a union of irreducible components of $Y$ (cf.~Lemma \ref{useful-lem2}). We will 
justify this reduction in a more general 
setting in Proposition \ref{taisetsu} below. 
We put $K_{Y'}+B_{Y'}=(K_Y+B_Y)|_{Y'}$ and 
$Y''=Y-Y'$. We claim that $[X', \omega']$ is a 
quasi-log pair and 
that $f:(Y', B_{Y'})\to X'$ is a quasi-log 
resolution.  
By the construction, $f^*\omega'\sim _{\mathbb R}K_{Y'}+B_{Y'}$ on $Y'$ 
is obvious. 
We put $A=\ulcorner -(B^{<1}_Y)\urcorner$ and 
$N=\llcorner B^{>1}_Y\lrcorner$. We consider the following 
short exact sequence 
$$
0\to \mathcal O_{Y''}(-Y')\to \mathcal O_Y\to \mathcal O_{Y'}\to 0. 
$$
By applying $\otimes \mathcal O_Y(A-N)$, 
we have 
$$
0\to \mathcal O_{Y''}(A-N-Y')\to \mathcal O_Y(A-N)\to \mathcal O_{Y'}
(A-N)\to 0. 
$$ 
By applying $f_*$, we 
obtain 
\begin{align*}
0&\to f_*\mathcal O_{Y''}(A-N-Y')\to f_*\mathcal O_Y(A-N)\to f_*\mathcal O_{Y'}
(A-N)
\\ 
&\to R^1f_*\mathcal O_{Y''}(A-N-Y')\to \cdots. 
\end{align*}
By Theorem \ref{ap1} (i), the support of any 
non-zero local section of $R^1f_*\mathcal O_{Y''}(A-N-Y')$ can not be 
contained in $X'=f(Y')$. 
We note 
that 
$$
(A-N-Y')|_{Y''}-(K_{Y''}+\{B_{Y''}\}+B^{=1}_{Y''}-Y'|_{Y''})
=-(K_{Y''}+B_{Y''})\sim _{\mathbb R}f^*\omega|_{Y''}, 
$$
where $K_{Y''}+B_{Y''}=(K_Y+B_Y)|_{Y''}$. 
Therefore, the connecting homomorphism 
$f_*\mathcal O_{Y'}(A-N)\to 
R^1f_*\mathcal O_{Y''}(A-N-Y')$ is a zero map. 
Thus, 
$$
0\to f_*\mathcal O_{Y''}(A-N-Y')\to \mathcal I_{X_{-\infty}}\to f_*\mathcal O_{Y'}
(A-N)\to 0 
$$ 
is exact. 
We put $\mathcal I_{X'}=f_*\mathcal O_{Y''}(A-N-Y')$. 
Then $\mathcal I_{X'}$ defines 
a scheme structure on $X'$. 
We define $\mathcal I_{X'_{-\infty}}=\mathcal I_{X_{-\infty}}/\mathcal I_{X'}$. 
Then $\mathcal I_{X'_{-\infty}}\simeq f_*\mathcal O_{Y'}(A-N)$ 
by the above exact sequence. 
By the following diagram: 
$$
\xymatrix{
0
\ar[r] 
&
f_*\mathcal O_{Y''}(A-N-Y') 
\ar[d]\ar[r]
&
f_*\mathcal O_Y(A-N)
\ar[d]\ar[r] 
&
f_*\mathcal O_{Y'}(A-N) 
\ar[r]\ar[d] 
& 0 \\ 
0
\ar[r] 
&
f_*\mathcal O_{Y''}(A-Y') 
\ar[r]
&
f_*\mathcal O_Y(A)
\ar[r] 
&
f_*\mathcal O_{Y'}(A) 
\\
0
\ar[r] 
&
\mathcal I_{X'}
\ar[u]\ar[r]
&
\mathcal O_{X}
\ar[u]\ar[r] 
&
\mathcal O_{X'}
\ar[r]\ar[u] 
& 0, 
}
$$
we can see that $\mathcal O_{X'}\to f_*\mathcal O_{Y'}(
\ulcorner -(B^{<1}_{Y'})\urcorner)$ induces 
an isomorphism $\mathcal I_{X'_{-\infty}}\to 
f_*\mathcal O_{Y'}(\ulcorner -(B^{<1}_{Y'})\urcorner-
\llcorner B^{>1}_{Y'}\lrcorner)$. 
Therefore, $[X', \omega']$ is a quasi-log pair 
such that $X'_{-\infty}=X_{-\infty}$. By the construction, the property 
about 
qlc centers are obvious. 
So, we finish the proof of (i). 

Let $f:(Y, B_Y)\to X$ be a quasi-log resolution as in the 
proof of (i). 
Then $f^*(L-\omega)\sim _{\mathbb R}
f^*L-(K_{Y''}+B_{Y''})$ on $Y''$, 
where $K_{Y''}+B_{Y''}=(K_Y+B_Y)|_{Y''}$. 
Note that 
$$
f^*L-(K_{Y''}+B_{Y''})=
(f^*L+A-N-Y')|_{Y''}
-(K_{Y''}+\{B_{Y''}\}+B^{=1}_{Y''}-Y'|_{Y''}) 
$$ 
and that any stratum of $(Y'', B^{=1}_{Y''}-Y'|_{Y''})$ is not mapped 
to $X_{-\infty}=X'_{-\infty}$. Then 
by Theorem \ref{nlb-vani-th} (Theorem \ref{ap1} (ii) 
when $L-\omega$ is $\pi$-ample), 
$$
R^p\pi_*(f_*\mathcal O_{Y''}(f^*L+A-N-Y'))
=R^p\pi_*(\mathcal I_{X'}\otimes \mathcal O_X(L))=0$$ for 
any $p>0$. Thus, we finish the proof of (ii). 
\end{proof}

\begin{rem}
We make a few comments on Theorem \ref{adj-th} for the 
reader's convenience. 
We slightly changed the big diagram 
in the proof of \cite[Theorem 4.4]{ambro} and 
incorporated \cite[Theorem 7.3]{ambro} into 
\cite[Theorem 4.4]{ambro}. 
Please compare Theorem \ref{adj-th} 
with the original statements in \cite{ambro}. 
\end{rem}

\begin{cor}\label{irre-cor} 
Let $[X, \omega]$ be a qlc pair and 
let $X'$ be an irreducible 
component of $X$. 
Then $[X', \omega']$, where 
$\omega'=\omega|_{X'}$, is a qlc pair. 
\end{cor}
\begin{proof}
It is because $X'$ is a qlc center 
of $[X, \omega]$ by Remark \ref{hosoku}. 
\end{proof}

The next example shows that the definition of quasi-log varieties 
is reasonable. 

\begin{ex}\label{311}
Let $(X, B_X)$ be an lc pair. 
Let $f:Y\to (X, B_X)$ be a resolution such that 
$K_Y+S+B=f^*(K_X+B_X)$, where 
$\Supp (S+B)$ is simple normal crossing, $S$ is 
reduced, and $\llcorner B\lrcorner \leq 0$. 
We put $K_S+B_S=(K_X+S+B)|_{S}$ and 
consider the short exact sequence 
$$0\to \mathcal O_Y(\ulcorner -B\urcorner-S)\to 
\mathcal O_Y(\ulcorner -B\urcorner)\to \mathcal O_S(\ulcorner 
-B_S\urcorner)\to 0.$$ Note that $B_S=B|_S$ since $Y$ is 
smooth. 
By the Kawamata--Viehweg vanishing theorem, 
$R^1f_*\mathcal O_Y(\ulcorner -B\urcorner-S)=0$. 
This implies that $f_*\mathcal O_S(\ulcorner 
-B_S\urcorner)\simeq 
\mathcal O_{f(S)}$ since $f_*\mathcal O_Y(\ulcorner -B
\urcorner)\simeq \mathcal O_X$. 
This argument is well known as the proof of the connectedness 
lemma. We put $W=f(S)$ and $\omega=(K_X+B_X)|_{W}$. Then 
$[W, \omega]$ is a quasi-log pair with only qlc singularities and 
$f:(S, B_S)\to W$ is a quasi-log resolution. 
\end{ex}

Example \ref{311} is a very special case of 
Theorem \ref{adj-th} (i), 
that is, adjunction from $[X, K_X+B_X]$ to 
$[W, \omega]$. For 
other examples, see \cite[\S 5]{fujino7} or 
Section \ref{to-sec}, where 
we treat toric polyhedra as quasi-log varieties. 
In the proof of Theorem \ref{adj-th} (i), we used Theorem 
\ref{ap1} (i), which is a generalization of Koll\'ar's theorem, 
instead of the Kawamata--Viehweg vanishing theorem.

\subsection{Miscellanies on qlc centers}

The notion of {\em{lcs locus}} is important for 
$X$-method on quasi-log varieties. 

\begin{defn}[LCS locus]\index{lcs locus}\label{lcs} 
The {\em{LCS locus}} of a quasi-log pair $[X, \omega]$, denoted 
by $\LCS(X)$ or $\LCS(X, \omega)$,\index{$\LCS(X, \omega)$, 
lcs locus} is 
the union of $X_{-\infty}$ with all qlc centers of $X$ that are not maximal 
with respect to the inclusion. 
The subscheme structure is defined in 
Theorem \ref{adj-th} (i), and 
we have a natural embedding $X_{-\infty}\subseteq \LCS(X)$. 
In this book and \cite{fuj-lec}, 
$\LCS(X, \omega)$ is denoted by 
$\Nqklt (X, \omega)$.\index{$\Nqklt(X, \omega)$} 

When $X$ is normal and $B$ is an 
effective $\mathbb R$-divisor 
such that $K_X+B$ is $\mathbb R$-Cartier, $\Nqklt 
(X, K_X+B)$ is denoted by $\Nklt (X, B)$ and 
is called the {\em{non-klt locus}}\index{non-klt locus}\index{$\Nklt(X, B)$, non-klt locus} of the pair $(X, B)$. 
\end{defn}

The next proposition is easy to prove. However, 
in some applications, it may be useful. 
 
\begin{prop}[{cf.~\cite[Proposition 4.7]{ambro}}]\label{normal}
Let $X$ be a quasi-log variety whose LCS locus is empty. 
Then $X$ is normal. 
\end{prop}
\begin{proof}
Let $f:(Y, B_Y)\to X$ be a quasi-log resolution. 
By the assumption, 
every stratum of $Y$ dominates $X$. 
Therefore, $f:Y\to X$ passes through 
the normalization $X^{\nu}\to X$ of $X$. 
This implies that 
$X$ is normal since $f_*\mathcal O_Y\simeq 
\mathcal O_X$ by Remark \ref{hosoku}. 
\end{proof}

\begin{thm}[{cf.~\cite[Proposition 4.8]{ambro}}]\label{qlc-cent-prop} 
Assume that $[X, \omega]$ is a qlc pair. 
We have the following properties: 
\begin{itemize}
\item[{\em{(i)}}] The intersection of two 
qlc centers is a union of 
qlc centers. 
\item[{\em{(ii)}}] 
For any point $P\in X$, the 
set of all qlc centers passing 
through $P$ has a unique 
element $W$. Moreover, $W$ is 
normal at $P$. 
\end{itemize}
\end{thm}
\begin{proof}
Let $C_1$ and $C_2$ be two qlc centers of $[X,\omega]$. 
We fix $P\in C_1\cap C_2$. 
It is enough to find a qlc center $C$ such that 
$P\in C\subset C_1\cap C_2$. The union 
$X'=C_1\cup C_2$ with 
$\omega'=\omega|_{X'}$ is a qlc pair having 
two 
irreducible components. 
Hence, it is not normal at $P$. 
By Proposition \ref{normal}, 
$P\in \Nqklt (X', \omega')$. Therefore, 
there exists a qlc center $C\subset C_1$ with 
$\dim C<\dim C_1$ such that 
$P\in C\cap C_2$. 
If $C\subset C_2$, 
we are done. 
Otherwise, we repeat the argument with $C_1=C$ and 
reach the conclusion in a finite number of 
steps. So, we finish the proof of (i). 
The uniqueness of the minimal qlc center follows 
from (i) and the normality of the minimal center 
follows from Proposition \ref{normal}. 
Thus, we have (ii). 
\end{proof}

\begin{thm}[{cf.~\cite[Theorem 1.1]{ambro2}}]\label{a1} 
We assume that $(X, B)$ is log canonical. Then 
we have the following 
properties. 
\begin{enumerate}
\item[$(1)$] $(X, B)$ has at most finitely 
many lc centers. 
\item[$(2)$] An intersection of two lc centers 
is a union of lc centers. 
\item[$(3)$] Any 
union of lc centers of $(X, B)$ is semi-normal. 
\item[$(4)$] Let $x\in X$ be a closed point such that 
$(X, B)$ is log canonical but not 
Kawamata log terminal at $x$. 
Then there is a unique minimal lc center $W_x$ 
passing through $x$, and $W_x$ is normal at $x$.  
\end{enumerate}
\end{thm}
\begin{proof}
Let $f:(Y, B_Y)\to (X, B)$ be 
a resolution such that 
$K_Y+B_Y=f^*(K_X+B)$ and 
$\Supp B_Y$ is a simple normal crossing divisor. 
Then an lc center of $(X, B)$ is the 
image of some stratum of a simple normal crossing 
variety $B^{=1}_Y$. 
Therefore, $(X, B)$ has at most finitely 
many lc centers. This is (1). 
The statements (2) and (4) are obvious 
by Theorem \ref{qlc-cent-prop}. 
Let $\{C_i\}_{i\in I}$ be a set of 
lc centers 
of $(X, B)$. 
We put $X'=\bigcup _{i\in I}C_i$ and $\omega'=
(K_X+B)|_{X'}$. 
Then $[X', \omega']$ is a qlc pair. 
Therefore, $X'$ is semi-normal by Remarks \ref{hosoku} 
and \ref{hosoku2}. 
This is (3). 
\end{proof}

The following result is an easy consequence of 
adjunction and 
the vanishing theorem:~Theorem \ref{adj-th}. 

\begin{thm}[{cf.~\cite[Theorem 6.6]{ambro}}]
\label{qlc-cent-th} 
Let $[X, \omega]$ be a quasi-log pair and let 
$\pi:X\to S$ be a proper 
morphism such that 
$\pi_*\mathcal O_X\simeq \mathcal O_S$ and $-\omega$ 
is nef and log big over $S$. 
Let $P\in S$ be a closed point. 
\begin{itemize}
\item[{\em{(i)}}] 
Assume that $X_{-\infty}\cap \pi^{-1}(P)\ne \emptyset$ and 
$C$ is a qlc center such that 
$C\cap \pi^{-1}(P)\ne \emptyset$. 
Then $C\cap X_{-\infty}\cap \pi^{-1}(P)\ne \emptyset$. 
\item[{\em{(ii)}}] 
Assume that $[X, \omega]$ is a qlc pair. 
Then the set of all qlc centers intersecting 
$\pi^{-1}(P)$ has a unique 
minimal element with respect to inclusion. 
\end{itemize}
\end{thm}
\begin{proof}
Let $C$ be a qlc center of $[X, \omega]$ such that 
$P\in \pi(C)\cap \pi(X_{-\infty})$. 
Then $X'=C\cup X_{-\infty}$ with $\omega'=\omega|_{X'}$ 
is a quasi-log variety and the restriction map 
$\pi_*\mathcal O_X\to \pi_*\mathcal O_{X'}$ is 
surjective by Theorem \ref{adj-th}. 
Since $\pi_*\mathcal O_X\simeq \mathcal O_S$, $X_{-\infty}$ 
and $C$ intersect over a neighborhood of $P$. 
So, we have (i). 

Assume that $[X, \omega]$ is a qlc pair, that is, 
$X_{-\infty}=\emptyset$. 
Let $C_1$ and $C_2$ be two qlc centers of $[X, \omega]$ 
such that $P\in \pi(C_1)\cap \pi(C_2)$. 
The union $X'=C_1\cup C_2$ with 
$\omega'=\omega|_{X'}$ is a qlc pair and 
the restriction map 
$\pi_*\mathcal O_X\to \pi _*\mathcal O_{X'}$ is 
surjective. 
Therefore, $C_1$ and $C_2$ intersect over $P$. 
Furthermore, the intersection $C_1\cap C_2$ 
is a union of qlc centers by 
Proposition \ref{qlc-cent-prop}. 
Therefore, there exists a unique qlc center 
$C_P$ over a neighborhood 
of $P$ such that 
$C_P\subset C$ for every qlc center $C$ with 
$P\in \pi(C)$. So, we finish the proof of (ii). 
\end{proof}

The following corollary is obvious by 
Theorem \ref{qlc-cent-th}. 

\begin{cor}
Let $(X, B)$ be a proper lc pair. 
Assume that $-(K_X+B)$ is nef and log big and 
that $(X, B)$ is not klt. 
Then there exists a unique minimal lc center $C_0$ such 
that every lc center contains $C_0$. In particular, 
$\Nklt (X, B)$ 
is connected. 
\end{cor}

The next theorem easily follows from \cite[Section 2]{abun}. 

\begin{thm}
Let $(X, B)$ be a projective lc pair. 
Assume that $K_X+B$ is numerically trivial. 
Then $\Nklt (X, B)$ has at most two connected components. 
\end{thm} 
\begin{proof}
By \cite{bchm}, there is a birational morphism 
$f:(Y, B_Y)\to (X, B)$ such that 
$K_Y+B_Y=f^*(K_X+B)$, $Y$ is projective 
and $\mathbb Q$-factorial, $B_Y$ is 
effective, and $(Y, \{B_Y\})$ is klt. 
Therefore, it is sufficient to 
prove that 
$\llcorner B_Y\lrcorner$ has at most two connected components. 
We assume that $\llcorner B_Y\lrcorner \ne 0$. 
Then $K_Y+\{B_Y\}$ is $\mathbb Q$-factorial klt and  
is not pseudo-effective. Apply the arguments in \cite[Proposition 2.1]{abun} 
with using 
the LMMP with scaling (see \cite{bchm}). Then we obtain that 
$\llcorner B_Y\lrcorner$ and 
$\Nklt(X, B)$ have at most two connected components. 
\end{proof}

\subsection{Useful lemmas}
In this subsection, we prepare some useful lemmas 
for making quasi-log resolutions 
with good properties. 

\begin{prop}\label{taisetsu}
Let $f:Z\to Y$ be a proper birational morphism between 
smooth varieties and let $B_Y$ be an 
$\mathbb  R$-divisor on $Y$ such 
that $\Supp B_Y$ is simple normal crossing. 
Assume that $K_Z+B_Z=f^*(K_Y+B_Y)$ and 
that $\Supp B_Z$ is simple normal crossing. 
Then we have $$f_*\mathcal O_Z(\ulcorner -(B^{<1}_Z)\urcorner 
-\llcorner B^{>1}_Z\lrcorner)\simeq 
\mathcal O_Y(\ulcorner -(B^{<1}_Y)\urcorner 
-\llcorner B^{>1}_Y\lrcorner). $$ 
Furthermore, 
let $S$ be a simple normal crossing divisor on $Y$ such 
that $S\subset \Supp B^{=1}_Y$. Let $T$ be the union of the 
irreducible 
components of $B^{=1}_Z$ that are mapped into $S$ by $f$. 
Assume that 
$\Supp f^{-1}_*B_Y\cup \Exc (f)$ is simple normal crossing on $Z$. 
Then we have 
$$f_*\mathcal O_T(\ulcorner -(B^{<1}_T)\urcorner 
-\llcorner B^{>1}_T\lrcorner)\simeq 
\mathcal O_S(\ulcorner -(B^{<1}_S)\urcorner 
-\llcorner B^{>1}_S\lrcorner),$$ 
where 
$(K_Z+B_Z)|_T=K_T+B_T$ and $(K_Y+B_Y)|_S=K_S+B_S$. 
\end{prop}
\begin{proof}
By $K_Z+B_Z=f^*(K_Y+B_Y)$, we 
obtain 
\begin{align*}
K_Z=&f^*(K_Y+B^{=1}_Y+\{B_Y\})\\&+f^*
(\llcorner B^{<1}_Y\lrcorner+\llcorner B^{>1}_Y\lrcorner)
-(\llcorner B^{<1}_Z\lrcorner+\llcorner B^{>1}_Z\lrcorner)
-B^{=1}_Z-\{B_Z\}.
\end{align*} 
If $a(\nu, Y, B^{=1}_Y+\{B_Y\})=-1$ for a prime divisor 
$\nu$ over $Y$, then 
we can check that $a(\nu, Y, B_Y)=-1$ by using 
\cite[Lemma 2.45]{km}. 
Since $f^*
(\llcorner B^{<1}_Y\lrcorner+\llcorner B^{>1}_Y\lrcorner)
-(\llcorner B^{<1}_Z\lrcorner+\llcorner B^{>1}_Z\lrcorner)$ is 
Cartier, we can easily see that 
$f^*(\llcorner B^{<1}_Y\lrcorner+\llcorner B^{>1}_Y\lrcorner)
=\llcorner B^{<1}_Z\lrcorner+\llcorner B^{>1}_Z\lrcorner+E$, 
where $E$ is an effective $f$-exceptional divisor. 
Thus, we obtain 
$$f_*\mathcal O_Z(\ulcorner -(B^{<1}_Z)\urcorner 
-\llcorner B^{>1}_Z\lrcorner)\simeq 
\mathcal O_Y(\ulcorner -(B^{<1}_Y)\urcorner 
-\llcorner B^{>1}_Y\lrcorner).$$  
Next, we consider 
\begin{align*}
0&\to \mathcal O_Z(\ulcorner -(B^{<1}_Z)\urcorner-
\llcorner B^{>1}_Z\lrcorner-T)\\ &\to 
\mathcal O_Z(\ulcorner -(B^{<1}_Z)\urcorner-
\llcorner B^{>1}_Z\lrcorner)
\to \mathcal O_T(\ulcorner -(B^{<1}_T)\urcorner-
\llcorner B^{>1}_T\lrcorner)\to 0. 
\end{align*} 
Since $T=f^*S-F$, where $F$ is an effective $f$-exceptional 
divisor, we can easily see that 
$$f_*\mathcal O_Z(\ulcorner -(B^{<1}_Z)\urcorner 
-\llcorner B^{>1}_Z\lrcorner-T)\simeq 
\mathcal O_Y(\ulcorner -(B^{<1}_Y)\urcorner 
-\llcorner B^{>1}_Y\lrcorner-S).$$  
We note that 
\begin{align*} 
(\ulcorner -(B^{<1}_Z)\urcorner 
-\llcorner B^{>1}_Z\lrcorner-T)-(K_Z+\{B_Z\}+B^{=1}_Z-T)
\\ =-f^*(K_Y+B_Y). 
\end{align*} 
Therefore, every local section 
of $R^1f_*\mathcal O_Z(\ulcorner -(B^{<1}_Z)\urcorner 
-\llcorner B^{>1}_Z\lrcorner-T)$ contains in its support the $f$-image 
of some strata of $(Z, \{B_Z\}+B^{=1}_Z-T)$ by 
Theorem \ref{ap1} (i). 
\begin{claim}
No strata of $(Z, \{B_Z\}+B^{=1}_Z-T)$ are 
mapped into $S$ by $f$. 
\end{claim}
\begin{proof}[Proof of Claim]
Assume that there is a stratum $C$ of $(Z, \{B_Z\}+B^{=1}_Z-T)$ such that 
$f(C)\subset S$. Note that 
$\Supp f^*S\subset \Supp f^{-1}_*B_Y\cup \Exc (f)$ and 
$\Supp B^{=1}_Z\subset \Supp f^{-1}_*B_Y\cup \Exc (f)$. 
Since $C$ is also a stratum of $(Z, B^{=1}_Z)$ and 
$C\subset \Supp f^*S$, 
there exists an irreducible component $G$ of $B^{=1}_Z$ such that 
$C\subset G\subset \Supp f^*S$. 
Therefore, by the definition of $T$, $G$ is an 
irreducible component of $T$ because $f(G)\subset S$ and $G$ is an 
irreducible component of $B^{=1}_Z$. So, 
$C$ is not a stratum of $(Z, \{B_Z\}+B^{=1}_Z-T)$. It is 
a contradiction. 
\end{proof}
On the other hand, $f(T)\subset S$. Therefore, 
$$f_*\mathcal O_T(\ulcorner -(B^{<1}_T)\urcorner
-\llcorner B^{>1}_T\lrcorner)
\to R^1f_*\mathcal O_Z(\ulcorner -(B^{<1}_Z)\urcorner
-\llcorner B^{>1}_Z\lrcorner-T)$$ is a zero map by 
the assumption on the strata of $(Z, B^{=1}_Z-T)$. 
Thus, 
$$f_*\mathcal O_T(\ulcorner -(B^{<1}_T)\urcorner 
-\llcorner B^{>1}_T\lrcorner)\simeq 
\mathcal O_S(\ulcorner -(B^{<1}_S)\urcorner 
-\llcorner B^{>1}_S\lrcorner).$$ 
We finish the proof.  
\end{proof}

The following corollary is obvious by Proposition \ref{taisetsu}. 

\begin{cor}\label{346346}
Let $X$ be a normal variety and let $B$ be an 
effective $\mathbb R$-divisor 
on $X$ such that $K_X+B$ is $\mathbb R$-Cartier. 
Let $f_i:Y_i\to X$ be a resolution of $(X, B)$ for 
$i=1, 2$. 
We put $K_{Y_i}+B_{Y_i}=f^*_i(K_X+B)$ 
and assume that $\Supp B_{Y_i}$ is simple 
normal crossing. 
Then $f_i: (Y_i, B_{Y_i})\to X$ defines a quasi-log 
structure on $[X, K_X+B]$ for $i=1, 2$. 
By taking a common log resolution of $(Y_1, B_{Y_1})$ 
and $(Y_2, B_{Y_2})$ suitably and applying {\em{Proposition 
\ref{taisetsu}}}, we can see that these 
two quasi-log structures coincide. 
Moreover, let $X'$ be the union of $X_{-\infty}$ with 
a union of some qlc centers of $[X, K_X+B]$. 
Then we can see that $f_1:(Y_1, B_{Y_1})\to 
X$ and $f_2:(Y_2, B_{Y_2})\to X$ induce the same 
quasi-log structure on $[X', (K_X+B)|_{X'}]$ by {\em{Proposition \ref{taisetsu}}}. 
\end{cor}

The final results in this section are very useful and 
indispensable for some applications. 

\begin{prop}\label{taisetsu3}
Let $[X, \omega]$ be a quasi-log pair and let $f:(Y, B_Y)\to X$ be a 
quasi-log resolution. Assume that $(Y, B_Y)$ is a global 
embedded simple normal crossing 
pair as in {\em{Definition \ref{gsnc}}}. 
Let $\sigma:N\to M$ be a proper birational morphism from a smooth variety 
$N$. 
We define $K_N+D_N=\sigma^*(K_M+D+Y)$ and 
assume that $\Supp \sigma^{-1}_*(D+Y)\cup \Exc 
(\sigma)$ is simple normal 
crossing on $N$. 
Let $Z$ be the union of the irreducible components of $D^{=1}_N$ that are 
mapped into $Y$ by $\sigma$. 
Then $f\circ \sigma:(Z, B_Z)\to 
X$ is a quasi-log resolution of $[X, \omega]$, where $K_Z+B_Z=(K_N+D_N)|_{Z}$. 
\end{prop}

The proof of Proposition \ref{taisetsu3} is obvious by Proposition \ref{taisetsu}. 

\begin{rem}
In Proposition \ref{taisetsu3}, $\sigma:(Z, B_Z)
\to (Y, B_Y)$ is not necessarily a composition 
of {\em{embedded log transformations}} 
and blow-ups whose centers contain 
no strata of the pair $(Y, B^{=1}_Y)$ 
(see \cite[Section 2]{ambro}). 
Compare Proposition \ref{taisetsu3} 
with \cite[Remark 4.2.(iv)]{ambro}. 
\end{rem}

The final proposition in this subsection will 
play very important roles in the following sections. 

\begin{prop}\label{tai4}
Let $f:(Y, B_Y)\to X$ be a quasi-log resolution of a quasi-log pair 
$[X, \omega]$, where $(Y, B_Y)$ is a global embedded 
simple normal crossing pair as in {\em{Definition \ref{gsnc}}}. 
Let $E$ be a Cartier divisor on $X$ 
such that $\Supp E$ contains no qlc centers of $[X, \omega]$. 
By blowing up $M$, the ambient space of 
$Y$, 
inside $\Supp f^*E$, 
we can assume that $(Y, B_Y+f^*E)$ is a global 
embedded simple normal crossing 
pair. 
\end{prop}
\begin{proof}
First, we take a blow-up of $M$ along $f^*E$ and apply Hironaka's 
resolution theorem to $M$. 
Then we can assume that there exists a Cartier divisor 
$F$ on $M$ such that 
$\Supp (F\cap Y)=\Supp f^*E$. 
Next, we apply Szab\'o's resolution lemma 
to $\Supp (D+Y+F)$ on $M$. 
Thus, we obtain the desired 
properties by Proposition \ref{taisetsu}. 
\end{proof}

\subsection{Ambro's original formulation}\label{3-2-6} 

Let us recall Ambro's original definition of 
quasi-log varieties. 

\begin{defn}[Quasi-log varieties]\index{quasi-log variety}
\label{qlog-def-ambro}
A {\em{quasi-log variety}} is a scheme $X$ endowed with an 
$\mathbb R$-Cartier $\mathbb R$-divisor 
$\omega$, a proper closed subscheme 
$X_{-\infty}\subset X$, and a finite collection $\{C\}$ of reduced 
and irreducible subvarieties of $X$ such that there is a 
proper morphism $f:(Y, B_Y)\to X$ from an {\em{embedded 
normal crossing pair}} 
satisfying the following properties: 
\begin{itemize}
\item[(1)] $f^*\omega\sim_{\mathbb R}K_Y+B_Y$. 
\item[(2)] The natural map 
$\mathcal O_X
\to f_*\mathcal O_Y(\ulcorner -(B_Y^{<1})\urcorner)$ 
induces an isomorphism 
$$
\mathcal I_{X_{-\infty}}\to f_*\mathcal O_Y(\ulcorner 
-(B_Y^{<1})\urcorner-\llcorner B_Y^{>1}\lrcorner),  
$$ 
where $\mathcal I_{X_{-\infty}}$ is the defining ideal sheaf of 
$X_{-\infty}$. 
\item[(3)] The collection of subvarieties $\{C\}$ coincides with the image 
of $(Y, B_Y)$-strata that are not included in $X_{-\infty}$. 
\end{itemize}
\end{defn}

For the definition of normal crossing pairs, see 
Definition \ref{625}. 

\begin{rem}\label{def-ni-tuite} 
We can always construct an embedded {\em{simple}} 
normal crossing pair $(Y', B_{Y'})$ and a proper morphism $f':(Y', 
B_{Y'})\to X$ with the above conditions (1), (2), and (3) by blowing up 
$M$ suitably, where $M$ is the ambient space 
of $Y$ (see \cite[p.218, 
embedded log transformations, and Remark 4.2.(iv)]{ambro}). 
We leave the details for the reader's 
exercies (see also Lemmas \ref{63}, 
\ref{64}, and \ref{65}, and the proof of 
Proposition \ref{taisetsu}).  
Therefore, we can assume that $(Y, B_Y)$ is a simple normal crossing 
pair in Definition \ref{qlog-def-ambro}. 
We note that the proofs of the vanishing and 
injectivity theorems on normal crossing 
pairs are much harder than on {\em{simple}} normal crossing 
pairs (see Chapter \ref{chap2}). 
Therefore, there are no advantages 
to adopt {\em{normal crossing pairs}} in the definition of 
quasi-log varieties. 
\end{rem}

The next proposition is the main result in this section. 
Proposition \ref{taisetsu} becomes very powerful if 
it is combined with Proposition \ref{taisetsu2}. 
See Proposition \ref{taisetsu3}. 

\begin{prop}\label{taisetsu2} 
We assume that 
$(Y, B_Y)$ is an embedded simple normal crossing 
pair in {\em{Definition \ref{qlog-def-ambro}}}. 
Let $M$ be the ambient space 
of $Y$. We can assume that 
there exists an $\mathbb R$-divisor 
$D$ on $M$ such that 
$\Supp (D+Y)$ is simple normal crossing and $B_Y=D|_Y$. 
\end{prop}
\begin{proof}
We can construct a sequence of blow-ups 
$M_k\to M_{k-1}\to\cdots\to M_0=M$ with 
the following properties. 
\begin{itemize}
\item[(i)] $\sigma_{i+1}:M_{i+1}\to M_i$ is the 
blow-up along a smooth irreducible component of $\Supp B_{Y_i}$ for 
any $i\geq 0$, 
\item[(ii)] we put $Y_0=Y$, $B_{Y_0}=B_Y$, and $Y_{i+1}$ is the 
strict transform of $Y_i$ for any $i\geq 0$, 
\item[(iii)] we define $K_{Y_{i+1}}+B_{Y_{i+1}}=\sigma^*_{i+1}
(K_{Y_i}+B_{Y_i})$ for any $i\geq 0$, 
\item[(iv)] there exists an $\mathbb R$-divisor $D$ on $M_k$ such that 
$\Supp (Y_k+D)$ is simple normal crossing 
on $M_k$ and that $D|_{Y_k}=B_{Y_k}$, and 
\item[(v)] $\sigma_*\mathcal O_{Y_k}(\ulcorner 
-(B^{<1}_{Y_k})\urcorner-\llcorner B^{>1}_{Y_k}\lrcorner)\simeq 
\mathcal O_Y(\ulcorner-(B^{<1}_Y)\urcorner-\llcorner B^{>1}_Y\lrcorner)$, 
where $\sigma: M_k\to M_{k-1}\to\cdots\to M_0=M$.  
\end{itemize}
We note that we can directly check $\sigma_{i+1*}
\mathcal O_{Y_{i+1}}(\ulcorner 
-(B^{<1}_{Y_{i+1}})\urcorner -\llcorner 
B^{>1}_{Y_{i+1}}\lrcorner)\simeq 
\mathcal O_{Y_i}(\ulcorner -(B^{<1}_{Y_i})\urcorner 
-\llcorner B^{>1}_{Y_i}\lrcorner)$ for any $i\geq 0$ 
by computations similar to the proof of Proposition \ref{taisetsu}. 
We replace $M$ and $(Y, B_Y)$ with $M_k$ and $(Y_k, B_{Y_k})$. 
\end{proof}

\begin{rem}\label{uun}
In the proof of Proposition \ref{taisetsu2}, 
$M_k$ and $(Y_k, B_{Y_k})$ depend on 
the order of blow-ups. If we 
change the order of blow-ups, we have another 
tower of blow-ups $\sigma':M'_k\to 
M'_{k-1}\to \cdots \to M'_0=M$, 
$D'$, $Y'_k$ on $M'_k$, and 
$D'|_{Y'_k}=B_{Y'_k}$ with 
the desired properties. 
The relationship between $M_k, Y_k, D$ and 
$M'_k, Y'_k, D'$ is not clear.  
\end{rem}

\begin{rem}[Multicrossing vs simple normal crossing]\label{multi}
In \cite[Section 2]{ambro}, Ambro discussed {\em{multicrossing 
singularities}} and {\em{multicrossing pairs}}.\index{multicrossing singularity}\index{multicrossing pair}  
However, we think that {\em{simple normal crossing 
varieties}} and {\em{simple normal crossing divisors}} on 
them 
are sufficient for the later arguments in \cite{ambro}. 
Therefore, we did not introduce the notion of 
{\em{multicrossing singularities}} and their simplicial 
resolutions. 
For the theory of quasi-log varieties, 
we may not even need the notion of {\em{simple normal crossing 
pairs}}. The notion of {\em{global 
embedded simple normal crossing pairs}} seems 
to be sufficient. 
\end{rem}

\subsection{A remark on the ambient space}

In this subsection, we make a remark on 
the ambient space $M$ of a quasi-log resolution 
$f:(Y, B_Y)\to X$ in Definition \ref{quasi-log}. 

The following lemma is essentially the 
same as Proposition \ref{taisetsu2}. 
We repeat it here since it is important. 
The proof is obvious. 

\begin{lem}\label{se-bl}
Let $(Y, B_Y)$ be a simple normal crossing 
pair. 
Let $V$ be a smooth 
variety such that 
$Y\subset V$. 
Then we can construct a sequence of 
blow-ups 
$$V_k\to V_{k-1}\to\cdots\to V_0=V$$ with 
the following properties. 
\begin{itemize}
\item[$(1)$] $\sigma_{i+1}:V_{i+1}\to V_i$ is the 
blow-up along a smooth irreducible 
component of $\Supp B_{Y_i}$ for 
any $i\geq 0$, 
\item[$(2)$] we put $Y_0=Y$, 
$B_{Y_0}=B_Y$, and $Y_{i+1}$ is the 
strict transform 
of $Y_i$ for any $i\geq 0$, 
\item[$(3)$] 
we define $K_{Y_{i+1}}+B_{Y_{i+1}}=\sigma^*_{i+1}
(K_{Y_i}+B_{Y_i})$ for 
any $i\geq 0$, 
\item[$(4)$] there exists 
an $\mathbb R$-divisor $D$ on $V_k$ such 
that $D|_{Y_k}=B_{Y_k}$, and 
\item[$(5)$] $\sigma_*\mathcal O_{Y_k}(\ulcorner 
-(B^{<1}_{Y_k})\urcorner-\llcorner B^{>1}_{Y_k}\lrcorner)\simeq 
\mathcal O_Y(\ulcorner-(B^{<1}_Y)\urcorner-\llcorner B^{>1}_Y\lrcorner)$, 
where $\sigma: V_k\to V_{k-1}\to\cdots\to V_0=V$.  
\end{itemize}
\end{lem}

When a simple normal crossing variety $Y$ is 
quasi-projective, we can make 
a {\em{singular}} ambient space whose singular locus 
dose not contain any strata of $Y$.  

\begin{lem}\label{gene-int}
Let $Y$ be a simple normal crossing 
variety. Let $V$ be a smooth 
quasi-projective variety such that 
$Y\subset V$. Let $\{P_i\}$ be 
any finite set of closed points of 
$Y$. 
Then we can find a quasi-projecive 
variety $W$ such that 
$Y\subset W\subset V$, $\dim W=\dim Y+1$, and 
$W$ is smooth 
at $P_i$ for any $i$. 
\end{lem}
\begin{proof}
Let $\mathcal I_Y$ be the defining ideal 
sheaf of $Y$ on $V$. 
Let $H$ be an ample Cartier divisor. 
Then $\mathcal I_Y\otimes \mathcal O_Y(dH)$ is 
generated by global sections for $d\gg 0$. We can 
further assume that 
$$
H^0(V, \mathcal I_Y\otimes 
\mathcal O_V(dH))\to \mathcal I_Y\otimes 
\mathcal O_V(dH)\otimes \mathcal O_V/m^2_{P_i}
$$ 
is surjective for any $i$, where 
$m_{P_i}$ is the maximal ideal corresponding to 
$P_i$. 
By taking a complete intersection of 
$(\dim V-\dim Y-1)$ general members in 
$H^0(V, \mathcal I_Y\otimes \mathcal O_V(dH))$, we 
obtain a desired variety $W$. 
\end{proof}

Of course, we can not always make 
$W$ smooth in Lemma \ref{gene-int}. 

\begin{ex}
Let $V\subset \mathbb P^5$ be the Segre 
embedding of $\mathbb P^1\times \mathbb P^2$. 
In this case, there are no smooth 
hypersurfaces of $\mathbb P^5$ containing $V$. 
We can check it as follows. 
If there exists a smooth hypersurface $S$ such that 
$V\subset S\subset \mathbb P^5$, then 
$\rho (V)=\rho (S)=\rho (\mathbb P^5)=1$ by 
the Lefschetz hyperplane theorem. 
It is a contradiction. 
\end{ex}

By the above lemmas, we can prove the final lemma. 

\begin{lem}\label{le-amb}
Let $(Y, B_Y)$ be a simple normal 
crossing 
pair such that 
$Y$ is quasi-projective. Then 
there exist a global embedded simple normal 
crossing 
pair $(Z, B_Z)$ and 
a morphism 
$\sigma: Z\to Y$ such that 
$$\sigma _*\mathcal O_Z(\ulcorner -(B^{<1}_Z)\urcorner
-\llcorner B^{>1}_Z\lrcorner)
\simeq \mathcal O_Y(\ulcorner -(B^{<1}_Y)\urcorner
-\llcorner B^{>1}_Y\lrcorner). $$ 
\end{lem}
\begin{proof}
Let $V$ be a smooth quasi-projective 
variety such that 
$Y\subset V$. 
By Lemma \ref{se-bl}, we can assume 
that there exists an $\mathbb R$-divisor 
$D$ on $V$ such that 
$D|_Y=B_Y$. 
Then we apply Lemma \ref{gene-int}. 
We can find a quasi-projectie 
variety $W$ such that 
$Y\subset W\subset V$, $\dim W=\dim Y+1$, and 
$W$ is smooth at the generic point of any stratum 
of $(Y, B_Y)$. 
Of course, we can make $W\not\subset \Supp D$ 
(see the proof of Lemma \ref{gene-int}). 
We apply 
Hironaka's resolution to $W$ and 
use Szab\'o's resolution lemma. 
Then we obtain a desired global embedded 
simple normal crossing 
pair $(Z, B_Z)$. 
\end{proof}

Therefore, we obtain the following statement. 

\begin{thm}\label{goodbye}
In {\em{Definition \ref{quasi-log}}}, 
it is sufficient to assume that $(Y, B_Y)$ is 
a simple normal crossing pair if $Y$ is 
quasi-projective. 
\end{thm}

We note that 
we have a natural quasi-projective 
ambient space $M$ 
in almost all the applications of the theory of 
quasi-log varieties to log canonical pairs. 
Therefore, Definition \ref{quasi-log} 
seems to be reasonable. 

We close this subsection with a remark on Chow's lemma. 
Proposition \ref{bottle} is a 
bottleneck to construct a good ambient space 
of a simple normal crossing pair. 

\begin{prop}\label{bottle}
There exists a complete simple normal crossing 
variety $Y$ with the following property. 
If $f:Z\to Y$ is a proper surjective morphism from 
a simple normal crossing variety $Z$ such 
that $f$ is an isomorphism 
at the generic point of any stratum of $Z$, then 
$Z$ is non-projective. 
\end{prop}

\begin{proof}
We take a smooth complete non-projective toric variety 
$X$ (cf.~Example \ref{fp-example}). 
We put $V=X\times \mathbb P^1$. Then 
$V$ is a toric variety. 
We consider $Y=V\setminus T$, where 
$T$ is the big torus of $V$. 
We will see that $Y$ has the desired property. 
By the above construction, there is 
an irreducible component $Y'$ of $Y$ that is 
isomorphic to $X$. 
Let $Z'$ be the irreducible component of $Z$ mapped 
onto $Y'$ by $f$. 
So, it is sufficient to see that 
$Z'$ is not projective. 
On $Y'\simeq X$, there is an torus invariant effective 
one cycle $C$ such that 
$C$ is numerically trivial. 
By the construction and the assumption, $g=f|_{Z'}: 
Z'\to Y'\simeq X$ is birational and an isomorphism 
over the generic point of any torus invariant curve 
on $Y'\simeq X$. 
We note that any torus invariant curve on $Y'\simeq 
X$ is a stratum of $Y$. 
We assume that $Z'$ is projective, 
then there is a very ample 
effective divisor $A$ on $Z'$ such that 
$A$ does not contain any irreducible components of 
the inverse image of $C$. 
Then $B=f_*A$ is an effective Cartier divisor 
on $Y'\simeq X$ such that 
$\Supp B$ contains no irreducible 
components of $C$. 
It is a contradiction because $\Supp B\cap C\ne \emptyset$ 
and $C$ is numerically trivial. 
\end{proof}

The phenomenon described in Proposition \ref{bottle} 
is annoying when we treat non-normal 
varieties. 

\section{Fundamental Theorems}\label{33-sec}

In this section, we will prove the fundamental theorems 
for quasi-log pairs. 
First, we prove the base point free theorem for quasi-log 
pairs in the subsection \ref{331-ssec}. 
The reader can find that the notion of quasi-log 
pairs is very useful for inductive arguments. 
Next, we give a proof to the rationality theorem 
for quasi-log pairs in the subsection \ref{332-ssec}. 
Our proof is essentially the same as the proof for 
klt pairs. 
In the subsection \ref{333-ssec}, 
we prove the cone theorem for quasi-log varieties. 
The cone and contraction theorems are the main results 
in this section. 

\subsection{Base Point Free Theorem}\label{331-ssec}
The next theorem is the main theorem of this subsection. 
It is \cite[Theorem 5.1]{ambro}. 
This formulation is useful for 
the inductive treatment of log canonical pairs. 

\begin{thm}[Base Point Free Theorem]\index{base 
point free theorem}\label{bpf-th} 
Let $[X, \omega]$ be a quasi-log pair and let $\pi:X\to S$ be 
a projective morphism. 
Let $L$ be a $\pi$-nef Cartier divisor on $X$. 
Assume that 
\begin{itemize}
\item[{\em{(i)}}] $qL-\omega$ is $\pi$-ample for some real number 
$q>0$, and 
\item[{\em{(ii)}}] $\mathcal O_{X_{-\infty}}(mL)$ is 
$\pi|_{X_{-\infty}}$-generated for $m\gg 0$. 
\end{itemize} 
Then $\mathcal O_X(mL)$ is $\pi$-generated for $m\gg0$, 
that is, there exists a positive number 
$m_0$ such that 
$\mathcal O_X(mL)$ is $\pi$-generated for 
any $m\geq m_0$.  
\end{thm}

\begin{proof} 
Without loss of generality, we can assume that 
$S$ is affine. 
\setcounter{cla}{0}
\begin{cla}\label{bpf-c1}  
$\mathcal O_X(mL)$ is $\pi$-generated 
around $\Nqklt(X, \omega)$ for $m\gg 0$. 
\end{cla}
We put $X'=\Nqklt(X, \omega)$. 
Then $[X', \omega']$, 
where $\omega'=\omega|_{X'}$, is a quasi-log pair 
by adjunction (see Theorem \ref{adj-th} (i)). 
If $X'=X_{-\infty}$, then 
$\mathcal O_{X'}(mL)$ is $\pi$-generated for $m \gg 0$ 
by the assumption (ii). 
If $X'\ne X_{-\infty}$, then 
$\mathcal O_{X'}(mL)$ 
is $\pi$-generated for $m\gg 0$ 
by the induction on the dimension of $X\setminus X_{-\infty}$.  
By the following commutative diagram: 
$$
\xymatrix{
\pi^*\pi_*\mathcal O_X(mL)
\ar[d]\ar[r]^{\alpha}& 
\pi^*\pi_*\mathcal O_{X'}(mL)
\ar[d]\ar[r]&0\\
\mathcal O_X(mL)
\ar[r]&
\mathcal O_{X'}(mL)
\ar[r]&0, 
}
$$
we know that $\mathcal O_X(mL)$ is $\pi$-generated 
around $X'$ for $m\gg 0$. 

\begin{cla}\label{bpf-c2} 
$\mathcal O_X(mL)$ is $\pi$-generated 
on a non-empty Zariski open set for $m\gg 0$. 
\end{cla} 

By Claim \ref{bpf-c1}, 
we can assume that $\Nqklt(X, \omega)$ is empty. 
We will see that we can also assume that $X$ is irreducible. 
Let $X'$ be an irreducible component of $X$. 
Then $X'$ with $\omega'=\omega|_{X'}$ 
has a natural quasi-log structure 
induced by $[X, \omega]$ by adjunction (see 
Corollary \ref{irre-cor}). 
By the vanishing theorem (see 
Theorem \ref{adj-th} (ii)), we have 
$R^1\pi_*(\mathcal I_{X'}\otimes 
\mathcal O_X(mL))=0$ for any $m\geq q$. 
We consider the following 
commutative diagram. 
$$
\xymatrix{
\pi^*\pi_*\mathcal O_X(mL)
\ar[d]\ar[r]^{\alpha}& 
\pi^*\pi_* \mathcal O_{X'}(mL)
\ar[d]\ar[r]&0\\
\mathcal O_X(mL)
\ar[r]&
\mathcal O_{X'}(mL)
\ar[r]&0
}
$$
Since $\alpha$ is surjective for $m\geq q$, we can assume that 
$X$ is irreducible when we prove this claim. 

If $L$ is $\pi$-numerically trivial, then 
$\pi_*\mathcal O_X(L)$ is not zero. 
It is because 
$h^0(X_\eta, \mathcal O_{X_\eta}
(L))=\chi (X_\eta, \mathcal O_{X_\eta}(L))
=\chi (X_\eta, \mathcal O_{X_\eta})=h^0(X_\eta, 
\mathcal O_{X_\eta})>0$  
by Theorem \ref{adj-th} (ii) and by 
\cite[Chapter II \S2 Theorem 1]{kleiman}, 
where $X_\eta$ is the generic fiber of $\pi:X\to S$. 
Let $D$ be a general member of 
$|L|$. Let $f:(Y, B_Y)\to X$ be a quasi-log resolution. 
By blowing up $M$, 
we can assume that $(Y, B_Y+f^*D)$ is a global embedded 
simple normal crossing pair by Proposition \ref{tai4}. 
We note that any stratum of $(Y, B_Y)$ is mapped onto $X$ by the 
assumption. 
We can take a positive real number $c\leq 1$ such 
that $B_Y+cf^*D$ is a subboundary 
and some stratum of $(Y, B_Y+cf^*D)$ does not dominate 
$X$. Note that 
$f_*\mathcal O_Y(\ulcorner -(B_Y^{<1})\urcorner)\simeq 
\mathcal O_X$. Then the pair 
$[X, \omega+cD]$ is qlc 
and $f:(Y, B_Y+cf^*D)\to X$ is a quasi-log resolution. 
We note that $qL-(\omega+cD)$ is $\pi$-ample. 
By Claim \ref{bpf-c1}, $\mathcal O_X(mL)$ is 
$\pi$-generated around 
$\Nqklt(X, \omega+cD)$ for $m\gg 0$. 
So, we can assume that $L$ is not $\pi$-numerically trivial. 

Let $x\in X$ be a general smooth point. Then 
we can take an $\mathbb R$-divisor $D$ such that 
$\mult _x D> \dim X$ and that 
$D\sim _{\mathbb R}(q+r)L-\omega$ for some $r>0$ 
(see \cite[3.5 Step 2]{km}). 
By blowing up $M$, 
we can assume that $(Y, B_Y+f^*D)$ is a global embedded 
simple normal crossing pair by Proposition \ref{tai4}. 
By the construction of $D$, 
we can find a positive real number $c<1$ 
such that $B_Y+cf^*D$ is a subboundary 
and some stratum of $(Y, B_Y+cf^*D)$ does not dominate $X$. 
Note that $f_*\mathcal O_Y(\ulcorner 
-(B^{<1}_Y)\urcorner)\simeq \mathcal O_X$. 
Then the pair $[X, \omega+cD]$ is qlc and $f:(Y, B_Y+cf^*D)\to 
X$ is a quasi-log resolution. 
We note that $q'L-(\omega+cD)$ is $\pi$-ample 
by $c<1$, where 
$q'=q+cr$. 
By the construction, $\Nqklt(X, \omega+cD)$ is non-empty. 
Therefore, by applying Claim \ref{bpf-c1} to 
$[X, \omega+cD]$, $\mathcal O_X(mL)$ is 
$\pi$-generated 
around $\Nqklt(X, \omega+cD)$ for $m\gg 0$. 
So, we finish the proof of Claim \ref{bpf-c2}. 

Let $p$ be a prime number and let $l$ be a large integer. 
Then $\pi_*\mathcal O_X(p^lL)\ne 0$ by Claim \ref{bpf-c2} and 
$\mathcal O_X(p^l L)$ is 
$\pi$-generated around $\Nqklt(X, \omega)$ by 
Claim \ref{bpf-c1}. 

\begin{cla}\label{bpf-c3} 
If the relative base locus $\Bs_\pi|p^lL|$ $($with 
reduced scheme 
structure$)$ is not empty, 
then $\Bs_\pi|p^lL|$ is not contained in $\Bs_\pi|p^{l'}L|$ for 
$l'\gg l$. 
\end{cla}

Let $f:(Y, B_Y)\to X$ be a quasi-log resolution. 
We take a general member $D\in |p^lL|$. We note 
that $S$ is affine and 
$|p^lL|$ is free around $\Nqklt(X, \omega)$. 
Thus, 
$f^*D$ intersects any strata of $(Y, \Supp B_Y)$ transversally 
over $X\setminus \Bs_\pi|p^lL|$ by Bertini and 
$f^*D$ contains no strata of $(Y, B_Y)$. 
By taking blow-ups of $M$ suitably, 
we can assume that $(Y, B_Y+f^*D)$ is a global 
embedded simple normal crossing pair. 
See the proofs of Propositions \ref{tai4} and \ref{taisetsu}. 
We take the maximal positive real number $c$ such 
that $B_Y+cf^*D$ is a subboundary over $X\setminus 
X_{-\infty}$.  
We note that $c\leq 1$. 
Here, we used $\mathcal O_X\simeq 
f_*\mathcal O_Y(\ulcorner -(B^{<1}_Y)\urcorner)$ over 
$X\setminus X_{-\infty}$. 
Then $f:(Y, B_Y+cf^*D)\to X$ is a quasi-log resolution 
of $[X, \omega'=\omega+cD]$. 
Note that $[X, \omega']$ has a qlc center $C$ that 
intersects $\Bs_\pi|p^lL|$ by the construction. 
By the induction, $\mathcal O_C(mL)$ is $\pi$-generated 
for $m\gg 0$ since $(q+cp^l)L-(\omega+cD)$ is 
$\pi$-ample. 
We can lift the sections of $\mathcal O_C(mL)$ to $X$ for 
$m\geq q+cp^l$ by Theorem \ref{adj-th} (ii). Then 
we obtain that $\mathcal O_X(mL)$ is $\pi$-generated 
around $C$ for $m\gg 0$. 
Therefore, $\Bs_\pi|p^{l'}L|$ is strictly smaller 
than $\Bs_\pi|p^{l}L|$ for $l'\gg l$. 

\begin{cla}\label{bpf-c4} 
$\mathcal O_X(mL)$ is $\pi$-generated 
for $m\gg 0$. 
\end{cla}
By Claim \ref{bpf-c3} and 
the noetherian induction, 
$\mathcal O_X(p^{l}L)$ and $\mathcal O_X({p'}^{l'}L)$ 
are $\pi$-generated 
for large $l$ and $l'$, 
where $p$ and $p'$ are prime numbers and they 
are relatively prime. 
So, there exists a positive number 
$m_0$ such that $\mathcal O_X(mL)$ is 
$\pi$-generated 
for any $m\geq m_0$. 
\end{proof}

The next corollary is a special case of Theorem \ref{bpf-th}. 

\begin{cor}[Base Point Free Theorem for lc pairs]
\label{bpf-co} 
Let $(X, B)$ be an lc pair and let $\pi:X\to S$ be a projective 
morphism. 
Let $L$ be a $\pi$-nef Cartier divisor on $X$. 
Assume that $qL-(K_X+B)$ is 
$\pi$-ample for some positive real number $q$. 
Then $\mathcal O_X(mL)$ is $\pi$-generated 
for $m\gg 0$. 
\end{cor}

\subsection{Rationality Theorem}\label{332-ssec} 
In this subsection, we prove the following rationality theorem (cf.~\cite[Theorem 5.9]{ambro}). 

\begin{thm}[Rationality Theorem]\index{rationality theorem}
\label{rat-th} 
Assume that $[X, \omega]$ is a quasi-log pair 
such that $\omega$ is $\mathbb Q$-Cartier. We note 
that this means $\omega$ is $\mathbb R$-linearly equivalent 
to a $\mathbb Q$-Cartier divisor on $X$ 
$($see {\em{Remark \ref{cano}}}$)$. 
Let $\pi:X\to S$ be a projective morphism and let $H$ be a 
$\pi$-ample Cartier divisor on $X$. Assume that 
$r$ is a positive number such that 
\begin{itemize}
\item[$(1)$] $H+r\omega$ is $\pi$-nef 
but not $\pi$-ample, and 
\item[$(2)$] $(H+r\omega)|_{X_{-\infty}}$ is $\pi|_{X_{-\infty}}$-ample. 
\end{itemize}
Then $r$ is a rational number, and 
in reduced form, $r$ has denominator at most $a(\dim X+1)$, where $a\omega$ is $\mathbb R$-linearly equivalent to 
a Cartier divisor on $X$. 
\end{thm} 

Before we go to the proof, we recall the following lemmas. 
\begin{lem}[{cf.~\cite[Lemma 3.19]{km}}]\label{pol1-le}  
Let $P(x, y)$ be a non-trivial 
polynomial of degree $\leq n$ and 
assume that 
$P$ vanishes for all sufficiently large integral solutions of 
$0<ay-rx<\varepsilon$ for some fixed positive integer $a$ and 
positive $\varepsilon$ for some $r\in \mathbb R$. 
Then $r$ is rational, and 
in reduced form, $r$ has denominator $\leq 
a(n+1)/\varepsilon$. 
\end{lem}

For the proof, see \cite[Lemma 3.19]{km}. 

\begin{lem}[{cf.~\cite[3.4 Step 2]{km}}]\label{pol2-le}  
Let $[Y, \omega]$ be 
a projective qlc pair and let $\{D_i\}$ 
be a finite collection of Cartier divisors. 
Consider the Hilbert polynomial 
$$
P(u_1, \cdots, u_k)=\chi (Y, \mathcal O_Y(\sum_{i=1}^k u_i D_i)). 
$$
Suppose that 
for some values of the $u_i$, 
$\sum _{i=1}^k u_i D_i$ is nef and 
$\sum _{i=1}^k u_i D_i-\omega$ is ample. 
Then $P(u_1, \cdots, u_k)$ is not identically 
zero by the base point free theorem for qlc 
pairs $($see {\em{Theorem \ref{bpf-th}}}$)$ 
and 
the vanishing theorem 
$($see {\em{Theorem \ref{adj-th} (ii)}}$)$, 
and its degree 
is $\leq \dim Y$. 
\end{lem}

Note that the arguments in \cite[3.4 Step 2]{km} work 
for our setting. 

\begin{proof}[Proof of {\em{Theorem \ref{rat-th}}}]
By using $mH$ with various large $m$ in place of $H$, we 
can assume that $H$ is very ample over $S$ 
(cf.~\cite[3.4 Step 1]{km}). 
For each $(p, q)\in \mathbb Z^2$, 
let $L(p, q)$ denote the relative base 
locus of the linear system $M(p, q)$ on $X$ (with 
reduced scheme structure), 
that is, 
$$
L(p, q)=\Supp (\Coker (\pi^*\pi_*
\mathcal O_X(M(p, q))\to \mathcal O_X(M(p, q)))), 
$$
where $M(p, q)=pH+qD$, where $D$ is a Cartier divisor 
such that $D\sim _{\mathbb R}a\omega$. By 
the definition, $L(p, q)=X$ if and only if $\pi_*\mathcal O_X(M(p, q))= 0$. 

\setcounter{cla}{0}
\begin{cla}[{cf.~\cite[Claim 3.20]{km}}]\label{rat-c1}  
Let $\varepsilon$ be a positive number. 
For $(p, q)$ sufficiently 
large and $0<aq-rp<\varepsilon$, $L(p, q)$ is the same 
subset of $X$. 
We call this subset $L_0$. 
We let $I\subset \mathbb Z^2$ be the set of $(p, q)$ for 
which $0<aq-rp<1$ and $L(p, q)=L_0$. 
We note that $I$ contains all sufficiently 
large $(p, q)$ with $0<aq-rp<1$. 
\end{cla}

For the proof, see \cite[Claim 3.20]{km}. 
See also the proof of Claim \ref{rat-c2} below. 

\begin{cla}\label{rat-c2} 
We have $L_0\cap X_{-\infty}=\emptyset$. 
\end{cla}

\begin{proof}[Proof of {\em{Claim \ref{rat-c2}}}] 
We take $(\alpha, \beta)\in \mathbb Q^2$ such that 
$\alpha>0$, $\beta>0$, 
and $\beta a/\alpha>r$ is sufficiently close to 
$r$. 
Then $(\alpha H+\beta a \omega)|_{X_{-\infty}}$ is 
$\pi|_{X_{-\infty}}$-ample because 
$(H+r\omega)|_{X_{-\infty}}$ is 
$\pi|_{X_{-\infty}}$-ample. 
If $0<aq-rp<1$ and $(p, q)\in \mathbb Z^2$ is 
sufficiently large, then 
$M(p, q)=mM(\alpha, \beta)+(M(p, q)-mM(\alpha, \beta))$ such 
that $M(p, q)-mM(\alpha, \beta)$ is $\pi$-very ample and 
that $m(\alpha H+\beta D)|_{X_{-\infty}}$ is 
also $\pi|_{X_{-\infty}}$-very ample. 
Therefore, $\mathcal O_{X_{-\infty}}(M(p, q))$ is $\pi$-very 
ample. 
Since $\pi_*\mathcal O_X(M(p, q))\to 
\pi_*\mathcal O_{X_{-\infty}}(M(p, q))
$ is surjective 
by the vanishing theorem 
(see Theorem \ref{adj-th} (ii)), 
$L(p, q)\cap X_{-\infty}=\emptyset$. 
We note that $M(p, q)-\omega$ is $\pi$-ample because 
$(p, q)$ is sufficiently large and $aq-rp<1$.  
By Claim \ref{rat-c1}, we have 
$L_0\cap X_{-\infty}=\emptyset$. 
\end{proof}

\begin{cla}\label{rat-c3} 
We assume that $r$ is not rational or that 
$r$ is rational and has denominator $>a(n+1)$ in reduced form, 
where $n=\dim X$. 
Then, for $(p, q)$ sufficiently large and $0<aq-rp<1$, 
$\mathcal O_X(M(p, q))$ is $\pi$-generated 
at the generic point of any qlc center 
of $[X, \omega]$. 
\end{cla}

\begin{proof}[Proof of {\em{Claim \ref{rat-c3}}}]
We note that $M(p, q)-\omega\sim _{\mathbb R}pH+(qa-1)\omega$. If 
$aq-rp<1$ and $(p, q)$ is sufficiently large, then $M(p, q)-\omega$ is 
$\pi$-ample. 
Let $C$ be a qlc center of $[X, \omega]$. 
We note that we can assume $C\cap 
X_{-\infty}=\emptyset$ by Claim \ref{rat-c2}. 
Then $P_{C_\eta}(p, q)=\chi (C_{\eta}, \mathcal 
O_{C_\eta}(M(p, q)))$ is a non-zero polynomial of degree 
at most $\dim C_{\eta}\leq \dim X$ by Lemma \ref{pol2-le} 
(see also Lemma \ref{ob-lem}). 
Note that $C_\eta$ is the generic fiber of $C\to \pi(C)$. 
By Lemma \ref{pol1-le}, there exists $(p, q)$ such that 
$P_{C_\eta}(p, q)\ne 0$ and 
that $(p, q)$ sufficiently large and $0<aq-rp<1$. 
By the $\pi$-ampleness of $M(p, q)-\omega$, $P_{C_\eta}
(p, q)=
\chi (C_\eta, \mathcal O_{C_\eta}(M(p, q)))
=h^0(C_\eta, \mathcal O_{C_\eta}(M(p, q)))$ and 
$
\pi_*\mathcal O_X(M(p, q))\to \pi_*\mathcal O_C(M(p, q))
$ 
is surjective. 
We note that $C'=C\cup X_{-\infty}$ has a 
natural quasi-log structure induced 
by $[X, \omega]$ and that 
$C\cap X_{-\infty}=\emptyset$. 
Therefore, $\mathcal O_X(M(p, q))$ 
is $\pi$-generated 
at the generic point of $C$. 
By combining this 
with Claim \ref{rat-c1}, $\mathcal O_X(M(p, q))$ is 
$\pi$-generated at the generic point of any 
qlc center of $[X, \omega]$ if 
$(p, q)$ is sufficiently large with 
$0<aq-rp<1$. 
So, we obtain Claim \ref{rat-c2}. 
\end{proof}

Note 
that $\mathcal O_X(M(p, q))$ is not $\pi$-generated 
for $(p, q)\in I$ 
because $M(p, q)$ is not $\pi$-nef. 
Therefore, $L_0\ne \emptyset$. 
We shrink $S$ to an affine open subset intersecting $\pi(L_0)$. 
Let $D_1, \cdots, D_{n+1}$ 
be general members of $\pi_*\mathcal O_X(M(p_0, q_0))
=H^0(X, \mathcal O_X(M(p_0, q_0)))$ with 
$(p_0, q_0)\in I$. 
Around the generic point of any irreducible component of 
$L_0$, 
by taking general hyperplane cuts and applying 
Lemma \ref{boun} below, we can check that $\omega+\sum _{i=1}^{n+1}D_i$ 
is not qlc at the generic point of any irreducible component of $L_0$. 
Thus, $\omega+\sum _{i=1}^{n+1}D_i$ 
is not qlc at the generic point of any irreducible 
component of $L_0$ and is qlc outside $L_0\cup X_{-\infty}$. 
Let $0<c<1$ be the maximal real number such that 
$\omega+c\sum _{i=1}^{n+1}D_i$ is qlc outside $X_{-\infty}$. 
Note that $c>0$ by Claim \ref{rat-c3}. 
Thus, the quasi-log 
pair $[X, \omega+c\sum _{i=1}^{n+1}D_i]$ has some 
qlc centers contained in $L_0$. Let $C$ be a qlc center 
contained in $L_0$. 
We note that $C\cap X_{-\infty}=\emptyset$. 
We consider 
$$\omega'=\omega+c\sum _{i=1}^{n+1} D_i
\sim _{\mathbb R}c(n+1)p_0H+(1+c(n+1)q_0a)\omega. $$ 
Thus we have 
$$pH+qa\omega-\omega'\sim _{\mathbb R}
(p-c(n+1)p_0)H+(qa-(1+c(n+1)q_0a))\omega.$$ 
If $p$ and $q$ are large enough and $0<aq-rp\leq aq_0-rp_0$, 
then $pH+qa\omega-\omega'$ is $\pi$-ample. 
It is because 
\begin{align*}
&(p-c(n+1)p_0)H+(qa-(1+c(n+1)q_0a))\omega\\
&=(p-(1+c(n+1))p_0)H+(qa-(1+c(n+1))q_0a)\omega
+p_0H+(q_0a-1)\omega.  
\end{align*}

Suppose that $r$ is not rational. 
There must be arbitrarily large $(p, q)$ such that 
$0<aq-rp<\varepsilon =aq_0-rp_0$ and 
$H^0(C_\eta, \mathcal O_{C_\eta}(M(p, q)))\ne 0$ 
by Lemma \ref{pol1-le}. 
It is because $M(p, q)-\omega'$ is $\pi$-ample by 
$0<aq-rp<aq_0-rp_0$, 
$P_{C_\eta}(p ,q)=\chi (C_\eta, \mathcal O_{C_\eta}
(M(p, q)))$ 
is a non-trivial polynomial of degree at most $\dim C_\eta$ 
by 
Lemma \ref{pol2-le}, 
and $\chi(C_\eta, \mathcal O_{C_\eta}(M(p, q)))=
h^0(C_\eta, \mathcal O_{C_\eta}(M(p, q)))$ 
by the ampleness of $M(p, q)-\omega'$. 
By the vanishing theorem, 
$
\pi_*\mathcal O_X(M(p, q))\to \pi_*\mathcal O_C(M(p, q))
$ 
is surjective because $M(p, q)-\omega'$ is $\pi$-ample.
We note that $C'=C\cup X_{-\infty}$ has a natural quasi-log structure 
induced by $[X, \omega']$ and that $C\cap X_{-\infty}=\emptyset$. 
Thus $C$ is not contained in $L(p, q)$. Therefore, 
$L(p, q)$ is a proper subset of $L(p_0, q_0)=L_0$, 
giving the desired contradiction. 
So now we know that $r$ is rational. 

We next suppose that the assertion of the theorem concerning 
the denominator of $r$ is false. 
Choose $(p_0, q_0)\in I$ such that 
$aq_0-rp_0$ is the maximum, 
say it is equal to $d/v$. 
If $0<aq-rp\leq d/v$ and 
$(p, q)$ is sufficiently large, then $\chi (C_\eta, 
\mathcal O_{C_\eta}(M(p, q)))=h^0(C_\eta, 
\mathcal O_{C_\eta}(M(p, q)))$ since 
$M(p, q)-\omega'$ is $\pi$-ample. 
There exists sufficiently 
large $(p, q)$ in the strip 
$0<aq-rp<1$ with $\varepsilon =1$ for which 
$h^0(C_\eta, \mathcal O_{C_\eta}(M(p, q)))\ne 0$ by 
Lemma \ref{pol1-le}. 
Note that $aq-rp\leq d/v=aq_0-rp_0$ holds automatically for 
$(p, q)\in I$. 
Since 
$\pi_*\mathcal O_X(M(p, q))
\to \pi_*\mathcal O_C(M(p,q))$ 
is surjective by the $\pi$-ampleness of $M(p,q)-\omega'$, 
we 
obtain the desired contradiction by the same reason as above. 
So, we finish the proof of the rationality theorem. 
\end{proof}

We used the following lemma in the proof of 
Theorem \ref{rat-th}. 

\begin{lem}\label{boun}
Let $[X, \omega]$ be a qlc pair and $x\in X$ 
a closed point. 
Let $D_1, \cdots, D_m$ be effective 
Cartier divisors passing through $x$. 
If $[X, \omega+\sum _{i=1}^m D_i]$ is qlc, 
then $m\leq \dim X$. 
\end{lem}

\begin{proof}
First, we assume $\dim X=1$. 
If $x\in X$ is a qlc center of $[X, \omega]$, 
then $m$ must be zero. 
So, we can assume that 
$x\in X$ is not a qlc center of $[X, \omega]$. 
Let $f:(Y, B_Y)\to X$ be a quasi-log resolution of $[X, \omega]$. 
By shrinking $X$ around $x$, we can assume that 
any stratum of $Y$ dominates $X$ and that 
$X$ is smooth by Proposition \ref{normal}. 
Since $f_*\mathcal O_Y(\ulcorner -(B_Y^{<1})\urcorner)\simeq 
\mathcal O_X$, we can easily check that 
$m\leq 1=\dim X$. 
In general, $[X, \omega+D_1]$ is qlc. 
Let $V$ be the union of qlc centers of $[X, \omega+D_1]$ contained 
in $\Supp D_1$. 
Then both $[V, (\omega+D_1)|_V]$ and $[V, 
(\omega+D_1)|_V+\sum _{i=2}^m D_i|_V]$ are 
qlc by adjunction. 
By the induction on the dimension, $m-1\leq \dim V$. Therefore, 
we obtain $m\leq \dim X$. 
\end{proof}

\subsection{Cone Theorem}\label{333-ssec}
The main theorem of this subsection is the cone theorem 
for quasi-log varieties (cf.~\cite[Theorem 5.10]{ambro}). 
Before we state the main theorem, let us 
fix the notation. 

\begin{defn}
Let $[X, \omega]$ be a quasi-log pair with 
the non-qlc locus $X_{-\infty}$. 
Let $\pi:X\to S$ be a projective morphism. 
We put\index{$\overline{NE}(X/S)_{-\infty}$} 
$$
\overline {NE}(X/S)_{-\infty}=
\xIm (\overline {NE}(X_{-\infty}/S)
\to \overline {NE}(X/S)). 
$$
For an $\mathbb R$-Cartier divisor $D$, we define 
$$
D_{\geq 0}=\{z\in N_1(X/S) \ | \ D\cdot z\geq 0\}. 
$$ 
Similarly, we can define $D_{>0}$, 
$D_{\leq 0}$, and 
$D_{<0}$. 
We also define 
$$
D^{\perp}=\{ z\in N_1(X/S) \ | \ D\cdot z=0\}. 
$$ 
We use the following notation 
$$
\overline {NE}(X/S)_{D\geq 0} =\overline {NE}(X/S)\cap 
D_{\geq 0}, 
$$ 
and similarly for $>0$, $\leq 0$, and $<0$.\index{$D_{\geq 0}$}\index{$D_{>0}$}\index{$D_{\leq 0}$}\index{$D_{<0}$}
\index{$D^{\perp}$} 
\end{defn}

\begin{defn}
An {\em{extremal face}}\index{extremal 
face} of $\overline {NE}(X/S)$ is a non-zero subcone $F\subset \overline {NE}(X/S)$ such that 
$z, z'\in F$ and $z+z'\in F$ imply that 
$z, z'\in F$. Equivalently, 
$F=\overline {NE}(X/S)\cap H^{\perp}$ for some 
$\pi$-nef $\mathbb R$-divisor $H$, which 
is called a {\em{supporting function}} of $F$.\index{supporting function} 
An {\em{extremal ray}}\index{extremal ray} 
is a one-dimensional 
extremal face. 
\begin{itemize}
\item[(1)] An extremal face $F$ is called 
{\em{$\omega$-negative}}\index{$\omega$-negative} 
if $F\cap \overline {NE}(X/S)_{\geq 0} 
=\{ 0\}$. 
\item[(2)] An extremal face $F$ is called {\em{rational}} 
if we can choose 
a $\pi$-nef $\mathbb Q$-divisor $H$ as a support 
function of $F$. 
\item[(3)] An extremal face $F$ is called {\em{relatively ample at infinity}}\index{relatively ample 
at infinity} if $F\cap \overline {NE}(X/S)_{-\infty}=\{0\}$. 
Equivalently, $H|_{X_{-\infty}}$ is $\pi|_{X_{-\infty}}$-ample for 
any supporting function $H$ of $F$. 
\item[(4)] An extremal face $F$ is called {\em{contractible 
at infinity}}\index{contractible at 
infinity} if it has a rational supporting function 
$H$ such that $H|_{X_{-\infty}}$ is $\pi|_{X_{-\infty}}$-semi-ample. 
\end{itemize}
\end{defn}

The following theorem is a direct consequence 
of Theorem \ref{bpf-th}. 

\begin{thm}[Contraction Theorem]\index{contraction theorem}
\label{cont-th}
Let $[X, \omega]$ be a quasi-log pair and 
let $\pi:X\to S$ be a projective 
morphism. 
Let $H$ be a $\pi$-nef 
Cartier divisor such that 
$F=H^{\perp}\cap \overline {NE}(X/S)$ is $\omega$-negative 
and contractible at infinity. 
Then there exists a projective morphism 
$\varphi_F:X\to Y$ over $S$ with the following properties. 
\begin{itemize}
\item[$(1)$] 
Let $C$ be an integral curve on $X$ such that 
$\pi(C)$ is a point. 
Then $\varphi_F(C)$ is a point if and 
only if 
$[C]\in F$. 
\item[$(2)$] $\mathcal O_Y\simeq (\varphi_F)_*\mathcal O_X$.  
\item[$(3)$] Let $L$ be a line bundle on $X$ such 
that $L\cdot C=0$ for 
every curve $C$ with $[C]\in F$. 
Then there is a line bundle $L_Y$ on $Y$ such 
that $L\simeq \varphi^*_FL_Y$. 
\end{itemize} 
\end{thm}
\begin{proof}
By the assumption, $qH-\omega$ is 
$\pi$-ample for some positive integer $q$ and 
$H|_{X_{-\infty}}$ is $\pi|_{X_{-\infty}}$-semi-ample. 
By Theorem \ref{bpf-th}, $\mathcal O_X(mH)$ is $\pi$-generated 
for $m\gg 0$. 
We take the Stein factorization of the associated 
morphism. Then, we have the contraction morphism 
$\varphi_F: X\to Y$ with the properties (1) and (2). 

We consider $\varphi_F:X\to Y$ and $\overline {NE}(X/Y)$. 
Then $\overline {NE}(X/Y)=F$, $L$ is 
numerically trivial over $Y$, and $-\omega$ is $\varphi_F$-ample. Applying the base point free theorem (cf.~Theorem 
\ref{bpf-th}) over $Y$, both $L^{\otimes m}$ and $L^{\otimes (m+1)}$ 
are pull-backs of line bundles on $Y$. Their difference 
gives a line bundle $L_Y$ such that 
$L\simeq \varphi^*_FL_Y$. 
\end{proof}

\begin{thm}[Cone Theorem]\index{cone theorem}\label{cone-thm} 
Let $[X, \omega]$ be a quasi-log pair and let 
$\pi:X\to S$ be a projective morphism. 
Then we have the following 
properties. 
\begin{itemize}
\item[$(1)$] $\overline {NE}(X/S)=\overline {NE}(X/S)_{\omega\geq 0} 
+\overline {NE}(X/S)_{-\infty}+\sum R_j$, 
where $R_j$'s are the $\omega$-negative 
extremal rays of $\overline {NE}(X/S)$ that are 
rational and relatively ample at infinity. 
In particular, each $R_j$ is spanned by 
an integral curve $C_j$ on $X$ such that 
$\pi(C_j)$ is a point.  
\item[$(2)$] Let $H$ be a $\pi$-ample $\mathbb R$-divisor 
on $X$. 
Then there are only finitely many $R_j$'s included in 
$(\omega+H)_{<0}$. In particular, 
the $R_j$'s are discrete in the half-space 
$\omega_{<0}$. 
\item[$(3)$] Let $F$ be an $\omega$-negative extremal 
face of $\overline {NE}(X/S)$ that is 
relatively ample at infinity. 
Then $F$ is a rational face. 
In particular, $F$ is contractible at infinity. 
\end{itemize}
\end{thm}

\begin{proof}
First, we assume that $\omega$ is $\mathbb Q$-Cartier. 
This means that $\omega$ is $\mathbb R$-linearly equivalent 
to a $\mathbb Q$-Cartier divisor. 
We can assume that $\dim _{\mathbb R}
N_1(X/S)\geq 2$ and $\omega$ is not $\pi$-nef. 
Otherwise, the theorem is obvious. 

\setcounter{step}{0}
\begin{step}\label{cone-st1} 
We have 
$$
\overline {NE}(X/S)=\overline 
{\overline {NE}(X/S)_{\omega\geq 0} 
+\overline {NE}(X/S)_{-\infty}+\sum _F F},  
$$ 
where $F$'s vary among all rational 
proper $\omega$-negative faces that are relatively 
ample at infinity and 
$\raise0.5ex\hbox{\textbf{-----}}$
denotes the closure 
with respect to 
the real topology. 
\end{step} 
\begin{proof}We put 
$$
B=\overline {\overline {NE}(X/S)_{\omega\geq 0} 
+\overline {NE}(X/S)_{-\infty}+\sum _F F}. 
$$ 
It is clear that $\overline {NE}(X/S)\supset B$. 
We note that each $F$ is spanned by curves 
on $X$ mapped to points on $S$ by 
Theorem \ref{cont-th} (1). Supposing
$\overline {NE}(X/S)\ne B$, we shall 
derive a contradiction. 
There is a separating function $M$ which 
is Cartier and is not a multiple of $\omega$ in $N^1(X/S)$ such 
that $M>0$ on $B\setminus \{0\}$ and 
$M\cdot z_0<0$ for some $z_0\in \overline {NE}(X/S)$. 
Let $C$ be the dual cone of $\overline {NE}(X/S)_{\omega\geq 0}$, 
that is, 
$$
C=\{D\in N^1(X/S)\ | \ D\cdot z\geq 0 \ \text{for}\ z\in 
\overline {NE}(X/S)_{\omega \geq 0}\}. 
$$
Then $C$ is generated by $\pi$-nef divisors and $\omega$. 
Since $M>0$ on $\overline {NE}(X/S)_{\omega\geq 0}\setminus \{0\}$, $M$ is in the interior of $C$, and hence there 
exists a $\pi$-ample $\mathbb Q$-Cartier 
divisor $A$ such that 
$M-A=L'+p\omega$ in $N^1(X/S)$, where 
$L'$ is a $\pi$-nef $\mathbb Q$-Cartier divisor 
on $X$ and $p$ is a non-negative rational number. 
Therefore, $M$ is expressed in the form $M=H+p\omega$ in 
$N^1(X/S)$, where $H=A+L'$ is a $\pi$-ample 
$\mathbb Q$-Cartier divisor. 
The rationality theorem (see Theorem \ref{rat-th}) 
implies that there exists a positive 
rational number $r<p$ such that $L=H+r\omega$ is 
$\pi$-nef but not 
$\pi$-ample, and $L|_{X_{-\infty}}$ is 
$\pi|_{X_{-\infty}}$-ample. 
Note that $L\ne 0$ in $N^1(X/S)$, since 
$M$ is not a multiple of $\omega$. 
Thus the extremal face $F_L$ associated to 
the supporting function $L$ is contained 
in $B$, which implies $M>0$ on $F_L$. 
Therefore, $p<r$. It is a contradiction. 
This completes the proof of our first claim. 
\end{proof}

\begin{step}
In the equality of Step \ref{cone-st1}, 
we may take such $L$ that 
has the extremal face $F_L$ of dimension one. 
\end{step}
\begin{proof}
Let $F$ be a rational proper $\omega$-negative extremal face 
that is relatively ample at infinity, and assume that $\dim F\geq 
2$. Let $\varphi_F:X\to W$ be the associated 
contraction. 
Note that $-\omega$ is $\varphi_F$-ample. 
By Step \ref{cone-st1}, 
we obtain 
$$
F=\overline {NE}(X/W)=\overline {\sum _G G}, 
$$ 
where the $G$'s are the rational proper 
$\omega$-negative extremal faces of $\overline {NE}(X/W)$. 
We note that $\overline {NE}(X/W)_{-\infty}=0$ because 
$\varphi_F$ embeds $X_{-\infty}$ into $W$. 
The $G$'s are also 
$\omega$-negative 
extremal faces of $\overline {NE}(X/S)$ that are 
ample at infinity, and $\dim G<\dim F$. By induction, we obtain 
\begin{equation}\label{siki1} 
\overline {NE}(X/S)=\overline {\overline {NE}(X/S)_{\omega\geq 0} 
+\overline {NE}(X/S)_{-\infty} +\sum R_j}, 
\end{equation}
where the $R_j$'s are $\omega$-negative rational 
extremal rays. 
Note that each $R_j$ does not 
intersect 
$\overline {NE}(X/S)_{-\infty}$.  
\end{proof}

\begin{step}\label{cone-st3} 
The contraction theorem (cf.~Theorem \ref{cont-th}) 
guarantees that 
for each extremal ray $R_j$ there exists a reduced irreducible curve $C_j$ on $X$ such that 
$[C_j]\in R_j$. 
Let $\psi_j:X\to W_j$ be the contraction 
morphism of $R_j$, and 
let $A$ be a $\pi$-ample 
Cartier divisor. 
We set 
$$
r_j=-\frac{A\cdot  C_j}{\omega\cdot C_j}. 
$$
Then $A+r_j\omega$ is $\psi_j$-nef but not $\psi_j$-ample, 
and 
$(A+r_j \omega)|_{X_{-\infty}}$ is $\psi_j|_{X_{-\infty}}$-ample. 
By the rationality theorem 
$($see Theorem \ref{rat-th}$)$, expressing 
$r_j =u_j /v_j$ with 
$u_j, v_j\in \mathbb Z_{>0}$ and 
$(u_j, v_j)=1$, we have 
the inequality $v_j\leq a(\dim X+1)$. 
\end{step}

\begin{step}
Now take $\pi$-ample Cartier divisors 
$H_1, H_2, \cdots, H_{\rho-1}$ such that 
$\omega$ and the $H_i$'s form a basis of $N^1(X/S)$, where 
$\rho =\dim _{\mathbb R}N^1(X/S)$. By 
Step \ref{cone-st3}, the intersection of  the extremal 
rays $R_j$ with the hyperplane 
$$\{z\in N_1(X/S)\ | \ a\omega \cdot z=-1\}$$ in 
$N_1(X/S)$ lie on the lattice 
$$
\Lambda=
\{z\in N_1(X/S)\ | \ a\omega \cdot z=-1, H_i \cdot z\in (a(a(\dim 
X+1))!)^{-1}\mathbb Z\}. 
$$ 
This implies that the extremal rays are discrete in the 
half space 
$$
\{z \in N_1(X/S) \ | \ \omega\cdot z<0\}. 
$$ 
Thus we can omit the closure sign 
$\raise0.5ex\hbox{\textbf{-----}}$ 
from 
the formula (\ref{siki1}) 
and this completes the proof of 
(1) when $\omega$ is $\mathbb Q$-Cartier. 
\end{step} 
\begin{step}
Let $H$ be a $\pi$-ample $\mathbb R$-divisor on $X$. 
We choose $0<\varepsilon_i\ll 1$ for $1\leq i\leq \rho-1$ such 
that $H-\sum_{i=1}^{\rho-1}\varepsilon _i H_i$ is $\pi$-ample. 
Then the $R_j$'s included in $(\omega+H)_{<0}$ correspond to 
some elements of the above lattice $\Lambda$ for which 
$\sum_{i=1}^{\rho-1}\varepsilon _i H_i\cdot z<1/a$.
 Therefore, we obtain (2). 
\end{step}
\begin{step}\label{cone-st6}
The vector space $V=F^{\perp}\subset N^1(X/S)$ is 
defined over $\mathbb Q$ because 
$F$ is generated by some of the $R_j$'s. 
There exists a $\pi$-ample $\mathbb R$-divisor 
$H$ such that $F$ is contained in $(\omega+H)_{<0}$. 
Let $\langle F\rangle$ be the vector space 
spanned by $F$. 
We put 
$$
W_F=\overline {NE}(X/S)_{\omega+H\geq 0}+
\overline {NE}(X/S)_{-\infty}+\sum _{R_j\not\subset F}R_j. 
$$ 
Then $W_F$ is a closed cone, 
$\overline {NE}(X/S)=W_F+F$, 
and $W_F\cap \langle F\rangle=\{0\}$. The supporting 
functions of $F$ are the elements of $V$ that are 
positive on $W_F\setminus \{0\}$. 
This is a non-empty open set and thus it 
contains a rational element that, after scaling, 
gives a 
$\pi$-nef Cartier divisor $L$ such that 
$F=L^{\perp}\cap \overline {NE}(X/S)$. Therefore, 
$F$ is rational. 
So, we have (3). 
\end{step}
From now on, $\omega$ is $\mathbb R$-Cartier. 
\begin{step}
Let $H$ be a $\pi$-ample $\mathbb R$-divisor on $X$. 
We shall prove (2). 
We assume that there are infinitely many 
$R_j$'s in $(\omega+H)_{<0}$ and 
get a contradiction. 
There exists an affine open subset $U$ of 
$S$ such that 
$\overline {NE}(\pi^{-1}(U)/U)$ has infinitely many 
$(\omega+H)$-negative extremal rays. 
So, we shrink $S$ and can assume that 
$S$ is affine. 
We can write $H=E+H'$, where 
$H'$ is $\pi$-ample, 
$[X, \omega+E]$ is a quasi-log pair 
with the same qlc centers and non-qlc 
locus as $[X, \omega]$, 
and $\omega+E$ is $\mathbb Q$-Cartier. 
Since $\omega+H=\omega+E+H'$, we 
have 
$$
\overline {NE}(X/S)=\overline {NE}(X/S)_{\omega+H\geq 0}
+\overline {NE}(X/S)_{-\infty}+\sum _{\text{finite}}R_j. 
$$
It is a contradiction. 
Thus, we obtain (2). 
The statement (1) is a direct consequence of 
(2). Of course, 
(3) holds by Step \ref{cone-st6} 
once we obtain (1). 
\end{step}
So, we finish the proof of the cone theorem. 
\end{proof}

We close this subsection 
with the following non-trivial example. 

\begin{ex}\label{cone-ex} We consider 
the first projection $p:\mathbb P^1\times \mathbb P^1\to \mathbb P^1$. 
We take a blow-up $\mu:Z\to \mathbb P^1\times \mathbb P^1$ at 
$(0, \infty)$. 
Let $A_{\infty}$ (resp.~$A_0$) be 
the strict transform of $\mathbb P^1\times \{\infty\}$ 
(resp.~$\mathbb P^1\times \{0\}$) on $Z$. 
We define $M=\mathbb P_Z(\mathcal O_Z\oplus \mathcal O_Z(A_0))$ and 
$X$ is the restriction of $M$ on $(p\circ \mu)^{-1}(0)$. 
Then $X$ is a simple normal crossing divisor on $M$. 
More explicitly, $X$ is a $\mathbb P^1$-bundle 
over $(p\circ \mu)^{-1}(0)$ and is obtained by gluing $X_1=
\mathbb P^1\times \mathbb P^1$ 
and $X_2=
\mathbb P_{\mathbb P^1}(\mathcal O_{\mathbb P^1}\oplus \mathcal 
O_{\mathbb P^1}(1))$ along a fiber. 
In particular, $[X, K_X]$ is a quasi-log pair with only qlc 
singularities. 
By the construction, $M\to Z$ has two sections. 
Let $D^+$ (resp.~$D^-$) be the restriction of the section 
of $M\to Z$ corresponding to 
$\mathcal O_Z\oplus \mathcal O_Z(A_0)\to \mathcal O_Z(A_0)\to 0$ 
(resp.~$\mathcal O_Z\oplus \mathcal O_Z(A_0)\to \mathcal O_Z\to 0$). 
Then it is easy to see that $D^+$ is 
a nef Cartier divisor on $X$ 
and that 
the linear system $|mD^+|$ is 
free for any $m>0$ by Remark \ref{free-re} below. 
We take a general member $B_0\in |mD^+|$ with 
$m\geq 2$. 
We consider $K_X+B$ with $B=D^-+B_0+B_1+B_2$, 
where $B_1$ and $B_2$ are general fibers of $X_1=\mathbb P^1\times 
\mathbb P^1\subset X$. 
We note that $B_0$ does not intersect 
$D^-$. 
Then $(X, B)$ is an embedded simple normal crossing 
pair. In particular, $[X, K_X+B]$ is a quasi-log pair with 
$X_{-\infty}=\emptyset$. 
It is easy to see that there exists only one 
integral curve $C$ on $X_2=
\mathbb P
_{\mathbb P^1}(\mathcal O_{\mathbb P^1}\oplus \mathcal O_{\mathbb P^1}
(1))\subset X$ such that 
$C\cdot (K_X+B)<0$. Note that 
$(K_X+B)|_{X_1}$ is ample on $X_1$. 
By the cone theorem, 
we obtain 
$$
\overline {NE}(X)=\overline {NE}(X)_{(K_X+B)\geq 0}+\mathbb R_{\geq 0} 
[C]. 
$$ 
By the contraction theorem, we have $\varphi:X\to W$ which 
contracts $C$. We can easily see that 
$W$ is a simple normal crossing surface but 
$K_W+B_W$, where $B_W=\varphi_*B$, is not $\mathbb Q$-Cartier. 
Therefore, we can not run the LMMP for reducible varieties. 
\end{ex}

The above example implies that the cone and contraction 
theorems for quasi-log varieties do not directly produce 
the LMMP for quasi-log varieties. 

\begin{rem}\label{free-re}
In Example \ref{cone-ex}, $M$ is a projective 
toric variety. 
Let $E$ be the section of $M\to Z$ corresponding 
to $\mathcal O_Z\oplus \mathcal O_Z(A_0)\to 
\mathcal O_Z(A_0)\to 0$. 
Then, it is easy to see that 
$E$ is a nef Cartier divisor 
on $M$. Therefore, the linear system $|E|$ is 
free. In particular, $|D^+|$ is free on $X$. Note 
that $D^+=E|_{X}$. So, 
$|mD^+|$ is free for any $m\geq 0$. 
\end{rem}

\chapter{Related Topics}\label{chap4}

In this chapter, we treat related topics. 
In Section \ref{34-sec}, we discuss the base point free 
theorem of Reid--Fukuda type. 
In Section \ref{b-dlt-sec}, we prove that 
the non-klt locus of a dlt pair is Cohen--Macaulay as an 
application of Lemma \ref{re-vani-lem}. 
Section \ref{sec-alex} is a description of Alexeev's 
criterion for Serre's $S_3$ condition. 
It is a clever application of Theorem \ref{8} (i). 
Section \ref{to-sec} is an introduction to the 
theory of toric polyhedra. 
A toric polyhedron has a natural quasi-log structure. 
In Section \ref{43-sec}, 
we quickly explain the notion of 
{\em{non-lc ideal sheaves}} and the restriction 
theorem in \cite{fujino10}. 
It is related to the inversion of adjunction 
on log canonicity. 
In the final section, we state effective base point free theorems 
for log canonical pairs. We give no proofs there. 

\section{Base Point Free Theorem of Reid--Fukuda 
type}\label{34-sec}

One of my motivations to study \cite{ambro} is 
to understand \cite[Theorem 7.2]{ambro}, 
which is a complete generalization of \cite{re-fu}. 
The following theorem is a special case of 
Theorem 7.2 in \cite{ambro}, 
which was stated without proof. 
Here, we will reduce it to Theorem \ref{bpf-th} by using 
Kodaira's lemma. 

\begin{thm}[Base point free theorem of 
Reid--Fukuda type]\label{bpf-rd-th}
Let $[X, \omega]$ be a quasi-log pair with $X_{-\infty}=\emptyset$, 
$\pi:X\to S$ a {\em{projective}} morphism, and $L$ 
a $\pi$-nef Cartier divisor on $X$ such that 
$qL-\omega$ is nef and log big over $S$ for some positive 
real number $q$. 
Then $\mathcal O_X(mL)$ is $\pi$-generated 
for $m\gg 0$. 
\end{thm}

\begin{rem}
In \cite[Section 7]{ambro}, Ambro said that 
the proof of \cite[Theorem 7.2]{ambro} 
is parallel to 
\cite[Theorem 5.1]{ambro}. 
However, I could not check it. 
Steps 1, 2, and 4 in the proof of 
\cite[Theorem 5.1]{ambro} 
work without any modifications. 
In Step 3 (see Claim \ref{bpf-c3} in 
the proof of Theorem \ref{bpf-th}), 
$q'L-\omega'$ is $\pi$-nef, but I think that 
$q'L-\omega'=qL-\omega$ 
is not always log big over $S$ with 
respect to $[X, \omega']$, 
where $\omega'=\omega+cD$ and $q'=q+cp^l$. So, 
we can not directly apply the argument in Step 1 
(see Claim \ref{bpf-c1} in the proof of 
Theorem \ref{bpf-th}) to 
this new quasi-log pair $[X, \omega']$. 
\end{rem}

\begin{proof}
We divide the proof into three steps. 
\setcounter{step}{0}
\begin{step}
We take an irreducible component $X'$ of $X$. 
Then $X'$ has a natural quasi-log 
structure induced by $X$ $($see 
Theorem \ref{adj-th} (i)$)$. 
By the vanishing theorem $($see 
Theorem \ref{adj-th} (ii)$)$, we have 
$R^1\pi_*
(\mathcal I_{X'}\otimes \mathcal O_X(mL))=0$ for $m\geq q$. 
Therefore, we obtain that 
$\pi_*\mathcal O_X(mL)\to \pi_*\mathcal O_{X'}
(mL)$ is surjective for 
$m\geq q$. 
Thus, we can assume that $X$ is irreducible for the proof of 
this theorem by the following commutative 
diagram. 
$$
\begin{CD}
\pi^*\pi_*\mathcal O_X(mL)@>>>\pi^*\pi_*\mathcal O_{X'}(mL)@>>>0\\
@VVV @VVV\\
\mathcal O_X(mL)@>>>\mathcal O_{X'}(mL)@>>>0
\end{CD}
$$
\end{step} 
\begin{step}
Without loss of generality, we can assume that $S$ is affine. 
Since $qL-\omega$ is nef and big over $S$, we can 
write $qL-\omega\sim _{\mathbb R}A+E$ by Kodaira's lemma, 
where $A$ is a $\pi$-ample $\mathbb Q$-Cartier 
$\mathbb Q$-divisor on $X$ and $E$ is an effective $\mathbb R$-Cartier 
$\mathbb R$-divisor 
on $X$. 
We note that $X$ is projective 
over $S$ and that $X$ is not necessarily normal. 
By Lemma \ref{key-le} below, we have a new 
quasi-log structure on $[X, \widetilde \omega]$, 
where $\widetilde \omega=\omega+\varepsilon E$,  
for $0<\varepsilon \ll 1$. 
\end{step}

\begin{step}
By the induction on the dimension, 
$\mathcal O_{\Nqklt(X, \omega)}(mL)$ is $\pi$-generated for 
$m\gg 0$. 
Note that $\pi_*\mathcal O_X(mL)\to \pi_*\mathcal 
O_{\Nqklt(X, 
\omega)}(mL)$ 
is surjective for $m\geq q$ by the vanishing theorem $($see 
Theorem \ref{adj-th} (ii)$)$. 
Then $\mathcal O_{\Nqklt(X, \widetilde \omega)}(mL)$ is $\pi$-generated 
for $m\gg 0$ by the above lifting result and by 
Lemma \ref{key-le}. 
In particular, $\mathcal O_{\widetilde X_{-\infty}}(mL)$ 
is $\pi$-generated for 
$m\gg 0$. 
We note that $qL-\widetilde \omega\sim _{\mathbb R}
(1-\varepsilon)(qL-\omega)+\varepsilon A$ is $\pi$-ample. 
Therefore, by Theorem \ref{bpf-th}, 
we obtain that $\mathcal O_X(mL)$ is $\pi$-generated for 
$m\gg 0$. 
\end{step}
We finish the proof. 
\end{proof}

\begin{lem}\label{key-le}
Let $[X, \omega]$ be a quasi-log pair with 
$X_{-\infty}=\emptyset$. 
Let $E$ be an effective $\mathbb R$-Cartier 
$\mathbb R$-divisor on $X$. 
Then $[X, \omega+\varepsilon E]$ is a 
quasi-log pair with the following 
properties for 
$0<\varepsilon \ll 1$. 
\begin{itemize}
\item[{\em{(i)}}] 
We put $[X, \widetilde \omega]=[X, 
\omega+\varepsilon E]$. 
Then $[X, \widetilde \omega]$ 
is a quasi-log pair and 
$\Nqklt (X, 
\widetilde \omega)=\Nqklt (X, \omega)$ as closed subsets of $X$. 
\item[{\em{(ii)}}] There exist natural surjective 
homomorphisms $\mathcal O_{\Nqklt(X, \widetilde \omega)}\to 
\mathcal O_{\Nqklt(X, \omega)}\to 0$ and 
$\mathcal O_{\Nqklt(X, \widetilde \omega)}
\to \mathcal O_{\widetilde X_{-\infty}}\to 0$, 
that is, $\Nqklt(X, \omega)$ and $\widetilde X_{-\infty}$ are closed 
subschemes of $\Nqklt(X, \widetilde \omega)$, 
where $\widetilde X_{-\infty}$ is the non-qlc locus 
of $[X, \widetilde\omega]$. 
\end{itemize}
\end{lem} 
\begin{proof}
Let $f:(Y, B_Y)\to X$ be a quasi-log resolution of $[X, \omega]$, 
where $(Y, B_Y)$ is a global embedded simple normal 
crossing pair. 
We can assume that 
the union of all strata of $(Y, B_Y)$ mapped into 
$\Nqklt (X, \omega)$, which we denote by $Y'$, is a 
union of irreducible components of $Y$. 
We put $Y''=Y-Y'$. Then 
we obtain that $f_*\mathcal O_{Y''}(A-Y'|_{Y''})$ 
is $\mathcal I_{\Nqklt (X, \omega)}$, 
that is, the defining ideal sheaf of $\Nqklt (X, \omega)$ on $X$, 
where $A=\ulcorner -(B^{<1}_Y)\urcorner$. 
For the details, see the proof of Theorem \ref{adj-th} (i). 
Let $M$ be the ambient space of $Y$ and $B_Y=D|_Y$. 
\begin{claim}
By modifying $M$ birationally, 
we can assume that 
there exists a simple normal crossing divisor $F$ on $M$ such 
that $\Supp (Y+D+F)$ is simple normal 
crossing, $F$ and $Y''$ have no common irreducible 
components, and 
$F|_{Y''}=(f'')^*E$, where 
$f''=f|_{Y''}$. 
Of course, $(f'')^*E+B_{Y''}$ has a simple 
normal crossing support on $Y''$, where 
$K_{Y''}+B_{Y''}=(K_Y+B_Y)|_{Y''}$. 
In general, $F$ may have common irreducible components with $D$ and $Y'$. 
\end{claim}
\begin{proof}[Proof of {\em{Claim}}] 
First, we note that $(f'')^*E$ contains no strata of $Y''$. 
We can construct a proper birational 
morphism $h:\widetilde M\to M$ from 
a smooth variety $\widetilde M$ such that 
$K_{\widetilde M}+D_{\widetilde M}=h^*(K_M+Y+D)$, 
$h^{-1}((f'')^*E)$ is a divisor 
on $\widetilde M$, and 
$\Exc (h)\cup \Supp h^{-1}_*(Y+D)\cup 
h^{-1}((f'')^*E)$ is a simple normal crossing 
on $\widetilde M$ as in the proof of 
Proposition \ref{tai4}. 
We note that 
we can assume that $h$ is an isomorphism outside 
$h^{-1}((f'')^*E)$ by 
Szab\'o's resolution lemma. 
Let $\widetilde Y$ be the union 
of the irreducible 
components of $D^{=1}_{\widetilde M}$ that 
are mapped into $Y$. 
By Proposition \ref{taisetsu}, 
we can replace $M$, $Y$, 
and $D$ with 
$\widetilde M$, $\widetilde Y$, 
and $\widetilde D=D_{\widetilde M}-\widetilde Y$. 
We finish the proof. 
\end{proof}
Let us go back to 
the proof of Lemma \ref{key-le}. 
Let $Y_2$ be the union of all 
the irreducible 
components of $Y$ that are 
contained in $\Supp F$. 
We put $Y_1=Y-Y_2$ and 
$\widetilde B=F|_{Y_1}$. 
We consider $f_1:(Y_1, B_{Y_1}+\varepsilon \widetilde B)\to 
X$ for $0<\varepsilon \ll 1$, 
where 
$K_{Y_1}+B_{Y_1}=(K_Y+B_Y)|_{Y_1}$ and 
$f_1=f|_{Y_1}$. 
Then, we have 
$f^*_1(\omega+\varepsilon E)\sim _{\mathbb R}
K_{Y_1}+
B_{Y_1}+\varepsilon \widetilde B$. 
Moreover, the natural 
inclusion 
$\mathcal O_X\to f_{1*}\mathcal O_{Y_1}(\ulcorner 
-((B_{Y_1}+\varepsilon \widetilde B)^{<1})\urcorner)$ defines 
an ideal 
$\mathcal I_{\widetilde X_{-\infty}}
=f_{1*}\mathcal O_{Y_1}(\ulcorner 
-((B_{Y_1}+\varepsilon \widetilde B)^{<1})\urcorner
-\llcorner (B_{Y_1}+\varepsilon \widetilde B)^{>1}\lrcorner)$. 
It is because 
$$f_{1*}\mathcal O_{Y_1}(\ulcorner 
-((B_{Y_1}+\varepsilon \widetilde B)^{<1})\urcorner
-\llcorner (B_{Y_1}+\varepsilon \widetilde B)^{>1}\lrcorner)
\subset 
f_*\mathcal O_{Y}(\ulcorner -(B_Y)^{<1}\urcorner)\simeq 
\mathcal O_X$$ when 
$0<\varepsilon \ll 1$. 
We note that 
$\llcorner (B_{Y_1}+\varepsilon \widetilde B)^{>1}\lrcorner
\geq {Y_2}|_{Y_1}$. 
Namely, the pair $[X, \widetilde \omega]$ has a quasi-log structure 
with a quasi-log resolution $f_1: (Y_1, B_{Y_1}+\varepsilon 
\widetilde B)\to X$. By the construction 
and the definition, 
it is obvious that 
there exist surjective 
homomorphisms 
$\mathcal O_{\Nqklt (X, \widetilde \omega)}\to 
\mathcal O_{\Nqklt (X, \omega)}\to 0$ and 
$\mathcal O_{\Nqklt (X, \widetilde \omega)}
\to \mathcal O_{\widetilde X_{-\infty}}\to 0$. 
It is not difficult to see that $\Nqklt (X, \omega)=\Nqklt (X, \widetilde \omega)$ as closed subsets of $X$ for $0<\varepsilon \ll 1$. 
We finish the proof. 
\end{proof}

As a special case, we obtain the following 
base point free theorem of Reid--Fukuda type for 
log canonical pairs. 

\begin{thm}\label{bpf-rdlc-th} 
{\em{(Base point free theorem of Reid--Fukuda type 
for lc pairs).}} 
Let $(X, B)$ be an lc pair. 
Let $L$ be a $\pi$-nef Cartier divisor on $X$, where 
$\pi:X\to S$ is a projective morphism. 
Assume that $qL-(K_X+B)$ is $\pi$-nef and $\pi$-log big 
for some positive real number $q$. 
Then $\mathcal O_X(mL)$ is $\pi$-generated for $m\gg 0$. 
\end{thm}

We believe that the above theorem holds 
under the assumption that $\pi$ is only {\em{proper}}. 
However, our proof needs projectivity of $\pi$. 

\begin{rem} 
In Theorem \ref{bpf-rdlc-th}, 
if $\Nklt(X, B)$ is 
projective over $S$, then we can 
prove Theorem \ref{bpf-rdlc-th} under the weaker assumption 
that $\pi:X\to S$ is only {\em{proper}}. 
It is because we can apply Theorem \ref{bpf-rd-th} to 
$\Nklt(X, B)$. 
So, we can assume that 
$\mathcal O_X(mL)$ is $\pi$-generated 
on a non-empty open subset containing 
$\Nklt(X, B)$. 
In this case, we can prove Theorem \ref{bpf-rdlc-th} by 
applying the usual X-method to $L$ on $(X, B)$. 
We note that $\Nklt(X, B)$ is always projective over 
$S$ when $\dim \Nklt(X, B)\leq 1$.  
The reader can find a different proof in \cite{fukuda} 
when $(X, B)$ is a log canonical surface, 
where Fukuda used the LMMP with scaling for dlt surfaces. 
\end{rem}

Finally, we explain the reason why we assumed 
that $X_{-\infty}=\emptyset$ and 
$\pi$ is projective in Theorem \ref{bpf-rd-th}. 

\begin{rem}[Why $X_{-\infty}$ is empty?]\label{emp-re}
Let $C$ be a qlc center of $[X, \omega]$. 
Then we have to consider a quasi-log variety 
$X'=C\cup X_{-\infty}$ for the inductive arguments. 
In general, $X'$ is reducible. 
It sometimes happens that $\dim C<\dim X_{-\infty}$. 
We do not know how to apply Kodaira's lemma to 
reducible varieties. So, we assume that 
$X_{-\infty}=\emptyset$ in Theorem \ref{bpf-rd-th}. 
\end{rem}

\begin{rem}[Why $\pi$ is projective?]\label{proj-re}
We assume that $S$ is a point in Theorem \ref{bpf-rd-th} for 
simplicity. If $X_{-\infty}=\emptyset$, then 
it is enough to treat irreducible quasi-log varieties by 
Step 1. 
Thus, we can assume that $X$ is irreducible. 
Let $f:Y\to X$ be a proper birational morphism 
from a smooth projective variety. 
If $X$ is normal, then $H^0(X, \mathcal O_X(mL))\simeq 
H^0(Y, \mathcal O_Y(mf^*L))$ for any $m\geq 0$. 
However, $X$ is not always normal (see Example 
\ref{rd-ex} below). 
So, it sometimes happens that 
$\mathcal O_Y(mf^*L)$ has many global sections 
but $\mathcal O_X(mL)$ has only a few global sections. 
Therefore, we can not easily reduce the problem to 
the case when $X$ is projective. 
This is the reason why we assume that 
$\pi:X\to S$ is projective. 
See also Proposition \ref{bottle}. 
\end{rem}

\begin{ex}\label{rd-ex}
Let $M=\mathbb P^2$ and let $X$ be a nodal curve on $M$. 
Then $(M, X)$ is an lc pair. 
By Example \ref{new-example}, 
$[X, K_X]$ is a quasi-log variety with only 
qlc singularities. 
In this case, $X$ is irreducible, but it is not 
normal.    
\end{ex}

\section{Basic properties of dlt pairs}\label{b-dlt-sec}  
In this section, we prove supplementary results on dlt pairs. 
First, let us reprove the following 
well-known theorem. 

\begin{thm}\label{49-dlt}
Let $(X, D)$ be a dlt pair. 
Then $X$ has only rational singularities. 
\end{thm}
\begin{proof}(cf.~\cite[Chapter VII, 1.1.Theorem]{nakayama2}).   
By the definition of dlt, we can take a resolution 
$f:Y\to X$ such that 
$\Exc (f)$ and $\Exc (f)\cup \Supp f^{-1}_*D$ are 
both simple normal crossing divisors 
on $Y$ and that $K_Y+f^{-1}_*D=f^*(K_X+D)+E$ 
with $\ulcorner E\urcorner \geq 0$. 
We can take an effective $f$-exceptional 
divisor $A$ on $Y$ such $-A$ is $f$-ample 
(see, for example, \cite[Proposition 3.7.7]{fujino0}). 
Then $\ulcorner E\urcorner -(K_Y+f^{-1}_*D+\{-E\}+\varepsilon 
A)=-f^*(K_X+D)-\varepsilon A$ is $f$-ample for $\varepsilon >0$. 
If $0<\varepsilon \ll 1$, then 
$(Y, f^{-1}_*D+\{-E\}+\varepsilon A)$ is dlt. 
Therefore, $R^if_*\mathcal O_Y(\ulcorner E\urcorner )=0$ for $i>0$ 
(see \cite[Theorem 1-2-5]{kmm}, Theorem \ref{lc}, 
or Lemma \ref{vani-rf-le} below) 
and $f_*\mathcal O_Y(\ulcorner E\urcorner)\simeq \mathcal O_X$. Note 
that $\ulcorner E\urcorner $ is effective and $f$-exceptional. 
Thus, the composition 
$\mathcal O_X\to Rf_*\mathcal O_Y\to 
Rf_*\mathcal O_Y(\ulcorner E\urcorner)\simeq \mathcal O_X$ is 
a quasi-isomorphism. 
So, $X$ has only rational singularities by \cite[Theorem 1]{kovacs2}. 
\end{proof}

In the above proof, we used the next lemma. 

\begin{lem}[Vanishing lemma of Reid--Fukuda type]
\label{vani-rf-le}
Let $V$ be a smooth variety and 
let $B$ be a boundary $\mathbb R$-divisor 
on $V$ such that 
$\Supp B$ is a simple normal crossing divisor. 
Let $f:V\to W$ be a proper morphism 
onto a variety $W$. 
Assume that $D$ is a Cartier divisor on $V$ such that 
$D-(K_V+B)$ is $f$-nef and 
$f$-log big. Then 
$R^if_*\mathcal O_V(D)=0$ for any $i>0$. 
\end{lem}
\begin{proof}
We use the induction on the number of 
irreducible components of $\llcorner B\lrcorner$ and 
on the dimension of $V$. 
If $\llcorner B\lrcorner =0$, then 
the lemma follows from the Kawamata--Viehweg 
vanishing theorem. Therefore, 
we can assume that there is an irreducible 
divisor $S\subset \llcorner B\lrcorner$. 
We 
consider the following short exact sequence 
$$
0\to \mathcal O_V(D-S)\to \mathcal O_V(D)\to 
\mathcal O_S(D)\to 0. 
$$ 
By induction, we see that 
$R^if_*\mathcal O_V(D-S)=0$ 
and $R^if_*\mathcal O_S(D)=0$ for 
any $i>0$. Thus, we have 
$R^if_*\mathcal O_V(D)=0$ for $i>0$. 
\end{proof}

\begin{say}[Weak log-terminal singularities] 
\index{weak log-terminal singularity}The 
proof of Theorem \ref{49-dlt} works 
for {\em{weak log-terminal}} singularities in 
the sense of \cite{kmm}. 
For the definition, see \cite[Definition 0-2-10]{kmm}. 
Thus, we can recover \cite[Theorem 1-3-6]{kmm}, 
that is, we obtain the following statement. 
\begin{thm}[cf.~{\cite[Theorem 1-3-6]{kmm}}]
All weak log-terminal singularities are rational. 
\end{thm}
We think that this theorem is one of the most difficult 
results in \cite{kmm}. 
We do not need the difficult vanishing theorem 
due to Elkik and Fujita (see \cite[Theorem 1-3-1]{kmm}) 
to obtain the above theorem. 
In Theorem \ref{49-dlt}, if we assume that $(X, D)$ is 
only weak log-terminal, then we can not 
necessarily make $\Exc(f)$ and $\Exc(f)\cup \Supp f^{-1}_*D$ 
{\em{simple}} normal crossing divisors. 
We can only make them {\em{normal crossing divisors}}. 
However, \cite[Theorem 1-2-5]{kmm} and Theorem \ref{lc} 
work in this setting. 
Thus, the proof of Theorem \ref{49-dlt} works for weak log-terminal. 
Anyway, the notion of weak log-terminal singularities 
is not useful in the recent log minimal model program. 
So, we do not discuss weak log-terminal 
singularities here. 
\end{say}

\begin{rem}
The proofs of Theorem \ref{cm-th} and 
Theorem \ref{z-rational} also work 
for weak log-terminal 
pairs once we adopt suitable 
vanishing theorems such as 
Theorem \ref{lc} and Theorem \ref{62}. 
\end{rem}

The following theorem generalizes \cite[17.5 Corollary]{FA}, 
where it was only proved that $S$ is semi-normal and 
satisfies Serre's $S_2$ condition. 
We use Lemma \ref{re-vani-lem} in the proof. 

\begin{thm}\label{cm-th} 
Let $X$ be a normal variety and 
$S+B$ a boundary $\mathbb R$-divisor such that 
$(X, S+B)$ is dlt, $S$ is reduced, and $\llcorner B\lrcorner=0$. 
Let $S=S_1+\cdots +S_k$ be the irreducible 
decomposition and $T=S_1+\cdots +S_l$ for $1\leq l \leq k$. 
Then $T$ is semi-normal, Cohen--Macaulay, and has only 
Du Bois singularities. 
\end{thm}
\begin{proof}
Let $f:Y\to X$ be a resolution such that 
$K_Y+S'+B'=f^*(K_X+S+B)+E$ with the 
following properties:~
(i) $S'$ (resp.~$B'$) is 
the strict transform of $S$ (resp.~$B$), 
(ii) $\Supp (S'+B')\cup \Exc (f)$ and $\Exc(f)$ are 
simple normal crossing divisors on $Y$, 
(iii) $f$ is an isomorphism over the generic point of any lc center of 
$(X, S+B)$, and (iv) $\ulcorner E\urcorner \geq 0$. 
We write $S=T+U$. Let $T'$ (resp.~$U'$) 
be the strict transform of $T$ (resp.~$U$) 
on $Y$. We consider the following 
short exact sequence 
$$0\to \mathcal O_Y(-T'+\ulcorner E\urcorner)\to 
\mathcal O_Y(\ulcorner E\urcorner)\to \mathcal O_{T'}(\ulcorner 
E|_{T'}\urcorner)\to 0. $$ 
Since 
$-T'+E\sim _{\mathbb R, f}K_Y+U'+B'$ and 
$E\sim _{\mathbb R, f}K_Y+S'+B'$, 
we have $-T'+\ulcorner E\urcorner \sim _{\mathbb R, f}K_Y
+U'+B'+\{-E\}$ and 
$\ulcorner E\urcorner \sim _{\mathbb R, f} K_Y+S'+B'+\{-E\}$. 
By the vanishing theorem, 
$R^if_*\mathcal O_Y(-T'+\ulcorner E\urcorner)=R^if_*\mathcal O_Y(\ulcorner 
E\urcorner)=0$ for any $i>0$. 
Note that we used the vanishing lemma of 
Reid--Fukuda type (see Lemma \ref{vani-rf-le}). 
Therefore, we have $$0\to 
f_*\mathcal O_Y(-T'+\ulcorner E\urcorner)\to 
\mathcal O_X\to f_*\mathcal O_{T'}(\ulcorner 
E|_{T'}\urcorner)\to 0$$ and 
$R^if_*\mathcal O_{T'}(\ulcorner E|_{T'}\urcorner)=0$ 
for all $i>0$. Note that $\ulcorner E\urcorner $ is effective 
and $f$-exceptional. 
Thus, $\mathcal O_T\simeq 
f_*\mathcal O_{T'}\simeq f_*\mathcal O_{T'} 
(\ulcorner E|_{T'}\urcorner)$. 
Since $T'$ is a simple normal crossing divisor, 
$T$ is semi-normal. 
By the above vanishing result, 
we obtain $Rf_*\mathcal O_{T'}(\ulcorner 
E|_{T'}\urcorner)\simeq 
\mathcal O_T$ in the derived category. 
Therefore, the composition 
$\mathcal O_T\to R f_*\mathcal O_{T'}\to 
Rf_*\mathcal O_{T'}(\ulcorner E|_{T'}\urcorner)\simeq 
\mathcal O_T$ is 
a quasi-isomorphism. 
Apply $R\mathcal Hom_T (\underline{\ \ \ } , \omega^{\bullet}_T)$ to the quasi-isomorphism 
$\mathcal O_T\to Rf_*\mathcal O_{T'}\to \mathcal O_T$. 
Then the composition 
$\omega^{\bullet}_T\to R f_*\omega^{\bullet}_{T'} 
\to \omega^{\bullet}_T$ is a quasi-isomorphism 
by the Grothendieck duality. 
By the vanishing theorem (see, for 
example, Lemma \ref{re-vani-lem}), 
$R^if_*\omega_{T'}=0$ for 
$i>0$. Hence, 
$h^i(\omega^{\bullet}_T)\subseteq R^if_*\omega^{\bullet}_{T'} 
\simeq R^{i+d}f_*\omega_{T'}$, where 
$d=\dim T=\dim T'$. Therefore, $h^i(\omega^{\bullet}_T)=0$ for $i\ne-d$. 
Thus, $T$ is Cohen--Macaulay. 
This argument is the same as the proof 
of Theorem 1 in \cite{kovacs2}. Since 
$T'$ is a simple normal crossing divisor, $T'$ has only Du Bois 
singularities. 
The quasi-isomorphism $\mathcal O_T\to 
Rf_*\mathcal O_{T'}\to \mathcal O_T$ implies that $T$ has 
only Du Bois 
singularities (cf.~\cite[Corollary 2.4]{kovacs1}). 
Since the composition $\omega_T\to f_*\omega_{T'}\to \omega_T$ 
is an isomorphism, we obtain 
$f_*\omega_{T'}\simeq \omega_T$. 
By the Grothendieck duality, 
$Rf_*\mathcal O_{T'}\simeq 
R\mathcal Hom _T(Rf_*\omega^{\bullet}_{T'}, 
\omega^{\bullet}_{T})\simeq 
R\mathcal Hom _T(\omega^{\bullet}_T, 
\omega^{\bullet}_T)\simeq \mathcal O_T$. 
So, $R^if_*\mathcal O_{T'}=0$ for all $i>0$. 
\end{proof}

We obtained the following vanishing theorem 
in the proof of Theorem \ref{cm-th}. 

\begin{cor}\label{35} 
Under the notation in the proof of {\em{Theorem 
\ref{cm-th}}}, 
$R^if_*\mathcal O_{T'}=0$ for 
any $i>0$ and $f_*\mathcal O_{T'}\simeq 
\mathcal O_T$.  
\end{cor}

We close this section with a non-trivial example. 

\begin{ex}[{cf.~\cite[Remark 0-2-11.~(4)]{kmm}}]\label{416ex}
We consider the $\mathbb P^2$-bundle 
$$\pi:V=\mathbb P_{\mathbb P^2}(
\mathcal O_{\mathbb P^2}\oplus \mathcal O_{\mathbb P^2}(1)
\oplus \mathcal O_{\mathbb P^2}(1))\to \mathbb P^2.$$ 
Let $F_1=\mathbb P_{\mathbb P^2}(\mathcal O_{\mathbb P^2}\oplus 
\mathcal O_{\mathbb P^2}(1))$ and 
$F_2=\mathbb P_{\mathbb P^2}(\mathcal O_{\mathbb P^2}\oplus 
\mathcal O_{\mathbb P^2}(1))$ be two 
hypersurfaces of $V$ which correspond to 
projections $\mathcal O_{\mathbb P^2}\oplus \mathcal O_{\mathbb P^2}(1)
\oplus \mathcal O_{\mathbb P^2}(1)\to \mathcal O_{\mathbb P^2}\oplus 
\mathcal O_{\mathbb P^2}(1)$ given by $(x, y, z)\mapsto (x, y)$ and 
$(x, y, z)\mapsto (x, z)$. 
Let $\Phi:V\to W$ be the flipping 
contraction that contracts the negative section of $\pi:V\to \mathbb P^2$, 
that is, the section corresponding to the projection $\mathcal O_{\mathbb P^2}
\oplus \mathcal O_{\mathbb P^2}(1)\oplus \mathcal O_{\mathbb P^2}(1)
\to \mathcal O_{\mathbb P^2}\to 0$. 
Let $C\subset  \mathbb P^2$ be an elliptic curve. 
We put $Y=\pi^{-1}(C)$, $D_1=F_1|_Y$, and 
$D_2=F_2|_Y$. 
Let $f:Y\to X$ be the Stein factorization 
of $\Phi|_Y: Y\to \Phi (Y)$. 
Then the exceptional locus of $f$ is $E=D_1\cap D_2$. 
We note that $Y$ is smooth, $D_1+D_2$ is a simple normal crossing divisor, 
and $E\simeq C$ is an elliptic curve. 
Let $g:Z\to Y$ be the blow-up along $E$. 
Then 
$$
K_Z+D'_1+D'_2+D=g^*(K_Y+D_1+D_2), 
$$ 
where $D'_1$ (resp.~$D'_2$) is the strict transform 
of $D_1$ (resp.~$D_2$) and $D$ is the exceptional 
divisor of $g$. 
Note that $D\simeq C\times \mathbb P^1$. 
Since 
$$
-D+(K_Z+D'_1+D'_2+D)-(K_Z+D'_1+D'_2)=0, 
$$ 
we obtain that 
$R^if_*(g_*\mathcal O_Z(-D+K_Z+D'_1+D'_2+D))=0$ 
for any $i>0$ by Theorem \ref{74} or 
Theorem \ref{nlb-vani-th}. 
We note that $f\circ g$ is an isomorphism outside $D$. 
We consider the following short exact sequence 
$$
0\to \mathcal I_E\to \mathcal O_X\to \mathcal O_E\to 0, 
$$ 
where $\mathcal I_E$ is the defining ideal sheaf of $E$. 
Since $\mathcal I_E=g_*\mathcal O_Z(-D)$, we obtain that 
\begin{eqnarray*}
0\to f_*(\mathcal I_E\otimes \mathcal O_Y(K_Y+D_1+D_2))\to 
f_*\mathcal O_Y(K_Y+D_1+D_2)\\ 
\to f_*\mathcal O_E(K_Y+D_1+D_2)
\to 0 
\end{eqnarray*} 
by $R^1f_*(\mathcal I_E\otimes \mathcal O_Y(K_Y+D_1+D_2))=0$. 
By adjunction, $\mathcal O_E(K_Y+D_1+D_2)\simeq \mathcal O_E$. 
Therefore, $\mathcal O_Y(K_Y+D_1+D_2)$ is $f$-free.  
In particular, $K_Y+D_1+D_2=f^*(K_X+B_1+B_2)$, 
where $B_1=f_*D_1$ and $B_2=f_*D_2$. 
Thus, $-D-(K_Z+D'_1+D'_2)\sim_{f\circ g}0$. So, 
we have 
$R^if_*\mathcal I_E=R^if_*(g_*\mathcal O_Z(-D))=0$ for 
any $i>0$ by Theorem \ref{74} or 
Theorem \ref{nlb-vani-th}. 
This implies that 
$R^if_*\mathcal O_Y\simeq R^if_*\mathcal O_E$ for 
every 
$i>0$. 
Thus, $R^1f_*\mathcal O_Y\simeq \mathbb C(P)$, 
where $P=f(E)$. 
We consider the following spectral sequence 
$$
E^{pq}=H^p(X, R^qf_*\mathcal O_Y\otimes \mathcal O_X(-mA))
\Rightarrow H^{p+q}(Y, \mathcal O_Y(-mA)), 
$$ 
where $A$ is an ample Cartier divisor on $X$ and 
$m$ is any positive integer. 
Since 
$H^1(Y, \mathcal O_Y(-mf^*A))=H^2(Y, \mathcal O_Y(-mf^*A))=0$ by the 
Kawamata--Viehweg vanishing theorem, 
we have 
$$H^0(X, R^1f_*\mathcal O_Y\otimes \mathcal O_X(-mA))\simeq 
H^2(X, \mathcal O_X(-mA)). $$ 
If we assume that $X$ is Cohen--Macaulay, 
then we have $H^2(X, \mathcal O_X(-mA))=0$ 
for $m\gg 0$ by the 
Serre duality and the Serre vanishing theorem. 
On the other hand, $H^0(X, R^1f_*\mathcal O_Y\otimes 
\mathcal O_X(-mA))\simeq \mathbb C(P)$ because 
$R^1f_*\mathcal O_Y\simeq \mathbb C(P)$. 
It is a contradiction. 
Thus, $X$ is not Cohen--Macaulay. In particular, 
$(X, B_1+B_2)$ is lc but not dlt.  
We note that $\Exc (f)=E$ is not a divisor on $Y$. 
See Definition \ref{dlt-def}. 

Let us recall that $\Phi:V\to W$ is a flipping contraction. 
Let $\Phi^+:V^+\to W$ be the flip of $\Phi$. 
We can check that $V^+=\mathbb P_{\mathbb P^1}(\mathcal O_{\mathbb P^1}
\oplus \mathcal O_{\mathbb P^1}(1) 
\oplus \mathcal O_{\mathbb P^1}(1)
\oplus \mathcal O_{\mathbb P^1}(1))$ and 
the flipped curve $E^+\simeq \mathbb P^1$ is the negative 
section of $\pi^+:V^+\to \mathbb P^1$, 
that is, the section corresponding to the projection 
$\mathcal O_{\mathbb P^1}
\oplus \mathcal O_{\mathbb P^1}(1)
\oplus \mathcal O_{\mathbb P^1}(1)
\oplus \mathcal O_{\mathbb P^1}(1)\to 
\mathcal O_{\mathbb P^1}\to 0$. 
Let $Y^+$ be the strict transform of $Y$ on $V^+$. Then 
$Y^+$ is Gorenstein, lc along 
$E^+\subset Y^+$, and 
smooth outside $E^+$. Let $D^+_1$ (resp.~$D^+_2$) 
be the strict transform 
of $D_1$ (resp.~$D_2$) on $Y^+$. 
If we take a Cartier divisor 
$B$ on $Y$ suitably, 
then $(Y, D_1+D_2)\dashrightarrow (Y^+, D^+_1+D^+_2)$ is the 
$B$-flop of $f:Y\to X$. 
We note that $(Y, D_1+D_2)$ is dlt. 
However, $(Y^+, D^+_1+D^+_2)$ is lc but not dlt. 
\end{ex}

\subsection{Appendix:~Rational singularities} 
In this subsection, 
we give a proof to the following well-known theorem 
again (see Theorem \ref{49-dlt}). 

\begin{thm}\label{z-rational} 
Let $(X, D)$ be a dlt pair. 
Then $X$ has only rational singularities. 
\end{thm}

Our proof is a combination of 
the proofs 
in \cite[Theorem 5.22]{km} and \cite[Section 11]{ko-sing}. 
We need no difficult duality theorems.
The arguments here will be used in Section \ref{sec-alex}. 
First, let us recall the definition of 
the rational singularities. 

\begin{defn}[Rational singularities]\index{rational singularity} 
A variety $X$ has {\em{rational singularities}} 
if there is a resolution $f:Y\to X$ such that 
$f_*\mathcal O_Y\simeq \mathcal O_X$ and 
$R^if_*\mathcal O_Y=0$ for all $i>0$. 
\end{defn}

Next, we give a dual form of the Grauert--Riemenschneider 
vanishing theorem. 

\begin{lem}\label{lem-gr}
Let $f:Y\to X$ be a proper 
birational morphism from a smooth 
variety $Y$ to a variety $X$. 
Let $x\in X$ be a closed point. 
We put $F=f^{-1}(x)$. 
Then we have 
$$H^i_F(Y, \mathcal O_Y)=0$$ for 
any $i<n=\dim X$. 
\end{lem}
\begin{proof}
We take a proper birational morphism 
$g:Z\to Y$ from a smooth variety $Z$ such that $f\circ g$ is projective. We consider the following spectral sequence 
$$
E^{pq}_2=H^p_F(Y, R^qg_*\mathcal O_Z)\Rightarrow 
H^{p+q}_E(Z, \mathcal O_Z), 
$$
where $E=g^{-1}(F)=(f\circ g)^{-1}(x)$. 
Since $R^qg_*\mathcal O_Z=0$ for $q>0$ and $g_*\mathcal O_Z\simeq 
\mathcal O_Y$, 
we have $H^p_F(Y, \mathcal O_Y)\simeq 
H^p_E(Z, \mathcal O_Z)$ for any $p$. 
Therefore, we can replace $Y$ with 
$Z$ and assume that $f:Y\to X$ is projective. 
Without loss of generality, we can 
assume that $X$ is affine. 
Then we compactify $X$ and assume that 
$X$ and $Y$ are projective. 
It is well known that 
$$
H^i_F(Y, \mathcal O_Y)\simeq 
\lim_{\underset{m}{\longrightarrow}}
\Ext ^i(\mathcal O_{mF}, \mathcal O_Y)
$$ 
(see \cite[Theorem 2.8]{hartshorne-local}) and that 
$$
\Hom (\Ext ^i(\mathcal O_{mF}, \mathcal O_Y), \mathbb C)
\simeq H^{n-i}(Y, \mathcal O_{mF}\otimes \omega_Y)  
$$ 
by duality on a smooth projective variety $Y$ 
(see \cite[Theorem 7.6 (a)]{hartshorne-ag}). 
Therefore, 
\begin{align*}
\Hom (H^i_F(Y, \mathcal O_Y), \mathbb C) 
&\simeq 
\Hom (\lim _{\underset{m}{\longrightarrow}}
\Ext ^i(\mathcal O_{mF}, 
\mathcal O_Y), \mathbb C)\\ &\simeq 
\lim _{\underset{m}{\longleftarrow}} H^{n-i}(Y, \mathcal O_{mF}\otimes 
\omega_Y)\\
&\simeq (R^{n-i}f_*\omega_Y)_x^{\wedge} 
\end{align*}
by the theorem on formal functions 
(see \cite[Theorem 11.1]{hartshorne-ag}), where $(R^{n-i}f_*\omega_Y)_x^{\wedge}$ is the completion 
of $R^{n-i}f_*\omega_Y$ at $x \in X$. 
On the other hand, $R^{n-i}f_*\omega_Y=0$ for 
$i<n$ by the Grauert--Riemenschneider 
vanishing theorem. Thus, $H^i_F(Y, \mathcal O_Y)=0$ for 
$i<n$. 
\end{proof}

\begin{rem}\label{lem-gr-ho} 
Lemma \ref{lem-gr} 
holds true even when 
$Y$ has rational 
singularities. 
It is because $R^qg_*\mathcal O_Z=0$ for 
$q>0$ and $g_*\mathcal O_Z\simeq \mathcal O_Y$ holds 
in the proof of 
Lemma \ref{lem-gr}. 
\end{rem}

Let us go to the proof of Theorem \ref{z-rational}. 

\begin{proof}[Proof of {\em{Theorem \ref{z-rational}}}]
Without loss of generality, 
we can assume that $X$ is affine. 
Moreover, by taking generic hyperplane sections of $X$, 
we can also assume that $X$ has only rational 
singularities outside a closed point $x\in X$. By the definition of dlt, 
we can take a resolution $f:Y\to X$ such that 
$\Exc (f)$ and $\Exc (f)\cup \Supp f^{-1}_*D$ are both 
simple normal crossing divisors on $Y$, 
$K_Y+f^{-1}_*D=f^*(K_X+D)+E$ with 
$\ulcorner E\urcorner\geq 0$, and that 
$f$ is projective. 
Moreover, we can make $f$ an isomorphism 
over the generic point of any lc center of $(X, D)$. 
Therefore, by Lemma \ref{vani-rf-le}, we can check 
that $R^if_*\mathcal O_Y(\ulcorner E\urcorner)=0$ for 
any $i>0$. 
See also the proof of Theorem \ref{49-dlt}. 
We note that $f_*\mathcal O_Y(\ulcorner E\urcorner)\simeq 
\mathcal O_X$ since $\ulcorner E\urcorner$ is effective 
and $f$-exceptional. 
For any $i>0$, 
by the above assumption, $R^if_*\mathcal O_Y$ is 
supported at a point $x\in X$ if it ever 
has a non-empty support at all. 
We put $F=f^{-1}(x)$. Then 
we have a spectral 
sequence 
$$
E^{ij}_2=H^i_x(X, R^jf_*\mathcal O_Y(\ulcorner E\urcorner))\Rightarrow 
H^{i+j}_F(Y, \mathcal O_Y(\ulcorner E\urcorner)). 
$$
By the above vanishing result, we have 
$$
H^i_x(X, \mathcal O_X)\simeq H^i_F(Y, \mathcal O_Y
(\ulcorner E\urcorner))
$$ for 
every $i\geq 0$. 
We obtain a commutative diagram 
$$
\begin{CD}
H^i_F(Y, \mathcal O_Y)@>>> H^i_F(Y, 
\mathcal O_Y(\ulcorner E\urcorner))\\ 
@A{\alpha}AA @AA{\beta}A\\ 
H^i_x(X, \mathcal O_X)@=H^i_x(X, \mathcal O_X). 
\end{CD}
$$
We have already checked that $\beta$ is an isomorphism 
for every $i$ and 
that $H^i_F(Y, \mathcal O_Y)=0$ for $i<n$ (see Lemma 
\ref{lem-gr}). 
Therefore, $H^i_x(X, \mathcal O_X)=0$ for 
any $i<n=\dim X$. Thus, $X$ is Cohen--Macaulay. 
For $i=n$, we obtain 
that 
$$
\alpha: H^n_x(X, \mathcal O_X)\to H^n_F(Y, \mathcal O_Y)   
$$ 
is injective. 
We consider the following spectral sequence 
$$
E^{ij}_2=H^i_x(X, R^jf_*\mathcal O_Y)\Rightarrow 
H^{i+j}_F(Y, \mathcal O_Y). 
$$
We note that $H^i_x(X, R^jf_*\mathcal O_Y)=0$ for any 
$i>0$ and 
$j>0$ since $\Supp R^jf_*\mathcal O_Y\subset \{x\}$ for 
$j>0$. 
On the other hand, $E^{i0}_2=H^i_x(X, \mathcal O_X)=0$ for 
any $i<n$. 
Therefore, $H^0_x(X, R^jf_*\mathcal O_Y)\simeq 
H^j_x(X, \mathcal O_X)=0$ for 
all $j\leq n-2$. Thus, $R^jf_*\mathcal O_Y=0$ for $1
\leq j\leq n-2$. Since 
$H^{n-1}_x(X, \mathcal O_X)=0$, 
we obtain that 
$$
0\to H^0_x(X, R^{n-1}f_*\mathcal O_Y)\to 
H^n_x(X, \mathcal O_X)\overset{\alpha}{\to}
H^n_F(Y, \mathcal O_Y)\to 0   
$$ 
is exact. 
We have already checked that $\alpha$ is injective. 
So, we obtain that 
$H^0_x(X, R^{n-1}f_*\mathcal O_Y)=0$. 
This means that $R^{n-1}f_*\mathcal O_Y=0$. Thus, 
we have $R^if_*\mathcal O_Y=0$ for 
any $i>0$. 
We complete the proof. 
\end{proof}

\section{Alexeev's criterion for $S_3$ condition}\label{sec-alex}

In this section, we explain Alexeev's 
criterion for Serre's $S_3$ condition (see Theorem \ref{alex-cri}). 
It is a clever application of Theorem \ref{8} (i). 
In general, log canonical singularities are not Cohen--Macaulay. 
So, the results in this section will be useful for the 
study of lc pairs. 

\begin{thm}[{cf.~\cite[Lemma 3.2]{alex}}]\label{alex-cri}
Let $(X, B)$ be an lc pair with $\dim X=n \geq 3$ and 
let $P\in X$ be a scheme theoretic 
point such that 
$\dim \overline {\{P\}}\leq n-3$. 
Assume that 
$\overline {\{P\}}$ is not an lc center of $(X, B)$. 
Then the local ring $\mathcal O_{X, P}$ satisfies Serre's 
$S_3$ condition. 
\end{thm}

We slightly changed the original formulation. 
The following proof is essentially the same as 
Alexeev's. We use local cohomologies to 
calculate depths. 

\begin{proof}
We note that $\mathcal O_{X, P}$ satisfies Serre's $S_2$ 
condition because $X$ is normal. 
Since the assertion is local, we can assume that 
$X$ is affine. 
Let $f: Y \to X$ be a resolution of $X$ such that 
$\Exc (f)\cup \Supp f^{-1}_*B$ is a simple normal crossing 
divisor on $Y$. Then we can write 
$$
K_Y+B_Y=f^*(K_X+B)   
$$ 
such that $\Supp B_Y$ is a simple normal crossing divisor on $Y$. 
We put $A=\ulcorner -(B^{<1}_Y)\urcorner\geq 0$. 
Then we obtain 
$$
A=K_Y+B^{=1}_Y+\{B_Y\}-f^*(K_X+B). 
$$ 
Therefore, by Theorem \ref{8} (i), the support 
of every non-zero local section of the sheaf 
$R^1f_*\mathcal O_Y(A)$ contains 
some lc centers of $(X, B)$. 
Thus, $P$ is not an associated point of $R^1f_*\mathcal O_Y(A)$. 

We put $X_P=\Spec \mathcal O_{X, P}$ and 
$Y_P=Y\times _X X_P$. Then 
$P$ is a closed point of $X_P$ and it is 
sufficient to prove that $H^2_P(X_P, \mathcal O_{X_P})=0$. 
We put $F=f^{-1}(P)$, where $f:Y_P\to X_P$. 
Then we have the following vanishing theorem. 
It is nothing but Lemma \ref{lem-gr} when 
$P$ is a closed point of $X$. 

\begin{lem}[{cf.~Lemma \ref{lem-gr}}]\label{lem-gr2} 
We have $H^i_F(Y_P, \mathcal O_{Y_P})=0$ for 
$i< n-\dim \overline{\{P\}}$. 
\end{lem}
\begin{proof}[Proof of {\em{Lemma \ref{lem-gr2}}}]
Let $I$ denote an injective hull of $\mathcal O_{X_P}/m_P$ 
as an $\mathcal O_{X_P}$-module, 
where $m_P$ is the maximal ideal corresponding 
to $P$. 
We have 
\begin{align*}
R\Gamma_F\mathcal O_{Y_P}&\simeq 
R\Gamma_P(Rf_*\mathcal O_{Y_P})\\
&\simeq\Hom (R\mathcal Hom (Rf_*\mathcal O_{Y_P}, 
\omega^{\bullet}_{X_P}), I)\\
&\simeq \Hom (Rf_*\mathcal O_Y(K_Y)\otimes \mathcal O_{X_P}
[n-m], I), 
\end{align*} 
where $m=\dim \overline {\{P\}}$, 
by the local duality theorem (\cite[Chapter V, 
Theorem 6.2]{residue}) and 
the Grothendieck duality theorem 
(\cite[Chapter VII, Theorem 3.3]{residue}). 
We note the shift that normalize the dualizing 
complex $\omega^{\bullet}_{X_P}$. Therefore, 
we obtain $H^i_F(Y_P, \mathcal O_{Y_P})=0$ 
for $i<n-m$ because 
$R^jf_*\mathcal O_Y(K_Y)=0$ for 
any $j>0$ by the 
Grauert--Riemenschneider vanishing theorem. 
\end{proof}
Let  us go back to the proof of the theorem. 
We consider the following 
spectral sequences 
$$
E^{pq}_2=H^p_P(X_P, R^qf_*\mathcal O_{Y_P}(A))
\Rightarrow 
H^{p+q}_F(Y_P, \mathcal O_{Y_P}(A)),  
$$ 
and 
$$
{}^{\prime}\!E^{pq}_2=H^p_P(X_P, R^qf_*\mathcal O_{Y_P})
\Rightarrow 
H^{p+q}_F(Y_P, \mathcal O_{Y_P}).   
$$ 
By the above spectral sequences, we have the next commutative diagram. 
$$
\xymatrix{
H^2_F(Y_P, \mathcal O_{Y_P}) \rto
&  H^2_F(Y_P, \mathcal O_{Y_P}(A)) \\
  H^2_P(X_P, f_*\mathcal O_{Y_P}) \uto \ar@{=}[d]
&  H^2_P(X_P, f_*\mathcal 
O_{Y_P}(A)) \uto_{\phi} \ar@{=}[d] \\
  H^2_P\big(X_P, \mathcal O_{X_P})  \ar@{=}[r]
&  H^2_P\big(X_P, \mathcal O_{X_P}) 
}
$$
Since $P$ is not an associated 
point of $R^1f_*\mathcal O_Y(A)$, 
we have $$E^{0,1}_2=H^0_P(X_P, R^1f_*\mathcal O_{Y_P}
(A))=0. $$ 
By the edge sequence 
$$
0\to E^{1,0}_2\to E^1\to E^{0, 1}_2\to 
E^{2, 0}_2\overset{\phi}\to E^2\to \cdots, 
$$
we know that $\phi:E^{2,0}_2\to E^2$ is injective. 
Therefore, $H^2_P(X_P, \mathcal O_{X_P})\to 
H^2_F(Y_P, \mathcal O_{Y_P})$ is injective 
by the above big commutative diagram. 
Thus, we obtain $H^2_P(X_P, \mathcal O_{X_P})=0$ since 
$H^2_F(Y_P, \mathcal O_{Y_P})=0$ by Lemma \ref{lem-gr2}. 
\end{proof}

\begin{rem}\label{comp-prob}
The original argument 
in the proof of \cite[Lemma 3.2]{alex} 
has some compactification problems when $X$ is not projective. 
Our proof does not need any compactifications of 
$X$. 
\end{rem}

As an easy application of Theorem \ref{alex-cri}, 
we have the following result. 
It is \cite[Theorem 3.4]{alex}. 

\begin{thm}[{cf.~\cite[Theorem 3.4]{alex}}]\label{alex-cri2}  
Let $(X, B)$ be an lc pair and let $D$ be an effective 
Cartier divisor. 
Assume that 
the pair $(X, B+\varepsilon D)$ 
is lc for some $\varepsilon >0$. 
Then $D$ is $S_2$. 
\end{thm}
\begin{proof}
Without loss of generality, we can assume that 
$\dim X=n\geq 3$. Let $P\in D\subset X$ 
be a scheme theoretic point such that 
$\dim \overline{\{P\}}\leq n-3$. 
We localize $X$ at $P$ and assume that 
$X=\Spec \mathcal O_{X, P}$. 
By the assumption, 
$\overline {\{P\}}$ is not an lc center of $(X, B)$. 
By Theorem \ref{alex-cri}, 
we obtain that 
$H^i_P(X, \mathcal O_X)=0$ for 
$i<3$. 
Therefore, 
$H^i_P(D, \mathcal O_D)=0$ for 
$i<2$ by the long exact sequence 
$$
\cdots \to H^i_P(X, \mathcal O_X(-D))\to 
H^i_P(X, \mathcal O_X)\to H^i_P(D, \mathcal O_D)\to \cdots. 
$$ 
We note that $H^i_P(X, \mathcal O_X(-D))\simeq 
H^i_P(X, \mathcal O_X)=0$ for $i<3$. 
Thus, $D$ satisfies Serre's $S_2$ condition. 
\end{proof}

We give a supplement to 
adjunction (see Theorem \ref{adj-th} (i)). 
It may be useful for the study of 
limits of stable pairs (see \cite{alex}). 

\begin{thm}[Adjunction for Cartier divisors on lc pairs]\label{tu-thm1}
Let $(X, B)$ be an lc pair and let 
$D$ be an effective 
Cartier divisor on $X$ 
such that $(X, B+D)$ is log canonical. 
Let $V$ be a union of lc centers of $(X, B)$. 
We consider $V$ as a reduced closed subscheme of 
$X$. We define a scheme structure on $V\cap D$ by 
the following 
short exact sequence
$$
0\to \mathcal O_V(-D)\to \mathcal O_V\to \mathcal O_{V\cap D}\to 0. 
$$
Then, 
$\mathcal O_{V\cap D}$ is reduced and 
semi-normal. 
\end{thm}

\begin{proof}
First, we note that $V\cap D$ is a union of lc centers 
of $(X, B+D)$ (see Theorem \ref{a1}). 
Let $f:Y\to X$ be a resolution 
such that 
$\Exc (f) \cup \Supp f^{-1}_*(B+D)$ is a simple normal 
crossing 
divisor on $Y$. We can write 
$$
K_Y+B_Y=f^*(K_X+B+D)  
$$ 
such that $\Supp B_Y$ is a simple normal crossing 
divisor on $Y$. 
We take more blow-ups 
and can assume that 
$f^{-1}(V\cap D)$ and $f^{-1}(V)$ are simple normal crossing 
divisors. Then 
the union of all strata 
of $B^{=1}_Y$ mapped to 
$V\cap D$ (resp.~$V$), which is denoted by $W$ (resp.~$U+W$), 
is a divisor on $Y$. We put $A=\ulcorner 
-(B^{<1}_Y)\urcorner\geq 0$ and consider the following 
commutative diagram. 
$$
\xymatrix{
0 \rto &  \mathcal O_Y(A-U-W)\dto \rto &\mathcal O_Y(A) \ar@{=}[d]
\rto &\mathcal O_{U+W}(A)\dto\rto &0  \\
0 \rto &  \mathcal O_Y(A-W)\rto &\mathcal O_Y(A) 
\rto &\mathcal O_{W}(A)\rto &0
}
$$
By applying $f_*$, we obtain the next big 
diagram by 
Theorem \ref{8} (i) and 
Theorem \ref{adj-th} (i). 
$$
\xymatrix{
&&& 0\dto &\\
&0\dto & & f_*\mathcal O_U(A-W)\dto&\\
0 \rto &  f_*\mathcal O_Y(A-U-W)\dto \rto &\mathcal O_X \ar@{=}[d]
\rto &\mathcal O_V\dto\rto &0  \\
0 \rto &  f_*\mathcal O_Y(A-W)\dto\rto &\mathcal O_X 
\rto &\mathcal O_{V\cap D}\dto \rto &0\\ 
&f_*\mathcal O_U(A-W)\dto &&0&\\
&0&&&
}
$$ 
A key point is that the connecting homomorphism 
$$f_*\mathcal O_U(A-W)\to R^1f_*\mathcal O_Y(A-U-W)$$ is 
a zero map by Theorem \ref{8} (i). 
We note that $\mathcal O_V$ and $\mathcal O_{V\cap D}$ in 
the above diagram are the structure sheaves 
of qlc pairs $V$ and $V\cap D$ induced by 
$(X, B+D)$. In particular, 
$\mathcal O_V\simeq f_*\mathcal O_{U+W}$ and 
$\mathcal O_{V\cap D}\simeq f_*\mathcal O_W$.  
So, $\mathcal O_V$ and $\mathcal O_{V\cap D}$ are 
reduced and semi-normal 
since $W$ and $U+W$ are simple normal crossing 
divisors on $Y$. 

Therefore, to prove this theorem, 
it is sufficient to see that 
$f_*\mathcal O_U(A-W)\simeq \mathcal O_V(-D)$. 
We can write 
$$
A=K_Y+B^{=1}_Y+\{B_Y\}-f^*(K_X+B+D) 
$$ 
and 
$$
f^*D=W+E+f^{-1}_*D, 
$$
where $E$ is an effective $f$-exceptional 
divisor. 
We note that $f^{-1}_*D\cap U=\emptyset$. 
Since $A-W=A-f^*D+E+f^{-1}_*D$, 
it is enough to 
see that 
$f_*\mathcal O_U(A+E+f^{-1}_*D)\simeq 
f_*\mathcal O_U(A+E)\simeq \mathcal O_V$. 
We consider the 
following short exact sequence 
$$
0\to \mathcal O_Y(A+E-U)\to \mathcal O_Y(A+E)\to \mathcal 
O_U(A+E)\to 0. 
$$ 
Note that 
$$
A+E-U=K_Y+B^{=1}_Y-f^{-1}_*D-U-W+\{B_Y\}-f^*(K_X+B). 
$$ 
Thus, the connecting homomorphism 
$f_*\mathcal O_U(A+E)\to R^1f_*\mathcal O_Y(A+E-U)$ 
is a zero map by Theorem \ref{8} (i). 
Therefore, 
we obtain that 
$$
0\to f_*\mathcal O_Y(A+E-U)\to \mathcal O_X\to f_*\mathcal 
O_U(A+E)\to 0. 
$$
So, we have $f_*\mathcal O_U(A+E)\simeq 
\mathcal O_V$. 
We finish the proof of this theorem. 
\end{proof}

The next corollary is one of the main results in 
\cite{alex}. 
The original proof in \cite{alex} 
depends on the $S_2$-fication. 
Our proof uses adjunction (see Theorem \ref{tu-thm1}). 
As a result, 
we obtain the semi-normality of $\llcorner B\lrcorner \cap D$. 

\begin{cor}[{cf.~\cite[Theorem 4.1]{alex}}]
Let $(X, B)$ be an lc pair and let $D$ be an effective Cartier 
divisor 
on $D$ such that 
$(X, B+D)$ is lc. 
Then $D$ is $S_2$ and the scheme 
$\llcorner B\lrcorner \cap D$ is reduced and 
semi-normal. 
\end{cor}
\begin{proof}
By Theorem \ref{alex-cri2}, 
$D$ satisfies Serre's $S_2$ condition. 
By Theorem \ref{tu-thm1}, 
$\llcorner B\lrcorner \cap D$ is reduced 
and semi-normal. 
\end{proof}

The following proposition may be useful. 
So, we contain it here. 
It is \cite[Lemma 3.1]{alex} with slight modifications 
as Theorem \ref{alex-cri}. 

\begin{prop}[{cf.~\cite[Lemma 3.1]{alex}}]\label{w-alex}
Let $X$ be a normal variety with $\dim X=n\geq 3$ and 
let $f:Y\to X$ be a resolution of 
singularities. Let $P\in 
X$ be a scheme theoretic point such 
that $\dim \overline {\{P\}}\leq n-3$. 
Then the local ring 
$\mathcal O_{X, P}$ is $S_3$ if and 
only if $P$ is not an associated point of 
$R^1f_*\mathcal O_Y$. 
\end{prop}
\begin{proof}
We put $X_P=\Spec \mathcal O_{X, P}$, 
$Y_P=Y\times _X X_P$, 
and $F=f^{-1}(P)$, where $f:Y_P\to X_P$. 
We consider the following 
spectral sequence 
$$
E^{ij}_2=H^i_P(X, R^jf_*\mathcal O_{Y_P})\Rightarrow 
H^{i+j}_F(Y_P, \mathcal O_{Y_P}).  
$$ 
Since 
$H^1_F(Y_P, \mathcal O_{Y_P})=H^2_F(Y_P, 
\mathcal O_{Y_P})=0$ by 
Lemma \ref{lem-gr2}, 
we have an isomorphism 
$H^0_P(X_P, R^1f_*\mathcal O_{Y_P})
\simeq H^2_P(X_P, \mathcal O_{X_P})$. 
Therefore, the depth of $\mathcal O_{X, P}$ is $\geq 3$ if 
and only if $H^2_P(X_P, \mathcal O_{X_P})=
H^0_P(X_P, R^1f_*\mathcal O_{Y_P})=0$. 
It is equivalent to 
the condition that $P$ is not an associated 
point of $R^1f_*\mathcal O_Y$. 
\end{proof}

\begin{say}[Supplements] Here, we give 
a slight generalization of 
\cite[Theorem 3.5]{alex}. 
We can prove it by a similar method to the proof of Theorem 
\ref{alex-cri}. 

\begin{thm}[{cf.~\cite[Theorem 3.5]{alex}}]\label{ale-18} 
Let $(X, B)$ be an lc pair and 
$D$ an effective Cartier divisor on $X$ such that 
$(X, B+\varepsilon D)$ is lc for some 
$\varepsilon >0$. Let $V$ be a union of 
some lc centers of $(X, B)$. 
We consider $V$ as a reduced closed subscheme 
of $X$. 
We can define 
a scheme structure on $V\cap D$ by 
the following exact sequence 
$$
0\to \mathcal O_V(-D)\to \mathcal O_V\to \mathcal O_{V\cap D}
\to 0. 
$$
Then the scheme $V\cap D$ satisfies 
Serre's $S_1$ condition. 
In particular, $\llcorner B\lrcorner \cap D$ has no embedded 
point. 
\end{thm}

\begin{proof}
Without loss of generality, 
we can assume that 
$X$ is affine. 
We take a resolution $f:Y\to X$ such that 
$\Exc (f)\cup \Supp f^{-1}_*B$ is a simple normal crossing 
divisor on $Y$. 
Then we can write 
$$
K_Y+B_Y=f^*(K_X+B)
$$
such that $\Supp B_Y$ is a simple normal crossing 
divisor on $Y$. 
We take more blow-ups and can 
assume that 
the union of all strata of $B^{=1}_Y$ mapped to 
$V$, which is denoted by $W$, 
is a divisor on $Y$. 
Moreover, for any lc center $C$ of $(X, B)$ contained in $V$, 
we can assume that $f^{-1}(C)$ is a divisor on $Y$. 
We consider the following short exact sequence 
$$
0\to \mathcal O_Y(A-W)\to \mathcal O_Y(A)\to \mathcal O_W(A)\to 0, 
$$ 
where $A=\ulcorner -(B^{<1}_Y)\urcorner\geq 0$. 
By taking higher direct images, 
we obtain 
$$
0\to f_*\mathcal O_Y(A-W)\to \mathcal O_X\to 
f_*\mathcal O_W(A)\to R^1f_*\mathcal O_Y(A-W)\to \cdots. 
$$ 
By Theorem \ref{8} (i), we have that 
$f_*\mathcal O_W(A)\to R^1f_*\mathcal O_Y(A-W)$ is 
a zero map, $f_*\mathcal O_W(A)\simeq \mathcal O_V$, and 
$f_*\mathcal O_Y(A-W)\simeq 
\mathcal I_V$, the defining ideal sheaf of $V$ on $X$. 
We note that 
$f_*\mathcal O_W\simeq \mathcal O_V$. 
In particular, $\mathcal O_V$ is reduced and semi-normal. 
For the 
details, see Theorem \ref{adj-th} (i). 

Let $P\in V\cap D$ be a scheme theoretic point such that 
the height of $P$ in $\mathcal O_{V\cap D}$ 
is $\geq 1$. 
We can assume that $\dim V\geq 2$ 
around $P$. Otherwise, the theorem is trivial. 
We put $V_P=\Spec \mathcal O_{V, P}$, 
$W_P=W\times _V V_P$, and $F=f^{-1}(P)$, 
where $f:W_P\to V_P$. We denote 
the pull back of $D$ on $V_P$ by $D$ for simplicity. 
To check this theorem, it is sufficient to 
see that 
$H^0_P(V_P\cap D, \mathcal O_{V_P\cap D})=0$. 
First, we note that $H^0_P(V_P, \mathcal O_{V_P})=
H^0_F(W_P, \mathcal O_{W_P})=0$ by $f_*\mathcal O_W\simeq 
\mathcal O_V$. Next, as in the proof of Lemma \ref{lem-gr2}, 
we have 
\begin{align*}
R\Gamma_F\mathcal O_{W_P}&\simeq 
R\Gamma_P(Rf_*\mathcal O_{W_P})\\
&\simeq\Hom (R\mathcal Hom (Rf_*\mathcal O_{W_P}, 
\omega^{\bullet}_{V_P}), I)\\
&\simeq \Hom (Rf_*\mathcal O_W(K_W)\otimes \mathcal O_{V_P}
[n-1-m], I), 
\end{align*} 
where $n=\dim X$, 
$m=\dim \overline {\{P\}}$, and 
$I$ is an injective hull of $\mathcal O_{V_P}/m_P$ 
as an $\mathcal O_{V_P}$-module such that 
$m_P$ is the maximal ideal corresponding 
to $P$. 
Once we obtain $R^{n-m-2}f_*\mathcal O_W(K_W)\otimes \mathcal O_{V_P}=0$, then 
$H^1_F(W_P, \mathcal O_{W_P})=0$. 
It implies that $H^1_P(V_P, \mathcal O_{V_P})=0$ since 
$H^1_P(V_P, \mathcal O_{V_P})
\subset H^1_F(W_P, \mathcal O_{W_P})$. 
By the long exact sequence 
\begin{align*}
\cdots \to H^0_P(V_P, \mathcal O_{V_P})\to 
H^0_P(V_P\cap D, \mathcal O_{V_P\cap D}) \\ \to 
H^1_P(V_P, \mathcal O_{V_P}(-D))\to \cdots, 
\end{align*}
we obtain $H^0_P(V_P\cap D, \mathcal O_{V_P\cap D})=0$. 
It is because $H^0_P(V_P, \mathcal O_{V_P})=0$ and 
$H^1_P(V_P, \mathcal O_{V_P}(-D))
\simeq H^1_P(V_P, \mathcal O_{V_P})=0$. 
So, it is sufficient to 
see that 
$R^{n-m-2}f_*\mathcal O_W(K_W)\otimes \mathcal O_{V_P}=0$. 

To check the vanishing of 
$R^{n-m-2}f_*\mathcal O_W(K_W)\otimes \mathcal O_{V_P}$, 
by taking general hyperplane cuts $m$ times, 
we can assume that $m=0$ and $P\in X$ is a closed point. 
We note that the dimension of any irreducible component 
of $V$ passing through $P$ is $\geq 2$ since 
$P$ is not an lc center of $(X, B)$ (see Theorem \ref{a1}). 

On the other hand, we can write $W=U_1+U_2$ such that 
$U_2$ is the union of all the irreducible components 
of $W$ whose images by $f$ have dimensions $\geq 2$ 
and 
$U_1=W-U_2$. 
We note that the dimension of the image of any stratum of 
$U_2$ by $f$ is $\geq 2$ 
by the construction of $f:Y\to X$. 
We consider the following exact sequence 
\begin{multline*}
\cdots \to R^{n-2}f_*\mathcal O_{U_2}(K_{U_2})
\to R^{n-2}f_*\mathcal O_W(K_W)\\ \to  
R^{n-2}f_*\mathcal O_{U_1}(K_{U_1}+U_2|_{U_1})\to 
R^{n-1}f_*\mathcal O_{U_2}(K_{U_2})\to 
\cdots. 
\end{multline*}
We have $R^{n-2}f_*\mathcal O_{U_2}(K_{U_2})=R^{n-1}
f_*\mathcal O_{U_2}(K_{U_2})=0$ around $P$ 
since the dimension of general fibers of $f:U_2\to f(U_2)$ is 
$\leq n-3$.  
Thus, we obtain $R^{n-2}f_*\mathcal O_W(K_W)\simeq 
R^{n-2}f_*\mathcal O_{U_1}(K_{U_1}+U_2|_{U_1})$ 
around $P$. 
Therefore, the support of 
$R^{n-2}f_*\mathcal O_W(K_W)$ around $P$ is contained 
in one-dimensional 
lc centers of $(X, B)$ in $V$ and 
$R^{n-2}f_*\mathcal O_W(K_W)$ has no zero-dimensional 
associated point around $P$ by Theorem \ref{8} (i). 
By taking a general hyperplane cut of $X$ again, we 
have the vanishing of $R^{n-2}f_*\mathcal O_W(K_W)$ around $P$ 
by Lemma \ref{lem-ele-vani} below. 
So, we finish the proof. 
\end{proof}

We used the following lemma in the proof of Theorem \ref{ale-18}. 

\begin{lem}\label{lem-ele-vani} 
Let $(Z, \Delta)$ be an $n$-dimensional 
lc pair and let $x\in Z$ be a closed point such that 
$x$ is an lc center of $(Z, \Delta)$. 
Let $V$ be a union of some lc centers 
of $(X, B)$ such that 
$\dim V>0$, $x\in V$, and $x$ is not isolated 
in $V$. 
Let $f:Y\to Z$ be a resolution such that 
$f^{-1}(x)$ and $f^{-1}(V)$ are divisors on $Y$ and that 
$\Exc (f)\cup \Supp f^{-1}_*\Delta$ is a simple normal 
crossing divisor on $Y$. We can 
write 
$$
K_Y+B_Y=f^*(K_Z+\Delta) 
$$
such that $\Supp B_Y$ is a simple normal crossing divisor 
on $Y$. Let $W$ be the union of all the irreducible 
components of $B^{=1}_Y$ mapped to 
$V$. Then 
$R^{n-1}f_*\mathcal O_W(K_W)=0$ around $x$. 
\end{lem}
\begin{proof}
We can write $W=W_1+W_2$, where 
$W_2$ is the union of all the irreducible 
components of $W$ mapped to $x$ by $f$ and 
$W_1=W-W_2$. 
We consider the following short exact sequence 
$$
0\to \mathcal O_Y(K_Y)\to \mathcal O_Y(K_Y+W)\to \mathcal 
O_W(K_W)\to 0.  
$$
By the Grauert--Riemenschneider vanishing theorem, 
we obtain that $$R^{n-1}f_*\mathcal O_Y(K_Y+W)\simeq 
R^{n-1}f_*\mathcal O_W(K_W). $$ 
Next, we consider the short exact sequence 
$$
0\to \mathcal O_Y(K_Y+W_1)\to 
\mathcal O_Y(K_Y+W)\to \mathcal O_{W_2}(K_{W_2}+W_1|_{W_2})\to 0. 
$$
Around $x$, the image of any irreducible component of $W_1$ by $f$ 
is positive dimensional. Therefore, 
$R^{n-1}f_*\mathcal O_Y(K_Y+W_1)=0$ near $x$. 
It can be checked by the induction 
on the number of irreducible components 
using the following exact sequence 
\begin{align*}
\cdots \to R^{n-1}f_*\mathcal O_Y(K_Y+W_1-S)
\to R^{n-1}f_*\mathcal O_Y(K_Y+W_1)\\
\to R^{n-1}f_*\mathcal O_S(K_S+(W_1-S)|_S)
\to \cdots,  
\end{align*}
where $S$ is an irreducible component of $W_1$. 
On the other hand, we have 
$$
R^{n-1}f_*\mathcal O_{W_2}(K_{W_2}+W_1|_{W_2})
\simeq H^{n-1}(W_2, \mathcal O_{W_2} 
(K_{W_2}+W_1|_{W_2})) 
$$ 
and 
$H^{n-1}(W_2, \mathcal O_{W_2} 
(K_{W_2}+W_1|_{W_2}))$ 
is dual to 
$H^0(W_2, \mathcal O_{W_2}(-W_1|_{W_2}))$. 
Note that 
$f_*\mathcal O_{W_2}\simeq \mathcal O_x$ and 
$f_*\mathcal O_{W}\simeq \mathcal O_V$ by the usual 
argument on adjunction 
(see Theorem \ref{adj-th} (i)). 
Since $W_2$ and $W=W_1+W_2$ are connected over $x$, 
$H^0(W_2, \mathcal O_{W_2}(-W_1|_{W_2}))=0$. 
We note that $W_1|_{W_2}\ne 0$ since $x$ is not isolated 
in $V$. 
This means that $R^{n-1}f_*\mathcal O_W(K_W)=0$ around 
$x$ by the above arguments. 
\end{proof}
\end{say}

\subsection{Appendix:~Cone singularities}\label{431sss} 
In this subsection, we collect some basic facts on cone 
singularities for the reader's convenience.  
First, we give two lemmas 
which can be proved by the same method as in 
Section \ref{sec-alex}. 
We think that these lemmas will be 
useful for the study of 
log canonical singularities. 

\begin{lem}\label{w-lem1}
Let $X$ be an $n$-dimensional normal 
variety and let $f:Y\to X$ be a resolution of 
singularities. 
Assume that $R^if_*\mathcal O_Y=0$ for $1\leq i\leq n-2$. Then 
$X$ is Cohen--Macaulay. 
\end{lem}
\begin{proof}
We can assume that $n\geq 3$. 
Since $\Supp R^{n-1}f_*\mathcal O_Y$ is zero-dimensional, 
we can assume that there exists a closed point $x\in X$ 
such that $X$ has only rational 
singularities outside $x$ by shrinking 
$X$ around $x$. 
Therefore, it is sufficient to see that 
the depth of $\mathcal O_{X, x}$ is $\geq n=\dim X$. 
We consider the following spectral sequence 
$$
E^{ij}_2=H^i_x(X, R^jf_*\mathcal O_Y)\Rightarrow 
H^{i+j}_F(Y, \mathcal O_Y), 
$$ 
where 
$F=f^{-1}(x)$. 
Then $H^i_x(X, \mathcal O_X)=E^{i0}_2\simeq 
E^{i0}_{\infty}=0$ for 
$i\leq n-1$. 
It is because $H^i_F(Y, \mathcal O_Y)=0$ for $i\leq n-1$ by 
Lemma \ref{lem-gr}. 
This means that 
the depth of $\mathcal O_{X,x}$ is $\geq n$. 
So, we have that $X$ is Cohen--Macaulay. 
\end{proof}

\begin{lem}\label{w-lem2}
Let $X$ be an $n$-dimensional normal 
variety and let $f:Y\to X$ 
be a resolution of singularities. 
Let $x\in X$ be a closed 
point. 
Assume that $X$ is Cohen-Macaulay and 
that $X$ has only rational 
singularities outside $x$. Then 
$R^if_*\mathcal O_Y=0$ for 
$1\leq i\leq n-2$. 
\end{lem}
\begin{proof}
We can assume that $n\geq 3$. 
By the assumption, 
$\Supp R^if_*\mathcal O_Y\subset \{x\}$ for 
$1\leq i\leq n-1$. 
We consider the following spectral 
sequence 
$$
E^{ij}_2=H^i_x(X, R^jf_*\mathcal O_Y)\Rightarrow 
H^{i+j}_F(Y, \mathcal O_Y), 
$$ 
where 
$F=f^{-1}(x)$. 
Then $H^0_x(X, R^jf_*\mathcal O_Y)
=E^{0j}_2\simeq E^{0j}_{\infty}=0$ for 
$j\leq n-2$ since 
$E^{ij}_2=0$ for $i>0$ and $j>0$, 
$E^{i0}_2=0$ for $i\leq n-1$, and 
$H^j_F(Y, \mathcal O_Y)=0$ for $j<n$. Therefore, 
$R^if_*\mathcal O_Y=0$ for $1\leq i\leq n-2$.  
\end{proof}

We point out the following fact explicitly for the reader's convenience. 
It is \cite[11.2 Theorem.~(11.2.5)]{ko-sing}. 

\begin{lem}
Let $f:Y\to X$ be a proper morphism, 
$x\in X$ a closed point, $F=f^{-1}(x)$ and 
$G$ a sheaf on $Y$. 
If $\Supp R^jf_*G\subset \{x\}$ for 
$1\leq i<k$ and $H^i_F(Y, G)=0$ for $i\leq k$, 
then $R^jf_*G\simeq H^{j+1}_x(X, f_*G)$ for $j=1, \cdots, k-1$. 
\end{lem}

The assumptions in Lemma \ref{w-lem2} are satisfied for $n$-dimensional isolated Cohen--Macaulay singularities. 
Therefore, we have the following corollary of Lemmas \ref{w-lem1} 
and \ref{w-lem2}. 
\begin{cor}
Let $x\in X$ be an $n$-dimensional normal 
isolated singularity. 
Then $x\in X$ is Cohen--Macaulay if and 
only if $R^if_*\mathcal O_Y=0$ for 
$1\leq i\leq n-2$, where $f:Y\to X$ is 
a resolution of singularities. 
\end{cor}

We note the following easy example. 

\begin{ex}
Let $V$ be a cone over a smooth plane cubic curve 
and let $\varphi:W\to V$ be the 
blow-up at the vertex. 
Then $W$ is smooth and 
$K_W=\varphi^*K_V-E$, where 
$E$ is an elliptic curve. 
In particular, $V$ is log canonical. 
Let $C$ be a smooth curve. We put 
$Y=W\times C$, $X=V\times C$, and 
$f=\varphi\times{\id_C}: Y\to X$, 
where $\id_C$ is the identity map of $C$. 
By the construction, $X$ is a log canonical threefold. 
We can check that $X$ is 
Cohen--Macaulay by 
Theorem \ref{alex-cri} or Proposition \ref{w-alex}. We note that 
$R^1f_*\mathcal O_Y\ne 0$ and that 
$R^1f_*\mathcal O_Y$ has no zero-dimensional 
associated components. 
Therefore, the Cohen--Macaulayness of $X$ does not 
necessarily imply the vanishing of $R^1f_*\mathcal O_Y$. 
\end{ex}

Let us go to cone singularities 
(cf.~\cite[3.8 Example]{ko-sing} and \cite[Exercises 70, 71]{ko-exe}). 

\begin{lem}[Projective normality]\index{projectively normal}
Let $X\subset \mathbb P^N$ be a normal projective irreducible 
variety and $V\subset \mathbb A^{N+1}$ the cone over 
$X$. 
Then $V$ is normal if and only if 
$H^0(\mathbb P^N, \mathcal O_{\mathbb P^N}(m))\to 
H^0(X, \mathcal O_X(m))$ is 
surjective for 
any $m\geq 0$. 
In this case, $X\subset \mathbb P^N$ is said to be 
{\em{projectively normal}}. 
\end{lem}
\begin{proof}
Without loss of generality, we can assume that 
$\dim X\geq 1$. 
Let $P\in V$ be the vertex of $V$. By the construction, 
we have $H^0_P(V, \mathcal O_V)=0$. We consider 
the following commutative 
diagram. 
$$
\xymatrix{
0 \rto &  H^0(\mathbb A^{N+1}, \mathcal O_{\mathbb A^{N+1}})
\dto \rto &  H^0(\mathbb A^{N+1}\setminus P, \mathcal O_{\mathbb A^{N+1}})
\dto
\rto & 0 &  \\
0 \rto & H^0(V, \mathcal O_V) \dto\rto & 
H^0(V\setminus P, \mathcal O_V) 
\rto &
H^1_P(V, \mathcal O_V)\rto &0\\
& 0 & &
}
$$ 
We note that 
$H^i(V, \mathcal O_V)=0$ for any $i>0$ 
since $V$ is affine. By the 
above commutative 
diagram, 
it is easy to see that 
the following conditions are equivalent. 
\begin{itemize}
\item[(a)] $V$ is normal. 
\item[(b)] the depth of $\mathcal O_{V, P}$ is $\geq 2$. 
\item[(c)] $H^1_P(V, \mathcal O_V)=0$. 
\item[(d)] $H^0(\mathbb A^{N+1}\setminus P, \mathcal O_{\mathbb A^{N+1}})\to 
H^0(V\setminus P, \mathcal O_V)$ is 
surjective. 
\end{itemize} 
The condition (d) is equivalent to 
the condition that 
$H^0(\mathbb P^N, \mathcal O_{\mathbb P^N}(m))\to 
H^0(X, \mathcal O_X(m))$ is 
surjective for any $m\geq 0$. 
We note that 
$$H^0(\mathbb A^{N+1}\setminus P, \mathcal O_{\mathbb A^{N+1}})
\simeq \bigoplus _{m\geq 0}
H^0(\mathbb P^N, \mathcal O_{\mathbb P^N}(m))$$ and 
$$H^0(V\setminus P, \mathcal O_V)
\simeq \bigoplus _{m\geq 0}H^0(X, \mathcal O_X(m)). 
$$
So, we finish the proof. 
\end{proof}

The next lemma is more or less well known to the experts. 

\begin{lem}\label{437lem} 
Let $X\subset \mathbb P^N$ be  a normal projective 
irreducible 
variety and $V\subset \mathbb A^{N+1}$ the cone over 
$X$. 
Assume that $X$ is projectively normal and 
that $X$ has only rational singularities. 
Then we have the following properties. 
\begin{itemize}
\item[$(1)$] $V$ is Cohen--Macaulay if and only if 
$H^i(X, \mathcal O_X(m))=0$ for any $0<i<\dim X$ and  
$m\geq 0$. 
\item[$(2)$] $V$ has only rational singularities 
if and only if 
$H^i(X, \mathcal O_X(m))=0$ for any $i>0$ and 
$m\geq 0$. 
\end{itemize}
\end{lem}
\begin{proof}
We put $n=\dim X$ and can assume $n\geq 1$. 
For (1), it is sufficient to 
prove that 
$H^i_P(V, \mathcal O_V)=0$ for 
$2\leq i \leq n$ if and only if 
$H^i(X, \mathcal O_X(m))=0$ for any $0<i<n$ and 
$m\geq 0$ since $V$ is normal, where 
$P\in V$ is the vertex of 
$V$. 
Let $f:W\to V$ be the 
blow-up at $P$ and $E\simeq X$ the exceptional divisor 
of $f$. 
We note that $W$ is the total space of 
$\mathcal O_X(-1)$ over $E\simeq X$ and 
that $W$ has only rational singularities. 
Since $V$ is affine, we obtain 
$H^i(V\setminus P, \mathcal O_V)\simeq 
H^{i+1}_P(V, \mathcal O_V)$ for any $i\geq 1$. 
Since $W$ has only rational singularities, 
we have $H^i_E(W, \mathcal O_W)=0$ 
for $i<n+1$ (cf.~Lemma \ref{lem-gr} and 
Remark \ref{lem-gr-ho}). Therefore, 
$$
H^i(V\setminus P, \mathcal O_V)\simeq 
H^i(W\setminus E, \mathcal O_W)\simeq H^i(W, \mathcal 
O_W)
$$ 
for $i\leq n-1$. 
Thus, 
$$
H^i_P(V, \mathcal O_V)\simeq H^{i-1}(V\setminus P, \mathcal O_V)
\simeq H^{i-1}(W, \mathcal O_W)\simeq 
\bigoplus _{m\geq 0}H^{i-1}(X, \mathcal O_X(m)) 
$$ 
for $2\leq i\leq n$. So, we 
obtain the desired equivalence. 

For (2), we consider the 
following commutative 
diagram. 
$$
\xymatrix{
&0\rto  &  H^n(V\setminus P, \mathcal O_V)
\dto^{\simeq} \rto &  H^{n+1}_P(V, \mathcal O_V)
\dto^{\alpha}
\rto & 0 &  \\
0 \rto & H^n(W, \mathcal O_W) \rto & 
H^n(W\setminus E, \mathcal O_W) 
\rto &
H^{n+1}_E(W, \mathcal O_W) &
}
$$ 
We note that $V$ is Cohen--Macaulay if and 
only if $R^if_*\mathcal O_W=0$ for 
$1\leq i\leq n-1$ (cf.~Lemmas \ref{w-lem1} and 
\ref{w-lem2}) 
since $W$ has only rational singularities. 
From now on, we assume that 
$V$ is Cohen--Macaulay. Then, $V$ has only rational singularities 
if and only if $R^nf_*\mathcal O_W=0$. By the 
same argument as in the proof of Theorem \ref{z-rational}, 
the kernel of $\alpha$ is $H^0_P(V, R^nf_*\mathcal O_W)$. 
Thus, $R^nf_*\mathcal O_W=0$ if and only if 
$H^n(W, \mathcal O_W)\simeq \bigoplus _{m\geq 0}
H^n(X, \mathcal O_X(m))=0$ by the 
above commutative diagram. 
So, we obtain the statement (2). 
\end{proof}

The following proposition is very useful when we construct 
some examples. 
We have already used it in this book. 

\begin{prop}\label{438p}
Let $X\subset \mathbb P^N$ be a normal projective 
irreducible variety and 
$V\subset \mathbb A^{N+1}$ the cone over $X$. 
Assume that 
$X$ is projectively normal. 
Let $\Delta$ be an effective $\mathbb R$-divisor on $X$and $B$ the 
cone over $\Delta$. 
Then, we have the following properties. 
\begin{itemize}
\item[$(1)$] $K_V+B$ is $\mathbb R$-Cartier if and only if 
$K_X+\Delta\sim _{\mathbb R}rH$ for some $r\in \mathbb R$, 
where $H\subset X$ is the hyperplane divisor on $X\subset \mathbb P^N$. 
\item[$(2)$] If $K_X+\Delta\sim _{\mathbb R}rH$, then 
$(V, B)$ is 
\begin{itemize}
\item[$(a)$] terminal if and only if $r<-1$ and $(X, \Delta)$ 
is terminal, 
\item[$(b)$] canonical if and only if $r\leq -1$ and 
$(X, \Delta)$ is canonical, 
\item[$(c)$] klt if and only if $r<0$ and $(X, \Delta)$ 
is klt, and 
\item[$(d)$] lc if and only if $r\leq 0$ and $(X, \Delta)$ is lc. 
\end{itemize}
\end{itemize}
\end{prop}
\begin{proof}
Let $f:W\to V$ be the blow-up at the 
vertex $P\in V$ and $E\simeq X$ the exceptional 
divisor of $f$. 
If $K_V+B$ is $\mathbb R$-Cartier, then 
$K_W+f^{-1}_*B\sim_{\mathbb R}f^*(K_V+B)+aE$ 
for some $a\in \mathbb R$. 
By restricting it to $E$, 
we obtain that 
$K_X+\Delta\sim _{\mathbb R}-(a+1)H$. 
On the other hand, if $K_X+\Delta\sim _{\mathbb R}rH$, then 
$K_W+f^{-1}_*B\sim_{\mathbb R}-(r+1)E$. Therefore, 
$K_V+B\sim _{\mathbb R}0$ on $V$. Thus, 
we have the statement (1). 
For (2), it is sufficient to 
note that 
$$
K_W+f^{-1}_*B=f^*(K_X+B)-(r+1)E 
$$ 
and that $W$ is the total space of $\mathcal O_X(-1)$ over 
$E\simeq X$. 
\end{proof}

\section{Toric Polyhedron}\label{to-sec}
 
In this section, we freely use the basic notation 
of the toric geometry. See, for example, \cite{fulton}. 

\begin{defn}
For a subset $\Phi$ of a fan $\Delta$, we say that 
$\Phi$ is {\em{star closed}} if $\sigma\in \Phi, 
\tau \in \Delta$ and $\sigma \prec \tau$ imply $\tau 
\in \Phi$. 
\end{defn}

\begin{defn}[Toric Polyhedron]\index{toric polyhedron} 
For a star closed subset $\Phi$ of a fan $\Delta$, 
we denote by $Y=Y(\Phi)$ 
the 
subscheme 
$\bigcup _{\sigma\in \Phi}V(\sigma)$ of $X=X(\Delta)$, 
and we call it the {\em{toric polyhedron associated to 
$\Phi$}}. 
\end{defn}

Let $X=X(\Delta)$ be a toric variety and let $D$ be the 
complement of the big torus. 
Then the following property is well known. 
\begin{prop}
The pair $(X, D)$ is log canonical 
and $K_X+D\sim 0$. 
Let $W$ be a closed subvariety of $X$. Then, 
$W$ is an lc center of $(X, D)$ if and only if 
$W=V(\sigma)$ for some $\sigma\in \Delta\setminus \{0\}$. 
\end{prop}

Therefore, we have the next theorem by adjunction 
(see Theorem \ref{adj-th} (i)). 

\begin{thm}\label{topo-th} 
Let $Y=Y(\Phi)$ be a toric polyhedron on $X=X(\Delta)$. 
Then, the log canonical 
pair $(X, D)$ induces a natural 
quasi-log structure on $[Y, 0]$. 
Note that $[Y, 0]$ has only qlc singularities. 
Let $W$ be a closed subvariety of $Y$. 
Then, $W$ is a qlc center of $[Y, 0]$ if and only if 
$W=V(\sigma)$ for some $\sigma \in \Phi$.  
\end{thm}

Thus, we can use the theory of quasi-log varieties 
to investigate toric varieties and toric polyhedra. 
For example, we have the following result as a special case of 
Theorem \ref{adj-th} (ii). 

\begin{cor}
We use the same notation as in {\em{Theorem \ref{topo-th}}}. 
Assume that $X$ is projective and $L$ is 
an ample Cartier divisor. 
Then $H^i(X, \mathcal I_Y\otimes \mathcal O_X(L))=0$ for 
any $i>0$, where $\mathcal I_Y$ is the defining 
ideal sheaf of $Y$ on $X$. 
In particular, $H^0(X, \mathcal O_X(L))\to 
H^0(Y, \mathcal O_Y(L))$ is 
surjective. 
\end{cor}

We can prove various vanishing theorems for 
toric varieties and toric polyhedra without 
appealing the results in Chapter \ref{chap2}. 
For the details, see \cite{fujino7}. 

\section{Non-lc ideal sheaves}\label{43-sec} 

In \cite{fujino10}, we introduced the notion of 
{\em{non-lc ideal sheaves}} and proved the 
restriction theorem. 
In this section, we quickly review the 
results in \cite{fujino10}. 

\begin{defn}[Non-lc ideal sheaf]\index{non-lc 
ideal sheaf}\label{nlc-def}
Let $X$ be a normal variety and let $\Delta$ be 
an $\mathbb R$-divisor on $X$ such that 
$K_X+\Delta$ is $\mathbb R$-Cartier. 
Let $f:Y\to X$ be a resolution with $K_Y+\Delta _Y=f^*(K_X+\Delta)$ 
such that 
$\Supp \Delta _Y$ is simple normal crossing. 
Then we put 
$$\mathcal J_{NLC}(X, \Delta)=f_*\mathcal O_Y(\ulcorner -(\Delta_Y^{<1})\urcorner-
\llcorner \Delta_Y^{>1}\lrcorner)=f_*\mathcal O_Y(-\llcorner \Delta_Y\lrcorner 
+\Delta^{=1}_Y)$$ and 
call it the {\em{non-lc ideal sheaf associated to $(X, \Delta)$}}. 
\end{defn}

In Definition \ref{nlc-def}, the ideal $\mathcal J_{NLC}(X, 
\Delta)$ coincides with 
$\mathcal I_{X_{-\infty}}$ 
for the quasi-log pair $[X, K_X+\Delta]$ 
when $\Delta$ is effective. 

\begin{rem}
In the same notation as in Definition \ref{nlc-def}, 
we put 
$$
\mathcal J(X, \Delta)=f_*\mathcal O_Y(-\llcorner \Delta_Y\lrcorner)
=f_*\mathcal O_Y(K_Y-\llcorner f^*(K_X+\Delta)\lrcorner). 
$$ 
It is nothing but the well-known 
{\em{multiplier ideal sheaf}}\index{multiplier ideal 
sheaf}. 
It is obvious that 
$\mathcal J(X, \Delta)\subseteq \mathcal J_{NLC}(X, \Delta)$. 
\end{rem}

The following theorem is the main theorem of \cite{fujino10}. 
We hope that it will have many applications. 
For the proof, see \cite{fujino10}. 

\begin{thm}[Restriction Theorem]\index{restriction 
theorem}\label{rest-th}
Let $X$ be a normal variety and let $S+B$ be an effective $\mathbb R$-divisor 
on $X$ such that $S$ is reduced and normal and that $S$ and $B$ have no 
common irreducible components. 
Assume that $K_X+S+B$ is $\mathbb R$-Cartier. 
We put 
$K_S+B_S=(K_X+S+B)|_S$. 
Then we obtain that 
$$
\mathcal J_{NLC}(S, B_S)= \mathcal J_{NLC}(X, S+B)|_S. 
$$
\end{thm}

Theorem \ref{rest-th} 
is a generalization of the inversion of 
adjunction on log canonicity in some sense. 

\begin{cor}[Inversion of Adjunction]\index{inversion of 
adjunction}\label{inv-cor}
We use the notation as in {\em{Theorem \ref{rest-th}}}. 
Then, $(S, B_S)$ is lc if and only if 
$(X, S+B)$ is lc around $S$. 
\end{cor}

In \cite{kawakita}, 
Kawakita proved the inversion of adjunction on log canonicity 
without assuming that $S$ is normal. 

\section{Effective Base Point Free Theorems}\label{44-sec} 

In this section, we state effective base point free theorems 
for log canonical pairs without proof. 
First, we state Angehrn--Siu type effective base point 
free 
theorems (see \cite{as} and \cite{ko-sing}). 
For the details of 
Theorems \ref{eff-th1} and \ref{eff-th2}, 
see \cite{fujino9}. 

\begin{thm}[Effective 
Freeness]\label{eff-th1} 
Let $(X, \Delta)$ be a projective {\em{log canonical}} pair 
such that $\Delta$ is an effective 
$\mathbb Q$-divisor and let $M$ be a 
line bundle on $X$. 
Assume that $M\equiv K_X+\Delta+N$, where $N$ is an ample 
$\mathbb Q$-divisor 
on $X$. 
Let $x\in X$ be a closed point and assume that there are 
positive numbers $c(k)$ with 
the following properties: 
\begin{enumerate}
\item[$(1)$] If $x\in Z\subset X$ is an irreducible 
$($positive dimensional$)$ 
subvariety, then 
$$
(N^{\dim Z}\cdot Z)>c(\dim Z)^{\dim Z}. 
$$ 
\item[$(2)$] The numbers $c(k)$ satisfy the inequality 
$$
\sum _{k=1}^{\dim X} \frac{k}{c(k)}\leq 1. 
$$
\end{enumerate} 
Then $M$ has a global section not vanishing at $x$. 
\end{thm}

\begin{thm}[Effective Point 
Separation]\label{eff-th2} 
Let $(X, \Delta)$ be a projective 
{\em{log canonical}} pair such 
that $\Delta$ is an effective 
$\mathbb Q$-divisor and let $M$ be a 
line bundle on $X$. 
Assume that $M\equiv K_X+\Delta+N$, where $N$ is an ample 
$\mathbb Q$-divisor 
on $X$. 
Let $x_1, x_2\in X$ be closed points and assume that there are 
positive numbers $c(k)$ with 
the following properties: 
\begin{enumerate}
\item[$(1)$] If $Z\subset X$ is an 
irreducible $($positive dimensional$)$ 
subvariety which contains $x_1$ or $x_2$, then 
$$
(N^{\dim Z}\cdot Z)>c(\dim Z)^{\dim Z}. 
$$ 
\item[$(2)$] The numbers $c(k)$ satisfy the inequality 
$$
\sum _{k=1}^{\dim X}
\sqrt[\leftroot{0}\uproot{0} k]{2}\frac{k}{c(k)}
\leq 1. 
$$
\end{enumerate} 
Then global sections of $M$ separates $x_1$ and $x_2$. 
\end{thm}

The key points of the proofs of Theorems \ref{eff-th1} and 
\ref{eff-th2} are the vanishing theorem (see Theorem 
\ref{adj-th} (ii)) and the inversion of adjunction on 
log canonicity (see Corollary \ref{inv-cor} 
and \cite{kawakita}). 

The final theorem in this book is a generalization of 
Koll\'ar's effective base point freeness (see \cite{kol-eff}). 
The proof is essentially the same as Koll\'ar's once 
we adopt Theorem \ref{adj-th} (ii) and Theorem 
\ref{bpf-rdlc-th}. For the details, see \cite{fujino8}. 

\begin{thm}\label{eff-th3} 
Let $(X, \Delta)$ be a {\em{log canonical}} 
pair with $\dim 
X=n$ and let $\pi:X\to V$ be a {\em{projective}} surjective morphism. 
Note that $\Delta$ is an effective $\mathbb Q$-divisor on $X$. 
Let $L$ be a $\pi$-nef Cartier divisor on $X$. Assume that 
$aL-(K_X+\Delta)$ is $\pi$-nef and $\pi$-log 
big for some $a\geq 0$. 
Then there exists a positive integer $m=m(n, a)$,  
which only depends on $n$ and $a$, 
such that 
$\mathcal O_X(mL)$ is $\pi$-generated. 
\end{thm} 

\chapter{Appendix}\label{chap5}

In this final chapter, we will explain some sample 
computations of flips. 
We use the toric geometry to construct explicit examples here. 

\section{Francia's flip revisited} 

We give an example of Francia's flip on a projective 
toric variety explicitly. 
It is a monumental example  (see \cite{francia}). 
So, we contain it here. 
Our description looks slightly different from 
the usual one because we use the toric geometry. 

\begin{ex}\label{fran11}
We fix a lattice $N\simeq \mathbb Z^3$ and 
consider the lattice points 
$e_1=(1, 0, 0)$, $e_2=(0, 1, 0)$, 
$e_3=(0, 0, 1)$, $e_4=(1, 1, -2)$, and 
$e_5=(-1, -1, 1)$. First, 
we consider the complete fan $\Delta_1$ spanned 
by $e_1, e_2, e_4$, and $e_5$. 
Since $e_1+e_2+e_4+2e_5=0$, $X_1=X(\Delta_1)$ 
is $\mathbb P(1,1,1,2)$. Next, we 
take the blow-up $f:X_2= X(\Delta_2)\to X_1$ 
along the ray $e_3=(0, 0, 1)$. Then 
$X_2$ is a projective $\mathbb Q$-factorial 
toric variety with only one $\frac{1}{2}(1, 1, 1)$-singular 
point. 
Since $\rho (X_2)=2$, we have one more 
contraction morphism $\varphi :X_2\to X_3=X(\Delta_3)$. This 
contraction morphism $\varphi$ corresponds to the 
removal of the wall $\langle e_1, e_2\rangle$ from $\Delta_2$. 
We can easily check that $\varphi$ is a flipping contraction. 
By adding the wall $\langle e_3, e_4\rangle$ to $\Delta_3$, 
we obtain a flipping diagram. 
$$
\begin{matrix}
X_2 & \dashrightarrow & \ X_4 \\
{\ \ \ \ \ \searrow} & \ &  {\swarrow}\ \ \ \ \\
 \ & X_3 &  
\end{matrix}
$$
It is an example of Francia's flip. 
We can easily check that 
$X_4\simeq \mathbb P_{\mathbb P^1}(\mathcal O_{\mathbb P^1}
\oplus \mathcal O_{\mathbb P^1}(1)\oplus \mathcal O_{\mathbb P^1}(2))$ and 
that the flipped curve $C\simeq \mathbb P^1$ is the 
section of $\pi:\mathbb P_{\mathbb P^1}(\mathcal 
O_{\mathbb P^1}\oplus \mathcal O_{\mathbb P^1}(1)\oplus 
\mathcal O_{\mathbb P^1}(2))\to \mathbb P^1$ defined 
by the projection 
$\mathcal O_{\mathbb P^1}\oplus \mathcal O_{\mathbb P^1}(1)\oplus \mathcal O_{\mathbb P^1}(2)\to \mathcal O_{\mathbb P^1}\to 0$. 
\end{ex}

By taking double covers, we have an interesting example 
(cf.~\cite{francia}). 

\begin{ex}
We use the same notation as in Example \ref{fran11}. 
Let $g:X_5\to X_2$ be the blow-up along the 
ray $e_6=(1, 1, -1)$. 
Then $X_5$ is a smooth projective 
toric variety. 
Let $\mathcal O_{X_4}(1)$ be the tautological 
line bundle of the $\mathbb P^2$-bundle 
$\pi:X_4=\mathbb P_{\mathbb P^1}(\mathcal O_{\mathbb P^1}
\oplus \mathcal O_{\mathbb P^1}(1)\oplus \mathcal O_{\mathbb P^1}(2))\to \mathbb P^1$. 
It is easy to see that $\mathcal O_{X_4}(1)$ is nef and 
$\mathcal O_{X_4}(1)\cdot C=0$. 
Therefore, there exists a line bundle 
$\mathcal L$ on $X_3$ such that 
$\mathcal O_{X_4}(1)\simeq \psi^*\mathcal L$, 
where $\psi:X_4\to X_3$. 
We take a general 
member $D\in |\mathcal L^{\otimes 8}|$. 
We note that $|\mathcal L|$ is free since 
$\mathcal L$ is nef. 
We take a double cover $X\to X_4$ (resp.~$Y\to X_5$) 
ramifying along $\Supp \psi ^{-1}D$ (resp.~$\Supp (\varphi
\circ g)^{-1}D$). 
Then $X$ is a smooth projective variety 
such that $K_X$ is ample. 
It is obvious that 
$Y$ is a smooth projective 
variety and is birational to $X$. 
So, $X$ is the unique minimal model of $Y$. 
We need flips 
to obtain the minimal model $X$ from $Y$ by running 
the 
MMP. 
\end{ex}

\section{A sample computation of a log flip}

Here, we treat an example of threefold log flips. 
In general, it is difficult to know 
what happens around the flipping curve. 
Therefore, the following nontrivial example 
is valuable because we can see the behavior 
of the flip explicitly.  
It helps us understand the proof of the special termination 
in \cite{special}.  
\begin{ex}
We fix a lattice $N=\mathbb Z^3$. 
We put $e_1=(1,0,0)$, $e_2=(-1,2,0)$, 
$e_3=(0,0,1)$, and $e_4=(-1,3,-3)$. 
We consider the fan $$\Delta=\{\langle e_1, e_3, e_4\rangle, 
\langle e_2, e_3, e_4\rangle, {\text{and their faces}}\}.$$ 
We put $X=X(\Delta)$, that is, $X$ is the toric variety 
associated to the fan $\Delta$. 
We define torus invariant prime divisors 
$D_i=V(e_i)$ for $1\leq i\leq 4$. We can easily 
check the following claim. 

\begin{claim}
The pair $(X, D_1+D_3)$ is a $\mathbb Q$-factorial 
dlt pair. 
\end{claim}
We put $C=V(\langle e_3, e_4\rangle)\simeq \mathbb P^1$, which 
is a torus invariant irreducible curve on $X$. 
Since $\langle e_2, e_3, e_4\rangle$ is a non-singular 
cone, the intersection number $D_2\cdot C=1$. 
Therefore, $C\cdot D_4=-\frac {2}{3}$ and 
$-(K_X+D_1+D_3)\cdot C=\frac{1}{3}$. 
We note the linear relation $e_1+3e_2-6e_3-2e_4=0$. 
We put 
$Y=X(\langle e_1, e_2, e_3, e_4\rangle)$, that is, 
$Y$ is the affine toric variety associated to the 
cone $\langle e_1, e_2, e_3, e_4\rangle$. Then 
we have the next claim. 
\begin{claim}
The birational 
map $f:X\to Y$ is an elementary pl flipping contraction with respect 
to $K_X+D_1+D_3$. 
\end{claim}
For the definition of 
pl flipping contractions, see 
\cite[Definition 4.3.1]{special}. 
We note the intersection numbers $C\cdot D_1=\frac{1}{3}$ and 
$D_3\cdot C=-2$. 
Let $\varphi:X\dashrightarrow X^+$ be the flip 
of $f$. We note that 
the flip $\varphi$ is an isomorphism 
around any generic points of lc centers 
of $(X, D_1+D_3)$. 
We restrict the flipping diagram 
$$
\begin{matrix}
X & \dashrightarrow & \ X^+ \\
{\ \ \ \ \ \searrow} & \ &  {\swarrow}\ \ \ \ \\
 \ & Y &  
\end{matrix}
$$
to $D_3$. 
Then we have the following diagram. 
$$
\begin{matrix}
D_3 & \dashrightarrow & \ D^+_3 \\
{\ \ \ \ \ \searrow} & \ &  {\swarrow}\ \ \ \ \\
 \ & f(D_3) &  
\end{matrix}
$$
It is not difficult to see that 
$D^+_3\to f(D_3)$ is an isomorphism. 
We put $(K_X+D_1+D_3)|_{D_3}=K_{D_3}+B$. 
Then $f:D_3\to f(D_3)$ is an extremal divisorial 
contraction with respect to $K_{D_3}+B$. 
We note that $B=D_1|_{D_3}$. 
\begin{claim}
The birational 
morphism 
$f:D_3\to f(D_3)$ contracts $E\simeq \mathbb P^1$ to a point 
$Q$ on $D^+_3\simeq f(D_3)$ 
and $Q$ is a $\frac{1}{2}(1,1)$-singular point 
on $D^+_3\simeq 
f(D_3)$. 
The surface 
$D_3$ has a $\frac{1}{3}(1,1)$-singular point $P$, which 
is the intersection of $E$ and $B$. We 
also have the adjunction formula 
$(K_{D_3}+B)|_B=K_B+\frac{2}{3}P$.  
\end{claim}
Let $D^+_i$ be the torus 
invariant 
prime divisor 
$V(e_i)$ on $X^+$ for all $i$ 
and $B^+$ the strict transform of $B$ on $D^+_3$. 
\begin{claim} 
We have 
$$(K_{X^+}+D^+_1+D^+_3)|_{D^+_3}=K_{D^+_3}+B^+$$ and 
$$(K_{D^+_3}+B^+)|_{B^+}=K_{B^+}+\frac{1}{2}Q.$$ 
\end{claim}
We note that $f^+:D^+_3\to 
f(D_3)$ is an isomorphism. 
In particular, 
$$
\begin{matrix}
D_3 & \dashrightarrow & \ D^+_3 \\
{\ \ \ \ \ \searrow} & \ &  {\swarrow}\ \ \ \ \\
 \ & f(D_3) &  
\end{matrix}
$$
is of type (DS) in the sense of \cite[Definition 4.2.6]{special}. 
Moreover, $f:B\to B^+$ is an isomorphism 
but $f:(B, \frac{2}{3}P)\to (B^+, \frac{1}{2}Q)$ is not an 
isomorphism of pairs (see \cite[Definition 4.2.5]{special}). 
We note that $B$ is an lc center 
of $(X, D_1+D_3)$. 
So, we need \cite[Lemma 4.2.15]{special}. 
Next, we restrict the flipping diagram to $D_1$. 
Then we obtain the diagram. 
$$
\begin{matrix}
D_1 & \dashrightarrow & \ D^+_1 \\
{\ \ \ \ \ \searrow} & \ &  {\swarrow}\ \ \ \ \\
 \ & f(D_1) &  
\end{matrix}
$$
In this case, $f:D_1\to f(D_1)$ is an 
isomorphism. 
\begin{claim}
The surfaces 
$D_1$ and $D^+_1$ are smooth. 
\end{claim} 
It can be directly checked. 
Moreover, we obtain the following adjunction formulas. 
\begin{claim}
We have 
$$(K_X+D_1+D_3)|_{D_1}
=K_{D_1}+B+\frac{2}{3}B',$$ where 
$B$ $($resp.~$B'$$)$ comes from 
the intersection of 
$D_1$ and $D_3$ $($resp.~$D_4$$)$. 
We also obtain 
$$(K_{X^+}+D^+_1+D^+_3)|_{D^+_1}
=K_{D^+_1}+B^++\frac{2}{3}{B'}^++\frac{1}{2}F,$$ 
where $B^+$ $($resp.~${B'}^+$$)$ is the strict transform 
of $B$ $($resp.~$B'$$)$ and $F$ is the exceptional 
curve of $f^+:D^+_1\to f(D_1)$. 
\end{claim}
\begin{claim}The birational morphism 
$f^+:D^+_1\to f(D_1)\simeq D_1$ is the 
blow-up at $P=B\cap B'$. 
\end{claim}
We can easily check that 
$$K_{D^+_1}+B^++\frac{2}{3}{B'}^++\frac{1}{2}F
={f^+}^*(K_{D_1}+B+\frac{2}{3}B')-\frac{1}{6}F.$$ 
It is obvious that $K_{D^+_1}+B^++\frac{2}{3}
{B'}^++\frac{1}{2}F$ is $f^+$-ample. 
Note that $F$ comes from the intersection of 
$D^+_1$ and $D^+_2$. 
Note that the diagram 
$$
\begin{matrix}
D_1 & \dashrightarrow & \ D^+_1 \\
{\ \ \ \ \ \searrow} & \ &  {\swarrow}\ \ \ \ \\
 \ & f(D_1) &  
\end{matrix}
$$
is of type (SD) in the sense of 
\cite[Definition 4.2.6]{special}. 
\end{ex}

\section{A non-$\mathbb Q$-factorial flip}
I apologize for the mistake in \cite[Example 4.4.2]{fujino0}. 
We give an example of a three-dimensional 
non-$\mathbb Q$-factorial canonical Gorenstein toric 
flip. 
See also \cite{fstu}. 
We think that it is difficult to construct 
such examples without using the toric geometry. 

\begin{ex}[Non-$\mathbb Q$-factorial canonical Gorenstein 
toric flip]\label{final-ex}
We fix a lattice $N=\mathbb Z^3$. 
Let $n$ be a positive integer with $n\geq 2$. 
We take lattice points 
$e_0=(0, -1, 0)$, 
$e_i=(n+1-i, \sum _{k=n+1-i}^{n-1}k, 1)$ for 
$1\leq i\leq n+1$, and 
$e_{n+2}=(-1, 0, 1)$. 
We consider the following fans. 
\begin{eqnarray*}
\Delta_X&=&\{\langle e_0, e_1, e_{n+2}\rangle, 
\langle e_1, e_2, \cdots, e_{n+1}, e_{n+2}\rangle, 
\text{and their faces}\},\\ 
\Delta_W&=&\{\langle e_0, e_1, \cdots, e_{n+1}, e_{n+2}\rangle, 
\text{and its faces}\}, {\text{and}}\\
\Delta_{X^+}&=&\{\langle e_0, e_i, e_{i+1}\rangle, {\text{for}}\  
i=1, \cdots, n+1, 
\text{and their faces}\}. 
\end{eqnarray*}
We define 
$X=X(\Delta_X), X^+=X(\Delta_{X^+})$, and $W=X(\Delta_W)$.
Then we have a diagram of toric varieties.  
$$
\begin{matrix}
X & \dashrightarrow & \ X^+ \\
{\ \ \ \ \ \searrow} & \ &  {\swarrow}\ \ \ \ \\
 \ & W &  
\end{matrix}
$$
We can easily check the following properties. 
\begin{enumerate}
\item[(i)] $X$ has only canonical Gorenstein singularities. 
\item[(ii)] $X$ is not $\mathbb Q$-factorial. 
\item[(iii)] $X^+$ is smooth. 
\item[(iv)] $-K_X$ is $\varphi$-ample and 
$K_{X^+}$ is $\varphi^+$-ample. 
\item[(v)] $\varphi:X\to W$ and $\varphi^+: X^+\to W$ are small projective 
toric morphisms. 
\item[(vi)] $\rho (X/W)=1$ and $\rho (X^+/W)=n$. 
\end{enumerate} 
Therefore, the above diagram is a desired 
flipping diagram. 
We note that $e_i+e_{i+2}=2e_{i+1}
+e_0$ for $i=1, \cdots, n-1$ and 
$e_n+e_{n+2}=2e_{n+1}+\frac{n(n-1)}{2}e_0$. 
We recommend the reader to draw pictures of 
$\Delta_X$ and $\Delta_{X^+}$. 
\end{ex}

By this example, we see that a flip sometimes 
increases the Picard number when the variety is not 
$\mathbb Q$-factorial.


\printindex

\begin{thebibliography}{CHKLM}

\bibitem[Al]{alex}
V.~Alexeev, Limits of stable pairs, 
Pure and Applied Math Quaterly, {\textbf{4}} 
(2008), no. 3, 1--15. 

\bibitem[AHK]{ahk}
V.~Alexeev, C.~Hacon, and 
Y.~Kawamata, 
Termination of (many) $4$-dimensional log flips, 
Invent. Math. {\textbf{168}} (2007), no. 2, 433--448.

\bibitem[Am1]{ambro}F.~Ambro, 
Quasi-log varieties, 
Tr. Mat. Inst. Steklova {\textbf{240}} (2003), 
Biratsion. Geom. Linein. Sist. Konechno 
Porozhdennye Algebry, 220--239; translation 
in Proc. Steklov Inst. Math. 2003, no. 1 (240), 214--233.  

\bibitem[Am2]{ambro2}F.~Ambro, 
Basic properties of log canonical centers, preprint 2006. 

\bibitem[AS]{as} 
U.~Angehrn, Y.-T.~Siu, 
Effective freeness and point separation for adjoint bundles, 
Invent. Math. {\textbf{122}} (1995), no. 2, 291--308. 

\bibitem[B]{birkar}C.~Birkar, On existence of log 
minimal models, preprint 2007. 

\bibitem[BCHM]{bchm} 
C.~Birkar, P.~Cascini, C.~Hacon, and 
J.~McKernan, 
Existence of minimal models for varieties of log general type, 
preprint 2006. 

\bibitem[Book]{corti} 
{\em{Flips for 3-folds and 4-folds}}, 
Edited by Alessio Corti. Oxford Lecture 
Series in Mathematics and its 
Applications, {\textbf{35}}. Oxford 
University Press, Oxford, 2007. 

\bibitem[CHKLM]{chklm} 
A.~Corti, 
P.~Hacking, J.~Koll\'ar, R.~Lazarsfeld, 
and M.~Musta\c{t}\v{a}, 
Lectures on flips and minimal models, preprint 2007. 

\bibitem[D1]{deligne0} 
P.~Deligne, 
Th\'eorie de Hodge. II, 
Inst. Hautes \'Etudes Sci. Publ. Math. No. {\textbf{40}} (1971), 5--57. 

\bibitem[D2]{deligne} 
P.~Deligne, 
Th\'eorie de Hodge. III, 
Inst. Hautes \'Etudes Sci. Publ. Math. No. {\textbf{44}} (1974), 5--77.

\bibitem[Dr]{druel} 
S.~Druel, 
Existence de mod\`eles minimaux 
pour les vari\'et\'es de type 
g\'en\'eral, 
S\'eminaire Bourbaki, 
Vol. 2007/08. Exp. No. 982. 

\bibitem[E1]{elzein} F.~Elzein, 
Mixed Hodge structures, Trans. Amer. Math. Soc. 
{\textbf{275}} (1983), 
no. 1, 71--106. 

\bibitem[E2]{elzein2} F.~Elzein, 
{\em{Introduction \'a la th\'eorie de Hodge mixte}}, 
Actualit\'es Math\'ematiques, Hermann, Paris, 1991. 

\bibitem[EV]{ev} H.~Esnault, E.~Viehweg, 
{\em{Lectures on vanishing theorems}}, DMV Seminar, {\textbf{20}}. 
Birkh\"auser Verlag, Basel, 1992. 

\bibitem[Fr]{francia} 
P.~Francia, 
Some remarks on minimal models. I, 
Compositio Math. {\textbf{40}} (1980), no. 3, 301--313.

\bibitem[F1]{abun} 
O.~Fujino, Abundance theorem for semi log canonical 
threefolds, Duke Math. J. {\textbf{102}} (2000), no. 3, 
513--532. 

\bibitem[F2]{re-fu} 
O.~Fujino, Base point free theorem of Reid--Fukuda type, 
J. Math. Sci. Univ. Tokyo {\textbf{7}} (2000), no. 1, 1--5. 

\bibitem[F3]{fuji-tor} 
O.~Fujino, 
Notes on toric varieties from Mori theoretic 
viewpoint, Tohoku Math. J. (2) {\textbf{55}} (2003), no. 4, 551--564.

\bibitem[F4]{fujino-high}O.~Fujino, 
Higher direct images of log canonical divisors, 
J. Differential Geom. {\textbf{66}} (2004), no. 3, 453--479. 

\bibitem[F5]{fuji-to2} 
O.~Fujino, 
Equivariant completions of toric contraction morphisms, 
With an appendix by Fujino and Hiroshi Sato, 
Tohoku Math. J. (2) {\textbf{58}} (2006), no. 3, 303--321.

\bibitem[F6]{fujino-km} 
O.~Fujino, On the Kleiman--Mori cone, 
Proc. Japan Acad. Ser. A Math. Sci. {\textbf{81}} 
(2005), no. 5, 80--84.

\bibitem[F7]{fujino0}O.~Fujino, 
What is log terminal?, in {\em{Flips for 
$3$-folds and $4$-folds}} (Alessio Corti, ed.), 
49--62, Oxford 
University Press, 2007.  

\bibitem[F8]{special} 
O.~Fujino, 
Special termination and reduction to 
pl flips, in {\em{Flips 
for $3$-folds and 
$4$-folds}} (Alessio Corti, ed.), 63--75, Oxford University 
Press, 2007. 

\bibitem[F9]{fujino} 
O.~Fujino, 
Vanishing and injectivity theorems for LMMP, preprint 2007. 

\bibitem[F10]{fuji-note} 
O.~Fujino, Notes on the log minimal model 
program, preprint 2007. 

\bibitem[F11]{fujino6} O.~Fujino, 
Base point free theorems---saturation, b-divisors, and 
canonical bundle formula---, preprint 2007. 

\bibitem[F12]{fujino7} 
O.~Fujino, 
Vanishing theorems for toric polyhedra, 
RIMS K\^oky\^uroku Bessatsu {\textbf{B9}} (2008), 81--95. 

\bibitem[F13]{fujino8} 
O.~Fujino, 
Effective base point free theorem for log canonical 
pairs---Koll\'ar type theorem---, 
preprint 2007.
 
\bibitem[F14]{fujino9} 
O.~Fujino, 
Effective base point free theorem for log canonical 
pairs II---Angehrn--Siu type theorems---, 
preprint 2007. 

\bibitem[F15]{fujino10} 
O.~Fujino, Theory of non-lc ideal 
sheaves---basic properties---, 
preprint 2007. 

\bibitem[F16]{fuj-lec} 
O.~Fujino, 
Lectures on the log minimal model program, 
preprint 2007. 

\bibitem[F17]{fuji-ka} 
O.~Fujino, 
On Kawamata's theorem, preprint 2007. 

\bibitem[F18]{fuji-finite} 
O.~Fujino, 
Finite generation of the log canonical ring 
in dimension four, preprint 2008. 

\bibitem[F19]{fuji-ronsetsu} 
O.~Fujino, 
New developments in the theory of 
minimal models (Japanese), to appear 
in 
S$\bar{\rm{u}}$gaku. 

\bibitem[F20]{fuji-a} 
O.~Fujino, On injectivity, vanishing and torsion-free theorems, 
preprint 2008. 

\bibitem[F21]{fuji-b}
O.~Fujino, Non-vanishing theorem for log canonical pairs, 
preprint 2008. 

\bibitem[FM]{fuji-mori} 
O.~Fujino, S.~Mori, 
A canonical bundle formula, 
J. Differential Geom. {\textbf{56}} (2000), no. 1, 167--188. 

\bibitem[FP]{fuji-pa} 
O.~Fujino, S.~Payne, 
Smooth complete toric threefolds with no nontrivial 
nef line bundles, 
Proc. Japan Acad. Ser. A 
Math. Sci. {\textbf{81}} (2005), no. 10, 
174--179.

\bibitem[FS]{fuji-sato} 
O.~Fujino, H.~Sato, 
Introduction to the toric Mori theory, 
Michigan Math. J. {\textbf{52}} (2004), no. 3, 649--665.

\bibitem[FSTU]{fstu}
O.~Fujino, 
H.~Sato, Y.~Takano, H.~Uehara, 
Three-dimensional terminal 
toric flips, to appear. 

\bibitem[Fk1]{fukuda} 
S.~Fukuda, 
A base point free theorem for log canonical surfaces, 
Osaka J. Math. {\textbf{36}} (1999), no. 2, 337--341.

\bibitem[Fk2]{fukuda2} 
S.~Fukuda, 
On numerically effective log canonical divisors, 
Int. J. Math. Math. Sci. {\textbf{30}} 
(2002), no. 9, 521--531.

\bibitem[Fl]{fulton} W.~Fulton, {\em{Introduction 
to toric varieties}}, 
Annals of Mathematics Studies, {\textbf{131}}. 
The William H. Roever Lectures in Geometry. Princeton University Press, Princeton, NJ, 1993. 

\bibitem[G]{hartshorne-local} 
A.~Grothendieck, {\em{Local cohomology}}, 
Lecture 
Notes in Mathematics, No. {\textbf{41}}. 
Springer-Verlag, Berlin-New York, 1967. 

\bibitem[HM]{hm} 
C.~Hacon, J.~McKernan, 
Extension theorems and the existence of flips, 
in {\em{Flips for 3-folds and 4-folds}} 
(Alessio Corti, ed.), 76--110, Oxford University Press, 2007. 

\bibitem[H1]{residue}R.~Hartshorne, 
{\em{Residues and duality}}, 
Lecture notes of a seminar on the 
work of A.~Grothendieck, given at 
Harvard 1963/64. With an appendix 
by P.~Deligne. Lecture Notes in 
Mathematics, No. {\textbf{20}}. 
Springer-Verlag, Berlin-New York, 1966.

\bibitem[H2]{hartshorne-ag} 
R.~Hartshorne,  
{\em{Algebraic geometry}}, 
Graduate Texts in Mathematics, 
No. {\textbf{52}}. Springer-Verlag, 
New York-Heidelberg, 1977. 

\bibitem[Kw]{kawakita} 
M.~Kawakita, 
Inversion of adjunction 
on log canonicity, 
Invent. Math. {\textbf{167}} (2007), no. 1, 129--133.

\bibitem[Ka1]{kawamata1} 
Y.~Kawamata, 
Pluricanonical systems on minimal algebraic varieties, 
Invent. Math. {\textbf{79}} (1985), no. 3, 567--588.  

\bibitem[Ka2]{kawamata} 
Y.~Kawamata, 
On the length of an extremal rational curve, 
Invent. Math. {\textbf{105}} (1991), no. 3, 609--611.

\bibitem[KMM]{kmm} 
Y.~Kawamata, K.~Matsuda, and K.~Matsuki, Introduction to the Minimal 
Model Problem, in {\em Algebraic Geometry, Sendai 1985,} Advanced Studies 
in Pure Math. {\textbf {10}}, (1987) Kinokuniya and North-Holland, 283--360. 

\bibitem[Kl]{kleiman} 
S.~L.~Kleiman, 
Toward a numerical theory of ampleness, 
Ann. of Math. (2) {\textbf{84}} (1966), 293--344. 

\bibitem[Ko1]{kollar-high} 
J.~Koll\'ar, Higher direct images of dualizing 
sheaves. I,  Ann. of Math. (2) {\textbf{123}} 
(1986), no. 1, 11--42.

\bibitem[Ko2]{kol-eff} 
J.~Koll\'ar, 
Effective base point freeness, 
Math. Ann. {\textbf{296}} (1993), no. 4, 595--605.
 
\bibitem[Ko3]{kollar} J.~Koll\'ar, 
{\em{Shafarevich maps and automorphic forms}}, 
M. B. Porter Lectures. Princeton University Press, Princeton, NJ, 
1995. 

\bibitem[Ko4]{ko-sing} 
J.~Koll\'ar, Singularities of pairs, 
Algebraic geometry---Santa Cruz 1995, 
221--287, Proc. Sympos. Pure Math., {\textbf{62}}, 
Part 1, Amer. Math. Soc., Providence, RI, 1997.

\bibitem[Ko5]{ko-exe} 
J.~Koll\'ar, 
Exercises in the birational 
geometry of algebraic varieties, preprint 2008. 

\bibitem[FA]{FA} 
J.~Koll\'ar et al., \emph{Flips and Abundance for algebraic threefolds}, 
Ast\'erisque \textbf{211}, (1992).  

\bibitem[KM]{km} 
J.~Koll\'ar, S.~Mori, {\em Birational geometry of 
algebraic varieties,} Cambridge Tracts in Mathematics, Vol. {\textbf {134}}, 
1998. 

\bibitem[Kv1]{kovacs1} S.~Kov\'acs, 
Rational, log canonical, Du 
Bois singularities: on the conjectures of Koll\'ar 
and Steenbrink, 
Compositio Math. {\textbf{118}} (1999), no. 2, 123--133.

\bibitem[Kv2]{kovacs} 
S.~Kov\'acs, 
Rational, log canonical, Du Bois 
singularities. II. Kodaira vanishing 
and small deformations, Compositio Math. {\textbf{121}} 
(2000), no. 3, 297--304. 

\bibitem[Kv3]{kovacs2} S.~Kov\'acs, 
A characterization of rational singularities, 
Duke Math. J. {\textbf{102}} (2000), no. 2, 187--191. 

\bibitem[KSS]{kss} 
S.~Kov\'acs, K.~Schwede, K.~Smith, 
Cohen--Macaulay semi-log 
canonical singularities are Du Bois, 
preprint 2008. 

\bibitem[M]{ma-bon} 
K.~Matsuki, {\em{Introduction to the 
Mori program}}, Universitext, 
Springer-Verlag, New York, 2002.

\bibitem[N]{nakayama2} 
N.~Nakayama, 
{\em{Zariski-decomposition and abundance}}, 
MSJ Memoirs, {\textbf{14}}. Mathematical 
Society of Japan, Tokyo, 2004.

\bibitem[PS]{ps}
C.~A.~M.~Peters, J.~H.~M.~Steenbrink, 
{\em{Mixed Hodge structures}}, 
Ergebnisse der Mathematik und ihrer Grenzgebiete. 3. Folge. A 
Series of Modern Surveys in Mathematics 
[Results in Mathematics and Related Areas. 3rd 
Series. A Series of Modern Surveys in Mathematics], 
{\textbf{52}}. Springer-Verlag, Berlin, 2008.

\bibitem[R]{reid-toric} 
M.~Reid, 
Decomposition of toric morphisms, 
Arithmetic and geometry, Vol. II, 395--418, 
Progr. Math., {\textbf{36}}, 
Birkh\"auser Boston, Boston, MA, 1983.

\bibitem[Sh1]{sho-pre}
V.~V.~Shokurov, 
Prelimiting flips, 
Tr. Mat. Inst. Steklova {\textbf{240}} (2003), 
Biratsion. Geom. Linein. Sist. Konechno 
Porozhdennye Algebry, 82--219; translation 
in Proc. Steklov Inst. Math. 2003, no. 1 (240), 75--213. 
 
\bibitem[Sh2]{sho-7}V.~V.~Shokurov, 
Letters of a Bi-Rationalist:~VII.~Ordered termination, 
preprint 2006. 

\bibitem[So]{som} 
A.~J.~Sommese, 
On the adjunction theoretic structure of 
projective varieties, 
Complex analysis and algebraic geometry 
(G\"ottingen, 1985),  175--213, 
Lecture Notes in Math., {\textbf{1194}}, 
Springer, Berlin, 1986.

\bibitem[St]{steenbrink}
J.~H.~M.~Steenbrink, 
Mixed Hodge structure on the vanishing cohomology, 
Real and complex singularities 
(Proc. Ninth Nordic Summer School/NAVF Sympos. Math., Oslo, 
1976), pp. 525--563. Sijthoff and Noordhoff, Alphen aan den Rijn, 1977.

\bibitem[Sz]{szabo} E.~Szab\'o, 
Divisorial log terminal singularities, 
J. Math. Sci. Univ. Tokyo {\textbf{1}} (1994), no. 3, 631--639. 

\end{thebibliography}
\end{document}